\documentclass[11pt,reqno]{amsart}
\usepackage{bm,amsmath,amsthm,amssymb,mathtools,verbatim,amsfonts,tikz-cd,mathrsfs,diagbox,makecell}
\usepackage{enumitem}
\usepackage{float}
\usepackage{xfrac}

\usepackage[hidelinks,colorlinks=true,linkcolor=blue, citecolor=black,linktocpage=true]{hyperref}
\usepackage{graphicx}
\usepackage[alphabetic]{amsrefs}
\usepackage{caption}
\usepackage{morefloats}
\usepackage[top=1.4in, bottom=1.2in, left=1.4in, right=1.4in,marginpar=1in]{geometry}
\usepackage{subfigure}
\usepackage{tikz}
\usepackage{adjustbox}
\usetikzlibrary{cd}
\usetikzlibrary{fit, patterns}
\usetikzlibrary{matrix,arrows,decorations.pathmorphing,decorations.pathreplacing}

\usepackage{chngcntr}
\counterwithin{figure}{section}

\newtheorem{thm}{Theorem}[section]
\newtheorem{prop}[thm]{Proposition}

\newtheorem{cor}[thm]{Corollary}
\newtheorem{lem}[thm]{Lemma}

\theoremstyle{definition}
\newtheorem{define}[thm]{Definition}

\theoremstyle{remark}
\newtheorem{rem}[thm]{Remark}
\newtheorem{example}[thm]{Example}

\newtheorem{question}[thm]{Question}

\newcommand{\ve}[1]{\boldsymbol{\mathbf{#1}}}

\newcommand{\R}{\mathbb{R}}

\newcommand{\Z}{\mathbb{Z}}
\newcommand{\N}{\mathbb{N}}
\newcommand{\Q}{\mathbb{Q}}

\renewcommand{\d}{\partial}
\renewcommand{\subset}{\subseteq}

\renewcommand{\bar}{\overline}

\newcommand{\iso}{\cong}

\DeclareMathOperator{\alg}{{alg}}

\DeclareMathOperator{\Cyl}{{Cyl}}

\DeclareMathOperator{\Ext}{{Ext}}

\DeclareMathOperator{\gr}{{gr}}

\DeclareMathOperator{\Hom}{{Hom}}

\DeclareMathOperator{\id}{{id}}

\DeclareMathOperator{\im}{{im}}

\DeclareMathOperator{\Maz}{{Maz}}

\DeclareMathOperator{\Spin}{{Spin}}

\newcommand{\lk}{\mathrm{lk}}

\newcommand{\bA}{\mathbb{A}}
\newcommand{\bB}{\mathbb{B}}

\newcommand{\bE}{\mathbb{E}}
\newcommand{\bF}{\mathbb{F}}

\newcommand{\bH}{\mathbb{H}}
\newcommand{\bI}{\mathbb{I}}
\newcommand{\bJ}{\mathbb{J}}

\newcommand{\bM}{\mathbb{M}}

\newcommand{\bT}{\mathbb{T}}

\newcommand{\bX}{\mathbb{X}}

\newcommand{\cA}{\mathcal{A}}
\newcommand{\cB}{\mathcal{B}}
\newcommand{\cC}{\mathcal{C}}
\newcommand{\cD}{\mathcal{D}}

\newcommand{\cF}{\mathcal{F}}

\newcommand{\cH}{\mathcal{H}}

\newcommand{\cK}{\mathcal{K}}

\newcommand{\cS}{\mathcal{S}}
\newcommand{\cT}{\mathcal{T}}
\newcommand{\cU}{\mathcal{U}}
\newcommand{\cV}{\mathcal{V}}

\newcommand{\cX}{\mathcal{X}}
\newcommand{\cY}{\mathcal{Y}}
\newcommand{\cZ}{\mathcal{Z}}

\newcommand{\frC}{\mathfrak{C}}

\newcommand{\frH}{\mathfrak{H}}

\newcommand{\frL}{\mathfrak{L}}

\newcommand{\frR}{\mathfrak{R}}

\newcommand{\frV}{\mathfrak{V}}

\newcommand{\frh}{\mathfrak{h}}

\newcommand{\fro}{\mathfrak{o}}

\newcommand{\frs}{\mathfrak{s}}

\newcommand{\frv}{\mathfrak{v}}

\newcommand{\frz}{\mathfrak{z}}

\newcommand{\scC}{\mathscr{C}}

\newcommand{\scE}{\mathscr{E}}
\newcommand{\scF}{\mathscr{F}}

\newcommand{\scH}{\mathscr{H}}

\newcommand{\scJ}{\mathscr{J}}

\newcommand{\scM}{\mathscr{M}}

\newcommand{\scU}{\mathscr{U}}
\newcommand{\scV}{\mathscr{V}}

\newcommand{\cCFL}{\mathcal{C\!F\!L}}
\newcommand{\cCFK}{\mathcal{C\hspace{-.5mm}F\hspace{-.3mm}K}}
\newcommand{\cHFL}{\mathcal{H\!F\! L}}
\newcommand{\cHFK}{\mathcal{H\!F\! K}}
\newcommand{\CF}{\mathit{CF}}

\newcommand{\HF}{\mathit{HF}}

\newcommand{\HFL}{\mathit{HFL}}

\newcommand{\PD}{\mathit{PD}}

\newcommand{\xs}{\ve{x}}
\newcommand{\ys}{\ve{y}}
\newcommand{\zs}{\ve{z}}
\newcommand{\ws}{\ve{w}}

\newcommand{\ps}{\ve{p}}

\renewcommand{\a}{\alpha}
\renewcommand{\b}{\beta}
\newcommand{\g}{\gamma}

\newcommand{\veps}{\varepsilon}

\usepackage{leftidx}

\DeclareMathOperator{\Cone}{{Cone}}

\numberwithin{equation}{section}

\newcommand{\ar}{\mathrm{a.r.}}

\newcommand{\llsquare}{[\hspace{-.5mm}[}
\newcommand{\rrsquare}{]\hspace{-.5mm}]}

\newcommand{\co}{\mathrm{co}}

\allowdisplaybreaks

\newcommand{\vecotimes}{\mathrel{\vec{\otimes}}}

\title{L-space satellite operators and knot Floer homology}

\author{Daren Chen}
\address{Department of Mathematics\\California Institute of Technology\\ Pasadena, CA, USA}
\email{darenc@caltech.edu}

\author{Ian Zemke}
\address{Department of Mathematics\\University of Oregon\\  Eugene, OR, USA}
\email{izemke@uoregon.edu}
\thanks{IZ was partially supported by NSF grant DMS-2204375 and a Sloan fellowship.}

\author{Hugo Zhou}
\address{Department of Mathematics\\University of Michigan\\  Ann Arbor, MI, USA}
\email{hugozhou@umich.edu}
\thanks{HZ was  supported by the Max Planck Institute for Mathematics during the course of this work.}

\begin{document}
\maketitle
\begin{abstract}
We consider satellite operators where the corresponding 2-component link is an L-space link. This family includes many commonly studied satellite operators, including cabling operators, the Whitehead operator, and a family of Mazur operators. We give a formula which computes the knot Floer complex of a satellite of $K$ in terms of the knot Floer complex of $K$. Our main tools are the Heegaard Floer Dehn surgery formulas and their refinements. A key step in our computation is a proof that 2-component L-space links have formal knot Floer complexes. We use this to show that the link Floer complexes of 2-component L-space links are determined by their multivariable Alexander polynomials. We implement our satellite formula in Python code, which we also make available.
\end{abstract}
\tableofcontents

\section{Introduction}

In this paper we study satellite operators and knot Floer homology. Knot Floer homology is an invariant of knots which was introduced independently by Ozsv\'{a}th and Szab\'{o} \cite{OSKnots} and Rasmussen \cite{RasmussenKnots}. In this paper, we focus on the full knot Floer complex. If $K\subset S^3$ is a knot, we write $\cCFK(K)$ for this invariant. It takes the form of a finitely generated, free chain complex over a 2-variable polynomial ring $R=\bF[W,Z]$. Note, many authors write $\cU$ in place of $W$, and $\cV$ in place of $Z$; we reserve $U$ to denote the product $U=WZ$.

Satellite constructions are some of the most fundamental operations in knot theory. Given a knot $P$ in $S^1\times D^2$, and a knot $K\subset S^3$, we may form a new knot 
\[
P(K,n)\subset S^3
\]
for each $n\in \Z$. To construct this knot, we consider an embedding
\[
i_{K,n}\colon S^1\times D^2\hookrightarrow S^3
\]
of a tubular neighborhood of $K$. We assume that the longitude $S^1\times \{\theta\}$, $\theta\in \d D^2$ is sent to an $n$-framed longitude of $K$. The knot $P(K,n)$ is defined as the image of $P$ under $i_{K,n}$:
\[
P(K,n)=i_{K,n}(P)\subset S^3.
\]

%
%

 Note that there is a corresponding 2-component link $L_P\subset S^3$ obtained by viewing $S^1\times D^2$ as the tubular neighborhood of an unknot $U$ in $S^3$. We set
\[
L_P=\mu\cup P
\]
where $\mu$ is the core of the solid torus complementary to solid torus containing $P$. See Figure~\ref{fig:11}.

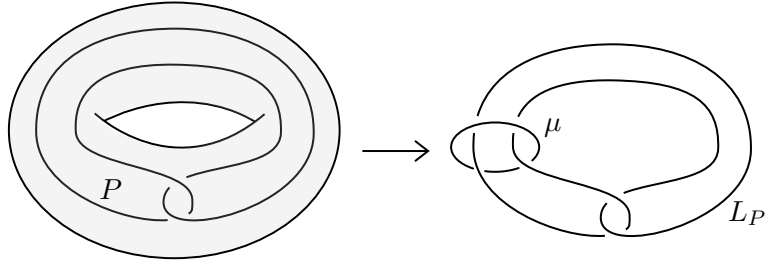
\begin{figure}[h]
\begingroup%
  \makeatletter%
  \providecommand\color[2][]{%
    \errmessage{(Inkscape) Color is used for the text in Inkscape, but the package 'color.sty' is not loaded}%
    \renewcommand\color[2][]{}%
  }%
  \providecommand\transparent[1]{%
    \errmessage{(Inkscape) Transparency is used (non-zero) for the text in Inkscape, but the package 'transparent.sty' is not loaded}%
    \renewcommand\transparent[1]{}%
  }%
  \providecommand\rotatebox[2]{#2}%
  \newcommand*\fsize{\dimexpr\f@size pt\relax}%
  \newcommand*\lineheight[1]{\fontsize{\fsize}{#1\fsize}\selectfont}%
  \ifx\svgwidth\undefined%
    \setlength{\unitlength}{284.04015121bp}%
    \ifx\svgscale\undefined%
      \relax%
    \else%
      \setlength{\unitlength}{\unitlength * \real{\svgscale}}%
    \fi%
  \else%
    \setlength{\unitlength}{\svgwidth}%
  \fi%
  \global\let\svgwidth\undefined%
  \global\let\svgscale\undefined%
  \makeatother%
  \begin{picture}(1,0.34407329)%
    \lineheight{1}%
    \setlength\tabcolsep{0pt}%
    \put(0,0){\includegraphics[width=\unitlength,page=1]{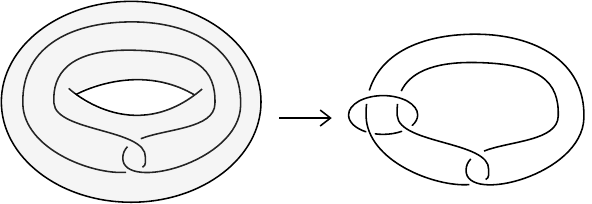}}%
    \put(0.71255525,0.17176173){\makebox(0,0)[lt]{\lineheight{1.25}\smash{\begin{tabular}[t]{l}$\mu$\end{tabular}}}}%
    \put(0.12249339,0.08126561){\makebox(0,0)[lt]{\lineheight{1.25}\smash{\begin{tabular}[t]{l}$P$\end{tabular}}}}%
    \put(0.95678889,0.05219165){\makebox(0,0)[lt]{\lineheight{1.25}\smash{\begin{tabular}[t]{l}$L_P$\end{tabular}}}}%
  \end{picture}%
\endgroup%

\caption{(Left) A knot $P\subset S^1\times D^2$. (Right) The 2-component link $L_P\subset S^3$.}
\label{fig:11}
\end{figure}

In this paper, we consider the problem of computing $\cCFK(P(K,n))$ given $\cCFK(K)$. We consider a restricted family of satellite operators:
\begin{define}
\item
\begin{enumerate}
\item A rational homology 3-sphere $Y$ is called an \emph{L-space} if $\widehat{\HF}(Y,\frs)\iso \Z/2$ for all $\frs\in \Spin^c(Y)$. 
 \item A link $L\subset S^3$ is called an \emph{L-space link} if the Dehn surgery $S^3_{\Lambda}(L)$ is an L-space for all integral framings $\Lambda \gg 0$.
\item A pattern $P\subset S^1\times D^2$ is an \emph{L-space pattern} if the 2-component link $L_P\subset S^3$ is an L-space link.
\end{enumerate}
\end{define}
Many important families of links are L-space links. For example, algebraic links are L-space links \cite{GorskyNemethiAlgebraicLinks}.  We refer the reader to Section~\ref{sec:background-L-space-link} more more details about L-space links.

In this paper, we completely describe $\cCFK(P(K,n))$ in terms of $\cCFK(K)$ when $P$ is an L-space pattern. Many common satellite operators are formed using L-space patterns. This includes all cabling operators, the Whitehead operator, and a family of generalized Mazur operators. We illustrate our techniques with numerous example computations. 

 Our main tool is the link surgery formula of Manolescu and Ozsv\'{a}th \cite{MOIntegerSurgery}, as well as its interpretation in terms of bordered manifolds with torus boundary due to the second author \cite{ZemBordered} \cite{ZemExact}.

Satellite operators in knot Floer theory is a well-trodden subject. Early papers by Hedden \cite{HeddenCabling1} \cite{HeddencablingII} \cite{HeddenWhiteheadDoubles} give substantial information about the knot Floer homology of cables and Whitehead doubles. More recent efforts have centered on using the bordered Heegaard Floer theory of Lipshitz, Ozsv\'{a}th and Thurston \cite{LOTBordered} to study satellite operations. This has led to numerous advances in the subject, notably work of Hom \cite{HomTauCables}, Petkova \cite{PetkovaCablesThin}, Levine \cite{LevineDoubling}, among many others.  Using the immersed curves interpretation of bordered Heegaard Floer homology \cite{HRWImmersedCurves}, Hanselman and Watson \cite{HWCabling} describe a very elegant description of the $U=0$ version of the knot Floer complex in terms of geometric operations on immersed curves. These techniques have been extended to compute $\cCFK(P(K,n))/U$ for $(1,1)$-patterns by Chen and Hanselman \cite{Chen11} \cite{ChenHanselmanSatellites}.  Comparing our work to \cite{ChenHanselmanSatellites}, we remark that there are L-space patterns that are not $(1,1)$-patterns, and vice versa. See Section \ref{sec: More general L space patterns} for further discussion.

 It would be interesting to compare our techniques with the immersed curve techniques, for example by interpreting them in terms of Hanselman's minus version of the immersed curves technology for knot Floer homology \cite{HanselmanMinus}.

We note that all of the above techniques compute a weaker version of knot Floer homology than $\cCFK(K)$. Indeed, they compute the versions over one of the rings $\bF[W,Z]/U$, $\bF[W,Z]/W$ or $\bF[W,Z]/(W,Z)$. We note that these invariants still contain substantial topological information. For example the $\bF[W,Z]/U$ version is sufficient to compute Hom's $\veps$-invariant \cite{HomEpsilon}, the $\varphi_k$ invariants of Dai, Hom, Stoffregen and Truong \cite{DHSTmore} and Ozsv\'{a}th and Szab\'{o}'s $\tau$-invariant \cite{OS4ballgenus}.

 We note that the $\bF[W,Z]/U$ theory is not sufficient to compute Ozsv\'{a}th, Stipsicz and Szab\'{o}'s $\Upsilon$-invariant \cite{OSSUpsilon}, or the $d$-invariants of Dehn surgeries. Our techniques also allow one to compute $\CF^-$ of Dehn surgeries along these satellite knots using the mapping cone formula \cite{OSIntegerSurgeries}.

In small examples, the complex $\cCFK(K)/U$ frequently determines $\cCFK(K)$, though there are simple algebraic examples of knot-like complexes $C$ over $\bF[W,Z]$ where $C/U$ does not determine $C$.

Our techniques are also suited for direct computation. We build a finite dimensional (albeit rather large) model for $\cCFK(P(K,n))$. We have implemented this algorithm in Python code which we have made available \cite{CZZCode}.

\subsection{Formality and 2-component L-space links}

To a link $L$ in $S^3$, Ozsv\'{a}th and Szab\'{o} define a link Floer complex $\cCFL(L)$ \cite{OSLinks}. For an $n$-component link, this invariant can be interpreted as being a finitely generated and free chain complex over the $2n$-variable polynomial ring
\[
R_n:=\bF[W_1,Z_1,\dots,W_n,Z_n].
\]
We write $U_i$ for the product $W_iZ_i$. We note that $U_i$ and $U_j$ give chain homotopic endomorphisms of $\cCFL(L)$, and therefore induce the same action on homology.

In this paper, we describe how to compute $\cCFL(L)$ for 2-component L-space links. The form of our computation is inspired by \cite{BLZLattice}, though our proof is substantially different. The key step in establishing our computation is the following result:

\begin{thm}
\label{thm:formality-intro}
If $L$ is a 2-component L-space link, then $\cCFL(L)$ is formal (i.e. there exists a quasi-isomorphism from $\cCFL(L)$ to its homology $\cHFL(L)$).
\end{thm}

Our proof of Theorem~\ref{thm:formality-intro} is purely algebraic, and makes use of the Koszul duality between the polynomial ring and the exterior algebra.

We now sketch how to use Theorem~\ref{thm:formality-intro} to compute $\cCFL(L)$. Firstly, by some basic homological algebra, it follows from Theorem~\ref{thm:formality-intro} that $\cCFL(L)$ is homotopy equivalent to a free resolution over $R_2$ of $\cHFL(L)$. Therefore it suffices to understand the homology $\cHFL(L)$ as an $R_2$-module.

 If $L$ is a link with $n$-components, the link Floer complex $\cCFL(L)$ admits an $n$-component Alexander grading, which takes values in the set
\[
\bH(L)=\prod_{i=1}^n \left(\Z+\frac{\lk(K_i,L\setminus K_i)}{2}\right)
\]
where $K_i$ is the $i$-th component of $L$. The differential preserves the Alexander grading, as do the actions of $U_1,\dots, U_n$.

 It is a well-known fact that $L$ is an L-space link if and only if the homology $\cHFL(L)$ is torsion free as an $\bF[U]$-module (where $U$ acts by any of the $U_i$; all $U_i$ have the same action on homology). Therefore the homology $\cHFL(L)$ decomposes as a direct sum of one copy of $\bF[U]$ in each Alexander grading:
\[
\cHFL(L)=\bigoplus_{\ve{s}\in \bH(L)} \bF[U].
\]
Since the actions of $W_i$ and $Z_i$ have simple Maslov and Alexander grading shifts, the structure of $\cHFL(L)$ as an $R_n$-module is determined by the set of data of the Maslov grading of the $\bF[U]$-generator in each Alexander grading $\ve{s}$. This data is exactly encoded by the $H$-function
\[
H_L\colon \bH(L)\to \Z^{\ge 0}.
\]
Gorsky and N\'{e}methi \cite{GorskyNemethiLattice} prove that for an L-space link $L\subset S^3$ the function $H_L$ is determined by multi-variable Alexander polynomial of $L$ and its sublinks. See Proposition~\ref{prop:GNH-function} below for a precise statement of their result. As a consequence:

\begin{cor} If $L$ is a 2-component L-space link, then $\cCFL(L)$ is computable from the Alexander polynomials of $L$ and its sublinks.
\end{cor}

We note that Theorem~\ref{thm:formality-intro} has two important antecedents. The first is Ozsv\'{a}th and Szab\'{o}'s computation of the knot Floer complex of L-space knots \cite{OSlens}. Therein, Ozsv\'{a}th and Szab\'{o} prove that L-space knots have $\cCFK(K)$ equal to a staircase complex. Note that staircase complexes are easily seen to be formal. 

The second antecedent to Theorem~\ref{thm:formality-intro} is \cite{BLZLattice}*{Theorem~1.2}, wherein it is proven that \emph{plumbed} L-space links have formal link Floer complexes.  Note that the proof of formality from \cite{BLZLattice} is substantially different than the proof we give in the present paper. The proof in \cite{BLZLattice} goes by way of extending the second author's proof of the equivalence of lattice homology and Heegaard Floer homology \cite{ZemLattice} to give a version of lattice link homology which computes the link Floer complex of plumbed links. A key property of the lattice link complex is that it decomposes as a chain complex of the form
\[
\begin{tikzcd}
C_q \ar[r, "\d"] & C_{q-1}\ar[r, "\d"] & \cdots \ar[r, "\d"] & C_0
\end{tikzcd}
\] 
where each $C_j$ is an $R_n$-module (with vanishing internal differential). The subscript $j$ is the lattice grading. Interestingly, the existence of the lattice grading is sufficient to force the link Floer complex of a plumbed L-space link to be formal.

The above discussion motivates the following question:

\begin{question} If $L$ is an L-space link, is $\cCFL(L)$ formal?
\end{question}
We note that \cite{OSlens}, \cite{BLZLattice} and the present work answer the above question in the affirmative when either $|L|\in \{1,2\}$ or $L$ is plumbed. We discuss the challenges of extending our proof beyond these cases in Section~\ref{sec:remarks}.

\subsection{L-space satellite operators}

In the present paper, we use Theorem~\ref{thm:formality-intro} to compute $\cCFK(P(K,n))$ for any satellite operation $P$ where the corresponding 2-component link $L_P$ is an L-space link. To do this, we use the link surgery formula of Manolescu and Ozsv\'{a}th \cite{MOIntegerSurgery} and its reinterpretation as a theory for manifolds with parametrized torus boundary due to the second author \cite{ZemBordered}.

The second author reformulated the link surgery formula in terms of an associative algebra $\cK$ called \emph{the surgery algebra}. We recall the definition of $\cK$. It is an associative algebra over the ring of two idempotents $\ve{I}=\ve{I}_0\oplus \ve{I}_1$ (where each $\ve{I}_i=\bF=\Z/2$. Furthermore
\[
\ve{I}_0\cdot \cK\cdot \ve{I}_0=\bF[W,Z],\quad \ve{I}_0 \cdot \cK \cdot \ve{I}_1=0
\]
\[
\ve{I}_1\cdot \cK\cdot \ve{I}_0=\bF[U,T,T^{-1}] \otimes \langle \sigma, \tau\rangle \quad \text{and}, \quad \ve{I}_1\cdot \cK\cdot \ve{I}_1=\bF[U,T,T^{-1}].
\]
The algebra is subject to the relations that
\[
\sigma W=UT^{-1} \sigma, \quad \sigma Z=T \sigma,\quad \tau W=T^{-1} \tau,\quad \tau  Z=UT  \tau.
\]

  To a knot $K\subset S^3$ with a framing $n\in \Z$, there is a type-$D$ module
\[
\cX_n(K)^{\cK}.
\]
To a two component link $L\subset S^3$, equipped with an integral framing $\Lambda$, there is a $DA$-bimodule
\[
{}_{\cK} \cX_{\Lambda}(L)^{\cK}.
\]
We think of each of the two actions of $\cK$ as corresponding to one of the link components. By ignoring the link surgery maps for the component of $L$ corresponding to the right action of $\cK$, we obtain a $DA$-bimodule
\[
{}_{\cK} \cX_{\Lambda}(L)^R,
\]
where $R=\bF[W,Z]$. This bimodule is computable from ${}_{\cK} \cX_{\Lambda}(L)^{\cK}$, but contains less information. Our proof of Theorem~\ref{thm:formality-intro} extends very naturally to prove the following:

\begin{thm}
\label{thm:compute-bimodule-intro}
 If $L\subset S^3$ is a 2-component L-space link, then the bimodule ${}_{\cK} \cX_{\Lambda}(L)^R$ is determined by the Alexander polynomials of $L$ and its sublinks.
\end{thm}

Our model of ${}_{\cK} \cX_{\Lambda}(L)^R$ is described in Section~\ref{sec:2-component-L-space}. We perform a number of example computations in Section~\ref{sec:Examples of bimodules} to illustrate the technique. The bimodule ${}_{\cK} \cX_{\Lambda}(L)^R$ is quite straightforward to read off from the $H_L$ function. Furthermore, we show how to encode the data in a finite list of staircase complexes and a finite list of morphisms between them.

If $P\subset S^1\times D^2$ is a satellite operator, and $L_P$ is the associated link, then we may use the module ${}_{\cK} \cX_{\Lambda}(L_P)^R$ and the connected sum formulas from \cite{ZemBordered} to obtain a homotopy equivalence
\begin{equation}
\cCFK(P(K,n))^R\simeq \cX_{n}(K)^{\cK}\boxtimes {}_{\cK}\cH^{\cK}\boxtimes {}_{\cK} \cX_{(0,0)}(L_P)^{R}.
\label{eq:satellite-formula}
\end{equation}
Here $\cX_n(K)^{\cK}$ is the type-$D$ surgery module for $K$, and ${}_{\cK} \cH^{\cK}$ is the type-$DA$ module for the $(0,0)$-framed Hopf link (computed in \cite{ZemBordered}). When $K\subset S^3$, the type-$D$ module $\cX_n(K)^{\cK}$ contains the same information as $\cCFK(K)$. The tensor product description in Equation~\eqref{eq:satellite-formula} corresponds to a simple Dehn surgery description of the knot $P(K,n)$, shown below in Figure~\ref{fig:satellite1}.

We spend a substantial effort describing how to use Equation~\eqref{eq:satellite-formula} to perform concrete computations. We will write $\bX(P,K,n)^R$ for the type-$D$ module on the right hand side of Equation~\eqref{eq:satellite-formula}. We will write this chain complex as a complex of the form:
\[
\bX(P,K,n)=\begin{tikzcd}[column sep=1.5cm, row sep=1.5cm]
 \bE
 	\ar[r, "\Phi^\mu+\Phi^{-\mu}"]
 	\ar[d, "\Phi^K+\Phi^{-K}",swap]
 	\ar[dr,dashed, "*"]
 	 &
 \bF
 	 \ar[d, "\Phi^K+\Phi^{-K}"]
 \\
\bJ
	 	\ar[r, "\Phi^\mu+\Phi^{-\mu}",swap]
& 
\bM
\end{tikzcd}
\]
Here $*$ denotes a sum of four maps $\Phi^{K,\mu}$, $\Phi^{-K,\mu}$, $\Phi^{K,-\mu}$ and $\Phi^{-K,-\mu}$.  The complexes $\bE$, $\bF$, $\bJ$ and $\bM$ are infinitely generated type-$D$ modules over $R$. The $\bE$, $\bF$, $\bJ$ and $\bM$ complexes are the analogs of the $\bA$ and $\bB$ complexes from Ozsv\'{a}th and Szab\'{o}'s mapping cone formula for Dehn surgeries \cite{OSIntegerSurgeries}. The complex $\bE$ decomposes as a direct product
\[
\bE=\prod_{(s,t)\in \bH(P)} E_{s,t}
\]
where $\bH(P)=(\tfrac{1}{2}+\Z)\times (\frac{l-1}{2}+\Z)$ and $l$ denotes the winding number of $P$. Each of the $E_{s,t}$ is a finitely generated type-$D$ module over $R$. The complexes $\bF$, $\bJ$ and $\bM$ admit similar descriptions.

In Section~\ref{sec:truncation}, we describe in detail how to truncate the complex $\bX(P,K,n)$ to obtain a finitely generated type-$D$ module over $R$ which computes $\cCFK(P(K,n))$. We perform a number of concrete example computations and also implement this algorithm into a computer program written in Python \cite{CZZCode}. 

\subsection{The identity and elliptic satellite operators}

We prove several additional results which are important for the bordered link surgery theory, and which are helpful in our proofs of the above.

The most basic satellite operator is the identity satellite operator. In terms of bordered manifolds with torus boundary, this corresponds to gluing the mapping cylinder of the identity map $\bI\colon \bT^2\to \bT^2$. Of additional interest is the elliptic involution $E\colon \bT^2\to \bT^2$. If we write $\bT^2=\R^2/\Z^2$, then $E(x,y)=(-x,-y)$.

 The construction from \cite{ZemExact} gives a $DA$-bimodule to any manifold with parametrized torus boundary components. We therefore obtain bimodules
\[
{}_{\cK}\cX(\Cyl_{\bI})^{\cK}, \quad \text{and} \quad {}_{\cK} \cX(\Cyl_E)^{\cK}.
\]

The algebra $\cK$ has a symmetry $E\colon \cK\to \cK$, given by the formula
\[
E(W)=Z\quad E(Z)=W\quad E(\sigma)=\tau\quad E(\tau)=\sigma,\quad E(T^{\pm 1})=T^{\mp 1} \quad \text{and} \quad E(U)=U.
\]
 
We recall additionally that given any algebra morphism $\phi\colon A\to B$, there is a canonical $DA$-bimodule ${}_A [\phi]^B$. See Section~\ref{sec:typeDA-background} for more details.

In this paper, we prove the following:
\begin{thm}
\label{thm:mapping-cylinder-identity-intro} There are homotopy equivalences
\[
{}_{\cK} \cX(\Cyl_{\bI})^{\cK}\simeq {}_{\cK} [\bI]^{\cK} \quad \text{and} \quad {}_{\cK} \cX(\Cyl_E)^{\cK}\simeq {}_{\cK}[E]^{\cK}.
\] 
\end{thm}

The above theorem has several helpful consequences. We recall that in \cite{ZemBordered} there was an algebraically defined bimodule ${}_{\cK|\cK} \bI^{\Supset}$ which we used to turn type-$D$ modules over $\cK$ into type-$A$ modules over $\cK$. One corollary of Theorem~\ref{thm:mapping-cylinder-identity-intro} is the following:

\begin{cor}\label{cor:quasi-inverse} The algebraic bimodule ${}_{\cK|\cK} \bI^{\Supset}$ admits a quasi-inverse. This quasi-inverse is given by $\cX(\Cyl_{\bI})^{\cK\otimes \cK}$. 
\end{cor}

In the above corollary, by \emph{quasi-inverses} we mean that there is a homotopy equivalence
\[
\cX(\Cyl_{\bI})^{\cK\otimes \cK} \boxtimes {}_{\cK|\cK} \bI^{\Supset}\simeq {}_{\cK} \bI^{\cK}.
\]
The tensor product is formed along just one copy of $\cK$ (and therefore is a $DA$-bimodule over $(\cK,\cK)$). The result holds regardless of which of the $\cK$ factors the tensor product is taken along.

\begin{rem}
There is a technical subtlety in Theorems~\ref{thm:mapping-cylinder-identity-intro} and Corollary~\ref{cor:quasi-inverse} that we are not emphasizing in the introduction. The algebra $\cK$ requires a choice of linear topology, and there are two natural choices described in \cite{ZemExact}: a \emph{chiral} topology and the \emph{$U$-adic} topology. The results stated above hold for the chiral topology, but not the $U$-adic topology. See Section~\ref{sec:identity-cobordism}.
\end{rem}

\subsection{Acknowledgements}

We thank Maciej Borodzik, Jonathan Hanselman, Kristen Hendricks, Jennifer Hom, Adam S. Levine, Beibei Liu, Yi Ni, Diego Santoro and Matthew Stoffregen for helpful discussions. 

\section{Satellite operations and Dehn surgery}

We recall the basics of satellite operators. Suppose that $P$ is a knot in the solid torus $S^1\times D^2$. If $K\subset Y$ is a knot with Morse framing $\lambda$, we can form a satellite knot $P(K,\lambda)$ by carving out $S^1\times D^2$ and gluing in $(S^1\times D^2,P)$. We assume that the longitude $S^1\times \{\theta\}$, $\theta\in \d D^2$ is sent to the longitude $\lambda$. We note that this construction also depends on an orientation for $K$. Of particular interest is  the case that $K$ is a knot in $S^3$. For such knots, a Morse framing is specified by an integer $n\in \Z$. We define $P(K,n)$ to be the knot obtained by gluing $S^1\times \{p\}$, $p\in \d D$, to the $n$-framed longitude of $K$.

Satellite operations can be achieved by a straightforward Dehn surgery construction, as we now describe. We can view $(S^1\times D^2, P)$ as being a knot in the complement of $\mu\subset S^3$. We write $L_P$ for the two component link $\mu\cup P$. We write $H$ for the negative Hopf link.  We then consider the 3-component link $K\# H\# L_P$, as shown in Figure~\ref{fig:satellite1}. If we perform $n$-surgery on $K$, and $0$-surgery on the knot $\mu$, we are left with the knot $P(K,n)\subset S^3$.

A particularly important special case of the construction is cabling:

\begin{define}
If $\gcd(p,q)=1$, the $(p,q)$-cable of $K$, denoted $K_{p,q}$, is the knot which lies on the 2-torus $\d \nu (K)$, which winds $p$ times longitudinally, and $q$ times meridianally.
\end{define}

In Section~\ref{sec:cables}, we will use the link Floer complex of $T(2,2q)$ to compute the cabling bimodule for the $(q,nq+ 1)$-satellite operator. The cabling bimodule for general $p$ and $q$ is more complicated, but can also be computed using our constructions. See Section~\ref{sec:general cabling} for more information.

\begin{figure}[h]
\begingroup%
  \makeatletter%
  \providecommand\color[2][]{%
    \errmessage{(Inkscape) Color is used for the text in Inkscape, but the package 'color.sty' is not loaded}%
    \renewcommand\color[2][]{}%
  }%
  \providecommand\transparent[1]{%
    \errmessage{(Inkscape) Transparency is used (non-zero) for the text in Inkscape, but the package 'transparent.sty' is not loaded}%
    \renewcommand\transparent[1]{}%
  }%
  \providecommand\rotatebox[2]{#2}%
  \newcommand*\fsize{\dimexpr\f@size pt\relax}%
  \newcommand*\lineheight[1]{\fontsize{\fsize}{#1\fsize}\selectfont}%
  \ifx\svgwidth\undefined%
    \setlength{\unitlength}{196.44827952bp}%
    \ifx\svgscale\undefined%
      \relax%
    \else%
      \setlength{\unitlength}{\unitlength * \real{\svgscale}}%
    \fi%
  \else%
    \setlength{\unitlength}{\svgwidth}%
  \fi%
  \global\let\svgwidth\undefined%
  \global\let\svgscale\undefined%
  \makeatother%
  \begin{picture}(1,0.35123729)%
    \lineheight{1}%
    \setlength\tabcolsep{0pt}%
    \put(0,0){\includegraphics[width=\unitlength,page=1]{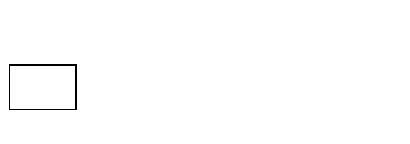}}%
    \put(0.08505322,0.12893885){\makebox(0,0)[lt]{\lineheight{1.25}\smash{\begin{tabular}[t]{l}$K$\end{tabular}}}}%
    \put(0,0){\includegraphics[width=\unitlength,page=2]{fig_satellite1.pdf}}%
    \put(0.43781318,0.16241423){\makebox(0,0)[lt]{\lineheight{1.25}\smash{\begin{tabular}[t]{l}$0$\end{tabular}}}}%
    \put(0.08459551,0.02595012){\makebox(0,0)[rt]{\lineheight{1.25}\smash{\begin{tabular}[t]{r}$n$\end{tabular}}}}%
    \put(0.80593731,0.02309524){\makebox(0,0)[lt]{\lineheight{1.25}\smash{\begin{tabular}[t]{l}$P(K,n)$\end{tabular}}}}%
    \put(0,0){\includegraphics[width=\unitlength,page=3]{fig_satellite1.pdf}}%
  \end{picture}%
\endgroup%

\caption{The knot $P(K,n)$ in terms of Dehn surgery on $K\# H\# L_P$.}
\label{fig:satellite1}
\end{figure}

\section{Algebraic background}

\subsection{Type-$D$ and $A$ modules}
\label{sec:typeDA-background}

We make use of Lipshitz, Ozsv\'{a}th and Thurston's formalism of type-$D$, type-$A$ and type-$DA$ modules  \cite{LOTBimodules} \cite{LOTBordered}.

\begin{define} Suppose that $A$ is an algebra over a ring $\ve{k}$. A \emph{type-$D$ module} $X^A$ consists of a right $\ve{k}$-module $X$ equipped with a $\ve{k}$-linear map $\delta^1\colon X\to X\otimes_{\ve{k}} A$ such that
\[
(\bI_{X}\otimes \mu_2)\circ (\delta^1\otimes \bI_{A})\circ \delta^1=0.
\]
\end{define}

\begin{define}Suppose that $A$ is an algebra over a ring $\ve{k}$. A \emph{type-$A$} module ${}_A X$ is a left $\ve{k}$-module $X$ equipped with $\ve{k}$-linear maps
\[
m_{n+1}\colon \underbrace{A\otimes_{\ve{k}}\cdots \otimes_{\ve{k}} A}_n\otimes_{\ve{k}} X\to X
\]
for $n\ge 0$. These maps are required to further satisfy
\[
\begin{split}
&\sum_{j=0}^n m_{n-j}(a_n,\dots, m_{j+1}(a_j,\dots, a_1, x))\\
+&\sum_{j=1}^{n-1} m_{n}(a_n,\dots, a_{j+1}a_j,\dots, a_1, x)=0
\end{split}
\]
\end{define}

Note that the definition of a type-$A$ module is identical to the more standard notion of an $A_\infty$-module. We refer the reader to \cite{KellerNotes} \cite{KellerAddendumNotes} for more extensive background.

We additionally need the notion of a $DA$-bimodule:
\begin{define} If $A$ and $B$ are algebras over $\ve{k}$ and $\ve{j}$, respectively, a $DA$-bimodule ${}_A X^B$ is a $(\ve{k},\ve{j})$-bimodule which is equipped with $(\ve{k},\ve{j})$-linear maps
\[
\delta_{n+1}^1 \colon \underbrace{A\otimes_{\ve{k}}\cdots \otimes_{\ve{k}} A}_n\otimes_{\ve{k}}X\to X\otimes_{\ve{j}} B,
\]
for $n\ge 0$. These maps are further required to satisfy the following relation:
\[
\begin{split}
&\sum_{j=0}^n (\bI\otimes \mu_2)(\delta_{n-j}^1\otimes \bI_{B})(a_n,\dots, \delta^1_{j+1}(a_j,\dots, a_1, x))\\
+&\sum_{j=1}^{n-1} \delta^1_{n}(a_n,\dots, a_{j+1}a_j,\dots, a_1, x)=0.
\end{split}
\]
\end{define}

An important example of a $DA$-bimodule is the following (see \cite{LOTBimodules}*{Definition~2.2.48} for further details): 
\begin{define}
\label{def:morphism->bimodule}
Suppose that $A$ and $B$ are $A_\infty$-algebras over the ring $\ve{k}$, and $\phi_*\colon A\to B$ is a morphism of $A_\infty$-algebras. Then there is an induced $DA$ bimodule ${}_A[\phi]^B$ as follows. The underlying vector $(\ve{k},\ve{k})$-bimodule of ${}_A [\phi]^B$ is $\ve{k}$. The structure map $\delta_1^1$ is zero, we have 
\[
\delta_{j+1}(a_j,\dots, a_1,1)=1\otimes \phi_{j}(a_j,\dots, a_1).
\]
In the special case that $A$ and $B$ are actually just associative algebras and $\phi$ is an algebra homomorphism, the only non-trivial action is $\delta_2^1(a,1)=1\otimes \phi(a)$. 
\end{define}

\subsection{Homological perturbation lemma}

We now recall the homological perturbation lemma for $A_\infty$-modules. Before stating the lemma, we recall the following notion:

\begin{define}
\label{def:strong-deformation-retraction} Suppose that $C$ and $Z$ are two chain complexes. We say that $Z$ is a \emph{strong deformation retraction} of $C$ if there exist maps $i,\pi,h$, fitting in to the following diagram:
\[
\begin{tikzcd}
C \ar[loop left, "h"] \ar[r, "\pi", shift left] &Z \ar[l, "i", shift left]
\end{tikzcd}
\]
such that 
\[
\pi\circ i=\id_Z,\quad i\circ \pi=\id_C+\d(h),\quad h\circ h=0,\quad h\circ i=0,\quad \text{and} \quad \pi \circ h=0.
\]
\end{define}
In the above, $\d(h)$ denotes the morphism differential, $\d(h)=\d\circ h+h\circ \d$. Note that definition of a strong deformation retraction makes sense in any $dg$-category (i.e. a category where the morphism spaces are chain complexes).

We now recall the homological perturbation lemma for $A_\infty$-modules. This is well known. See \cite{KontsevichSoibelman}*{Section~6.4} for the result in the context of $A_\infty$-algebras. See \cite{KellerNotes}*{Section~3.3} and \cite{SeidelFukaya}*{Remark~1.15} for more background and historical context. 

\begin{lem}
\label{lem:HPL-modules}
Let $A$ be an associative algebra over a ring $\ve{k}$ of characteristic 2. Suppose that ${}_A M$ is an $A_\infty$-module over $A$ and suppose that $Z$ is a chain complex of $\ve{k}$-modules, and that there is a strong deformation retraction of chain complexes
\[
\begin{tikzcd}
M \ar[loop left, "h"] \ar[r, "\pi", shift left] &Z \ar[l, "i", shift left].
\end{tikzcd}
\]
Then there is an induced $A_\infty$-module structure on ${}_A Z$, and the maps $i$, $\pi$ and $h$ extend to morphisms of $A_\infty$-modules $i_*$, $\pi_*$ and $h_*$ which give a strong deformation retraction from ${}_A M$ to ${}_A Z$. 
\end{lem}

One particularly helpful aspect of the lemma is that the actions on ${}_A Z$ and the maps $i_*$, $h_*$ and $\pi_*$ are given by concrete formulas. These are shown in Figure~\ref{fig:homological-perturbation}. Therein, if $\cT^* A=\bigoplus_n \otimes^n A$ denotes the tensor algebra over $A$, then $\Delta$ denotes the map
\[
\Delta\colon \cT^* A\to \cT^* A\otimes \cT^* A
\]
which sends elementary tensor $a_1\otimes \cdots \otimes a_n$  to the sum
\[
(a_1)\otimes (a_2\otimes \cdots\otimes a_n)+(a_1\otimes a_2)\otimes (a_3\otimes \cdots \otimes a_n)+\cdots +(a_1\otimes \cdots \otimes a_{n-1})\otimes a_n. 
\]

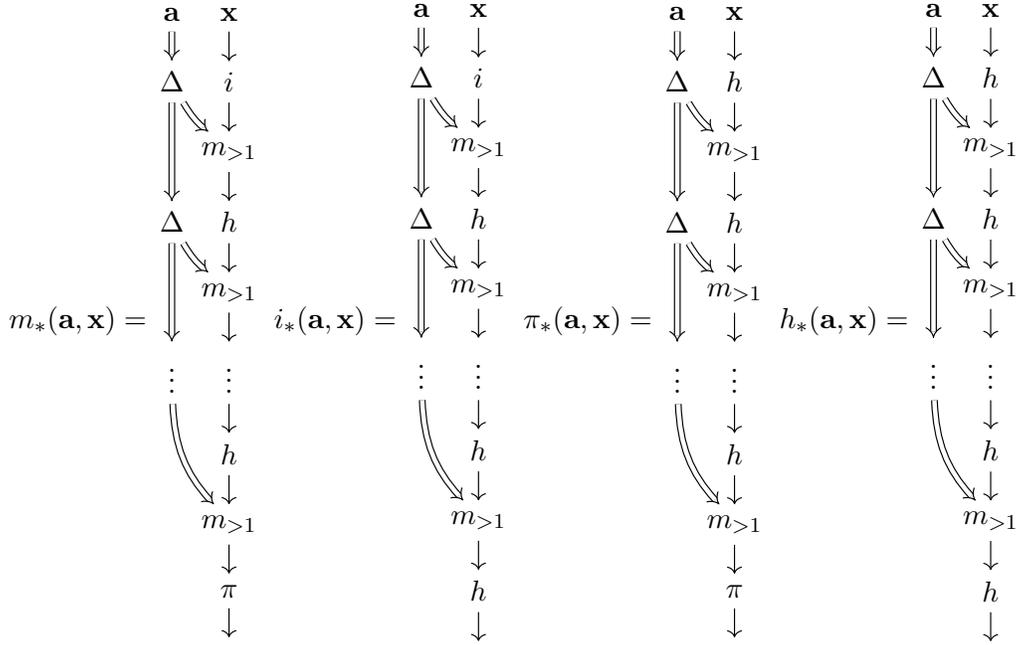
\begin{figure}[ht]
\[
m_*(\ve{a},\xs)=\hspace{-1mm}\begin{tikzcd}[column sep=-.1cm,row sep=.4cm]
\ve{a}\ar[d, Rightarrow]& \ve{x} \ar[d]\\
 \Delta\ar[dr,bend right=10, Rightarrow]\ar[dd,Rightarrow]&i\ar[d]\\
\,& m_{>1} \ar[d]\\
\Delta\ar[dr,bend right=10, Rightarrow]\ar[dd,Rightarrow]&h\ar[d]\\
\,&m_{>1}\ar[d]\\
\vdots \ar[ddr,bend right=20,Rightarrow] & \vdots \ar[d]\\
\,&h\ar[d]\\
\,&m_{>1}\ar[d]\\
\,&\pi\ar[d] \\
\, &\, 
\end{tikzcd}
\,
i_*(\ve{a}, \xs)=\hspace{-1mm}\begin{tikzcd}[column sep=-.1cm,row sep=.4cm]
\ve{a}\ar[d, Rightarrow]& \ve{x} \ar[d]\\
 \Delta\ar[dr,bend right=10, Rightarrow]\ar[dd,Rightarrow]&i\ar[d]\\
\,& m_{>1} \ar[d]\\
\Delta\ar[dr,bend right=10, Rightarrow]\ar[dd,Rightarrow]&h\ar[d]\\
\,&m_{>1}\ar[d]\\
\vdots \ar[ddr,bend right=20,Rightarrow] & \vdots\ar[d]\\
\,&h\ar[d]\\
\,&m_{>1}\ar[d]\\
\,&h\ar[d] \\
\, &\, 
\end{tikzcd}
\,
\pi_*(\ve{a}, \xs)=\hspace{-1mm}\begin{tikzcd}[column sep=-.1cm,row sep=.4cm]
\ve{a}\ar[d, Rightarrow]& \ve{x} \ar[d]\\
 \Delta\ar[dr,bend right=10, Rightarrow]\ar[dd,Rightarrow]&h\ar[d]\\
\,& m_{>1} \ar[d]\\
\Delta\ar[dr,bend right=10, Rightarrow]\ar[dd,Rightarrow]&h\ar[d]\\
\,&m_{>1}\ar[d]\\
\vdots \ar[ddr,bend right=20,Rightarrow] & \vdots\ar[d]\\
\,&h\ar[d]\\
\,&m_{>1}\ar[d]\\
\,&\pi\ar[d] \\
\, &\, 
\end{tikzcd}
\,
 h_*(\ve{a},\xs)=\hspace{-1mm}
\begin{tikzcd}[column sep=-.1cm,row sep=.4cm]
\ve{a}\ar[d, Rightarrow]& \ve{x} \ar[d]\\
 \Delta\ar[dr,bend right=10, Rightarrow]\ar[dd,Rightarrow]&h\ar[d]\\
\,& m_{>1} \ar[d]\\
\Delta\ar[dr,bend right=10, Rightarrow]\ar[dd,Rightarrow]&h\ar[d]\\
\,&m_{>1}\ar[d]\\
\vdots \ar[ddr,bend right=20,Rightarrow] & \vdots\ar[d]\\
\,&h\ar[d]\\
\,&m_{>1}\ar[d]\\
\,&h\ar[d] \\
\, &\, 
\end{tikzcd}
\]
\caption{The maps appearing in the homological perturbation lemma for $A_\infty$-modules.}
\label{fig:homological-perturbation}
\end{figure}

We note that the above homological perturbation results have a more basic formalism in terms of chain complexes, which we now review. We will have use for this perspective.

\begin{lem}
\label{lem:HPL-chain-complexes} Suppose that $(C,d)$ is a chain complex and $\a$ is an endomorphism of $C$ which satisfies the Mauer-Cartan relation
\[
\a^2+[d,\a]=0.
\]
(This implies that $(C,d+\a)$ is a chain complex). Suppose that $(Z,\delta)$ is a chain complex and there is a strong deformation retraction 
\[
\begin{tikzcd}
(C,d) \ar[loop left, "h"] \ar[r, "\pi", shift left] &(Z,\delta) \ar[l, "i", shift left]
\end{tikzcd}.
\]
Suppose additionally that $(1+h \circ \a)$ is invertible, and that
\[
\pi\circ h=0\quad h\circ i=0\quad \text{and}\quad h\circ h=0.
\]
 Then there is an endomorphism $\b\colon Z\to Z$ which satisfies the Mauer-Cartan condition $\b^2+[\delta,\b]=0$ such that there is a strong deformation retraction
\[
\begin{tikzcd}
(C,d+\a) \ar[loop left, "H"] \ar[r, "\Pi", shift left] &(Z,\delta+\beta) \ar[l, "I", shift left]
\end{tikzcd}.
\]
\end{lem}
We leave the proof to the reader since it is well-known and elementary. See, e.g. \cite{HK_Homological_Perturbation}. We remark, however, that the maps $H$, $\Pi$, $I$ and $\b$ have the following simple formulas:
\[
\begin{split}
\b&=\pi\circ \a\circ (1+h\circ \a)^{-1} \circ i\\
\Pi&= \pi\circ (1+\a \circ h)^{-1}\\
H&= (1+h \circ \a)^{-1} \circ h\\
I&= (1+h \circ \a)^{-1} \circ i.
\end{split}
\]
Note that typically $(1+h\circ \a)$ is shown to be invertible by showing that the infinite series
\[
(1+h\circ \a)^{-1}=\sum_{n=0}^\infty (h\circ \a)^n
\]
is convergent.

\begin{rem}
	\label{rem:relax-HPL}
 For our purposes, it is also helpful to use the perturbation lemma in situations where all of the conditions of Lemma~\ref{lem:HPL-chain-complexes} are satisfied, except that $h\circ h$ is potentially nonzero. It is straightforward to show that $\beta= \a \circ (1+h \circ \a)^{-1}$ still satisfies the Mauer-Cartan equation, and that $\Pi$ and $I$ (given by the same formulas as above) are chain maps satisfying $I\circ \Pi=1+\d(H)$. Less obviously, it is possible to show that $\Pi$ and $I$ are in fact chain homotopy equivalences. The only non-trivial relation is $\Pi \circ I\simeq 1$. This follows from work of Markl \cite{Markl_IdealHPL}, and is stated in \cite{HogancampHPL}*{Corollary~4.10}. See also \cite{Crainic}*{2.4}. Translating the formula of Markl to our present case, a chain homotopy $\Pi\circ I=1+\d(K)$ is given by the formula $K=\Pi \circ h\circ I.$ It is straightforward to verify that $K$ is a null-homotopy of $I\circ \Pi+1$.
\end{rem}

\subsection{Formality}

We now recall some basic definitions and well-known results about formality and $A_\infty$-modules. We recall that if ${}_A M$ is an $A_\infty$-module, then the action $m_2$ naturally induces an $A$-module structure on $H_*(M)$, for which we write ${}_A H_*(M)$. Note that here we equip ${}_A H_*(M)$ with no higher actions; it is only an $A$-module.

\begin{define} We say that an $A_\infty$-module ${}_A M$ is \emph{formal} if there exists a quasi-isomorphism of $A_\infty$-modules
\[
f\colon {}_A M\to {}_A H_*(M). 
\]
\end{define}

The following is well-known, though we include a proof for the sake of exposition:

\begin{lem}
\label{lem:SDR-module} Suppose that $A$ is an algebra over $\bF=\Z/2$, and ${}_A M$ is an $A_\infty$-module with a $\Z/2$-valued homological grading. Then there is an $A_\infty$-module ${}_A \scH$ with vanishing differential $m_1$ such that there is a homotopy equivalence of $A_\infty$-modules ${}_A M\simeq {}_A \scH$.  
\end{lem}

\begin{proof}
The main step in the proof is to  describe a strong deformation retraction of chain complexes 
\[
\begin{tikzcd}
M
 	\ar[r, "\pi", shift left]
	\ar[loop left, "h"]	
& H_*(M). \ar[l, "i", shift left] 
\end{tikzcd}
\]
Such a strong deformation retraction exists by straightforward linear algebra, though we repeat the argument for the reader. Write the chain complex $M$ as
\[
\begin{tikzcd}
M_0 
	\ar[r, "d_0", shift left]
& M_1 
	\ar[l, "d_1", shift left]
\end{tikzcd}
\]
where $M_i$ is the $i$-th graded subspace with respect to the $\Z/2$-grading. Inside of $M_i$, we pick a vector space complement $M_i'$ of $\im(d_{i+1})$. Note that $d_i$ vanishes on $\im d_{i+1}$, so we can describe our chain complex via the following diagram:
\[
\begin{tikzcd}[row sep=-.2cm]
M_0' \ar[r, "d_0", twoheadrightarrow] & \im d_0 
\\
\oplus& \oplus
\\
\im d_1 & M_1' \ar[l, "d_1", twoheadrightarrow]
\end{tikzcd} 
\]
We write $H_i=\ker(d_i|_{M_i'})$, and we pick vector space complements $W_i$ of $H_i$ in $M_i'$, so that our chain complex takes the following form:
\[
\begin{tikzcd}[row sep=-.2cm]
W_0' \ar[r, "d_0", "\sim"' ] & \im d_0 
\\
\oplus& \oplus
\\
H_0  & H_1
\\
\oplus& \oplus
\\
\im d_1 & W_1' \ar[l, "d_1", "\sim"']
\end{tikzcd} 
\]
The above diagram gives a strong deformation retraction of chain complexes of $C$ onto $H_0\oplus H_1$, which is isomorphic to $H_*(M)$. Lemma~\ref{lem:HPL-modules} equips $H=H_0\oplus H_1$ with an $A_\infty$-module structure, for which we write ${}_{A} \scH$, so that ${}_A \scH\simeq {}_A M$. 
\end{proof}

\begin{lem} Suppose that $A$ is an algebra over $\bF=\Z/2$ and $f_*\colon {}_A M\to {}_A N$ is a quasi-isomorphism of $A_\infty$-modules, then $f$ is a homotopy equivalence. 
\end{lem}
\begin{proof} By applying Lemma~\ref{lem:SDR-module}, it suffices to consider the case that $m_1$ vanishes on $M$ and $N$ and that $f_1\colon M\to N$ is an isomorphism of vector spaces. An inverse of $f_*$ is constructed by an inductive argument \cite{KellerNotes}*{Section~4}. See \cite{LOTDiagonals}*{Lemma~3.31} for a helpful exposition.
\end{proof}

\subsection{Hypercubes of chain complexes}

We now recall the formalism of hypercubes of chain complexes due to Manolescu and Ozsv\'{a}th. We write $\bE_n$ for the $n$-cube $\{0,1\}^n$. If $\veps,\veps'\in \bE_n$, we write $\veps\le \veps'$ if the inequality holds for each component.

\begin{define} A \emph{hypercube of chain complexes} of dimension $n$ consists of a collection $C_{\veps}$ of vector spaces over $\bF=\Z/2$, as well as a collection of maps $D_{\veps,\veps'}\colon C_{\veps}\to C_{\veps'}$ for  $\veps\le \veps'$. We assume that if $\veps\le \veps''$, then
\[
\sum_{\veps'| \veps\le \veps'\le \veps''} D_{\veps',\veps''}\circ D_{\veps,\veps'}=0. 
\]
\end{define}

We refer the reader to \cite{MOIntegerSurgery}*{Section~5} for additional background.

One helpful technique we will use is the extension of the homological perturbation lemma to hypercubes of chain complexes. This result is written out in \cite{HHSZDuals}*{Lemma~2.10}, though it also follows from earlier work on homological perturbation and filtered chain complexes \cite{HK_Homological_Perturbation}.

\begin{lem} 
\label{lem:homotopy perturbation of hypercube}	
Suppose that $\cC=(C_{\veps}, D_{\veps,\veps'})_{\veps\in \bE_n}$ is a hypercube of chain complexes. Suppose further that for each $\veps\in \bE_n$, we have a strong deformation retraction
\[
\begin{tikzcd}
C_{\veps} \ar[loop left, "h_{\veps}"] \ar[r, "\pi_{\veps}", shift left] &Z_{\veps} \ar[l, "i_{\veps}", shift left],
\end{tikzcd}
\]
for some chain complexes $Z_{\veps}$. 
Then there is a hypercube $\cZ=(Z_{\veps}, \delta_{\veps,\veps'})$ such that the maps $i_{\veps},$ $ \pi_{\veps}$ and $h_{\veps}$ extend to morphisms of hypercubes which give a strong deformation retraction of hypercubes
\[
\begin{tikzcd}
\cC \ar[loop left, "H"] \ar[r, "\Pi", shift left] &\cZ \ar[l, "I", shift left].
\end{tikzcd}
\]
\end{lem}
The structure maps $\delta_{\veps,\veps'}$, as well as the maps $I$, $\Pi$ and $H$ have simple formulas. Namely, if $\veps<\veps'$, they are given by
\[
\begin{split}
\delta_{\veps,\veps'}
=&\sum_{\veps=\veps_1<\cdots<\veps_n=\veps'} \pi_{\veps_n}\circ D_{\veps_{n-1}, \veps_n} \circ h_{\veps_{n-1}}\circ \cdots \circ h_{\veps_2} \circ D_{\veps_1,\veps_2} \circ i_{\veps_1} \\
I_{\veps,\veps'}
=&\sum_{\veps=\veps_1<\cdots<\veps_n=\veps'}   h_{\veps_{n}}\circ D_{\veps_{n-1},\veps_n} \cdots \circ h_{\veps_2} \circ D_{\veps_1,\veps_2} \circ i_{\veps_1}\\
\Pi_{\veps,\veps'}
=&\sum_{\veps=\veps_1<\cdots<\veps_n=\veps'} \pi_{\veps_n}\circ D_{\veps_{n-1}, \veps_n} \circ h_{\veps_{n-1}}\circ \cdots \circ  D_{\veps_1,\veps_2} \circ h_{\veps_1} \\
H_{\veps,\veps'}
=&\sum_{\veps=\veps_1<\cdots<\veps_n=\veps'} h_{\veps_n}\circ D_{\veps_{n-1}, \veps_n} \circ h_{\veps_{n-1}}\circ \cdots \circ h_{\veps_2} \circ D_{\veps_1,\veps_2} \circ h_{\veps_1}.
\end{split}
\]

There are some additional refinements of the homological perturbation lemma for hypercubes of $DA$-bimodules which we will use at a few places. We refer to \cite{ZemBordered}*{Section~4.4} for more background. 

\subsection{Linear topological spaces and completions}

We now recall some background on linear topological spaces and completions. One may consult \cite{AtiayhMacdonald}*{Chapter~10} for general background.

\begin{define} A \emph{linear topological space} consists of a vector space $X$ which is also a topological space so that the following are satisfied:
\begin{enumerate}
\item There is a basis of open sets at $0$ consisting of vector subspaces.
\item Addition is continuous as a map from $X\times X\to X$.
\end{enumerate}
\end{define}

A linear topological space is uniquely specified by a basis of open subspaces centered at 0. 

If $X$ is a linear topological space and $(X_{\a})_{\a\in A}$ is a basis of open subspaces centered at $0$, then the \emph{completion} of $X$ is defined as the inverse limit
\[
\ve{X}:=\varprojlim X/X_{\a}.
\]
It is straightforward to show that $\ve{X}$ is independent (up to canonical isomorphism) of the choice of basis of opens subspaces centered at 0. Note that $\ve{X}$ is itself a linear topological space (whose completion is canonically isomorphic to itself).

\begin{example} \label{ex:cofinite} Suppose that $(X_{\a})_{\a\in A}$ is a family of vector spaces (viewed as having the discrete topology). We can consider the linear topology on $X:=\bigoplus_{\a\in A} X_{\a}$ where the open sets consist of $\bigoplus_{\a\in A\setminus S} X_{\a}$ ranging over finite sets $S\subset A$.  The completion is the direct product $\prod_{\a\in A} X_{\a}$. When each $X_{\a}$ is one-dimensional, we call the topology the \emph{cofinite basis topology}.
\end{example}

\begin{rem}
Linear topological spaces form a category. For our purposes, we define a morphism from $X$ to $Y$ to be a continuous linear map from $\ve{X}$ to $\ve{Y}$ (i.e. a continuous map between completions). With this convention, a linear topological space $X$ is isomorphic (in the so-defined category of linear topological spaces) with its completion $\ve{X}$. Therefore we will not distinguish between linear topological spaces and their completions.
\end{rem}

We now discuss the tensor product of linear topological spaces. There is not a canonical choice of topology on the tensor product, but there are several natural topologies that one can consider. We refer to \cite{Positelski-Linear}*{Section~12} for a helpful background. 

Suppose that $X$ and $Y$ are linear topological spaces, with bases of open subspaces $(X_{\a})_{\a \in A}$ and $(Y_{\b})_{\b\in B}$. We consider the following two linear topologies on $X\otimes Y$:
\begin{enumerate}
\item (The standard tensor product $X\otimes^! Y$): A subspace $E\subset X\otimes Y$ is open if and only if $X_{\a}\otimes Y+X\otimes Y_{\b}\subset E$,for some $\a\in A$ and $\b\in B$.
\item (The chiral tensor product $X\vecotimes Y$): A subspace $E$ is open if and only if $X_{\a}\otimes Y \subset E$ for some $\a$, and for all $\xs\in X$ there is a $\b_{\xs}$ so that  $\xs\otimes Y_{\b_{\xs}}\subset E.$
\end{enumerate}

The standard tensor product is classical, while the chiral tensor product is due to Beilinson \cite{Beilinson-Tensors}.

\label{sec:linear topological spaces}

\section{Heegaard Floer background}

\subsection{Knot and link Floer homology}

Knot Floer homology is a refinement of Heegaard Floer homology for knots, due independently to Ozsv\'{a}th and Szab\'{o} \cite{OSKnots}, and Rasmussen \cite{RasmussenKnots}. Link Floer homology is a generalization of this theory to links, due to Ozsv\'{a}th and Szab\'{o} \cite{OSLinks}.

Given a knot in $S^3$, there is a finitely generated free chain complex $\cCFK(K)$ over the two variable polynomial ring 
\[
R=\bF[W,Z].
\]
 The chain complex $\cCFK(K)$ has a $\Z\times \Z$ bigrading, which we denote by $(\gr_{\ws}, \gr_{\zs})$. Here
\[
(\gr_{\ws},\gr_{\zs})(W)=(-2,0)\quad \text{and} \quad (\gr_{\ws}, \gr_{\zs})(Z)=(0,-2).
\]
It is also helpful to consider the Alexander grading, defined by
\[
A=\frac{ \gr_{\ws}-\gr_{\zs}}{2}.
\]

To a link $L\subset S^3$ there is a finitely generated, free chain complex $\cCFL(L)$ over the ring 
\[
R_n:=\bF[W_1,Z_1,\dots, W_n,Z_n].
\]
 This chain complex also has a Maslov bigrading $(\gr_{\ws}, \gr_{\zs})$, as well as an $n$-component Alexander grading $(A_1,\dots, A_n)$. The Alexander grading takes values in the set
\begin{equation}
\bH(L)=\prod_{i=1}^n\left( \Z+\frac{\lk(K_i,L\setminus K_i)}{2}\right).
\label{eq:Alexander-grading-def}
\end{equation}

The differential is graded as follows:
\[
(\gr_{\ws},\gr_{\zs})(\d)=(-1,-1)\quad \text{and} \quad A(\d)=(0,\dots, 0). 
\]

We now discuss the $H$-function of a link $L\subset S^3$, which takes the form of a map
\[
H_L\colon \bH(L)\to \Z^{\ge 0}.
\]
To define this map, if $\ve{s}\in \bH(L)$, we consider the subcomplex
\[
\cCFL(L,\ve{s})\subset \cCFL(L).
\]
We write 
\[
U_i=W_iZ_i
\] 
and we observe that the variables $U_1,\dots, U_n$ preserve Alexander grading.  Therefore $\cCFL(L,\ve{s})$ is a free, finitely generated chain complex over $\bF[U_1,\dots, U_n]$. It is not hard to see that $U_i\simeq U_j$ as endomorphisms of $\cCFL(L,\ve{s})$ for all $i$ and $j$. See, for example,  \cite{BLZNonCuspidal}*{Equation~2.8}. 
 Therefore we view the homology group $\cHFL(L,\ve{s})$ as a module over $\bF[U]$, where $U$ acts by any of the $U_i$. 

\begin{define} 
The map $H_L\colon \bH(L)\to \Z$ is defined via the equation
\[
H_L(\ve{s})=-\frac{1}{2} \max \{\gr_{\ws}(\xs): \xs\in \cHFL(L,\ve{s}), \xs \text{ is $U$ non-torsion}\}.
\]
\label{def:H function}
\end{define}

\subsection{L-space links}
\label{sec:background-L-space-link}

We recall that a rational homology 3-sphere is an \emph{L-space} if $\HF^-(Y,\frs)\iso \bF[U]$ for each $\frs\in \Spin^c(Y)$. A knot $K\subset S^3$ is called an \emph{L-space knot}  if $S_n^3(K)$ is an L-space for all $n\gg 0$.  Being an L-space knot puts a strong restriction on the knot Floer homology of $K$, as evidenced by the following result of Ozsv\'{a}th and Szab\'{o}:

\begin{thm}[\cite{OSlens}] \label{thm:L-space-knot-OS}
 If $K\subset S^3$ is an L-space knot, then $\cCFK(K)$ is a staircase complex, and hence determined completely by the Alexander polynomial $\Delta_K(t)$.
\end{thm}

In particular, using the Dehn surgery formulas of Ozsv\'{a}th and Szab\'{o} \cite{OSIntegerSurgeries} \cite{OSRationalSurgeries}, it is straightforward to compute the Heegaard Floer homologies of all Dehn surgeries on L-space knots. More background on staircase complexes is presented in Section~\ref{sec:staircase-complexes}.

 In this paper, we consider the analogous notion for links:

\begin{define}
 If $L\subset S^3$ is a link, we say that $L$ is an \emph{L-space link} if $S^3_{\Lambda}(L)$ is an L-space for all sufficiently positive framings $\Lambda$.
\end{define}

Note that using the surgery exact triangle, it is straightforward to show that a knot $K$ is an L-space knot if and only if it admits a single positive, integral surgery which is an L-space. This turns out not to be the case for links.

L-space links have been previously studied by a number of authors   \cite{GorskyNemethiLattice} \cite{GorskyNemethiAlgebraicLinks} \cite{GorskyHom} \cite{LiuLSpaceLinks} \cite{BorodzikGorskyImmersed} \cite{LiuB} \cite{BLZLattice} \cite{CavalloLiu}, though the theory is not as well understood as for L-space knots.

One helpful result concerning L-space links is the following:
\begin{lem}[\cite{LiuLSpaceLinks}*{Lemma~1.10}]\label{lem:sublinks} If $L\subset S^3$ is an L-space link, then all sublinks of $L$ are also L-space links. 
\end{lem}

A natural question is whether Theorem~\ref{thm:L-space-knot-OS} generalizes to links. One result in this spirit is the following:

\begin{thm}[\cite{BLZLattice}*{Theorem~1.2}] \label{thm:free-res-plumbed}
If $L\subset S^3$ is a plumbed L-space link, the chain complex $\cCFL(L)$ is homotopy equivalent to a free resolution of its homology as an $R_n$-module.
\end{thm}

Since staircase complexes are easily seen to be free resolutions of their homology, Theorem~\ref{thm:free-res-plumbed} naturally generalizes Theorem~\ref{thm:L-space-knot-OS}. A natural question is whether Theorem~\ref{thm:free-res-plumbed} holds for non-plumbed L-space links. In Section~\ref{sec:2-component-L-space}, we answer this question in the affirmative for 2-component L-space links.

We now describe how to compute the $R_n$-module $\cHFL(L)$ for an L-space link. 
Before presenting the technique, it is helpful to reformulate the condition of being an L-space link in terms of $\cHFL(L)$:

\begin{lem} A link $L\subset S^3$ is an L-space link if and only if $\cHFL(L,\ve{s})\iso \bF[U]$ for all $\ve{s}\in \bH(L)$. Equivalently, $L$ is an L-space link if and only $\cHFL(L)$ is free as an $\bF[U]$-module. 
\end{lem}
\begin{proof} The proof is well-known, but we include it for the benefit of the reader. The large surgeries formula of Manolescu and Ozsv\'{a}th \cite{MOIntegerSurgery}*{Theorem~12.1}, which identifies $\cHFL(L,\ve{s})$ with $\HF^-(S^3_{\Lambda}(L),\frs_{\ve{s}})$ for all sufficiently large $\Lambda$ (relative to $\ve{s}$). Here $\frs_{\ve{s}}$ denotes a particular $\Spin^c$ structure on $S^3_{\Lambda}$, depending on $\ve{s}$. The second statement follows from the first since 
\[
\cHFL(L)=\bigoplus_{\ve{s}\in \bH(L)} \cHFL(L,\ve{s}). 
\]
\end{proof}

In particular, we observe that for an L-space link $L\subset S^3$, the group $\cHFL(L)$ has rank 0 or 1 over $\bF$ in each $(\gr_{\ws},A)$ multi-grading. Since the actions of $W_i$ and $Z_i$ are homogeneously graded, the $R_n$-module structure on $\cHFL(L)$ is completely determined by the function $H_L$ for an L-space link $L$. A concrete generators and relations description of $\cHFL(L)$ in terms of the $H_L$-function is given in \cite{BLZLattice}*{Section~7.1}.

Gorsky and N\'{e}methi \cite{GorskyNemethiAlgebraicLinks} give a concrete description of the function $H_L$ for an L-space link $L$ in terms of the Alexander polynomials of $L$ and its sublinks:

\begin{prop}[\cite{GorskyNemethiAlgebraicLinks}*{Theorem~2.10}] \label{prop:GNH-function}Let $L$ be an oriented L-space link in $S^3$. Then
\[
H_L(\ve{s})=\sum_{L'\subset L} (-1)^{|L'|-1} \sum_{\substack{\ve{s}'\in \bH(L') \\ \ve{s}'\ge \pi_{L,L'}(\ve{s}+\ve{1})}} \chi(\HFL^-(L',\ve{s}')).
\]
\end{prop}

In the above, $\HFL^-(L')$ denotes the homology of $\cCFL(L')/(Z_1=\cdots=Z_m=0)$. In particular $\chi(\HFL^-(L',\ve{s}'))$ is a coefficient of a monomial in a suitable normalization of the multivariable Alexander polynomial by \cite{OSLinks}*{Section~1.1}. 
Furthermore, $\pi_{L,L'}\colon \bH(L)\to \bH(L')$ is the function which omits each component associated to link components of $L\setminus L'$, and which modifies the $i$-th component (whenever $K_i\subset L'$) via the transformation
\[
s_i\mapsto s_i-\frac{\lk(K_i, L\setminus L')}{2}.
\]
Also $\ve{1}$ denotes $(1,\dots, 1)$.

\subsection{The knot and link surgery formulas}

We now briefly review the knot and link surgery formulas of Manolescu, Ozsv\'{a}th and Szab\'{o} \cite{OSIntegerSurgeries} \cite{OSRationalSurgeries} \cite{MOIntegerSurgery}.

Given a knot $K$ in $S^3$ (or more generally a null-homologous knot in a 3-manifold $Y$) Ozsv\'{a}th and Szab\'{o} \cite{OSIntegerSurgeries} describe a chain complex $\bX_n(K)$ which is homotopy equivalent to $\ve{\CF}^-(S^3_n(K))$ and which is defined in terms of the knot Floer homology of $K$. The complex $\bX_n(K)$ takes the form
\[
\bX_n(K)=\Cone\left(v+h_n\colon \bA(K)\to \bB(K)\right).
\]
The complex $\bA(K)$ decomposes as
\[
\bA(K)=\prod_{s\in \Z}A_s(K).
\]
Here $A_s(K)$ is equal to the subspace of $\cCFK(K)$ in Alexander grading $s$, tensored over $\bF[U]$ with $\bF\llsquare U\rrsquare$. Similarly $\bB(K)$ decomposes as
\[
\bB(K)=\prod_{s\in \Z} B_s(K),
\]
where each $B_s(K)$ is a copy of $\CF^-(S^3)\otimes_{\bF[U]} \bF\llsquare U\rrsquare$. The map $v$ sends $A_s(K)$ to $B_s(K)$, and $h_n$ sends $A_{s}(K)$ to $B_{s+n}(K)$.

We now discuss the link surgery formula of Manolescu and Ozsv\'{a}th \cite{MOIntegerSurgery}.  For a link $L\subset S^3$ with integer framing $\Lambda$, Manolescu and Ozsv\'{a}th describe a chain complex $\cC_{\Lambda}(L)$ which computes $\CF^-(S^3_{\Lambda}(L))\otimes_{\bF[U]} \bF\llsquare U\rrsquare$. The complex $\cC_{\Lambda}(L)$ is an $n=|L|$ dimensional hypercube of chain complexes. As a hypercube, $\cC_{\Lambda}(L)$ decomposes over cube points $\veps\in \bE_n$. We write $\cC_{\Lambda}(L)=\bigoplus_{\epsilon\in \bE_n}\cC_{\veps}$. Frequently one identifies $\bE_n$ with the set of sublinks of $L$. We do this so that $(0,\dots, 0)\in \bE_n$ corresponds to $L\subset L$, and $(1,\dots, 1)$ corresponds to $\emptyset\subset L$.

The structure maps $D_{\veps,\veps'}$ further decompose over \emph{oriented} sublinks of $L$. If $L_{\veps'}\subset L_{\veps}$ are sublinks of $L$ and $M=L_{\veps}\setminus L_{\veps'}$, then 
\[
D_{\veps,\veps'}=\sum_{\vec{M}\in \fro(M)} \Phi^{\vec{M}}
\]
where $\fro(M)$ denotes the set of orientations on $M$.

Finally, each complex $\cC_{\veps}\subset \cC_{\Lambda}(L)$ decomposes as a direct product of chain complexes 
\[
\cC_{\veps}=\prod_{\ve{s}\in \bH(L)} \cC_{\veps}(\ve{s}).
\]
Recall that $\bH(L)$ is defined in Equation~\eqref{eq:Alexander-grading-def} and is the set of Alexander gradings on $\cCFL(L)$.

A choice of orientation on $L$ determines an identification of $H_1(S^3\setminus L)\iso \Z^{n}$. If $\vec{M}\subset L$ is an orientated sublink of $L$, then a Morse framing on $L$ allows us to construct an element $\Lambda_{\vec{M}}\in H_1(S^3\setminus L)$, obtained by pushing $M$ off of itself into $S^3\setminus L$. The map $\Phi^{\vec{M}}$ maps Alexander grading $\ve{s}$ to Alexander grading $\ve{s}+\Lambda_{\vec{M}}$.

Finally, we briefly describe the complexes $\cC_{\veps}$. The complex $\cC_{(0,\dots,0)}$ is isomorphic to a completion of $\cCFL(L)$. If $\veps$ corresponds to the sublink $L'\subset L$, then the complex $\cC_{\veps}$ can be identified (up to homotopy equivalence of $\bF[U_1,\dots, U_n]$-chain complexes) with the localization
\begin{equation}
\cC_{\veps}\simeq  S^{-1}_{\veps}\cdot \cCFL(L)
\label{eq:C_veps-def}
\end{equation}
where $S^{-1}_{\veps}\subset \bF[W_1,Z_1,\dots, W_n,Z_n]$ is the multiplicatively closed subset generated by $Z_i$ for components $K_i\subset L$ such that $\veps_i=1$.  See \cite{ZemBordered}*{Lemma~7.5}. 

It is sometimes helpful to view Equation~\eqref{eq:C_veps-def} instead as a twisted link Floer complex. That is, if we set $Z_i=T_i$ and $U_i=W_iZ_i$, we can instead view $\cC_{\veps}$ as a twisted and multi-pointed version of the link Floer complex of a sublink of $L$.
We let $L_{\veps}$ denote the sublink of $L$ consisting of components $K_i\subset L$ such that $\veps_i=1$. Then $\cC_{\veps}$ can be viewed as the link Floer complex of $L_{\veps}$, twisted by the classes of the link components $L_i$ with $\veps_i=1$, and with $|\veps|$ extra (non-link) basepoints added.  Note that since we are considering a link $L$ in $S^3$, the additional twisted coefficients have the effect of tensoring with an extra copy of $\bF[T^{-1},T]$ for each component of $L\setminus L_\veps$. In more detail,  let $\ps_{\veps}$ denote $\{w_i: \veps_i=1\}\subset \ws$. 
Then there is a homotopy equivalence
\begin{equation}
\cC_{\veps}\iso \cCFL(L_{\veps}, \ps_{\veps})\otimes \bigotimes_{ \{i| \veps_i=1\} } \bF[T_i, T_i^{-1} ].
\label{eq:twisted-Floer-identification}
\end{equation}
See \cite{ZemExact}*{Lemma~7.2}.

 \subsection{Gradings on the link surgery formula}
 \label{sec:Alexander-gradings}
 
 In this section, we describe several Maslov and Alexander gradings on the link surgery formula.
 
  We have already described an Alexander grading taking values in $\bH(L)$. This is related to the identification in Equation~\eqref{eq:C_veps-def}. Note that the identification is only well-defined up to overall multiplication by powers of $Z_i$ for $\veps_i=1$. These isomorphisms may be normalized by requiring the maps $\Phi^{+K_i}$ to preserve the Alexander grading. Therefore we will sometimes refer to this grading as the \emph{$\sigma$-normalized Alexander grading}. We will write 
  \[
  A^{\sigma}(\xs)\in \bH(L)
  \]
   for this grading.
 
 Additionally, the identification in Equation~\eqref{eq:twisted-Floer-identification} gives a $\bH(L_\veps)$-valued Alexander grading on $\cC_{\veps}$, which we refer to as the \emph{sublink normalized Alexander gradings}. (Here we give each $T_i$ Alexander grading 0).
 We write
 \[
A^{\veps}(\xs)\in \bH(L_{\veps})
 \]
 for this grading. 
 
 The gradings $A^\sigma$ and $A^\veps$ are related by the formula
 \begin{equation}
A^{\veps}(\xs)= \pi_{L,L_{\veps}}\left(A^\sigma(\xs)\right),\label{eq:A-veps-A-sigma-relation}
 \end{equation}
 where $\pi_{L,L_{\veps}}\colon \bH(L)\to \bH(L_{\veps})$ is the map which sends $\ve{s}=(s_1,\dots, s_n)$ to the tuple consisting of
 \[
 s_i-\frac{\lk(K_i,L\setminus L_\veps)}{2}
 \]
 ranging over $i$ with $\veps_i=0$.
 
 Additionally, there is a Maslov bigrading $(\gr_{\ws},\gr_{\zs})$. This is induced by the identification in Equation~\eqref{eq:twisted-Floer-identification}. Note that the maps $\Phi^{+K_i}$ preserves $\gr_{\ws}$, while the maps $\Phi^{-K_i}$ preserve $\gr_{\zs}$. Note that the $\gr_{\ws}$ grading is also induced by the identification in Equation~\eqref{eq:C_veps-def}, but $\gr_{\zs}$ is not. 
 
 Note that the Alexander and Maslov gradings on $\cC_{\veps}$ are related by the formula
 \begin{equation}
 \frac{\gr_{\ws}-\gr_{\zs}}{2}=\sum_{K_i\subset L_{\veps}} A_i^\veps.
 \label{eq:Relation-Maslov-Alexander}
 \end{equation}
 
 See \cite{ZemBordered}*{Lemma~7.4} for related results about gradings.
 
\subsection{The surgery algebra}
\label{sec:background-surgery-algebra}

We now describe an alternate perspective on the link surgery formula which is due to the second author \cite{ZemBordered}. This perspective reinterprets the surgery formula in terms of modules over an algebra, called the \emph{surgery algebra} $\cK$, defined as follows.

The surgery algebra $\cK$ is an algebra over the idempotent ring $\ve{I}=\ve{I}_0\oplus \ve{I}_1$, where $\ve{I}_{\veps}\iso \bF$. We set
\[
\ve{I}_0\cdot \cK\cdot \ve{I}_0=\bF[W,Z],\quad  \ve{I}_1 \cdot \cK \cdot \ve{I}_0=\bF[U,T,T^{-1}]\otimes_{\bF} \langle \sigma,\tau\rangle,
\]
\[
\ve{I}_1\cdot \cK\cdot \ve{I}_1=\bF[U,T,T^{-1}], \quad \text{and} \quad \ve{I}_0\cdot \cK\cdot \ve{I}_1=0.
\]
We now describe the multiplication in the algebra. Multiplication of two elements in $\ve{I}_\veps\cdot \cK\cdot \ve{I}_\veps$ is given by polynomial multiplication. Similarly 
\[
U^iT^j\cdot U^sT^t \sigma=U^{i+s}T^{j+t} \sigma
\]
and similarly when $\tau$ replaces $\sigma$ in the above formula. We additionally have the relations
\begin{equation}
\sigma W=UT^{-1} \sigma,\quad \sigma  Z=T \sigma, \quad \tau W=T^{-1} \tau, \quad \tau Z=UT \tau.
\label{eq:algebra-relations}
\end{equation}

It is convenient to package the above relations in terms of two maps
\[
\phi^\sigma,\phi^\tau\colon \bF[W,Z]\to \bF[U,T,T^{-1}],
\]
which are given by the formulas
\begin{equation}
\phi^{\sigma}(W^iZ^j)=U^i T^{j-i}\quad \text{and} \quad \phi^\tau(W^iZ^j)=U^jT^{j-i}. \label{eq:def-phisigma-phitau}
\end{equation}
Then the relations from Equation~\eqref{eq:algebra-relations} can be rewritten as
\[
\sigma\cdot a=\phi^\sigma(a)\cdot \sigma\quad \text{and} \quad \tau \cdot a=\phi^\tau(a)\cdot \tau.
\]

\begin{rem}
 In \cite{ZemBordered}, sometimes a different generating set for the algebra is used. Therein, $\ve{I}_0\cdot \cK\cdot \ve{I}_0\iso \bF[\scU,\scV]$, and $\ve{I}_1\cdot \cK \cdot \ve{I}_1\iso \bF[\scU,\scV,\scV^{-1}]$. In terms of our present notation, we identify $W=\scU$ and $Z=\scV$ in $\ve{I}_0\cdot \cK\cdot \ve{I}_0$. In $\ve{I}_1\cdot \cK \cdot \ve{I}_1$, we identify $U=\scU \scV$ and $T=\scV$. 
\end{rem}

There are two natural ways to topologize the algebra $\cK$, which lead to slightly different theories. The first is to complete at the ideal $(U)$, i.e. define a basis of opens to be the two sided ideals $(U^i)\subset \cK$. We refer to this as the \emph{$(U)$-adic topology}. Multiplication determines a continuous map
\[
\mu_2\colon \cK\otimes^!\cK\to \cK.
\]

Additionally, there is another topology that we can put on $\cK$ so that multiplication is continuous as a map
\[
\mu_2\colon \cK\vecotimes \cK\to \cK.
\]
In \cite{ZemExact}, this is referred to as the \emph{chiral topology}. The chiral topology is generated by the right ideals $(J_n)_{n\in \N}$ given by the following formulas:
\[
\begin{split}
\ve{I}_0\cdot J_n\cdot \ve{I}_0&=(W^n,Z^n)\subset \bF[W,Z]\\
\ve{I}_1\cdot J_n \cdot \ve{I}_0&= (\ve{I}_1\cdot \cK\cdot\ve{I}_0)\cdot (W^n,Z^n)\\
\ve{I}_1\cdot J_n\cdot \ve{I}_1&=(\ve{I}_1\cdot \cK\cdot \ve{I}_1) \cdot U^n.
\end{split}
\]

\begin{rem}
In \cite{ZemBordered}, the second author considers the chiral topology exclusively. In \cite{ZemExact}, the second author considers both the chiral and $U$-adic topologies (therein the algebra with the $U$-adic topology is denoted $\mathfrak{K}$). In this paper, we focus mostly the chiral algebra, though in Section~\ref{sec:identity-cobordism} we consider both the $U$-adic and chiral topologies. 
\end{rem}

\subsection{Modules over the surgery algebra}

We now describe some basic modules over the surgery algebra. If $K$ is a knot in a 3-manifold $Y$ which is equipped with Morse framing $\lambda$, there is a finitely generated type-$D$ module 
\[
\cX_{\lambda}(Y,K)^{\cK}.
\]
When $K$ is null-homologous, the data of $\cX_{\lambda}(Y,K)^{\cK}$  is essentially equivalent to the data of the mapping cone formula $\bX_{\lambda}(Y,K)$, defined by Ozsv\'{a}th and Szab\'{o} \cite{OSIntegerSurgeries}. We usually write $n$ in place of $\lambda$ in this case. For general $K$, the construction is given in \cite{ZemExact}. We briefly recall the construction of $\cX_{n}(Y,K)^{\cK}$ when $Y$ is an integer homology 3-sphere. We define
\[
\cX_{n}Y,K)\cdot \ve{I}_0
\]
to be the $\bF$ span of a free $\bF[W,Z]$ basis of $\cCFK(K)$. The space
\[
\cX_{n}(Y,K)\cdot \ve{I}_1
\]
consists of the $\bF$ span of a free $\bF[U]$-basis of $\CF^-(Y)$. 

The structure map $\delta^1$ of $\cX_{n}(Y,K)^{\cK}$ counts several terms. The terms of the differential $\delta^1$ which are weighted by elements of $\ve{I}_0\cdot \cK\cdot \ve{I}_0$ correspond to the differential of $\cCFK(K)$. The terms of the differential $\delta^1$ which are weighted by $\ve{I}_1\cdot \cK\cdot \ve{I}_1$ correspond to the differential of $\CF^-(Y)$. (In particular, in the model we are describing, no $T$ powers appear in terms of $\delta^1$ weighted by elements of $\ve{I}_1\cdot \cK\cdot \ve{I}_1$). 

Finally, there are the terms of $\delta^1$ which are weighted by multiples of $\sigma$ and $\tau$. To define these terms, we view the complex $\bB(K)$ appearing in the mapping cone formula as a completion of $\CF^-(Y)\otimes \bF[T,T^{-1}]$.

 If $\xs\in \cX_{n}(Y,K)\cdot \ve{I}_0$, and $v(\xs)$ has a summand of $U^i T^j \cdot \ys$, then $\delta^1(\xs)$ has a summand of $\ys\otimes U^i T^j \sigma$. Similarly if $h_{n}(\xs)$ contains a summand of $U^iT^j \cdot \ys$, then $\delta^1(\xs)$ contains a summand of $\ys\otimes U^iT^j \tau$. See \cite{ZemBordered}*{Section~8} for more details.

As an example, the type-$D$ module of an $n$-framed unknot in $S^3$ takes the form
\[
\cX_n(U)^{\cK}:=\begin{tikzcd}\xs_0 \ar[r, "\sigma+T^n \tau"] & \xs_1 \end{tikzcd}
\]
where $\xs_\veps$ is in idempotent $\veps\in \{0,1\}$. 

We now describe the type-$A$ module for a 0-framed solid torus ${}_{\cK} \cD$ (denoted ${}_{\cK} \cD_0$ in \cite{ZemBordered}*{Section~8.2}). We set 
\[
\ve{I}_0\cdot \cD=\bF\llsquare W,Z\rrsquare\qquad \text{and} \quad \ve{I}_1\cdot \cD=\bF\llsquare U,T,T^{-1}\rrsquare.
\]
The actions of $\ve{I}_\veps\cdot \cK\cdot \ve{I}_\veps$ are polynomial multiplication, in the obvious manner. The algebra elements $\sigma$ and $\tau$ act by $\phi^\sigma$ and $\phi^\tau$ as follows. If $x\in \bF\llsquare W,Z\rrsquare $, we set
\[
m_2(\sigma, x)=\phi^\sigma(x)\in \bF\llsquare U,T,T^{-1}\rrsquare=\ve{I}_1\cdot \cD.
\]
We define $m_2(\tau ,x)$ similarly.

When we work with the chiral topology on $\cK$, we give $\cD$ the cofinite basis topology (see Examples~\ref{ex:cofinite}) with respect to the $\bF$ basis of generators $(W^iZ^j)_{i,j\in \N}$ and $(U^iT^j)_{i\in \N, j\in \Z}$. 

If we work in the $U$-adic topology on $\cK$, then we need to topologize $\cD$ differently. In this case, we give $\cD$ the $U$-adic topology on $\bF[W,Z]\oplus \bF[U,T,T^{-1}]$.

The above type-$D$ and $A$ modules recover Ozsv\'{a}th and Szab\'{o}'s surgery formulas for Morse framed knots in the sense that there is a canonical isomorphism
\[
\cX_{n}(Y,K)^{\cK}\boxtimes {}_{\cK} \cD\simeq \bX_{n}(Y,K),
\]
whenever $K$ is rationally null-homologous. When $K$ is homologically essential, it follows from \cite{ZemExact}*{Theorem~1.5} that there is a homotopy equivalence
\[
\cX_{n}(Y,K)^{\cK}\boxtimes {}_{\cK} \cD\simeq \ve{\CF}^-(Y). 
\]

Manolescu and Ozsv\'{a}th's link surgery formula \cite{MOIntegerSurgery} also has a natural interpretation in terms of the algebra $\cK$. In this setting, the data of the link surgery complex $\cC_{\Lambda}(L)$ can naturally be interpreted in terms of a type-$D$ module
\[
\cX_{\Lambda}(S^3,L)^{\cK\otimes_{\bF}\cdots \otimes_{\bF} \cK}.
\]
The complex $\cC_{\Lambda}(L)$ can be recovered by tensoring the above type-$D$ module with $|L|$-copies of the type-$A$ module ${}_{\cK} \cD$. 

\subsection{Tensor product formulas}
\label{sec:tensor-product}
In \cite{ZemBordered}, the second author described several connected sum formulas for the link surgery formula. The most basic version of these formulas computes the link surgery complex of a connected sum of links $L_1\#L_2$ in terms of the surgery complexes of $L_1$ and $L_2$.

 To state the connected sum formula, we first introduce a bimodule ${}_{\cK|\cK} \bI^{\Supset}$. Ignoring the completions, this is a type-$A$ module over the algebra $\cK\otimes_{\bF} \cK$. The underlying $\ve{I}\otimes \ve{I}$-modules is as follows. We set 
\[
(\ve{I}_0|\ve{I}_0)\cdot \bI^{\Supset}=\bF\llsquare W,Z\rrsquare\quad \text{and} \quad (\ve{I}_1|\ve{I}_1)\cdot \bI^{\Supset}=\bF\llsquare U,T,T^{-1}\rrsquare.
\]
The module ${}_{\cK|\cK} \bI^{\Supset}$ vanishes in other idempotents. The action $m_2$ is supported only on tensors $(a|a')\otimes \xs$ where $a,a'\in \cK$ are elements in a single idempotent. In this case, we set $m_2(a|a', \xs)=aa'\cdot \xs$ (using ordinary polynomial multiplication). Additionally, the module $\bI^{\Supset}$ has an additional $m_3$, determined by the formulas
\[
m_3(\sigma|1,1|\sigma, \xs)=\phi^\sigma(\xs),\qquad m_3(1|\sigma, \sigma|1, \xs)=0\quad \text{and}
\]
\[
m_3(\tau|1,1|\tau, \xs)=\phi^\tau(\xs)\quad \text{and} \quad m_3(1|\tau, \tau|1, \xs)=0.
\]

For our purposes, we will need the following version of the connected sum formula:
\begin{thm}\label{thm:connected-sum} Suppose that $L_1,L_2\subset S^3$ are links. Then there is a homotopy equivalence
\[
\cX_{\Lambda_1+\Lambda_2}(L_1\# L_2)^{\otimes^{n_1+n_2-1} \cK}\boxtimes {}_{\cK} \cD\simeq \left(\cX_{\Lambda_1}(L_1)^{\otimes^{n_1} \cK}\otimes \cX_{\Lambda_2}(L_2)^{\otimes^{n_2} \cK}\right)\boxtimes {}_{\cK|\cK} \bI^{\Supset}. 
\]
\end{thm}

The above follows by combining \cite{ZemBordered}*{Theorems~12.1 and 13.1} 

\subsection{Invariance and the sublink surgery formula}

We now discuss two related properties of the link surgery modules which are proven in \cite{ZemExact}. The first is invariance of the modules under Dehn surgery. Given an arbitrary link $L$ in a closed 3-manifold $Y$ (not necessarily null-homologous) with Morse framing $\Lambda$, the construction from \cite{ZemExact} gives a type-$D$ module $\cX_{\Lambda}(Y,L)^{\cK\otimes\cdots\otimes \cK}$. This invariant is shown to be an invariant of $(Y,L,\Lambda)$ up to homotopy equivalence in \cite{ZemExact}*{Theorem~1.3}. Similarly, the $DA$-bimodules obtained by tensoring the above module with ${}_{\cK|\cK} \bI^{\Supset}$ are also invariants of $(Y,L,\Lambda)$. In \cite{ZemExact}*{Theorem~1.5}, it is proven that $\cX_{\Lambda}(Y,L)$ are natural with respect to performing Dehn surgery, in the following sense:

\begin{thm}
\label{thm:invariance}
Suppose that $(Y,L,\Lambda)$ is a Morse framed link, with a distinguished component $K\subset L$, and let $\lambda=\Lambda|_K$. Let $L'=L\setminus K$, viewed as a link in the 3-manifold $Y':=Y_{\lambda}(K)$. Write $\Lambda'=\Lambda|_{L'}$. Then
\[
\cX_{\Lambda'}(Y',L')^{\otimes^{n-1} \cK}\simeq \cX_{\Lambda}(Y,L)^{\otimes^n \cK}\boxtimes {}_{\cK} \cD.
\]
\end{thm}

Additionally, a very helpful property of the surgery modules is that there is a natural equivalence
\begin{equation}
\cX_{\lambda}(Y,K)^{\cK}\cdot \ve{I}_0\simeq \cCFK(Y,K)^{\bF[W,Z]}.
\label{eq:sublink-surgery-formula}
\end{equation}
See \cite{ZemExact}*{Section~1.4}.

\section{2-component L-space links}

\label{sec:2-component-L-space}

In this section, we study 2-component L-space links. In this section, we describe a procedure for computing the bimodule ${}_{\cK} \cX_{\Lambda}(L)^R$ when $L\subset S^3$ is an L-space link.  Here $\Lambda$ is a  Morse framing on $L$ (though only the framing on one component affects the module). 

The organization is as follows.  In Sections~\ref{sec:Y0}, ~\ref{sec:Y1} and~\ref{sec:definition-L-sigma/tau}, we define a candidate bimodule
\[
{}_{\cK} \cY_{\Lambda}(L)^{R},
\] which is combinatorial and computable from the $H$-function of $L$. After describing the candidate module ${}_{\cK} \cY_{\Lambda}(L)^R$, we prove the following:

\begin{thm}
\label{thm:isomorphism-resolution}
There is a homotopy equivalence ${}_{\cK} \cX_{\Lambda}(L)^{R}\simeq {}_{\cK} \cY_{\Lambda}(L)^{R}$.
\end{thm}

We will prove that the following result is a consequence of the above theorem:

\begin{thm}
\label{thm:proof-free-resolution} If $L$ is a 2-component L-space link in $S^3$, then ${}_{R_2}\cCFL(L)$ is a free resolution of its homology. Equivalently, ${}_{R_2}\cCFL(L)$ is formal.
\end{thm}

See \cite{BLZLattice}*{Section~2} for background material on formality of chain complexes in a closely related context.

\begin{rem}
 The construction of the candidate module generalizes rather straightforwardly to the case of $n$-component L-space links in $S^3$. In this case, there is a candidate module
 \[
 {}_{\cK} \cY_{\Lambda}(L)^{R_{n-1}}.
 \]
 Our proof of Theorem~\ref{thm:isomorphism-resolution} does \emph{not} seem to extend to $n>2$. We are able to describe the failure of our proof in terms of the vanishing of certain $\Ext$ groups. For specific $L$, it may be possible to prove Theorem~\ref{thm:isomorphism-resolution} by computing these groups.
\end{rem}

\begin{rem} It would be natural to ask whether one could extend these results to compute the entire $DA$-bimodules ${}_{\cK} \cX_{\Lambda}(L)^{\cK}$ for 2-component L-space links. This is a somewhat more subtle question. For example, in \cite{ZemBordered}*{Section~16}, it is proven that the full $DA$-bimodule ${}_{\cK} \cX_{\Lambda}(H)^{\cK}$  for the Hopf link depends non-trivially on the choice of arc system, whereas Theorem~\ref{thm:isomorphism-resolution} shows that ${}_{\cK} \cX_{\Lambda}(H)^{R}$ does not. 
\end{rem}

\subsection{Staircases complexes}
\label{sec:staircase-complexes}
In this section we describe some background about staircase complexes, and maps between them. 

We recall that a \emph{staircase complex} $\cS$ consists of a finitely generated free chain complex over $R:=\bF[W,Z]$ (i.e. a finitely generated type-$D$ module over $R$). The complex $\cS$ has a free $R$-basis $\xs_0,\ys_1,\xs_2,\dots, ,\ys_{2n-1},\xs_{2n}$ with respect to which the differential takes the following form:
\[
\d(\xs_{2i})=0\quad \text{and} \quad \d(\ys_{2i+1})=\xs_{2i}\otimes W^{\a_{i}}+\xs_{2i+2}\otimes Z^{\b_{i}}
\]
for some $\a_{i},\b_{i}>0$. 

We say that a staircase complex has \emph{standard gradings} if it is equipped with a $(\gr_{\ws},\gr_{\zs})$-bigrading so that $\gr_{\ws}(\xs_{0})=0$ and $\gr_{\zs}(\xs_{2n})=0$. This is the case whenever $\cS$ is the staircase complex of a knot in $S^3$. 

Note that we can write a staircase complex as a mapping cone complex $\cS_1\to \cS_0$, where $\cS_1$ is spanned by the $\ys_{i}$ and the $\cS_0$ is spanned by the $\xs_i$.  We  refer to the subscript $\veps$ of $\cS_\veps$ as the \emph{algebraic grading} of $\cS$.

 The following lemma is helpful for our purposes:

\begin{lem} \label{lem:monomial-ideals-resolutions}
\item
\begin{enumerate}
\item If $\cS$ is a staircase with standard gradings, then $H_*(\cS)$ is canonically isomorphic to a monomial ideal in $\bF[W,Z]$. 
\item A staircase complex $\cS$ is a free resolution over $R$ of its homology, i.e., if $\cS=(\cS_1\to \cS_0)$ is a staircase complex, then there is an exact sequence
\[
\begin{tikzcd}0\ar[r] &\cS_1\ar[r, "\d"] & \cS_0\ar[r, twoheadrightarrow] & \cS_0/\im \cS_1=H_*(\cS).
\end{tikzcd}
\] 
\end{enumerate}
\end{lem}
\begin{proof} Both claims are straightforward and left to the reader. See Figure~\ref{fig:cable6} for an example.
\end{proof}

\begin{figure}[h]
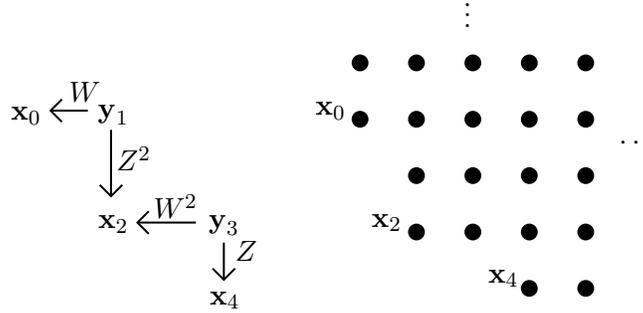
\caption{A staircase complex (left) and its homology (right). On the right, each dot represents a generator over $\bF$. The action of $W$ shifts rightward, and the action of $Z$ shifts upward.}
\label{fig:cable6}
\end{figure}

We now prove several straightforward lemmas which will help us later on.

\begin{lem}\label{lem:homology-to-chain}
 Let $\cA$ and $\cB$ be staircase complexes and let $f\colon H_*(\cA)\to H_*(\cB)$ be an $R$-equivariant map which is homogeneously graded. Then $f$ is induced by a chain map $F\colon \cA\to \cB$ which preserves the algebraic grading. Any two maps $F,G\colon \cA\to \cB$ which induce the same map on homology and preserve the algebraic grading are chain homotopic.
\end{lem}

\begin{proof}
Staircase complexes are free resolutions of their homology. Therefore the claim follows from the homological algebra of free resolutions.
\end{proof}

\begin{lem} Let $f\colon \cA\to \cB$ be a chain map between staircase complexes. Suppose that $f$ preserves the algebraic grading,  and $f=0$ in algebraic grading 0. Then $f=0$.
\end{lem}
\begin{proof} The map $f$ vanishes on $\cA_0$ so $f\circ \d=0$. Since $f$ is a chain map we have $\d \circ f+f\circ \d=\d\circ f=0$ on $\cA_1$. Since $\d$ is injective on $\cB_1$, it follows that $f=0$ on $\cA_1$ and hence on all of $\cA$. 
\end{proof}

\begin{lem}\label{lem:increase-cube-grading}
 Suppose that $f\colon \cA\to \cB$ is a chain map between staircase complexes which increases the algebraic grading. Then $f=0$.
\end{lem}
\begin{proof} By assumption, $f$ can only be non-trivial on $\cA_0$. Since $\d=0$ on $\cA_0$, and $f$ is a chain map, we conclude that $\d\circ f=0$.  Since $\d$ is injective on $\cB_1$, we conclude that $f=0$ on $\cA_0$, and hence on all of $\cA$.
\end{proof}

\subsection{The module $\cY^0(L)$}
\label{sec:Y0}

In this section, we define the module ${}_{R} \cY^0(L)^{R}$. This will be the idempotent 0 portion of ${}_{\cK} \cY_{\Lambda}(L)^{R}$.

 We begin by describing the underlying type-$D$ module $\cY^0(L)^{R}$, gotten by forgetting about the type-$A$ action of $R$. The module $\cY^0(L)^R$ will take the form
\[
\cY^0(L)^{R}:=\bigoplus_{s\in \lk(L_1,L_2)/2+\Z} \cC_s^{R},
\]
for a collection of staircase complexes $\cC_s^R$, which we describe presently. For each $s\in \Z+\lk(L_1,L_2)/2$, the staircase complex is determined by the function $H_L(s,-)$ . The generators of $\cC_s$ are as follows:
\begin{enumerate}[label=($y_0$-\arabic*), ref=$y_0$-\arabic*]
\item\label{Y-0} For each $s\in \Z+\lk(L_1,L_2)/2$, we consider the set of $t$ such that 
\[
H_L(s,t+1)=H_L(s,t)\quad \text{and} \quad H_L(s,t-1)=H_L(s,t)+1.
\]
For each $t$, as above, we add a generator $\xs_{(s,t)}$ to $\cC_s$ with
\[
\gr_{\ws}(\xs_{(s,t)})=-2H_L(s,t)\qquad \text{and} \quad A(\xs_{(s,t)})=(s,t).
\]
We enumerate these generators as $\xs_0,\xs_2,\dots, \xs_{2n}$, ordered so that $\gr_{\ws} (\xs_i)>\gr_{\ws}(\xs_{i+2})$.
\item\label{Y-1} Additionally, we add $n$ generators $\ys_1,\ys_{3},\dots, \ys_{2n-1}$ to $\cC_s$ and declare
\[
\d \ys_{2i+1}=\xs_{2i}\otimes \a_{2i+1}+ \xs_{2i+2}\otimes \b_{2i+1}
\]
for some non-zero monomials $\a_{2i+1}$ and $\b_{2i+1}$ in $R=\bF[W,Z]$, determined uniquely by the property that
\[
\gr(\xs_{2i})+\gr(\a_{2i+1})=\gr(\xs_{2i+2})+\gr(\b_{2i+1})=\gr(\ys_{2i+1})-(1,1),
\] 
where $\gr=(\gr_{\ws},\gr_{\zs})$. 
\end{enumerate}

We now describe the left type-$A$ action of $R$ on ${}_R \cY^0(L)^R$. For each monomial $a\in R$ and $s\in \Z+\lk(L_1,L_2)/2$, we define
 \[
\delta_2^1(a,-)\colon \cC_s^{R}\to \cC_{s+A(a)}^{R},
\]
by picking a chain map
\[
L_a\colon \cC_s^{R}\to \cC_{s+A(a)}^{R}
\]
 which is non-zero on homology and has Maslov bigrading
\[
(\gr_{\ws}, \gr_{\zs})(L_a)=(\gr_{\ws}, \gr_{\zs})(a).
\] 
We set $\delta_2^1(a,-)=L_a$. Note that any such $L_a$ will preserve the algebraic grading.

For monomials $a,b\in F$, we define $\delta_3^1(a,b,-)$ to be any choice of type-$D$ morphism
\[
h_{a,b}\colon \cC_{s}^R\to \cC_{s+A(a)+A(b)}^R
\]
which satisfies
\begin{equation}
\d(h_{a,b})=L_{a} \circ L_{b}+L_{ab}. \label{eq:def-h-a-b}
\end{equation}
and which increases algebraic grading by 1 and has Maslov bigrading 
\[
\gr(h_{a,b})=\gr(a)+\gr(b)+(1,1).
\]
Such a map $h_{a,b}$ exists by Lemma~\ref{lem:homology-to-chain}. 
We declare $\delta_n^1$ to vanish if $n>3$. Note that in Equation~\eqref{eq:def-h-a-b} and elsewhere in the paper, we write $\d(h_{a,b})$ for the morphism differential applied to $h_{a,b}$, i.e. $\d(h_{a,b})=\delta^1_1\circ h_{a,b}+h_{a,b}\circ \delta^1_1$. 

\begin{prop}
\label{prop:DA-relations}
If $L$ is a 2-component L-space link,  then ${}_R\cY^0(L)^R$ satisfies the type-$DA$ structure relations. Furthermore,  ${}_{R}\cY^0(L)^{R}$ is independent of the choices in the construction, up to homotopy equivalence. 
\end{prop}
\begin{proof} The structure relations for one module input and no algebra inputs amount to the claim that each $\cC_s$ is a chain complex. The structure relations for one algebra input and one module input follow from the fact that $L_W$ and $L_Z$ are chain maps. The structure relations for two algebra inputs and one module input follow from the fact that $\delta_3^1(a,b,-)$ is a chain homotopy between $\delta_2^1(a,\delta_2^1(b,-))$ and $\delta_2^1(ab,-)$. The structure relations for more inputs are satisfied because each summand of the structure relation for three or more algebra inputs  would increase algebraic grading by at least 2, and all such maps vanish. 

We now consider the statement that the construction produces a unique module, up to homotopy equivalence.  We assume that we use maps $L_a$ and $h_{a,b}$ to construct $\cY^0(L)$, and that we use $L_a'$ and $h_{a,b}'$ to construct a model $\cZ^0(L)$. 
We now define a morphism of $DA$-bimodules
\[
F_*^1\colon {}_{R}\cY^0(L)^{R}\to {}_{R}\cZ^0(L)^{R}.
\]
The two modules are identified as graded vector spaces, so we set $F_1^1(\xs)=\xs\otimes 1$, for all $\xs$. Next, we observe that for all monomials $a\in \bF[W,Z]$, the maps  $L_a$ and $L_a'$ are chain homotopic by Lemma~\ref{lem:homology-to-chain}. We declare therefore $F_2^1(a,-)$ to be any suitably graded chain homotopy between $L_a$ and $L_a'$, which increases algebraic grading by 1. We declare $F_n^1=0$ for $n>2$. A similar argument as above shows that $F_*^1$ is a chain map. By adapting the above line of reasoning, it is straightforward to see that $F_*^1$ is a homotopy equivalence.
\end{proof}

\subsection{The module $\cY^1(L)$}
\label{sec:Y1}
We now describe the idempotent 1 portion of $\cY(L)$. We denote this as $\cY^1(L)$, which we view as a type-$DA$ module 
\[
{}_{\bF[U,T,T^{-1}]} \cY^1(L)^{R}.
\]

 Write $\cS^R$ for the complex $\cCFK(L_2)^{R}$, where $L_2\subset L$ is the second component. Firstly, we write
 $\bF[\Z+\lk(L_1,L_2)/2]$ for the vector space with generators $T^s$ for $s\in \Z+\lk(L_1,L_2)/2$. This is naturally an $\bF[T,T^{-1}]$-module. We define
\[
\cY^1(L)^{R}:=\bF[\Z+\lk(L_1,L_2)/2]\otimes_{\bF} \cS^{R}
\]
 and write $\cS_s$ for $T^s\otimes \cS$.
 
We give this a type-$A$ action of $\bF[U,T,T^{-1}]$ by declaring $T$ to act only on the $\bF[\Z+\lk(L_1,L_2)/2]$ factor, and declaring $\delta_2^1(U^i,\xs)=\xs\otimes U^i$ for all $\xs\in \cY^1(L)$ and $i\in \N$. If $a\in \bF[U,T,T^{-1}]$, we write $L_a\colon \cS_s\to \cS_{s+A(a)}$  for $\delta_2^1(a,-)$, as defined above.

 We equip $\cY^1(L)^R$ with the $(\gr_{\ws},\gr_{\zs})$ grading induced from its identification with $\bF[\Z+\lk(L_1,L_2)/2]\otimes \cCFK(L_2)^{R}$ (declaring $(\gr_{\ws},\gr_{\zs})(T^s)=(0,0)$).

 \subsection{The actions of $\sigma$ and $\tau$}
 \label{sec:definition-L-sigma/tau}

We now describe the actions involving $\sigma$ and $\tau$. Similar to the actions of $W$ and $Z$ in idempotent 0, these are specified by their gradings. Namely, for each $s$ we define $L_{a\sigma}:=\delta_2^1(a\sigma,-)$ by picking a map
\[
L_{a\sigma}\colon \cC_s^{R}\to \cS_{s+A(a)}^{R}
\]
which has 
\[
\gr_{\ws}(L_{\a\sigma})=\gr_{\ws}(a)
\quad \text{and} \quad \gr_{\zs}(L_{a\sigma})=\gr_{\zs}(a)+2s+\lk(L_1,L_2),
\]
and which is furthermore not null-homotopic. (Such a map is unique up to chain homotopy). We define $L_{a\tau}:=\delta_2^1(a\tau,-)$ by picking a map
\[
L_{a\tau}\colon \cC_s^{R}\to \cS_{s+\lambda_1+A(a)}^{R}
\]
which has
\[
\gr_{\ws}(L_{a\tau})=\gr_{\ws}(a)+\lk(L_1,L_2)-2s\quad \text{and} \quad \gr_{\zs}(L_{\a\tau})= \gr_{\zs}(a).
\]
In the above, we are writing $\lambda_1$ for the first component of the framing $\Lambda=(\lambda_1,\lambda_2)$.

Next, we define $\delta_3^1$ terms involving $\sigma$ and $\tau$ similarly to how we defined $\delta_3^1$ in idempotent 0. For example, we set $\delta_3^1(a \sigma, b,-)$ to be any choice of type-$D$ morphism
\[
h_{a\sigma,b}\colon \cC_s^R\to \cS_{s+A(a)+A(b)}^R
\]
to be any map which has $\gr_{\ws}$-grading $\gr_{\ws}(a)+\gr_{\ws}(b)+1$ and which has $\gr_{\zs}$-grading equal to $\gr_{\zs}(a)+\gr_{\zs}(b)+2s+\lk(L_1,L_2)+1$, which furthermore increases the algebraic grading by $1$, and which satisfies
\[
\d(h_{a\sigma ,b})=L_{a \sigma b}+L_{a\sigma}\circ L_{b}.
\]
(Note that we can always select $\delta_3^1(a, b\sigma,-)=0$).

 We set $\delta_j^1=0$ for $j>3$.

 Essentially the same argument as in Proposition~\ref{prop:DA-relations} yields the following:

\begin{prop} The module ${}_{\cK} \cY_{\Lambda}(L)^{R}$ satisfies the type-$DA$ structure relations. Furthermore, the construction above specifies this module uniquely, up to homotopy equivalence. 
\end{prop}

\subsection{$\bF[U]$-equivariant models}
\label{sec:U-equivariant}

There is a convenient set of choices in the construction of the module ${}_{\cK} \cY_{\Lambda}(L)^R$ which make the actions $U$-equivariant and also encodable in a relatively small collection of maps. In all of the examples computed in this paper, we use this construction. 

For each $s$ we pick non-zero chain maps
 \[
L_{W}\colon \cC_s^{R}\to \cC_{s-1}^{R}\qquad \text{and} \qquad L_{Z}\colon \cC_s^{R}\to \cC_{s+1}^{R}
\]
of $(\gr_{\ws}, \gr_{\zs})$ bigrading $(-2,0)$ and $(0,-2)$, respectively. We also pick maps 
\[
h_{W,Z},h_{Z,W}\colon \cC_{s}^{R}\to \cC_{s}^{R}
\]
of bigrading $(-1,-1)$ such that
\[
\d(h_{W,Z})=L_W\circ L_Z+\id\otimes U\qquad \text{and} \qquad \d(h_{Z,W})=L_Z\circ L_W+\id\otimes U. 
\]
We may then define $\delta_2^1$ to be $\bF[U]$-equivariant and satisfy
\[
\delta_2^1(Z^n,-)=\overbrace{L_Z\circ \cdots \circ L_Z}^n\qquad \text{and} \qquad
\delta_2^1(W^n,-)=\overbrace{L_W\circ \cdots \circ L_W}^n.
\]
We define $\delta_3^1$ to be $\bF[U]$-equivariant and to satisfy
\[
\begin{split}
\delta_3^1(W^i, W^j, -)&=0\\
\delta_3^1(Z^i, Z^j, -)&=0\\
\delta_3^1(W^i, Z^j,-)&=\sum_{k=1}^{\min(i,j)} U^{k-1} L_W^{i-k}\circ h_{W,Z}\circ L_Z^{j-k}\\
\delta_3^1(Z^i, W^j, -)&=\sum_{k=1}^{\min(i,j)} U^{k-1} L_Z^{i-k} \circ h_{Z,W}\circ L_W^{j-k}
\end{split}
\]

We can extend construction to encode the actions of $\sigma$ and $\tau$ as follows. We first pick chain maps $L_\sigma$ and $L_\tau$ as described in the construction of ${}_{\cK} \cY_{\Lambda}(L)^R$. We then pick (families of) maps
\[
h_{\sigma,W}, h_{\sigma,Z},  h_{\tau,W}, h_{\tau, Z}.
\]
The map $h_{\sigma,W}$ satisfies
\[
\d(h_{\sigma,W})=L_{\sigma}\circ L_W+UT^{-1} L_{\sigma},
\]
and similarly for the other three maps. We then define the map $\delta_3^1$ as
\[
\delta_3^1(U^iT^j \sigma, U^n W^m)=\sum_{s=0}^{m-1} U^{n+i+s-1} T^{j-s+1}  h_{\sigma,W}\circ \underbrace{L_{W}\circ\cdots\circ L_W}_{m-s-1}.
\]
The other $\delta_3^1$ terms are similarly determined by the maps $h_{\sigma,Z}$, $h_{\tau,W}$, $h_{\tau,Z}$ and the operators $L_W$ and $L_Z$.

\subsection{Koszul duality}

We now recall some well-known facts about the Koszul duality between the exterior algebra and the polynomial ring. Priddy's work \cite{PriddyKoszul} is the original historical reference. Let $A$ denote the polynomial ring $\bF[X_1,\dots, X_n]$ and let $\Lambda$ denote the exterior algebra $\Lambda^*(\theta_1,\dots, \theta_n)$. 

There is a $DD$-bimodule ${}^{A} \Lambda^A$, whose differential is given by
\[
\delta^{1,1}(\theta_i)=X_i\otimes 1\otimes 1+1\otimes 1\otimes X_i,
\]
extended over all of $\Lambda$ via the Leibniz rule.

The following can be viewed as a consequence of the Koszul duality between the polynomial ring and the exterior algebra:
\begin{lem}
\label{lem:quasi-inverses} Let $A=\bF[X_1,\dots, X_n]$ and $\Lambda=\Lambda^*(\theta_1,\dots, \theta_n)$, as above. The bimodules ${}_A A_A$ and ${}^A \Lambda^A$ are quasi-inverses, i.e.
\[
{}_{A} A_A\boxtimes {}^{A} \Lambda^A\simeq {}_A [\bI]^{A}\quad \text{and} \quad {}^{A} \Lambda^A\boxtimes {}_{A}A_A\simeq {}^{A}[\bI]_A.
\]
\end{lem}
\begin{proof} The proof is well-known, though we describe it for the benefit of the reader. We first observe that it is sufficient to show the claim when $n=1$. This is because the algebra $A$ is the external tensor product (i.e. tensor product over $\bF$) of $n$ copies of $\bF[X]$, and similarly $\Lambda$ is the external tensor product of $n$ copies of $\Lambda^*(\theta)$. The modules ${}_A A_A$ and ${}^A \Lambda^A$ are themselves isomorphic to external tensor products. Since the external tensor product of a collection of homotopy equivalences is a homotopy equivalence, it suffices to show the claim for $n=1$.

In this case, the $DA$-bimodule ${}_{\bF[X]} \bF[X]_{\bF[X]} \boxtimes {}^{\bF[X]} \Lambda^{\bF[X]}$ has underlying chain complex
\[
\begin{tikzcd}[column sep=2cm]
\bF[X]\otimes \theta \ar[r,"d"] &\bF[X]\otimes 1
\end{tikzcd}
\]
where 
\[
d(X^i\otimes \theta)=(X^i\otimes 1)\otimes X+(X^{i+1}\otimes 1)\otimes 1.
\]
There is a strong deformation retraction of the above complex onto the type-$D$ module $\bF^{\bF[X]}$ (with vanishing differential). The map $\Pi$ sends $X^i\otimes 1$ to $1\otimes X^i$. The map $I$ sends $1$ to $(1\otimes 1)\otimes 1$. The homotopy $H$ sends $X^i\otimes 1$ to $\sum_{j=0}^{i-1} (X^{j}\otimes \theta)\otimes X^{i-j-1}$. Homological perturbation theory gives an induced $DA$-bimodule structure on $\bF^{\bF[X]}$, which is homotopy equivalent to ${}_{\bF[X]} \bF[X]_{\bF[X]}\boxtimes{}^{\bF[X]} \Lambda^{\bF[X]}$, and it is straightforward to see that it coincides with the identity bimodule ${}_{\bF[X]} [\bI]^{\bF[X]}$.  
\end{proof}

\begin{rem} When working over the 2-variable polynomial ring $R=\bF[W,Z]$, we will typically write $\Lambda^*(\theta_W,\theta_Z)$ for the Koszul dual.
\end{rem}

\subsection{Free resolutions}

We now prove that the module ${}_{R} \cY^0(L)^{R}$ is homotopy equivalent to a free resolution of  $H_*(\cCFL(L))$.

We now recall some standard terminology from homological algebra.  Let ${}_R M_R$ be an $(R,R)$-bimodule (i.e. a type-$AA$ bimodule with only $m_{1|1|0}$ and $m_{0|1|1}$ non-trivial). We say that a pair $({}^{R} \cF^R,\epsilon)$ is a \emph{free resolution} of ${}_R M_R$ if the following hold:
\begin{enumerate}
\item  ${}^{R} \cF^R$ is a $DD$-bimodule and can be written as
\[
\begin{tikzcd}
\cdots
	\ar[r ,"\delta^{1,1}"] 
&
 {}^R\cF_2{}^R
 	\ar[r,"\delta^{1,1}"] 
&
 {}^R\cF_1^R
 	 \ar[r,"\delta^{1,1}"]
& 
{}^R\cF_0^R
\end{tikzcd}
\]
\item The map $\epsilon$ is a quasi-isomorphism
\[
\epsilon\colon {}_R R_R \boxtimes {}^R \cF^R\boxtimes {}_R R_R \to {}_R M_R
\]
which is supported on $R\boxtimes \cF_0\boxtimes R$.
\end{enumerate} 

Note that the terminology is justified by the fact that if ${}^R \cF^R$ is a $DD$-bimodule, then ${}_R R_R\boxtimes {}^R \cF^R\boxtimes {}_R R_R$ may be viewed as a free $dg$-module over $R\otimes R$.

\begin{prop} Let ${}^R\cF^R$ be a free resolution of ${}_R H_*(\cCFL(L))_R$ (where we view $H_*(\cCFL(L))$ as having no higher actions). There is a homotopy equivalence
\[
{}_{R} \cY^0(L)^{R}\simeq {}_{R}R_{R}\boxtimes {}^{R} \cF^{R}.
\]
\end{prop}
\begin{proof}
Write $R=\bF[W,Z]$. We consider the quasi-invertible bimodule ${}^{R} \Lambda^R$, from Section~\ref{lem:quasi-inverses} whose quasi-inverse is ${}_{R} R_R$.  We consider the tensor product
\[
{}^R \cF^{R}:={}^{R} \Lambda^R\boxtimes {}_{R} \cY^0(L)^{R}. 
\]
This is a type-$DD$ module. It suffices to show that the above is a free resolution of $H_*(\cCFL(L))$. As a first step, we observe that $\cF$ has an algebraic grading. The algebraic grading $\gr_{\alg}$ is gotten by setting
\[
\gr_{\alg}(1)=0\quad \gr_{\alg}(\theta_W)=\gr_{\alg}(\theta_Z)=1 \qquad \text{and} \qquad \gr_{\alg}(\theta_W\theta_Z)=2.
\]
We define an algebraic grading on $\Lambda\boxtimes \cY^0(L)$ tensorially, using the algebraic grading on $\cY^0(L)$ described in Section~\ref{sec:Y0} (where we view each staircase as being concentrated in algebraic gradings $0$ and $1$). We observe that $\delta^{1,1}$ decreases the algebraic grading on $\cF$ by 1. Therefore, we may write
\[
{}^R\cF^R=
\left(\begin{tikzcd}
{}^R\cF_3^R
	\ar[r ,"\delta^{1,1}"] 
&
 {}^R\cF_2{}^R
 	\ar[r,"\delta^{1,1}"] 
&
 {}^R\cF_1^R
 	 \ar[r,"\delta^{1,1}"]
& 
{}^R\cF_0^R
\end{tikzcd}\right).
\]

It suffices to show the homology of $R\boxtimes \cF\boxtimes R$ is supported in $\cF_0$.  The homology of $\cY^0(L)\boxtimes R$ is easily seen to be supported in algebraic grading 0. On the other hand, the  homotopy equivalence
\[
f_*^1\colon {}_{R}[\bI]^{R}\to {}_{R} R_R\boxtimes {}^{R} \Lambda^R
\]
from Lemma~\ref{lem:quasi-inverses} has $f_1^1(1)=(1_R\otimes 1_\Lambda)\otimes 1.$ The map
\[
f_*^1\boxtimes \bI\colon {}_{R}[\bI]^R\boxtimes {}_{R} \cY^0(L)^R\boxtimes {}_R R_R\to {}_RR_R\boxtimes {}^R\Lambda^R\boxtimes {}_{R} \cY^0(L)^R\boxtimes {}_RR_R
\] 
is a quasi-isomorphism. Furthermore, $f_*^1\boxtimes \bI$ sends the generators of homology of $\cY^0(L)\boxtimes R$ to algebraic grading 0 in $R\boxtimes \Lambda\boxtimes \cY^0(L)\boxtimes R$. Therefore the homology of $R\boxtimes \Lambda\boxtimes \cY^0(L)\boxtimes R$ is supported in algebraic grading 0, completing the proof. 
\end{proof}

\subsection{L-space knots}

Before proving Theorem~\ref{thm:isomorphism-resolution} for 2-component L-space links, we recall the situation of L-space knots. We recall the following theorem of Ozsv\'{a}th and Szab\'{o}:

\begin{thm}[\cite{OSlens}]\label{thm:free-res-L-space-knot} If $K$ is an L-space knot in $S^3$, then $\cCFK(K)$ is homotopy equivalent to a staircase complex.
\end{thm} 
We now give a proof of the above using Koszul duality, which will be a helpful starting point for our proof of Theorem~\ref{thm:isomorphism-resolution} for 2-component links.

\begin{proof}[Proof of Theorem~\ref{thm:free-res-L-space-knot}] Let $K$ be an L-space knot. We consider two \emph{a-priori} different $A_\infty$-modules. Firstly, we define the $R$-module ${}_R H(K)$ as the homology $\cHFK(K)$, equipped with its natural $R$-action, but no higher actions.

Next, we write ${}_R \cHFK(K)$ for the $A_\infty$-module obtained whose higher actions are obtained by applying homological perturbation to the free chain complex ${}_R\cCFK(K)$, so that there is a homotopy equivalence of $A_\infty$-modules
\[
{}_{R} \cHFK(K)\simeq {}_{R} \cCFK(K).
\]

Since $K$ is an L-space knot, $H(K)$ is supported in even gradings. There is a chain map from $\cCFK(K)$ to $\CF^-(S^3)$ obtained by sending $W^iZ^j$ to $U^i$. This map preserves $\gr_{\ws}$ and sends $\bF[U]$-non-torsion elements to $\bF[U]$-non-torsion elements. Therefore $H(K)$ is supported in non-positive $\gr_{\ws}$-gradings. A similar argument shows that $H(K)$ is supported in non-positive $\gr_{\zs}$-gradings. Therefore $H(K)$ can naturally be viewed as a monomial ideal in $\bF[W,Z]$. We recall from Lemma~\ref{lem:monomial-ideals-resolutions} that a free resolution of ${}_RH(K)$ is therefore a staircase complex. Quasi-isomorphisms are invertible in the category of $A_\infty$-modules over $R$  \cite{KellerNotes}*{Section~4}. See also \cite{SeidelFukaya}*{Corollary~1.14} for a similar result. Therefore we conclude that ${}_R H(K)$ is homotopy equivalent as a type-$A$ module to a staircase complex. 
 To prove Theorem~\ref{thm:free-res-L-space-knot}, it therefore suffices to show that there is a homotopy equivalence of $A_\infty$-modules
\begin{equation}
{}_R H(K)\simeq {}_R \cHFK(K).
\label{eq:H-versus-cHFK}
\end{equation}

By Lemma~\ref{lem:quasi-inverses}, the bimodule ${}^{R} \Lambda^R$ is quasi-invertible, so it suffices to show Equation~\eqref{eq:H-versus-cHFK} after tensoring both modules with ${}^{R} \Lambda^R$.  We claim that  in fact there is an equality
\begin{equation}
{}^{R} \Lambda^R\boxtimes {}_R H(K)={}^{R} \Lambda^R\boxtimes {}_R\cHFK(K).\label{eq:equality-after-boxing}
\end{equation}
To see this, observe that $(\delta^{1,1})^n\colon \Lambda\to R^{\otimes n}\otimes  \Lambda\otimes R^{\otimes n}$ vanishes if $n>2$. Therefore only $m_2$ and $m_3$ of ${}_R H(K)$ and ${}_R \cHFK(K)$ can contribute to the box tensor product. On the other hand, the underlying groups of both ${}_R H(K)$ and ${}_R \cHFK(K)$ are concentrated in even gradings. By assumption, the maps $m_2(a,-)$ coincide on the two modules. Since the algebra $R$ is itself concentrated in even gradings, it follows that for all $a,b\in R$, the map $m_3(a,b,-)$ will shift the $\gr_{\ws}$ and $\gr_{\zs}$ gradings by an odd integer. Therefore the map $m_3(a,b,-)$ is zero for all $a$ and $b$. Since no higher multiplications contribute to the tensor products in Equation~\eqref{eq:equality-after-boxing}, Theorem~\ref{thm:free-res-L-space-knot} follows.
\end{proof} 

\subsection{Proof of Theorem~\ref{thm:isomorphism-resolution}}

We now consider L-space links with 2 components and prove Theorem~\ref{thm:isomorphism-resolution}.  The proof takes three steps:
\begin{enumerate}
\item Show that ${}_{R} \cY^0(L)^{R}$ is homotopy equivalent to ${}_{R}[\ve{I}_0]^{\cK}\boxtimes {}_{\cK} \cX_\Lambda(L)^R$.
\item Show that ${}_{\bF[U,T,T^{-1}]} \cY^1(L)^{R}$ is homotopy equivalent to ${}_{\bF[U,T,T^{-1}]}[\ve{I}_1]^{\cK}\boxtimes {}_{\cK} \cX_\Lambda(L)^R$.
\item Show that the actions of $\sigma$ and $\tau$ on ${}_{\cK}\cY(L)^R$ and ${}_{\cK}\cX(L)^R$ coincide up to homotopy.
\end{enumerate}

Let us write 
\[
{}_R \cX^0(L)^R:={}_{R}[\ve{I}_0]^{\cK}\boxtimes {}_{\cK} \cX_\Lambda(L)^R\quad \text{and} \quad {}_{\bF[U,T,T^{-1}]} \cX^1(L)^R:={}_{\bF[U,T,T^{-1}]}[\ve{I}_1]^{\cK}\boxtimes {}_{\cK} \cX_\Lambda(L)^R
\]

In the above, ${}_R[\ve{I}_0]^{\cK}$ denotes the bimodule induced by the inclusion of $R$ into the idempotent 0 subspace of $\cK$. Similarly ${}_{\bF[U,T,T^{-1}]} [\ve{I}_1]^{\cK}$ denotes the inclusion of $\bF[U,T,T^{-1}]$ into the idempotent 1 subspace of $\cK$. Here, the vector space of $[\ve{I}_\veps]$ is the subspace of the idempotent ring $\ve{I}_{\veps}\subset \ve{I}$.

\subsubsection{Idempotent 0}

In this section we prove the following:

\begin{prop}
\label{prop:step-1} There is a homotopy equivalence ${}_R \cY^0(L)^{R}\simeq {}_R \cX^0(L)^R$.
\end{prop}

 We recall that  ${}_{R} \cX^{0}(L)^R$ can be constructed by viewing ${}^{R} \cCFL(L)^{R}$ as a $DD$-bimodule and tensoring on the left with ${}_R R_R$. We will write $\cX^0(L)^{R}$ and $\cY^0(L)^{R}$ for the underlying type-$D$ modules of ${}_{R} \cX^0(L)^R$ and ${}_R \cY^0(L)^{R}$, respectively. We may assume, via homological perturbation theory, that $\cX^0(L)^R$ is \emph{reduced} as a type-$D$ module (i.e. no components of $\delta_1^1$ are weighted by $1\in R$). 

We will now show that there is a homotopy equivalence of type-$D$ modules
\[
\cX^0(L)^R\simeq \cY^0(L)^R.
\]
In fact, this homotopy equivalence will be unique up to chain homotopy.

We consider the generators of both ${}_R \cX^0(L)^R$ and ${}_R \cY^0(L)^R$ which lie in a fixed $A_1$-Alexander grading, i.e. those which have Alexander grading $(s_1,*)$ for a fixed $s_1$. The differential on $\cCFL(L)$ preserves the Alexander bigrading. Similarly, our $\delta_1^1$ map on ${}_R \cY^0(L)^R$ also preserves the Alexander bigrading. Therefore, as type-$D$ modules, we can decompose over the $A_1$-grading as follows:
 \[
 \cY^0(L)^R=\prod_{q}\cC_q^R,\qquad \cX^0(L)^R=\prod_{q} \cT_q^R.
 \]
Here the direct products are over $q\in \Z+\tfrac{\lk(L_1,L_2)}{2}$.

 \begin{lem} For all $q$, there is a homotopy equivalence $\cC_q^R\simeq \cT_q^R$.
 \end{lem}
 \begin{proof} By construction, both
 \[
 \cC_q^R\boxtimes {}_R R_R\qquad \text{and} \qquad \cT_q^R\boxtimes {}_R R_R
 \]
 have homology isomorphic to $\cHFL_{(q,*)}(L)_R$ as an $R$-module (i.e. ignoring higher actions). Since $L$ is an L-space link, the right $R$-module $\cHFL_{(q,*)}(L)_R$ is concentrated in only even bigradings. The argument from our proof of Theorem~\ref{thm:free-res-L-space-knot} implies that both $\cC_q^R$ and $\cT_q^R$ are staircase complexes. Since staircase complexes are free resolutions of their homologies, and the homologies of $\cC_q$ and $\cT_q$ coincide, we conclude that $\cC_q$ and $\cT_q$ are homotopy equivalent.
 \end{proof}

 \begin{rem} Note that two staircase complexes which are homotopy equivalent are in fact identical as staircase complexes (i.e. have same step lengths). Also Lemma~\ref{lem:homology-to-chain} implies that the homotopy equivalence is unique up to chain homotopy.
 \end{rem}

 \begin{lem} The isomorphism $\cY^0(L)^R\simeq \cX^0(L)^{R}$ intertwines the map $L_W$ with $\delta_2^1(W,-)$ up to chain homotopy. Similarly the map intertwines $L_Z$  and $\delta_2^1(Z,-)$ up to chain homotopy.
 \end{lem}
 \begin{proof} In the construction of ${}_R \cY^0(L)^R$, we can pick $L_W$ to be any chain map with $(\gr_{\ws},\gr_{\zs})$ bigrading $(-2,0)$ from $\cC_q^R$ to $\cC_{q-1}^R$ which is non-zero on homology and preserves the algebraic grading. By Lemma~\ref{lem:homology-to-chain} such a map is determined up to chain homotopy.  On the other hand, the map $\delta_2^1(W,-)$ has this bigrading, preserves the algebraic grading (because it has even Maslov bigrading) and is non-zero on homology. Therefore, we may use the map $\delta_2^1(W,-)$ from ${}_{R}\cX^{0}(L)^R$ as our definition of $L_W$ in  ${}_R \cY^0(L)^R$. The same reasoning applies to $\delta_2^1(Z,-)$.
 \end{proof}
 
 Note that if we use a reduced model for $\cX^0(L)^{R}$, then we can choose  $\cX^0(L)^R$ to coincide with $\cY^0(L)^{R}$. Furthermore, by the previous lemma, we can choose $\delta_2^1(W,-)$ and $\delta_2^1(Z,-)$ to coincide for ${}_R\cX^0(L)^{R}$ and ${}_R\cY^0(L)^R$.

 We now compare the tensor products ${}^R \Lambda^R\boxtimes {}_R \cY^0(L)^R$ and ${}^R \Lambda^R\boxtimes {}_R \cX^0(L)^R$. The tensor product with ${}_R\cX^0(L)^R$ takes the following form:
 \begin{equation}
 \begin{tikzcd}[column sep=6cm, row sep=1.7cm, labels=description]
 \cX^0(L)^R
 	\ar[r, "{1\otimes \delta_2^1(W,-)+W\otimes \id}"]
 	\ar[dr, dashed, " {1\otimes (\delta_3^1(W,Z,-)+\delta_3^1(Z,W,-))}"] 
 	\ar[d, "{1\otimes \delta_2^1(Z,-)+Z\otimes \id}"]
 &\cX^0(L)^R
 	\ar[d, "{1\otimes \delta_2^1(Z,-)+Z\otimes \id}"]	
 	\\
 \cX^0(L)^R \ar[r, "{1\otimes \delta_2^1(W,-)+W\otimes \id}"]& \cX^0(L)^R
 \end{tikzcd}
\label{eq:H-cube}
 \end{equation}
The complex ${}^{R} \Lambda^R \boxtimes {}_R \cY^0(L)^{R}$ is identical except that along the diagonal we have
\[
1\otimes (h_{W,Z}+h_{Z,W}).
\]

To show the complexes ${}^R \Lambda^R \boxtimes {}_R \cY^0(L)^R$ and ${}^R \Lambda^R \boxtimes {}_R \cX^0(L)^R$ are homotopy equivalent it suffices to show that
\begin{equation}
h_{W,Z}+\delta_3^1(W,Z,-)\simeq 0 \quad \text{and}\quad  h_{Z,W}+\delta_3^1(Z,W,-)\simeq 0
\label{eq:maps-null-homotopic}
\end{equation}
as endomorphisms of $\cX^0(L)^R$. 

 Note that by construction, the maps $h_{W,Z}$ and $h_{Z,W}$ increase the algebraic grading by 1. \emph{A-priori} we do not know how $\delta_3^1(Z,W,-)$ and $\delta_3^1(W,Z,-)$ interact with the algebraic grading.  We observe, however, that both $h_{W,Z}+\delta_3^1(W,Z,-)$ and $h_{Z,W}+\delta_3^1(Z,W,-)$ are chain maps. These maps both shift the Maslov $(\gr_{\ws},\gr_{\zs})$-grading by $(-1,-1)$. 
 
 We focus on 
 \[
 f:=h_{W,Z}+\delta_3^1(W,Z,-)
 \]
first. We can view $f$ as a collection of chain maps
\[
f_s\colon \cC_s^R\to \cC_s^R
\]
ranging over $s\in \Z+\lk(L_1,L_2)/2$. 

Write $\cC_s$ as the cone of a map $\delta^1$ from $\cC_s^1$ to $\cC_s^0$. The map $f_s$ can be written as a sum of two maps $f_{s}^{01}+f_s^{10}$, where $f_s^{01}$ maps 
 $\cC_s^{1}$ to $\cC_s^0$, and $f_s^{10}$ maps $\cC_s^0$ to $\cC_s^1$.  See the below diagram:
\[
\begin{tikzcd}[column sep=1.5cm, row sep=1.7cm]
\cC_s^1 
	\ar[r, "\delta^1"]
	\ar[dr, "f_{s}^{01}",pos=.2,swap]
& 
\cC_s^0
	\ar[dl, "f_{s}^{10}", pos=.2,crossing over]
\\
\cC_s^1 
	\ar[r, "\delta^1"]
 & 
\cC_s^0
\end{tikzcd}
\]

Since $f_s$ is a chain map and $\delta^1_1$ is injective as a map from $\cC_s^1$ to $\cC_s^0$, we conclude that $f^{10}_s=0$. It therefore suffices to show that $f_s^{01}$ is chain homotopic to zero. To show this, we will in fact show that all endomorphisms of $\cC_s^R$ which are chain maps and have $(\gr_{\alg}, \gr_{\ws},\gr_{\zs})$-degree $(-1,-1,-1)$ are null-homotopic.

\begin{lem}
\label{lem:computation-ext-complex} Let $\cS^R$ be a finite staircase complex. Then $\Hom(\cS^R,\cS^R)$ has trivial homology in $(\gr_{\alg}, \gr_{\ws}, \gr_{\zs})$-degree $(-1,-1,-1)$. 
\end{lem}
\begin{proof}
Write $\xs_0,\ys_1,\xs_2, \dots,\ys_{2n-1},\xs_{2n}$ for the generators of $\cS^R$, so that $\d(\ys_{2i+1})=\xs_{2i}\otimes W^{\a_{2i+1}}+\xs_{2i+2}\otimes Z^{\b_{2i+1}}.$

A morphism in algebraic grading $-1$ is supported on the $\ys_i$ generators, and sends $\ys_i$ to a sum of the $\xs_j$ generators (tensored with elements of $R$). As a first step, we enumerate all morphisms in this grading. Let $f\in \Hom(\cS^R,\cS^R)$ be a map with degree $(-1,-1,-1)$. For each $\ys_{2i+1}$, we claim that there are exactly four possibilities of $f(\ys_{2i+1})$, as below:
\begin{enumerate}
\item $f(\ys_{2i+1})=0$;
\item $f(\ys_{2i+1})=\xs_{2i}\otimes W^{\a_{2i+1}}+\xs_{2i+2}\otimes Z^{\b_{2i+1}}$;
\item $f(\ys_{2i+1})=\xs_{2i}\otimes W^{\a_{2i+1}}$;
\item $f(\ys_{2i+1})=\xs_{2i+2}\otimes Z^{\b_{2i+1}}$.
\end{enumerate}
There are no other possibilities for $f(\ys_{2i+1})$ because the only generators of $\cS$ in algebraic grading 0 with Maslov grading $(\gr_{\ws}(\ys_{2i+1})-1,\gr_{\zs}(\ys_{2i+1})-1)$ are $\xs_{2i}\otimes W^{\a_{2i+1}}$ and $\xs_{2i+2}\otimes Z^{\b_{2i+1}}$. To see this, note that any other element in this grading would have a summand the form $\xs_j\otimes a$, for some $a\in \bF[W,Z]$, and therefore we would have in particular that
\[
\begin{split}
\gr_{\ws}(\xs_j)&\ge \gr_{\ws}(\xs_{2i+2}\otimes Z^{\b_{2i+1}})=\gr_{\ws}(\xs_{2i+2})\quad \text{and}\\
\ \gr_{\zs}(\xs_j)&\ge \gr_{\zs}(\xs_{2i}\otimes W^{\a_{2i+1}})=\gr_{\zs}(\xs_{2i}).
\end{split}
\]
Note that $\gr_{\ws}(\xs_j)\ge \gr_{\ws}(\xs_{2i+2})$ implies that $j\le 2i+2$ while $\gr_{\zs}(\xs_j)\ge \gr_{\zs}(\xs_{2i})$ implies that $j \ge 2i$. Therefore $j\in \{2i,2i+2\}$, as claimed.

We may define a chain homotopy $H$ by setting $H(\ys_{2i+1})=\ys_{2i+1}$ (for a fixed $i$) and setting $H$ to vanish on all other generators. We note that adding $\d(H)$ to $f$ has the effect of switching cases (1) and (2), and switching cases (3) and (4) without changing the value of $f$ on any other generators. 

Let $f_{1},f_3,\dots, f_{2n-1}$ denote endomorphisms of $\cS^R$, where $f_{2i+1}(\ys_{2i+1})=\xs_{2i}\otimes W^{\a_{2i+1}}$ and $f_{2i+1}$ vanishes on all other generators. The above argument shows that any element of $\Hom(\cS^R,\cS^R)$ in Maslov gradings $(-1,-1)$ and algebraic grading $-1$ will be chain homotopic to a sum of some of the $f_{2i+1}$'s. We now observe, however, that
\[
f_{2i+1}=\d(J),
\]
where $J(\xs_j)=\xs_j\otimes 1$ if $j<2i+1$, $J(\ys_j)=\ys_j\otimes 1$ if $j<2i+1$ and vanishes on other generators. Therefore all morphisms in the gradings of interest are boundaries, and the main claim follows.
\end{proof}

\begin{rem}
\label{rem:hom=ext} Note that since a staircase $\cS^R$ is quasi-isomorphic to its homology, we observe that the  subspace of $H_* \Hom(\cS^R,\cS^R)$ in algebraic grading $-1$ is isomorphic to the more familiar module $\Ext^1_R(H_*(\cS),H_*(\cS))$. Therefore Lemma~\ref{lem:computation-ext-complex} can be restated as saying that the subspace of $\Ext^1_R(H_*(\cS),H_*(\cS))$ in $(\gr_{\ws},\gr_{\zs})$-bigrading $(-1,-1)$ is trivial. 
\end{rem}

Finally, we conclude from Lemma~\ref{lem:computation-ext-complex} that Equation~\eqref{eq:maps-null-homotopic} is satisfied, so we may conclude that
\[
{}_R \cX^0(L)^R\simeq {}_R \cY^0(L)^R.
\]

\subsubsection{Idempotent 1}

In this section, we consider the idempotent 1 subspace of the module ${}_{\cK} \cX_{\Lambda}(L)^{R}$. Our main result is the following:

\begin{prop}
\label{prop:iso-idempotent-1} There is a homotopy equivalence of type-$DA$ bimodules
\[
{}_{\bF[U,T,T^{-1}]} \cX^1(L)^{R}\simeq {}_{\bF[U,T,T^{-1}]} \cY^1(L)^{R}
\]
\end{prop}

\begin{proof}
We recall how $\cX^1(L)$ is constructed. It is defined as the twisted knot Floer complex
\[
\underline{\cCFK}(S^3,K_2,w_1)^{\bF[U,T,T^{-1}]\otimes \bF[W,Z]}.
\]
Here, we have an extra free-basepoint $w_1$, which contributes the $U$ variable. Additionally, we use coefficients which are twisted by the class of $K_1$ (this contributes the $T$-variable). Since $K_1$ is null-homologous in $S^3$, we can trivialize the local coefficients and we see that there is a homotopy equivalence
\begin{equation}
\underline{\cCFK}(S^3,K_2,w_1)^{\bF[U,T,T^{-1}]\otimes \bF[W,Z]}\simeq \cCFK(S^3,K_2)^{\bF[W,Z]}\otimes (\begin{tikzcd}\xs\ar[r, "U+WZ"]& \ys \end{tikzcd}).
\label{eq:homotopy-equivalence-idempotent-1}
\end{equation}
In particular, the differential on the right hand side has no $T$-powers. (The tensor factor of $(\begin{tikzcd}\xs\ar[r, "U+WZ"]& \ys \end{tikzcd})$  is the result of adding a basepoint to $Y$; see \cite{OSLinks}*{Proposition~6.5}). 

We recall that changing from type-$D$ to type-$A$ is accomplished by tensoring with the type-$A$ module $(\ve{I}_1|\ve{I}_1)\cdot \bI^{\Supset}\iso \bF\llsquare U,T,T^{-1}\rrsquare$ described in Section~\ref{sec:tensor-product}.   It is straightforward to check that the $DA$-bimodule obtained by tensoring the right hand side of Equation~\eqref{eq:homotopy-equivalence-idempotent-1} is homotopy equivalent the bimodule
\[
\bF[T,T^{-1}]\otimes_{\bF} \cCFK(S^3,K_2)^{\bF[W,Z]},
\]
which we can also identify with $\bF[\Z+\lk(L_1,L_2)/2]\otimes_{\bF} \cCFK(S^3,K_2)^{\bF[W,Z]}$. Here, the structure map $\delta_1^1$ is the same as the structure map for $\cCFK(S^3,K_2)^{\bF[W,Z]}$, and 
\[
\delta_2^1(T^iU^j, T^s\otimes \xs)=(T^{s+i}\otimes \xs)\otimes U^j.
\]
Furthermore, $\delta_j^1=0$ for $j>2$.

Note that by Lemma~\ref{lem:sublinks}, the property of being an L-space link is closed under sublinks. Therefore $\cCFK(S^3,K_2)$ is itself a staircase complex. This is exactly the definition of ${}_{\bF[U,T,T^{-1}]}\cY^1(L)^{R}$ given in Section~\ref{sec:Y1}. 
\end{proof}

\subsubsection{The actions of $\sigma$ and $\tau$}

In this section, we consider the actions of $\sigma$ and $\tau$. It is helpful to recall that we can view the module ${}_{\cK} \cX_{\Lambda}(L)^{R}$ itself as a mapping cone
\[
{}_{\cK}\cX_{\Lambda}(L)^R\simeq \Cone(v_*^1+h_*^1\colon  {}_{\cK} \cX^0(L)^R\to {}_{\cK} \cX^1(L)^R)
\]
where we view ${}_{\cK} \cX^\veps(L)^R$ as being supported in a single idempotent, for $\veps\in \{0,1\}$. The morphism $v_*^1$ is the sum of all of the non-trivial actions on the full module ${}_{\cK} \cX_{\Lambda}(L)^R$ which have $\sigma$ as an algebra input, while $h_*^1$ encodes all actions which have $\tau$ as an input. The construction described in Section~\ref{sec:definition-L-sigma/tau} gives a specific algebraic candidate for the $DA$-module morphisms $v_*^1$ and $h_*^1$. We will write $\mathfrak{v}_*^1$ and $\mathfrak{h}_*^1$ for the maps constructed in Section~\ref{sec:definition-L-sigma/tau}. I.e. $\mathfrak{v}_2^1(\sigma,-)=L_\sigma$, and  $\mathfrak{h}_2^1(\tau,-)=L_\tau$. There are also $\mathfrak{v}_3^1$ and $\mathfrak{h}_3^1$, whose formulas coincide with the terms of $\delta_3^1$ described in Section~\ref{sec:Y1}.

\begin{prop}\label{prop:module-structure} The homotopy equivalences ${}_{R} \cX^0(L)^{R}\simeq {}_R \cY^0(L)^R$ and ${}_{\bF[U,T,T^{-1}]} \cX^1(L)^R\simeq {}_{\bF[U,T,T^{-1}]} \cY^1(L)^R$ constructed in Propositions~\ref{prop:step-1} and~\ref{prop:iso-idempotent-1} intertwine the maps $v_*^1$ and $h_*^1$ with $\mathfrak{v}_*^1$ and $\mathfrak{h}_*^1$, up to chain homotopy.
\end{prop}

We now describe the proof of Proposition~\ref{prop:module-structure}. We can give a parallel description of the type-$D$ module $\cX_{\Lambda}(L)^{\cK\otimes R}$ as a mapping cone:
\[
\cX_{\Lambda}(L)^{\cK\otimes R}= \Cone\left(v^1+h^1\colon \cX^0(L)^{\cK\otimes R}\to \cX^1(L)^{\cK\otimes R}\right).
\]
Here the superscript $\veps$ in $\cX^\veps(L)$ indicates the idempotent. Also, $\cX^0(L)$ is concentrated in idempotent 0, and can be viewed as a type-$D$ module $\cX^0(L)^{R\otimes R}.$ Similarly $\cX^1(L)$ can be viewed as a type-$D$ module $\cX^1(L)^{\bF[U,T,T^{-1}]\otimes R}$.  By construction, the map $v^1$ is induced by a homotopy equivalence
\[
V^1\colon \cX^0(L)^{R\otimes R}\boxtimes {}_{R} [\phi^\sigma]^{\bF[U,T,T^{-1}]}\to \cX^1(L)^{\bF[U,T,T^{-1}]\otimes R}.
\] 
The map $h^1$ is similarly induced by a homotopy equivalence
\[
H^1\colon \cX^0(L)^{R\otimes R}\boxtimes {}_{R} [\phi^\tau]^{\bF[U,T,T^{-1}]}\to \cX^1(L)^{\bF[U,T,T^{-1}]\otimes R}.
\]
Here, $\phi^\sigma,\phi^\tau\colon R\to \bF[U,T,T^{-1}]$ are the algebra homomorphisms described in Equation~\eqref{eq:def-phisigma-phitau}, and ${}_R[\phi^\sigma]^{\bF[U,T,T^{-1}]}$ and ${}_{R}[\phi^\tau]^{\bF[U,T,T^{-1}]}$ denote the corresponding rank 1 bimodules (see Definition~\ref{def:morphism->bimodule}).

We recall that by definition, ${}_{\cK}\cX_{\Lambda}(L)^{R}$ is obtained by tensoring the type-$D$ module $\cX_{\Lambda}(L)^{\cK\otimes R}$ with the type-$A$ module ${}_{\cK|\cK} \bI^{\Supset}$. In Corollary~\ref{cor:quasi-inverse-text}, below we will prove that the bimodule ${}_{\cK|\cK} \bI^{\Supset}$ is quasi-invertible. Therefore it follows that our candidate $DA$-bimodule ${}_{\cK} \cY_{\Lambda}(L)^{R}$ also admits a type-$D$ description as  $\cY_{\Lambda}(L)^{\cK\otimes R}$ (related to the type-$DA$ module ${}_{\cK} \cY_{\Lambda}(L)^{R}$ by tensoring with ${}_{\cK|\cK}\bI^{\Supset}$). We write $\frv^1$ and $\frh^1$ for the maps so that it takes the form of a mapping cone
\[
\cY_{\Lambda}(L)^{\cK\otimes R}=\Cone(\frv^1+\frh^1\colon \cY^0(L)^{\cK\otimes R}\to \cY^1(L)^{\cK\otimes R}).
\]
Note by construction of the surgery algebra, the maps $\frv^1$ and $\frh^1$ are induced by chain maps
\[
\frV^1\colon \cY^0(L)^{R\otimes R}\boxtimes {}_{R} [\phi^\sigma]^{\bF[U,T,T^{-1}]}\to \cY^1(L)^{\bF[U,T,T^{-1}]\otimes R}
\]
and
\[
\frH^1\colon \cY^0(L)^{R\otimes R}\boxtimes {}_{R} [\phi^\tau]^{\bF[U,T,T^{-1}]}\to \cY^1(L)^{\bF[U,T,T^{-1}]\otimes R}
\]

As in the proof of Proposition~\ref{prop:iso-idempotent-1}, $\cX^1(L)^{\bF[U,T,T^{-1}]\otimes R}$ is homotopy equivalent to the type-$D$ module
\[
\cCFK(S^3,K_2)^{R}\otimes (\begin{tikzcd}\xs\ar[r, "U+WZ"]& \ys \end{tikzcd})
\]
(Note here that we are abusing notation and writing $U\in \bF[U,T,T^{-1}]$ and $WZ\in R$ for distinct algebraic elements).
Since $K_2$ is an L-space knot, $\cCFK(S^3,K_2)$ is a staircase complex. We note that the maps $\frV^1$ and $\frH^1$ are obviously non-zero on homology. Therefore, Proposition~\ref{prop:module-structure} is a consequence of the following lemma:

\begin{lem} Let $\cS$ be a staircase complex over $R=\bF[W,Z]$, and write $\cS'$ for the complex
\[
\cS'=\begin{tikzcd}\cS \otimes \xs \ar[r, "WZ+U"]& \cS \otimes \ys\end{tikzcd}
\]
viewed as a type-$D$ module over $\bF[W,Z,U]$. Suppose that $F\colon \cS'\to \cS'$ is a $(\gr_{\ws},\gr_{\zs})$-grading preserving map which is non-zero on homology. Then $F\simeq \id$. 
\end{lem}

\begin{proof}
Write $\cS'$ as the tensor product of $\cS$ with a complex 
\[
\cC=\begin{tikzcd}[column sep=2cm] \bF[W,Z,U]\otimes\xs\ar[r, "WZ+U"]&  \bF[W,Z,U]\otimes\ys\end{tikzcd}.
\]
 We can write $\cS'$ as
\[
\begin{tikzcd} [labels=description ]
\cS'_2 \arrow[r,"d_2"]& \cS_1'\arrow[r, "d_1"]& \cS_0',
\end{tikzcd}
\]
where we give $\xs$ algebraic grading 1, $\ys$ algebraic grading 0, so that the grading $\cS'$ is the sum of the algebraic gradings from $\cS$ and $\begin{tikzcd} \xs \ar[r, "U+WZ"] & \ys\end{tikzcd}$.

Write $\d_{\cS}$ for the differential on $S$, and write $\d_{\cC}$ for the differential on $\cC$. We consider a $(\gr_{\ws},\gr_{\zs})$-grading preserving chain map $f$ on $\cS'$. This map must take the following form:
\[
\begin{tikzcd}[labels=description, row sep=1.5cm, column sep=1.5cm] \cS' \ar[d, "f"] \\ \cS'
\end{tikzcd}
=
\begin{tikzcd}[labels=description, row sep=1.5cm, column sep=1.5cm]
\cS_2'
	\ar[r, "d_2"]
	\ar[d, "f_2"]
	 &
\cS_1'
	\ar[r, "d_1"]
	\ar[d,  "f_1"]
&
\cS_0'
	\ar[d, "f_0"]
	\ar[dll, "f_{20}", pos=.3, bend right =15,shorten=0.1cm,crossing over]
\\
 \cS_2'
 \ar[r, "d_2"]
 &
 \cS_1'
 	\ar[r, "d_1"]
 & \cS_0'
 	\ar[from = ull, bend right=15, "f_{02}", pos=.7,shorten=0.1cm,crossing over]
\end{tikzcd}
\]
We consider the component of $\d_{\cS'}\circ f+f\circ \d_{\cS'}$ starting at $\cS_0\otimes \ys$ and ending at $\cS_0\otimes \xs$. This component is equal to $\d_{\cS}\circ f_{20}$. Since $f$ is a chain map, this is zero. Since $\d_{\cS}$ is injective, we conclude that $f_{20}=0$.

Next, we observe that the map $f_{02}\colon \cS_2'\to \cS_0'$ can be viewed as a map from $\cS_1$ to $\cS_0$ which drops the bigrading by $(-1,-1)$. We can view the map $f_{02}$ therefore as determining an endomorphism of $\cS$ (supported only on $\cS_1$), which drops the $(\gr_{\ws},\gr_{\zs})$-bigrading by $(-1,-1)$. This map is clearly a chain map.  Up to chain homotopy, such chain maps are described by the Maslov grading $(-1,-1)$ subspace of $\Ext^1_{R}(\cS,\cS)$, which we saw in Lemma~\ref{lem:computation-ext-complex} vanishes.

If $h\colon \cS\to \cS$ is a null-homotopy of $f_{02}$, viewed as an endomorphism of $\cS$, then we can use $h$ to build a null-homotopy of $f_{02}$ as an endomorphism of $\cS'$, as follows. Let $h_{\veps}\colon \cS_{\veps}\to \cS_\veps$, denote the restriction of $h$, for $\veps\in \{0,1\}$. Then we define a null-homotopy $H$ of $f_{02}$ on $\cS'$ as in the following diagram (the gray arrows denote the differential on $\cS'$)
\[H=
\begin{tikzcd}[labels=description]
& \cS_1\otimes \ys	\ar[dr, "\d_S",gray]
\\
\cS_1\otimes \xs \ar[ur, bend left, "h_1"] \ar[ur, "U+WZ",gray] \ar[dr,"\d_{\cS}",gray]
 && \cS_0\otimes \ys\\
& \cS_0\otimes \xs \ar[ur, "U+WZ",gray] \ar[ur, "h_0", bend right]
\end{tikzcd}
\]
 Therefore we may take $f_{20}=f_{02}=0$. We observe now that $\cS'$ is a free resolution of its homology. Since the homology of $\cS'$ coincides with that of $\cS$, which is a monomial ideal in $\bF[W,Z]$, and $f$ is grading preserving and non-zero on homology, we conclude that $f$ is the identity on homology.  Since we have shown that $f$ preserves the algebraic grading, by standard homological algebra we conclude that $f$ is homotopic to the identity.
\end{proof}

\subsection{Remarks about $n$-component links}
\label{sec:remarks}
We now make some remarks about $n$-component links. The techniques used in the previous section do not generalize easily to $n$-component links. There are two natural strategies to generalize the proofs in the previous section to the case that $n>2$:
\begin{enumerate}
\item Via $\cCFL(L)$ as a type-$DA$ bimodule ${}_R \cCFL(L)^{R_{n-1}}$ and attempt a Koszul duality argument as in Equation~\eqref{eq:H-cube}.
\item Via $\cCFL(L)$ as a  type-$DA$ bimodule ${}_{R_{n-1}} \cCFL(L)^{R}$ and perform a similar Koszul duality argument, using the external tensor product of $n-1$ copies of ${}^R\Lambda^R$.
\end{enumerate}

Both strategies would be of an inductive nature. In the first case, we could write the underlying type-$D$ structure of ${}_R \cCFL(L)^{R_{n-1}}$ as a direct sum 
\[
\bigoplus_{s_1\in \frac{\lk(K_1,L\setminus K_1)}{2}+\Z} \cCFL(L, \{s_1\}\times \Q^{n-1})^{R_{n-1}}.
\]
Assuming (e.g. by an inductive argument) that one could show $\cCFL(L,\{s_1\}\times \Q^{n-1})^{R_{n-1}}$ is formal, it does not seem that the even the maps $\delta_2^1(W,-)$ and $\delta_2^1(Z,-)$ would be algebraically determined up to chain homotopy. When $n=2$ (as in the previous sections), we repeatedly used the fact that a morphism between two staircase complexes which was a chain map and preserved the mod 2 grading was uniquely determined by the induced map on homology. For general formal type-$D$ structures over $R_{n-1}$ for $n>2$, this is not the case. We illustrate this with the link Floer complex of $T(2,4).$ This takes the form shown below:
\[
\cCFL(T(2,4))=\begin{tikzcd}[row sep =.3cm, column sep=2cm,labels=description,ampersand replacement=\&]
\&\&\xs_0
\\
\&\ve{a}_1
	\ar[ur, "W_2"]
	\ar[ddr, "Z_1"]
\\
\&\ve{b}_1
	\ar[uur,crossing over, "W_1"]
	\ar[dr, "Z_2"]
\\
\ve{e}_2
	\ar[ruu, "W_1"]
	\ar[ru, "W_2"]
	\ar[dr, "Z_2"]
	\ar[ddr, "Z_1"]
\&\&\ve{y}_0
\\
\&\ve{c}_1
	\ar[ur, "W_2"]
	\ar[ddr, "Z_1"]
\\
\&\ve{d}_1
	\ar[uur,crossing over, "W_1"]
	\ar[dr, "Z_2"]
\\
\&\&\ve{z}_0
\end{tikzcd}
\]
Note that the endomorphism $\phi$ which sends $\ve{e}_2$ to $\ve{y}_0$ and vanishes on other generators preserves the $\gr_{\ws}$ and Alexander gradings. Further $\phi$ is a chain map and is not null-homotopic, and therefore represents an element of $\Ext^2_{R_2}(\cCFL(L), \cCFL(L))$ of $\gr_{\ws}$ and Alexander gradings 0.

We now consider the second strategy, where one could analyze the $DA$-bimodule ${}_{R_{n-1}} \cCFL(L)^R$. In this case, the underlying type-$D$ module would decompose (up to chain homotopy) as a direct sum of staircase complexes. One could attempt a Koszul duality argument similar to Equation~\eqref{eq:H-cube}, however one would need to tensor with $n-1$ copies of the dualizing bimodule ${}^R \Lambda^R$. This would entail considering a large number of higher homotopies, such as $\delta_{6}^1(W_1,W_2,W_3,W_4,Z_5,-)$. It is not clear what a suitable analog of Lemma~\ref{lem:computation-ext-complex} would be.

\section{Examples of bimodules}
\label{sec:Examples of bimodules}

In this section, we use the techniques of Section~\ref{sec:2-component-L-space} to compute the bimodules ${}_{\cK} \cX(L)^{R}$ for several L-space links. We will consider the torus links $T(2,2q)$, several cables of the Hopf link, the Whitehead link, and a family of Mazur links.

We remind the reader that we normalize the multi-variable Alexander polynomial as the graded Euler characteristic of $\HFL^-(L)=H_* \left(\cCFL(L)/(Z_1=Z_2=\cdots=Z_n=0)\right)$. Therefore, we consider normalizations of the Alexander polynomial which satisfy
\[
\Delta_L(t_1^{-1},t_2^{-1})=t_1^{-1}t_2^{-1} \Delta_L(t_1,t_2).
\]
Equivalently, we can write
\[
\Delta_L(t_1,t_2)=t_1^{1/2}t_2^{1/2} \delta_L(t_1,t_2),
\]
where $\delta_L(t_1,t_2)$ satisfies $\delta_L(t_1^{-1},t_2^{-1})=\delta_L(t_1,t_2)$.

\subsection{The torus link $T(2,2q)$}

\label{sec:T(2,2q)}

In this section, we compute the $DA$-bimodule
\[
{}_{\cK} \cX_{(0,0)}(T_{2,2q})^{\bF[W,Z]}.
\]
The link $T_{2,2q}$ is algebraic and therefore an L-space link by \cite{GorskyNemethiAlgebraicLinks}*{Theorem~2}. 

We first compute the multivariate Alexander polynomial. This is well-known to be
\[
\Delta_{T_{2,2q}}(t_1,t_2)=t_1^{1/2}t_2^{1/2}\left((t_1t_2)^{(q-1)/2}+(t_1t_2)^{(q-3)/2}+\cdots+(t_1t_2)^{-(q-1)/2}\right).
\]
One procedure for computing the above is to first view $T(2,2q)$ as being the $(q,1)$-cable of the Hopf link, which has multivariable Alexander polynomial equal to $(t_1t_2)^{1/2}$. Then we apply a well-known cabling formula for the multivariable Alexander polynomial \cite{Turaev_torsion}*{Theorem~1.3}.

\begin{rem}
The mirror of $T_{2,2q}$ is not an L-space link. See \cite{GorskyNemethiAlgebraicLinks}*{Figure~4}. On the other hand, if we let $T'_{2,2q}$ denote the link $T_{2,2q}$ with the orientation of one strand reversed, then $T'_{2,2q}$ is an L-space link, since Dehn surgery is independent of string orientation. Note that the multi-variable Alexander polynomial is not preserved by string orientation reversal. 
To compute ${}_{\cK} \cX_{(0,0)}(T_{2,2q}')^{\bF[W,Z]}$, we can either apply the same strategy as for $T_{2,2q}$, or we can observe that reversing the orientation of a component has the effect of tensoring with the elliptic bimodule ${}_{\cK} [E]^{\cK}$, so that
\[
{}_{\cK} \cX_{(0,0)}(T_{2,2q}')^{\bF[W,Z]}={}_{\cK} [E]^{\cK} \boxtimes {}_{\cK} \cX_{(0,0)}(T_{2,2q})^{\bF[W,Z]}.
\]

See Section \ref{sec:identity-cobordism} for a more detailed discussion of the elliptic bimodule ${}_{\cK} [E]^{\cK}$.
\end{rem}

Returning to the link $T_{2,2q}$, we may use Gorsky and N\'{e}methi's formula \cite{GorskyNemethiAlgebraicLinks}*{Theorem~2.10} to compute $H_{T_{2,2q}}$. We remind the reader that their formula states that if $L$ be an oriented L-space link in $S^3$, then
\[
H_L(\ve{s})=\sum_{L'\subset L} (-1)^{|L'|-1} \sum_{\substack{\ve{s}'\in \bH(L') \\ \ve{s}'\ge \pi_{L,L'}(\ve{s}+\ve{1})}} \chi(\HFL^-(L',\ve{s}')).
\]
We first illustrate the $H$ function of the link $T_{2,6}$ as well as the bimodule ${}_{\cK} \cX_{(0,0)}(T_{2,6})^R$ in Figure~\ref{fig:T226-diagram-H}. The general case of $T_{2,2q}$ is illustrated in Figure~\ref{fig:T22q-complex}.

\begin{figure}[h]
\begin{tabular}{m{4cm}p{8cm}}
 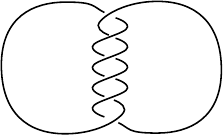
&$ \begin{array}{|c|ccccccccc|}
\hline
\hbox{\diagbox{$s_2$}{$s_1$}}&\cdots&\sfrac{-5}{2}&\sfrac{-3}{2}&\sfrac{-1}{2}&\sfrac{1}{2}&\sfrac{3}{2}&\sfrac{5}{2}&\cdots&\\
\hline	\vdots&\ddots&4&3&2&1&0&0&\reflectbox{$\ddots$}&\\
	\sfrac{3}{2}&\cdots&4&3 &2&1&\boxed{0} & \boxed{0}&\cdots& \\
	\sfrac{1}{2}&\cdots&4&3 &2&\boxed{1}&1 &1&\cdots&\\
	\sfrac{-1}{2}&\cdots&4&3 &\boxed{2}&2&2& 2&\cdots&\\
	\sfrac{-3}{2}&\cdots&\boxed{4}&\boxed{3} &3&3&3&3&\cdots&\\
	\vdots&\reflectbox{$\ddots$}&5&4 &4&4&4&4&\ddots &\,\\
\hline
\end{array}
$
\end{tabular}
\[
\begin{tikzcd}[labels=description, column sep=1.2cm, row sep=0cm]
\cdots
	\ar[r, bend left, "Z|U"]
&\xs_{-\frac{5}{2}}
	\ar[r, bend left, "Z|U"]
	\ar[l, bend left, "W|1"]
	\ar[d, "\substack{\sigma|W^3U \\ \tau|1}"]
& \xs_{-\frac{3}{2}}
	\ar[r, bend left, "Z|W"]
	\ar[l, bend left, "W|1"]
	\ar[d, "\substack{\sigma|W^3 \\ \tau|1}"]
&\xs_{-\frac{1}{2}}
	\ar[r, "Z|W", bend left]
	\ar[l, "W|Z", bend left]
	\ar[d, "\substack{\sigma|W^2\\ \tau|Z}"]
&\xs_{\frac{1}{2}}
	\ar[r, bend left, "Z|W"]
	\ar[l, bend left, "W|Z"]
	\ar[d, "\substack{\sigma|W \\ \tau|Z^2}"]
&\xs_{\frac{3}{2}}
	\ar[r, bend left, "Z|1"]
	\ar[l, bend left,"W|Z"]
	\ar[d, "\substack{\sigma|1\\ \tau|Z^3}"]
&\xs_{\frac{5}{2}}
	\ar[r, bend left, "Z|1"]
	\ar[l, bend left,"W|U"]
	\ar[d,"\substack{\sigma|1\\ \tau|UZ^3}"]
& \cdots
	\ar[l, bend left,"W|U"]
\\[2cm]
\cdots
	\ar[r, bend left, "T|1"]
&T^{-\frac{5}{2}}
	\ar[r, bend left, "T|1"]
	\ar[l, bend left, "T^{-1}|1"]
	\ar[loop below,looseness=20, "U|U"]
& T^{-\frac{3}{2}}
	\ar[r, bend left, "T|1"]
	\ar[l, bend left, "T^{-1}|1"]
	\ar[loop below,looseness=20, "U|U"]
& T^{-\frac{1}{2}}
	\ar[r, "T|1", bend left]
	\ar[l, "T^{-1}|1", bend left]
	\ar[loop below,looseness=20, "U|U"]
&T^{\frac{1}{2}}
	\ar[r, bend left, "T|1"]
	\ar[l, bend left,"T^{-1}|1"]
	\ar[loop below,looseness=20, "U|U"]
&T^{\frac{3}{2}}
	\ar[r, bend left, "T|1"]
	\ar[l, bend left,"T^{-1}|1"]
	\ar[loop below,looseness=20, "U|U"]
& T^{\frac{5}{2}}
	\ar[l, bend left,"T^{-1}|1"]
	\ar[loop below,looseness=20, "U|U"]
	\ar[r, "T|1", bend left]
&\cdots 
	\ar[l, "T^{-1}|1", bend left]
\end{tikzcd}
\]
\caption{Top right: the function $H_{T_{2,6}}$. The boxes indicate the generators of type~\eqref{Y-0} from Section~\ref{sec:Y0}. On the bottom is the bimodule ${}_{\cK} \cX_{(0,0)}(T_{2,6})^{\bF[W,Z]}$. The $\sigma$ and $\tau$ arrows to the left and right of the region shown are weighted by either $1$, $W^3 U^i$ or $Z^3 U^i$ }
\label{fig:T226-diagram-H}
\end{figure}
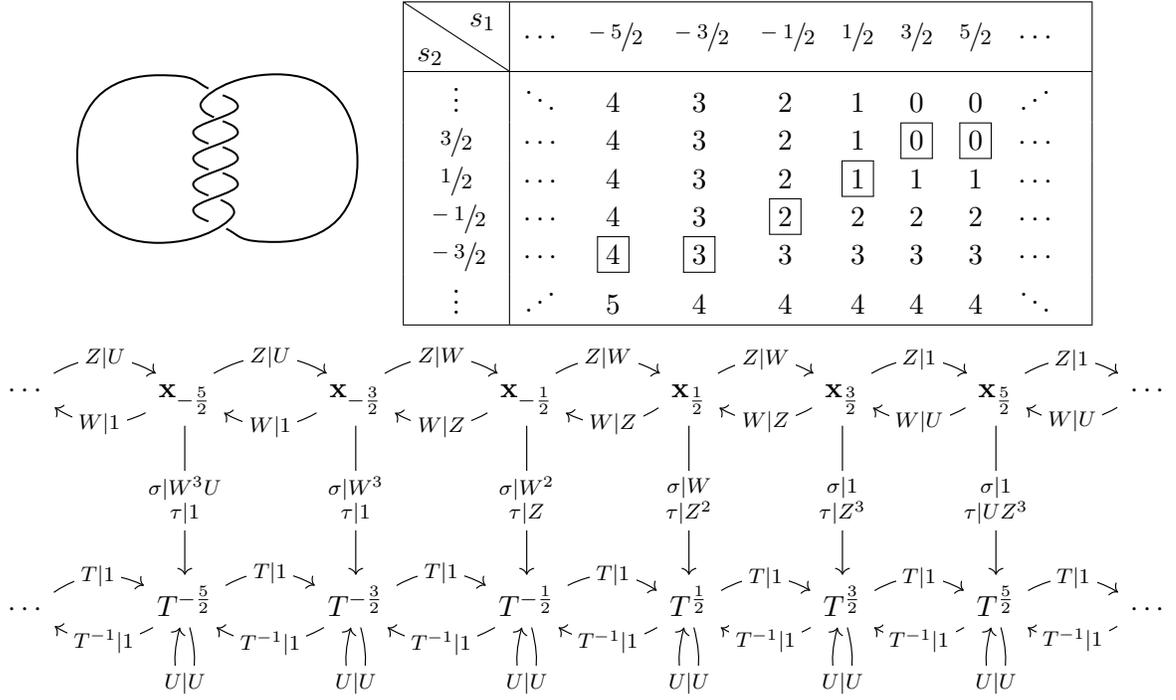

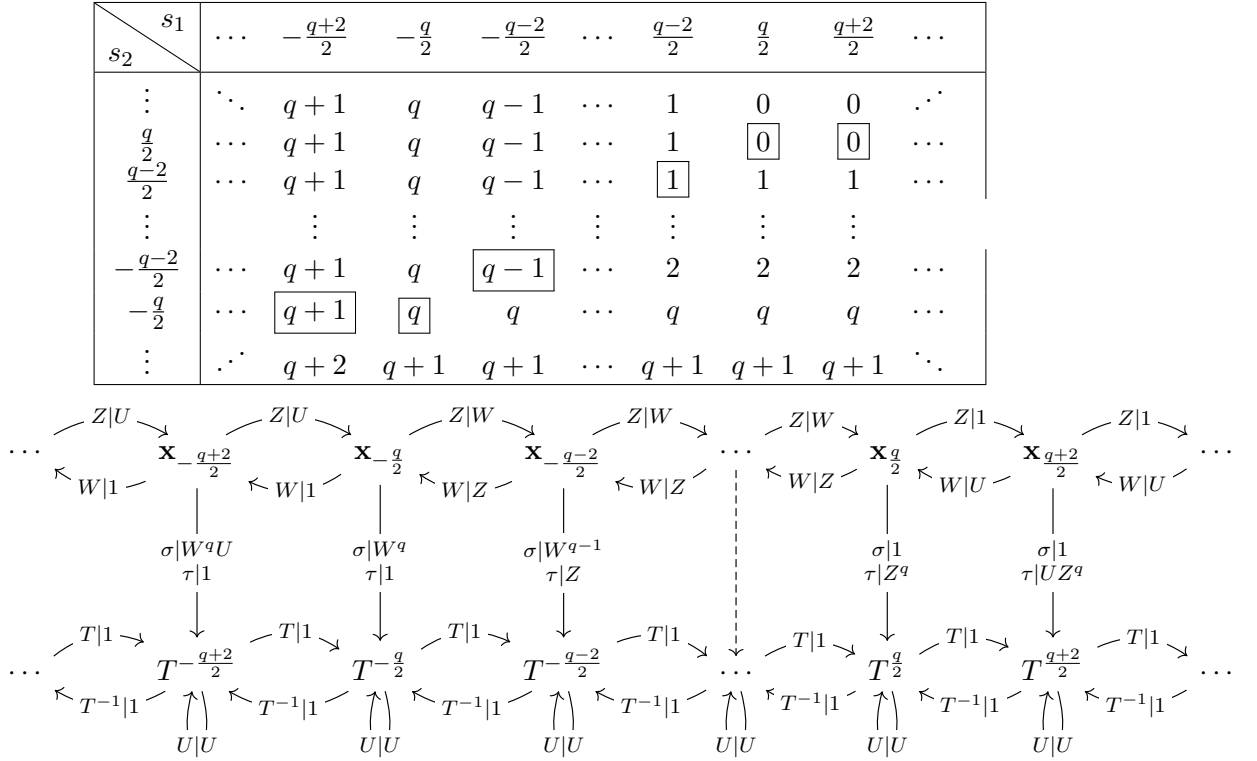
\begin{figure}[h]
\[
\begin{array}{|c|cccccccccc|}
\hline
\hbox{\diagbox{$s_2$}{$s_1$}}&\cdots&-\tfrac{q+2}{2}&-\tfrac{q}{2}&-\tfrac{q-2}{2}&\cdots &\tfrac{q-2}{2}&\tfrac{q}{2}&\tfrac{q+2}{2}&\cdots&\\
\hline	\vdots&\ddots&q+1&q&q-1&\cdots&1&0&0&\reflectbox{$\ddots$}&\\
	\tfrac{q}{2}&\cdots&q+1&q &q-1&\cdots&1&\boxed{0} & \boxed{0}&\cdots& \\
	\tfrac{q-2}{2}&\cdots&q+1&q &q-1&\cdots&\boxed{1}&1 &1&\cdots&\\
\vdots&&\vdots&\vdots & \vdots& \vdots & \vdots &\vdots&\vdots
\\
	-\tfrac{q-2}{2}&\cdots&q+1&q &\boxed{q-1}&\cdots&2&2& 2&\cdots&\\
	-\tfrac{q}{2}&\cdots&\boxed{q+1}&\boxed{q} &q&\cdots&q&q&q&\cdots&\\
	\vdots&\reflectbox{$\ddots$}&q+2&q+1 &q+1&\cdots&q+1&q+1&q+1&\ddots &\,\\
\hline
\end{array}
\]
\[
\begin{tikzcd}[labels=description, column sep=1.2cm, row sep=0cm]
\cdots
	\ar[r, bend left, "Z|U"]
&\xs_{-\frac{q+2}{2}}
	\ar[r, bend left, "Z|U"]
	\ar[l, bend left, "W|1"] 
	\ar[d, "\substack{\sigma|W^qU \\ \tau|1}"]
& \xs_{-\frac{q}{2}}
	\ar[r, bend left, "Z|W"]
	\ar[l, bend left, "W|1"]
	\ar[d, "\substack{\sigma|W^q \\ \tau|1}"]
&\xs_{-\frac{q-2}{2}}
	\ar[r, "Z|W", bend left]
	\ar[l, "W|Z", bend left]
	\ar[d, "\substack{\sigma|W^{q-1}\\ \tau|Z}"]
&\cdots
	\ar[r, bend left, "Z|W"]
	\ar[l, bend left, "W|Z"]
	\ar[d,dashed]
&\xs_{\frac{q}{2}}
	\ar[r, bend left, "Z|1"]
	\ar[l, bend left,"W|Z"]
	\ar[d, "\substack{\sigma|1\\ \tau|Z^q}"]
&\xs_{\frac{q+2}{2}}
	\ar[r, bend left, "Z|1"]
	\ar[l, bend left,"W|U"]
	\ar[d,"\substack{\sigma|1\\ \tau|UZ^q}"]
& \cdots
	\ar[l, bend left,"W|U"]
\\[2cm]
\cdots
	\ar[r, bend left, "T|1"]
&T^{-\frac{q+2}{2}}
	\ar[r, bend left, "T|1"]
	\ar[l, bend left, "T^{-1}|1"]
	\ar[loop below,looseness=20, "U|U"]
& T^{-\frac{q}{2}}
	\ar[r, bend left, "T|1"]
	\ar[l, bend left, "T^{-1}|1"]
	\ar[loop below,looseness=20, "U|U"]
& T^{-\frac{q-2}{2}}
	\ar[r, "T|1", bend left]
	\ar[l, "T^{-1}|1", bend left]
	\ar[loop below,looseness=20, "U|U"]
&\cdots
	\ar[r, bend left, "T|1"]
	\ar[l, bend left,"T^{-1}|1"]
	\ar[loop below,looseness=20, "U|U"]
&T^{\frac{q}{2}}
	\ar[r, bend left, "T|1"]
	\ar[l, bend left,"T^{-1}|1"]
	\ar[loop below,looseness=20, "U|U"]
& T^{\frac{q+2}{2}}
	\ar[l, bend left,"T^{-1}|1"]
	\ar[loop below,looseness=20, "U|U"]
	\ar[r, "T|1", bend left]
&\cdots 
	\ar[l, "T^{-1}|1", bend left]
\end{tikzcd}
\]
\caption{The $H$-function and $DA$-bimodule for $T_{2,2q}$.}
\label{fig:T22q-complex}
\end{figure}

\subsection{The Whitehead link}

\label{sec:Whitehead double}
We now compute the $DA$-module of the (positively clasped) Whitehead link $W^+$, shown in Figure~\ref{fig:7}. The Whitehead link is an L-space link \cite{LiuLSpaceLinks}*{Example~3.1}.  Note however that the mirror of the Whitehead link (the negatively clasped Whitehead link) is not an L-space link \cite{Liu2Bridge}*{Proposition~6.9} (see also \cite{GLM_dinvariant}*{Theorem~5.1} or \cite{Santoro_Whitehead}*{Theorem~1.1}).

The Alexander polynomial of the Whitehead link is given by
\[
\Delta_{W^+}(t_1,t_2)=-(t_1-1)(t_2-1)
\]
Using this, we can compute the function $H_{W^+}$. Using the construction from Section~\ref{sec:2-component-L-space}, we compute the bimodule ${}_{\cK} \cX(W^+)^{\bF[W,Z]}$. These are shown in Figure~\ref{fig:7}.

\begin{figure}[h]
\begin{tabular}{m{4cm}p{8cm}}
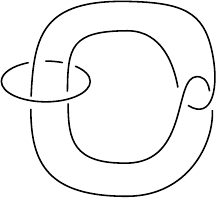
&$ \begin{array}{|c|cccccccc|}
\hline
\hbox{\diagbox{$s_2$}{$s_1$}}&\cdots&-2&-1&0&1&2&\cdots&\\
\hline	\vdots&\ddots&2&1&0&0&0&\reflectbox{$\ddots$}&\\
	1&\cdots&2&1 &\boxed{0}&0&0 & \cdots& \\
	0&\cdots&\boxed{2}&\boxed{1} &1&\boxed{0}&\boxed{0} &\cdots&\\
	-1&\cdots&3&2 &\boxed{1}&1&1& \cdots&\\
	\vdots&\reflectbox{$\ddots$}&4&3 &2&2&2&\ddots &\,\\
\hline
\end{array}
$
\end{tabular}
\[
\begin{tikzcd}[labels=description, column sep=1.4cm]
&&& \xs_0
	\ar[dl, bend left, "W|Z",pos=0.7]
	\ar[dr, "Z|Z", bend left]
	\ar[ddd,bend left=20, "\substack{\sigma|Z\\ +\tau|Z}",pos=.7]
	\ar[d,bend left =80, "{(W,Z)|Z}"]
&&
\\[.5cm]
\cdots
	\ar[r,bend left, "Z|U"]
& \xs_{-2}
	\ar[r, bend left, "Z|U"]
	\ar[l, bend left, "W|1"]
	\ar[dd,"\substack{\sigma|U^2\\ +\tau|1}"]
&
\xs_{-1}
	\ar[ur, bend left, "Z|W"]
	\ar[l, "W|1", bend left]
	\ar[dd, "\substack{\sigma|U \\ +\tau|1 }"]
&
\ys_0
	\ar[u,gray, "W"]
	\ar[d,gray, "Z"]
&
\xs_1
	\ar[r, bend left, "Z|1"]
	\ar[dl, "W|Z", bend left]
	\ar[dd, "\substack{\sigma|1 \\ +\tau|U}"]
&
\xs_2
	\ar[l, bend left, "W|U"]
	\ar[r,bend left, "Z|1"]
	\ar[dd, "\substack{\sigma|1\\+\tau|U^2}"]
&
\cdots
	\ar[l,bend left, "W|U"]
\\[.5cm]
&&&
\zs_0
	\ar[ul, "W|W", bend left]
	\ar[ur, "Z|W", bend left,crossing over,pos = 0.75]
	\ar[d,bend right, "\substack{\sigma|W\\ +\tau|W}"]
	\ar[u,bend left =70, "{(Z,W)|W}"]
&&
\\[1.5cm]
\cdots
	\ar[r, bend left, "T|1"]
&T^{-2}
	\ar[r, bend left, "T|1"]
	\ar[l, bend left, "T^{-1}|1"]
	\ar[loop below,looseness=20, "U|U"]
& T^{-1}
	\ar[r, bend left, "T|1"]
	\ar[l, bend left, "T^{-1}|1"]
	\ar[loop below,looseness=20, "U|U"]
&T^{0}
	\ar[r, "T|1", bend left]
	\ar[l, "T^{-1}|1", bend left]
	\ar[loop below,looseness=20, "U|U"]
&T^{1}
	\ar[r, bend left, "T|1"]
	\ar[l, bend left,"T^{-1}|1"]
	\ar[loop below,looseness=20, "U|U"]
&T^{2}
	\ar[r, bend left, "T|1"]
	\ar[l, bend left,"T^{-1}|1"]
	\ar[loop below,looseness=20, "U|U"]
&\cdots 
	\ar[l, bend left,"T^{-1}|1"]
\end{tikzcd}
\]
\caption{The $H$-function and $DA$-bimodule of the positively clasped Whitehead link $W^+$. }
\label{fig:7}
\end{figure}

In Figure~\ref{fig:7}, we record only $\delta^1_1$, $L_W=\delta_2^1(W,-)$, $L_Z = \delta_2^1(Z,-)$ and $h_{W,Z}=\delta_{3}^1(W,Z,-)$ and $h_{Z,W}=\delta_3^1(Z,W,-)$. 
The remaining actions are determined via the formulas in Section~\ref{sec:U-equivariant}.

\subsection{Mazur links}
The Mazur link is shown in Figure~\ref{fig:Mazur}. This is an L-space link by \cite{LiuLSpaceLinks}*{Theorem~3.8}. (See Section~\ref{sec:family-Mazur} below for more details).  The Alexander polynomial of the Mazur link $M$ is given by
\[
\Delta_{M}(t_1,t_2) = t_1^{1/2}t_2^{1/2} ( t_1^{-1}+t_2^{-1}+t_1+t_2-t_1^{-1}t_2^{-1}-1-t_1t_2).
\]
The resulting $H$-function and $DA$-bimodule are shown in Figure~\ref{fig:Mazur}. 

\begin{figure}[h]
\begin{tabular}{m{4cm}p{8cm}}
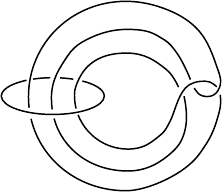
&$ \begin{array}{|c|ccccccccc|}
\hline
\hbox{\diagbox{$s_2$}{$s_1$}}&\cdots&\sfrac{-5}{2}&\sfrac{-3}{2}&\sfrac{-1}{2}&\sfrac{1}{2}&\sfrac{3}{2}&\sfrac{5}{2}&\cdots&\\
\hline	\vdots&\ddots&3&2&1&0&0&0&\reflectbox{$\ddots$}&\\
	\sfrac{3}{2}&\cdots&3&2&1 &\boxed{0}&0&0 & \cdots& \\
	\sfrac{1}{2}&\cdots&3&2&\boxed{1} &1&\boxed{0}&\boxed{0} &\cdots&\\
	\sfrac{-1}{2}&\cdots&\boxed{3}&\boxed{2}&2 &\boxed{1}&1&1& \cdots&\\
        \sfrac{-3}{2}&\cdots&4&3&\boxed{2} &2&2&2& \cdots&\\
	\vdots&\reflectbox{$\ddots$}&5&4&3 &3&3&3&\ddots &\,\\
\hline
\end{array}
$
\end{tabular}
\[
\begin{tikzcd}[labels=description, column sep=1.5cm]
&&& \xs_{\frac{1}{2}}
	\ar[dl, bend left, "W|Z",pos=0.65]	
        \ar[dr, bend left, "Z|Z"]	
	\ar[dddd,bend left=20, "\substack{\sigma|Z\\ +\tau|Z^2}",pos=.7]
	\ar[d,bend left =80, "{(W,Z)|Z}"]
&&
\\[.5cm]
&&
\xs_{-\frac{1}{2}}
	\ar[dl, bend left, "W|Z",pos = 0.75]
 \ar[ur, "Z|W", bend left,pos=0.7 ]
	\ar[ddd, bend left=30, "\substack{\sigma|U \\ +\tau|Z }",pos=.7]
&
\ys_{\frac{1}{2}}
	\ar[u,gray, "W"]
	\ar[d, gray, "Z"]
	\arrow[dl, bend left, "W|Z",pos = 0.65,crossing over]
&
\xs_{\frac{3}{2}}
	\ar[r, bend left, "Z|1"]
	\ar[dl, "W|Z",bend left,pos = 0.6]
	\ar[ddd, "\substack{\sigma|1 \\ +\tau|UZ}"]
&
\xs_{\frac{5}{2}}
	\ar[l, bend left, "W|U"]
	\ar[r,bend left, "Z|1"]
	\ar[ddd, "\substack{\sigma|1\\+\tau|U^2 Z}"] 
&
 \cdots
 \ar[l,bend left, "W|U"]
\\[.5cm]
\cdots
\ar[r,bend left, "Z|U"]
& 
\xs_{-\frac{3}{2}}
	\ar[ur, bend left, "Z|W"]
	\ar[l, bend left, "W|1"]
	\ar[dd,"\substack{\sigma|UW\\ +\tau|1}"]
&
\ys_{-\frac{1}{2}}
	\ar[u,gray, "W"]
	\ar[d, gray, "Z"]
 \ar[ur, "Z|W",bend left, crossing over,pos = 0.7]
&
\zs_{\frac{1}{2}}
	\ar[dl, "W|Z", bend left, crossing over, pos = 0.65]
	\ar[ur, "Z|W", bend left,crossing over,pos = 0.7]
	\ar[dd,bend right=15, "\substack{\sigma|W\\ +\tau|U}"]
&&
\\[.5cm]
&&
\zs_{-\frac{1}{2}}
	\ar[ul, "W|W", bend left]
	\ar[ur, "Z|W", bend left,crossing over,pos = 0.7]
	\ar[d,bend right, "\substack{\sigma|W^2\\ +\tau|W}"]
	\arrow[u,bend left =70, "{(Z,W)|W}"]
&&&
\\[1.5cm]
\cdots
	\ar[r, bend left, "T|1"]
&T^{-\frac{3}{2}}
	\ar[r, bend left, "T|1"]
	\ar[l, bend left, "T^{-1}|1"]
	\ar[loop below,looseness=20, "U|U"]
& T^{-\frac{1}{2}}
	\ar[r, bend left, "T|1"]
	\ar[l, bend left, "T^{-1}|1"]
	\ar[loop below,looseness=20, "U|U"]
&T^{\frac{1}{2}}
	\ar[r, "T|1", bend left]
	\ar[l, "T^{-1}|1", bend left]
	\ar[loop below,looseness=20, "U|U"]
&T^{\frac{3}{2}}
	\ar[r, bend left, "T|1"]
	\ar[l, bend left,"T^{-1}|1"]
	\ar[loop below,looseness=20, "U|U"]
&T^{\frac{5}{2}}
	\ar[r, bend left, "T|1"]
	\ar[l, bend left,"T^{-1}|1"]
	\ar[loop below,looseness=20, "U|U"]
&\cdots 
	\ar[l, bend left,"T^{-1}|1"]
\end{tikzcd}
\]
\caption{The $H$-function and $DA$-bimodule of the Mazur link $M$. }
\label{fig:Mazur}
\end{figure}

Similar as before, in the above diagram we record only $\delta^1_1$, as well as $L_W$, $L_Z$. We also have
\[
h_{Z,W}(\zs_{-\frac{1}{2}}) = \ys_{-\frac{1}{2}}\otimes W,\quad h_{W,Z}(\xs_{\frac{1}{2}} )= \ys_{\frac{1}{2}}\otimes Z. 
\]
Using the construction in Section~\ref{sec:U-equivariant}, we obtain the full $DA$-bimodule structure.

\subsection{A family of generalized Mazur patterns}
\label{sec:family-Mazur}
The Mazur link and the Whitehead link belong to a family of two-bridge links $K(6m+2,3)$ for $m\geq 1,$ where taking $m=1$ gives us the Whitehead link and taking $m=2$ gives us the Mazur link. See Figure \ref{fig:family}. Here we follow the convention in \cite{rationalknotskl} for the slope of a two-bridge link; in particular $T_{2,3}$ is $K(3,1).$ Note that the opposite convention is also used in the literature e.g. in \cite{LiuLSpaceLinks} the link $b(p,q)$ is what we call $K(p,-q).$ By taking $r=2m+1$ and $q=3$ in \cite{LiuLSpaceLinks}*{Theorem 3.8} it follows that links $K(6m+2,3)$ for $m\geq 1$ are all  L-space.  It follows from \cite{Hostetwobridge}*{Algorithm 3} that the Alexander polynomial of a link in this family is given by
\[
\Delta_{K(6m+2,3)}(t_1,t_2)= (t_1t_2)^{(1-m)/2}\Bigl(\sum^{m-1}_{i=0} (t_1 t_2)^i (t_1 + t_2) - \sum^m_{i=0}(t_1 t_2)^i\Bigr)
\]
Therefore we can compute the $H$-function and $DA$-bimodule associated to any link of this family. We carry out the computation for the case $m=3$ as shown in Figure~\ref{fig:K(20,3)}.

\begin{figure}[h]
\begin{tikzpicture}
 \node at (-3,2) {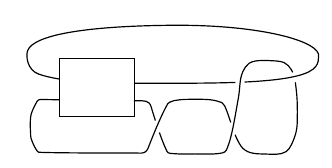};
 \node at (4,2) {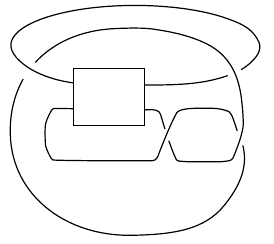};
  \node at (-3,-2.5) {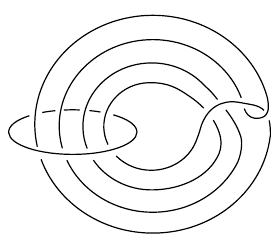};
       \node at (4,-3) {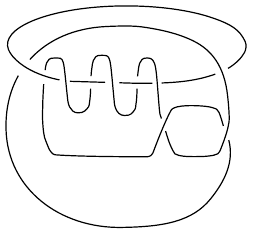};
       \node at (0.7,1.5) {\Large $\sim$};
       \node[rotate=90] at (4,-0.5) {\Large $\sim$};
        \node at (0.6,-2.5) {\Large $\sim$};
   \node at (-4.2,1.9) {$m$};
   \node at (3.5,2.4) {$m$};
    \draw [decorate,decoration={brace,amplitude=4pt,mirror},xshift=0cm,yshift=0pt]
      (-2.6,-1.1) -- (-2.6,-1.8) node [midway,yshift=0.1cm,xshift=-0.5cm] {$m$};
\end{tikzpicture}
\caption{The two-bridge link $K(6m+2,3)$. The number is the box indicates the number of right-handed full-twists. The case when $m=3$ is depicted.}
\label{fig:family}
\end{figure}

\begin{figure}[h]
\begin{tabular}{m{5cm}p{8cm}}
\begin{tikzpicture}
    \
    \node at (0,0) {\input{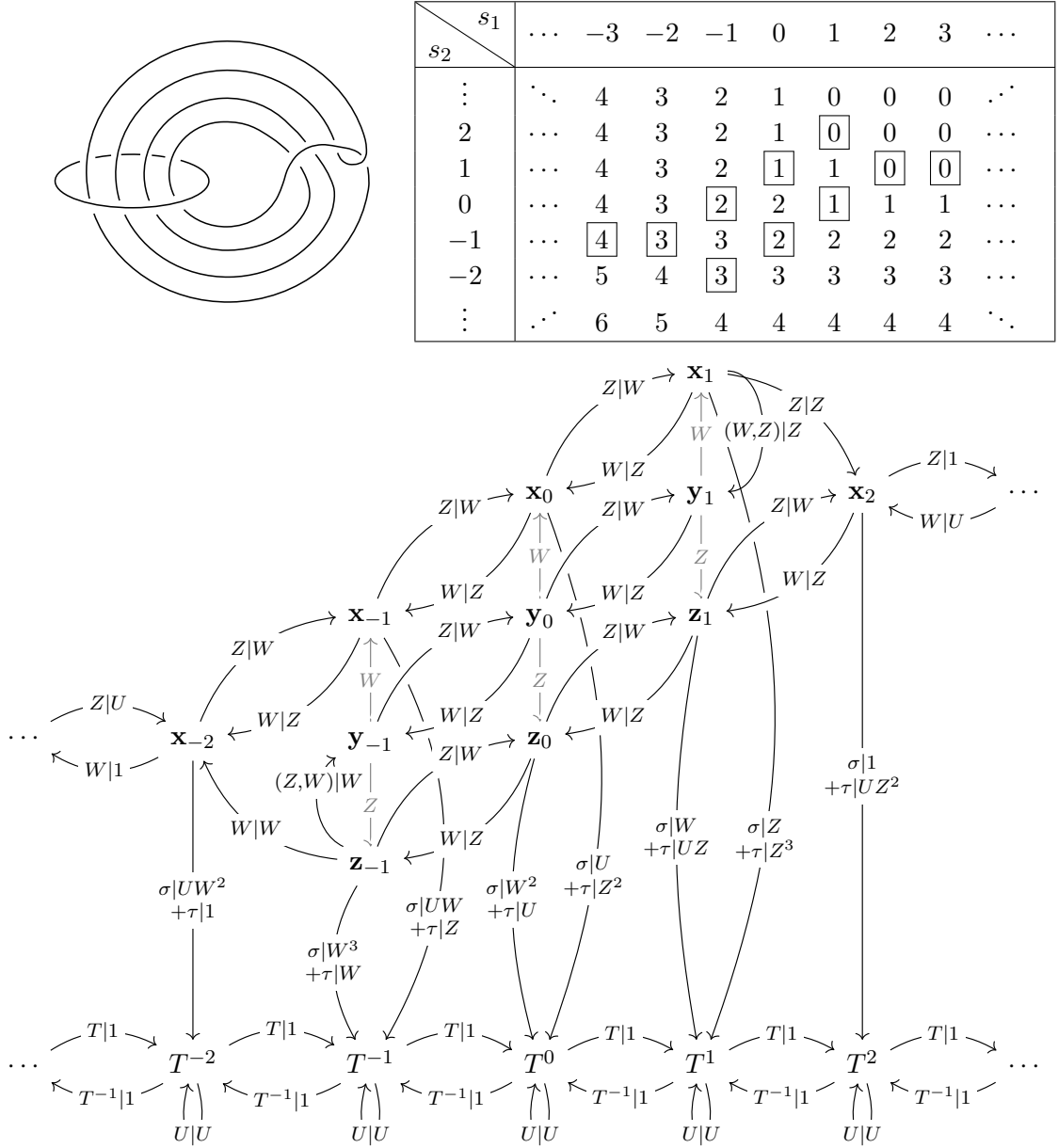}};
\end{tikzpicture}
&
$ \begin{array}{|c|cccccccccc|}
\hline
\hbox{\diagbox{$s_2$}{$s_1$}}&\cdots&-3&-2&-1&0&1&2&3&\cdots&\\
\hline	\vdots&\ddots&4&3&2&1&0&0&0&\reflectbox{$\ddots$}&\\
	2&\cdots&4&3&2&1 &\boxed{0}&0&0 & \cdots& \\
	1&\cdots&4&3&2&\boxed{1} &1&\boxed{0}&\boxed{0} &\cdots&\\
	0&\cdots&4&3&\boxed{2}&2 &\boxed{1}&1&1& \cdots&\\
        -1&\cdots&\boxed{4}&\boxed{3}&3&\boxed{2} &2&2&2& \cdots&\\
         -2&\cdots&5&4&\boxed{3}&3 &3&3&3& \cdots&\\
	\vdots&\reflectbox{$\ddots$}&6&5&4&4 &4&4&4&\ddots &\,\\
\hline
\end{array}
$
\end{tabular}
\[
\begin{tikzcd}[labels=description, column sep=1.5cm]
&&&& \xs_1
\ar[dl, bend left, "W|Z",pos =0.65]	
        \ar[dr, bend left, "Z|Z"]	
	\ar[ddddd,bend left=20, "\substack{\sigma|Z\\ +\tau|Z^3}",pos=.7]
	\arrow[d,bend left =90, "{(W,Z)|Z}"]
& 
\\[.5cm]
&&& \xs_0
	\ar[ur, bend left, "Z|W",pos = 0.7]
	\ar[dl, "W|Z",bend left,pos= 0.65]
	\ar[dddd,bend left=20, "\substack{\sigma|U\\ +\tau|Z^2}",pos=.7]
&
\ys_1
	\ar[u,gray, "W"]
	\ar[d, gray, "Z"]
	\arrow[dl, bend left, pos = 0.65, "W|Z",crossing over]
&
\xs_2
	\ar[dl, bend left, "W|Z",crossing over]
	\ar[r,bend left, "Z|1"]
	\ar[dddd, "\substack{\sigma|1\\+\tau|U Z^2}"] 
 &
  \cdots
   \ar[l,bend left, "W|U"]
\\[.5cm]
&&
\xs_{-1}
	\ar[dl, bend left, "W|Z",pos = 0.7]
 \ar[ur, "Z|W", bend left,pos=0.7 ]
	\ar[ddd, bend left=30, "\substack{\sigma|UW \\ +\tau|Z }",pos=.7]
&
\ys_0
	\ar[u,gray, "W"]
	\ar[d, gray, "Z"]
	\ar[dl, bend left, pos =0.65, "W|Z",crossing over]
 \ar[ur, "Z|W",bend left, crossing over,pos=0.7]
&
\zs_1
	\ar[ur, bend left, "Z|W", crossing over,pos =0.75]
	\ar[dl, "W|Z", bend left,pos=0.65, crossing over]
	\ar[ddd, bend right=10, "\substack{\sigma|W \\ +\tau|UZ}"]
&&
\\[.5cm]
\cdots
\ar[r,bend left, "Z|U"]
& 
\xs_{-2}
	\ar[ur, bend left, "Z|W"]
	\ar[l, bend left, "W|1"]
	\ar[dd,"\substack{\sigma|UW^2\\ +\tau|1}"]
&
\ys_{-1}
	\ar[u,gray, "W"]
	\ar[d, gray, "Z"]
 \ar[ur, "Z|W",bend left , crossing over,pos =0.7]
&
\zs_0
	\ar[dl, "W|Z", bend left,pos=.65, crossing over]
	\ar[ur, "Z|W", bend left,pos=0.7,crossing over]
	\ar[dd,bend right=15, "\substack{\sigma|W^2\\ +\tau|U}"]
&&
\\[.5cm]
&&
\zs_{-1}
	\ar[ul, "W|W", bend left]
	\ar[ur, "Z|W", bend left,pos =0.7,crossing over]
	\ar[d,bend right, "\substack{\sigma|W^3\\ +\tau|W}"]
	\ar[u, bend left =60, "{(Z,W)|W}", pos=.7]
&&&
\\[1.5cm]
\cdots
	\ar[r, bend left, "T|1"]
&T^{-2}
	\ar[r, bend left, "T|1"]
	\ar[l, bend left, "T^{-1}|1"]
	\ar[loop below,looseness=20, "U|U"]
& T^{-1}
	\ar[r, bend left, "T|1"]
	\ar[l, bend left, "T^{-1}|1"]
	\ar[loop below,looseness=20, "U|U"]
&T^{0}
	\ar[r, "T|1", bend left]
	\ar[l, "T^{-1}|1", bend left]
	\ar[loop below,looseness=20, "U|U"]
&T^{1}
	\ar[r, bend left, "T|1"]
	\ar[l, bend left,"T^{-1}|1"]
	\ar[loop below,looseness=20, "U|U"]
&T^{2}
	\ar[r, bend left, "T|1"]
	\ar[l, bend left,"T^{-1}|1"]
	\ar[loop below,looseness=20, "U|U"]
&\cdots 
	\ar[l, bend left,"T^{-1}|1"]
\end{tikzcd}
\]
\caption{The $H$-function and $DA$-bimodule of  $K(20,3)$.}
\label{fig:K(20,3)}
\end{figure}
 In the diagram of the $DA$-bimodule (bottom diagram of Figure \ref{fig:K(20,3)}) we record only $\delta^1$, $L_W $, $L_Z$, $h_{W,Z}$ and $h_{Z,W}$.  The rest of the $\delta_2^1$ and $\delta_3^1$ terms are determined as in Section~ \ref{sec:U-equivariant}.

\subsection{General cabling operations}
\label{sec:general cabling}

Section \ref{sec:T(2,2q)} discusses the case when $L_P=T_{2,2q}$. This link can also be viewed as the $(q,1)$-cable of the Hopf link (where we cable just one component). In general, we will write $C_{p,q}$ for the two component link obtained by taking the $(p,q)$-cable of the Hopf link. Using earlier notation, we think of $P$ as being the $(p,q)$ torus knot inside of the solid torus.

If $p,q>0$ and $\gcd(p,q)=1$, the link $C_{p,q}$ will be an L-space link since $C_{p,q}$ is an algebraic link:
\[
C_{p,q}= \left\{(w,z)\mid z(w^p+z^q)=0, w,z\in \mathbb{C}\right\} \cap S^3.
\]
By \cite{GorskyNemethiAlgebraicLinks}, all algebraic links are L-space links.

We will typically abuse notation and write $C_{p,q}$ for both the two-component link and the satellite operator.

Note that 
\[
C_{p,q}(K,n)=C_{p,q+pn}(K,0)\quad C_{-p,-q}(K,n)=C_{p,q}(rK,n)
\]
where $rK$ denotes $K$ with its string orientation reversed. Therefore for the sake of computing the cabling operators, it is sufficient to restrict to positive $p$ and $q$.

Using \cite{Turaev_torsion}*{Theorem~1.3}   the normalized multivariable Alexander polynomial of $L_P$ is 
\[
\Delta_{C_{p,q}}(t_1,t_2)=
t_1^{-\frac{p}{2}+1}t_2^{\frac{-pq+q+1}{2}} \left(\frac{t_1^p t_2^{pq}-1}{t_1 t_2^q-1}\right)
\]
where $t_1$ represents the meridian component, and $t_2$ represents the $(p,q)$-torus knot component.

We can perform a similar procedure as before to obtain the $H$-function of $C_{p,q}$. From the $H$-function, we can compute the $DA$-bimodule structure of ${}_{\cK}\cX(C_{p,q})^{\bF[W,Z]}$.

The main difference between the general case with the situation of $(q,1)$-torus knot as in Section \ref{sec:T(2,2q)} is now $\cS$ represents the knot Floer complex of the $(p,q)$-torus knot, which is now a non-trivial staircase, and there will usually be non-trivial $\delta_3^1(\sigma,W,-), \delta_3^1(\sigma,Z,-)$, $\delta_3^1(\tau, W,-)$, $\delta_3^1(\tau,Z,-)$, etc.

In Figure~\ref{fig:(3,2)_torus} we compute the $H$-function and $DA$-bimodule of the link $C_{3,2}$. Therein, $\cS$ denotes the complex of the trefoil $T_{2,3}$:
\[
\cS=(\begin{tikzcd} \xs & \ys \ar[l, "W",swap] \ar[r, "Z"]& \zs\end{tikzcd}).
\]
The generators labeled $\xs_{\pm \frac{1}{2}}^0$ denote single generators. To explain the maps which are not given explicit formulas in the diagram, we expand the center most region below.

 We also do not label the $\delta_3^1$ terms, though they are easily computed. The non-trivial components of the maps $h_{W,Z}$, $h_{Z,W}$, $h_{\sigma,Z}$, $h_{\sigma, W}$, $h_{\tau,Z}$ and $h_{\tau,W}$ are shown below. Therein, we write $\xs^{\veps}_{s}$ for a copy of $\xs\in \cS$, concentrated in Alexander grading $s$ and idempotent $\veps\in \{0,1\}$, likewise for $\ys^{\veps}_{s} $ and $ \zs^{\veps}_{s}$. The remaining $\delta_3^1$ actions are defined using the construction from Section~\ref{sec:U-equivariant}.
\begin{align*}
	&\delta_3^1( W,Z,\zs^0_{-\frac{3}{2}}) = \ys^0_{-\frac{3}{2}} \otimes W, &\delta_3^1( Z,W, \xs^0_{\frac{3}{2}}) = \ys^0_{\frac{3}{2}} \otimes Z, \\
	&\delta_3^1(\tau, Z,\zs^0_{-\frac{3}{2}}) = \ys^1_{-\frac{1}{2}} \otimes W, &\delta_3^1(\sigma, W,\xs^0_{\frac{3}{2}}) = \ys^1_{\frac{1}{2}} \otimes Z,\\
	&\delta_3^1(\sigma,Z,\xs^0_{-\frac{3}{2}}) = \ys^1_{-\frac{1}{2}} \otimes W^2, &\delta_3^1(\tau, W,\zs^0_{\frac{3}{2}}) = \ys^1_{\frac{1}{2}}\otimes Z^2,\\
	&\delta_3^1(\sigma,W,\xs^0_{-\frac{1}{2}}) = \ys^1_{-\frac{3}{2}}\otimes W^2, &\delta_3^1(\tau, Z, \xs^0_{\frac{1}{2}}) = \ys^{1}_{\frac{3}{2}}\otimes Z^2.
\end{align*}

\begin{figure}[p]
	\begin{tabular}{m{5cm}p{8cm}}
		\begin{tikzpicture}
			\
			\node at (0,0) {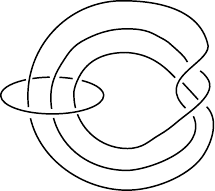};
		\end{tikzpicture}
		&
		$\begin{array}{|c|ccccccccc|}
			\hline
			\hbox{\diagbox{$s_2$}{$s_1$}}&\cdots&\sfrac{-5}{2}&\sfrac{-3}{2}&\sfrac{-1}{2}&\sfrac{1}{2}&\sfrac{3}{2}&\sfrac{5}{2}&\cdots&\\
			\hline	\vdots&\ddots&4&3&2&1&0&0&\reflectbox{$\ddots$}&\\
			\sfrac{5}{2}&\cdots&4&3 &2&1&\boxed{0} & \boxed{0}&\cdots& \\
			\sfrac{3}{2}&\cdots&4&3 &2&1&1 &1&\cdots&\\
			\sfrac{1}{2}&\cdots&4&3 &2&\boxed{1}&\boxed{1}& \boxed{1}&\cdots&\\
			\sfrac{-1}{2}&\cdots&\boxed{4}&\boxed{3} &\boxed{2}&2&2&2&\cdots&\\
			\sfrac{-3}{2}&\cdots&5&4 &3&3&3&3&\cdots&\\
			\sfrac{-5}{2}&\cdots&\boxed{5}&\boxed{4} &4&4&4&4&\cdots&\\
			\vdots&\reflectbox{$\ddots$}&6&5 &5&5&5&5&\ddots &\,\\
			\hline
		\end{array}
		$
	\end{tabular}
\begin{tikzcd}[labels=description, column sep=1.6cm, row sep=-.5cm]
\cdots
&[-1.7cm]\cS
	\ar[r, bend left, "Z|U"]
	\ar[d,"\substack{\sigma|UW^3\\ +\tau|1}"]
& \cS
	\ar[r, bend left, "Z|L_Z"]
	\ar[l, bend left, "W|1"]
	\ar[d,"\substack{\sigma|W^3\\ +\tau|1}"]
&\xs_{-\frac{1}{2}}^0
	\ar[r, bend left, "Z|W"]
	\ar[l, bend left, "W|L_W"]
	\ar[d, "\substack{\sigma|L_\sigma\\ +\tau|L_\tau}"]
& \xs_{\frac{1}{2}}^0
	\ar[r, bend left, "Z|L_Z"]
	\ar[l, bend left, "W|Z"]
	\ar[d, "\substack{\sigma|L_\sigma\\ +\tau|L_\tau}"]
& \cS
	\ar[r, bend left, "Z|1"]
	\ar[l, bend left, "W|L_W"]
	\ar[d, "\substack{\sigma|1\\ +\tau|Z^3}"]
& \cS
	\ar[l, bend left, "W|U"]
	\ar[d, "\substack{\sigma|1\\ +\tau|UZ^3}"]
&[-1.7cm] \cdots
\\[2cm]
\cdots
& \cS
	\ar[r, bend left, "T|1"]
	\ar[loop below,looseness=20, "U|U"]
& \cS
	\ar[r, bend left, "T|1"]
	\ar[l, bend left, "T^{-1}|1"]
	\ar[loop below,looseness=20, "U|U"]
&\cS
	\ar[r, bend left, "T|1"]
	\ar[l, bend left, "T^{-1}|1"]
	\ar[loop below,looseness=20, "U|U"]
&\cS
	\ar[r, bend left, "T|1"]
	\ar[l, bend left, "T^{-1}|1"]
	\ar[loop below,looseness=20, "U|U"]
&\cS
	\ar[l, bend left, "T^{-1}|1"]
	\ar[loop below,looseness=20, "U|U"]
		\ar[r, bend left, "T|1"]
& \cS
	\ar[l, bend left, "T^{-1}|1"]
	\ar[loop below,looseness=20, "U|U"]
&\cdots 
\end{tikzcd}
 \[
 	 \adjustbox{scale=.9}{
\begin{tikzcd}[labels=description, column sep=1.3cm, row sep=0cm]
 \cS
	\ar[r, bend left, "Z|L_Z"]
&\xs_{-\frac{1}{2}}^0
	\ar[r, bend left, "Z|W"]
	\ar[l, bend left, "W|L_W"]
	\ar[d, "\substack{\sigma|L_\sigma\\ +\tau|L_\tau}"]
& \xs_{\frac{1}{2}}^0
	\ar[r, bend left, "Z|L_Z"]
	\ar[l, bend left, "W|Z"]
	\ar[d, "\substack{\sigma|L_\sigma\\ +\tau|L_\tau}"]
& \cS
	\ar[l, bend left, "W|L_W"]
\\[2cm]
\,
&\cS
	\ar[r, bend left, "T|1"]
	\ar[loop below,looseness=20, "U|U"]
&\cS
	\ar[l, bend left, "T^{-1}|1"]
	\ar[loop below,looseness=20, "U|U"]
\end{tikzcd}
\hspace{-.5cm}
=
 	 	\begin{tikzcd}[labels={description},column sep=1.2cm, row sep=.6cm,shorten = -.5mm]
 \phantom{\cdots} 
 \xs^0_{-\frac{3}{2}} 
 	\arrow[dr, bend left, "Z|U" ]
 	\arrow[from = dr, bend left=40, "W|1"]
 & \phantom{\xs_{-\frac{1}{2}}}
  &[2mm] \phantom{\xs_{\frac{1}{2}}}
 & \xs^0_{\frac{3}{2}}
 &
  \\
  \ys^0_{-\frac{3}{2}}
 	\arrow[u,gray, "W"]
 	\arrow[d,gray,"Z"] 
 &
 \xs^0_{-\frac{1}{2}}
 	\arrow[r, bend left, "Z|W"]
 	\arrow[from =r, bend left, "W|Z"]
 	\arrow[dd, "\tau|1"]
 	\arrow[dddd,pos=.4, "\sigma|W", bend right =20]
 &
 \xs^0_{\frac{1}{2}}
 	\arrow[from = ur, bend right, "W|Z^2"]
 	\arrow[dd,  "\tau|Z"]
 	\arrow[dddd,pos=.4, bend left=20, "\sigma|1"]
 &
 \ys^0_{\frac{3}{2}} \arrow[u,gray, "W"]
 	\arrow[d,gray,"Z"]
 &
 \\
  \zs^0_{-\frac{3}{2}}
 	\arrow[ur, bend right, pos = 0.4,"Z|W^2"]
 	&
 \phantom{\xs_{-\frac{1}{2}}} 
 & \phantom{\xs_{\frac{1}{2}}}
 & 
 \zs^0_{\frac{3}{2}} 
 	\arrow[from = ul, bend left, "Z|1",pos = 0.3]
 	 \arrow[ul, bend left, "W|U"] 
 \\[4mm]	 	  
 \,
 & 
 \xs^1_{-\frac{1}{2}} 
 	\ar[r, bend left=20, "T"]
 &
 \xs^1_{\frac{1}{2}}
 	\ar[l,bend left=20, "T^{-1}"]
 \\
\,
 & 
 \ys^1_{-\frac{1}{2}}  
 	\arrow[u,gray, "W"] 
 	\arrow[d,gray,"Z"]  
	 \ar[r, bend left=20, "T"]
 &  
 \ys^1_{\frac{1}{2}} 
 	\arrow[u,gray, "W"] 
 	\arrow[d,gray,"Z"]  
 	\ar[l,bend left=20, "T^{-1}"]
 \\
 & \zs^1_{-\frac{1}{2}}
 	\ar[r, bend left=20, "T"]  
 &  \zs^1_{\frac{1}{2}}
 	\ar[l,bend left=20, "T^{-1}"]
 	 	\end{tikzcd}
 }
 \]
 \caption{The $H$-function and $DA$-bimodule of the link $C_{3,2}$ (the $(3,2)$-cable of the Hopf link). Here, $\cS$ denotes the the staircase complex of $T_{2,3}$.  There are also maps $h_{W,Z}$, $h_{Z,W}$, $h_{\sigma,Z}$, $h_{\sigma,W}$, $h_{\tau,Z}$ and $h_{\tau,W}$ which are not shown, but are described in Section~\ref{sec:general cabling}}
 \label{fig:(3,2)_torus}
\end{figure}
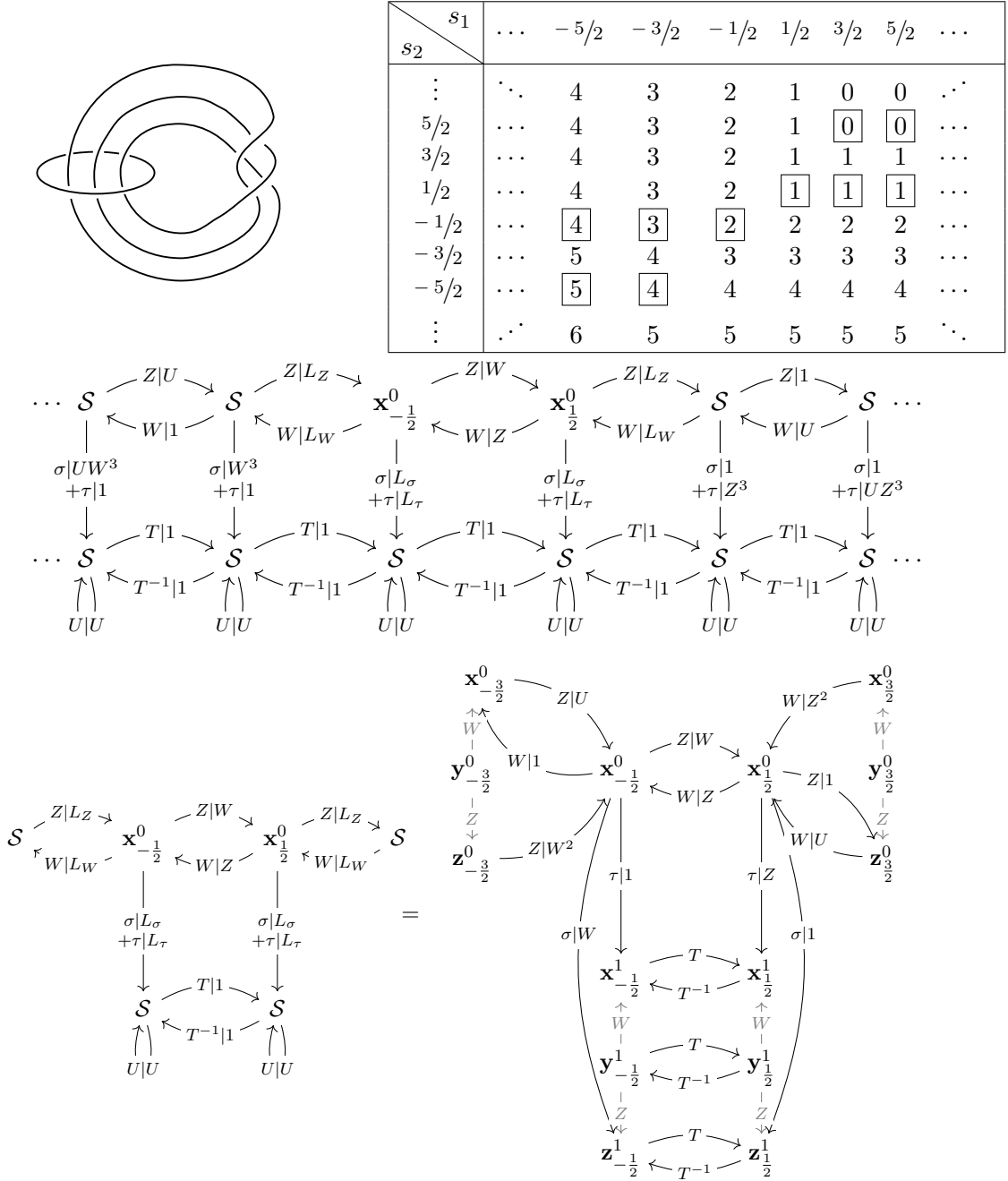

\subsection{More general L-space patterns}

\label{sec: More general L space patterns}
We now describe some additional families of L-space links which one can use to obtain L-space satellite operators. 

\begin{lem}
\label{lem:L-space-link-constructor-lemma}
Suppose that $L$ is a 2-component link with two components $L=J\cup \mu$, where $\mu$ is an unknot (if we ignore $J$). Let $J'$ be obtained by inserting a negative twist along $\mu$ to the knot $J$. Then $L$ is an L-space link if both $J$ and $J'$ are L-space knots.  
	\label{lem:L space links}
\end{lem}
\begin{proof}
We consider $\Lambda = (1,n)$ framed surgery on $L$ for some $n \gg 0$, i.e. we give $J$ framing $n$ and the meridian $\mu$ framing 1.  Note that a $\Lambda$-framed surgery on $L$ is homeomorphic to a $(n-\lk(J,\mu)^2)$-framed surgery on $J'$ by blowing down $\mu$. 	Then \cite{LiuLSpaceLinks}*{Lemma~2.6} applies and shows that $L$ is an L-space link. 
\end{proof}
\begin{rem}
	The converse of Lemma~\ref{lem:L-space-link-constructor-lemma} does not need to hold: If $L$ is an L-space link, then $J$ must be an L-space knot by Lemma~\ref{lem:sublinks}, but $J'$ does not need to be an L-space knot.
\end{rem}

In general, let $J$ be a knot in $S^3$, and $\mu$ be an unknot in the complement of $J$. Let $J_m$ denote the knot obtained by inserting $m$ positive twists to $J$ along $\mu$. If $m$ is negative, then $J_m$ is obtained by inserting $|m|$ negative twists to $J$ along $\mu$. Suppose there exists some integer $m$ such that both $J_{m-1}$ and $J_{m}$ are L-space knots. Then Lemma~\ref{lem:L-space-link-constructor-lemma} shows that the link $L=J_m \cup \mu$ is an L-space link. Therefore our technique can compute the $DA$-bimodule ${}_{\cK} \cX_{(0,0)}(L)^{R}$. See \cite{motegi2016space} for more examples of knots $J$ satisfying this property.

If we view $J_m$ now as being a knot in the solid torus complementary to $\mu$, we note that the $(n-m)$-framed satellite operation with pattern $J_m$ produces the same knot as the $n$-framed satellite operation with pattern $J$. Therefore our techniques can compute satellite operators constructed from patterns as above.

\begin{rem}
	In particular, \cite{motegi2016space}*{Theorem~8.1} gives an infinite family $\{J_m\}_{m\in \Z}$ of L-space knots with tunnel number two, each obtained by $m$ twists to a knot $J$ along an unknot $\mu$. By Lemma~\ref{lem:L-space-link-constructor-lemma}, the links $\mu \cup J_m$ are L-space links for all $m\in \Z$. On the other hand, if $P$ is a $(1,1)$ pattern, then $P(U,0)$ (the underlying knot of $P$ included into $S^3$) is a $(1,1)$-knot. Since any $(1,1)$-pattern knot has tunnel number $1$, this provides examples of L-space patterns which are not $(1,1)$-patterns.
	
	Conversely, there are also many $(1,1)$-pattern $P$ such that $L_P$ is not an L-space link. For example, the mirror  of the positive Whitehead pattern in Section \ref{sec:Whitehead double} is given by a $(1,1)$-pattern, but the corresponding $L_P$ (the negatively clasped Whitehead link) is not an L-space link. 
\end{rem}

\section{The identity cobordism and the elliptic involution}

\label{sec:identity-cobordism}

In this section, we compute the $DA$-bimodule for the mapping cylinder of the identity map $\bT^2\times [0,1]$ and the mapping cylinder of the elliptic involution. We recall that the elliptic involution of the torus is the map $z\mapsto -z$, if we identity $\bT^2=\R^2/\Z^2$.

We presently introduce our candidate, which we call the \emph{elliptic bimodule}, ${}_{\cK}[E]^{\cK}$. This module is rank 1 over the idempotent ring $\ve{I}$ and has structure map given by
\[
\delta_2^1(a,1)=1\otimes E(a)
\]
where $E\colon \cK\to \cK$ is the algebra morphism satisfying 
\[
E(W)=Z, \quad E(Z)=W,\quad E(\sigma)=\tau,\quad E(\tau)=\sigma,\quad E(T)=T^{-1},\quad \text{and}\quad  E(U)=U. 
\]

Recall that the identity module ${}_{\cK} [\bI]^{\cK}$ is also a rank 1 bimodule, but has structure map
\[
\delta_2^1(a,1)=1\otimes a.
\]

The main theorem of this section is the following (stated as Theorem~\ref{thm:mapping-cylinder-identity-intro} in the introduction):
\begin{thm}
\label{thm:identity-cobordism} There are homotopy equivalences
\[
{}_{\cK} \cX(\Cyl_{\bI})^{\cK}\simeq {}_{\cK} [\bI]^{\cK} \quad \text{and} \quad {}_{\cK} \cX(\Cyl_E)^{\cK}\simeq {}_{\cK}[E]^{\cK}.
\] 
\end{thm}

We will prove Theorem~\ref{thm:identity-cobordism} by representing the corresponding bordered manifolds in terms of surgery on a link. The link for the mapping cylinder of the identity map will be the connected sum of a negative and a positive Hopf link, while the link for the elliptic involution will be the connected sum of two negative Hopf links. See Figure~\ref{fig:cable4}.

\begin{figure}[h]
\begingroup%
  \makeatletter%
  \providecommand\color[2][]{%
    \errmessage{(Inkscape) Color is used for the text in Inkscape, but the package 'color.sty' is not loaded}%
    \renewcommand\color[2][]{}%
  }%
  \providecommand\transparent[1]{%
    \errmessage{(Inkscape) Transparency is used (non-zero) for the text in Inkscape, but the package 'transparent.sty' is not loaded}%
    \renewcommand\transparent[1]{}%
  }%
  \providecommand\rotatebox[2]{#2}%
  \newcommand*\fsize{\dimexpr\f@size pt\relax}%
  \newcommand*\lineheight[1]{\fontsize{\fsize}{#1\fsize}\selectfont}%
  \ifx\svgwidth\undefined%
    \setlength{\unitlength}{185.66089085bp}%
    \ifx\svgscale\undefined%
      \relax%
    \else%
      \setlength{\unitlength}{\unitlength * \real{\svgscale}}%
    \fi%
  \else%
    \setlength{\unitlength}{\svgwidth}%
  \fi%
  \global\let\svgwidth\undefined%
  \global\let\svgscale\undefined%
  \makeatother%
  \begin{picture}(1,0.20071841)%
    \lineheight{1}%
    \setlength\tabcolsep{0pt}%
    \put(0,0){\includegraphics[width=\unitlength,page=1]{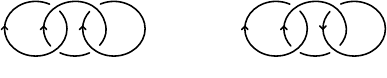}}%
    \put(0.19332656,0.00627387){\makebox(0,0)[t]{\lineheight{1.25}\smash{\begin{tabular}[t]{c}$0$\end{tabular}}}}%
    \put(0.80387161,0.00627387){\makebox(0,0)[t]{\lineheight{1.25}\smash{\begin{tabular}[t]{c}$0$\end{tabular}}}}%
  \end{picture}%
\endgroup%

\caption{Representing the mapping cylinder of the identity map on the torus (left) and the elliptic involution (right) in terms of links. These bordered 3-manifolds are obtained by removing neighborhoods of the left and right most link components, and surgering on the middle.}
\label{fig:cable4}
\end{figure}

We will write ${}_{\cK}\cH_+^{\cK}$ for the $DA$-module of the positive Hopf link, where we give both components framing 0.  We will write ${}_{\cK} \cH_-^{\cK}$ for the $DA$-bimodule of the negative Hopf link. The $DA$ bimodule ${}_{\cK} \cH_-^{\cK}$ is shown in Figure~\ref{fig:Hopf-link-diagram}. 

\begin{rem} In \cite{ZemBordered}, we wrote ${}_{\cK} \cZ_{(0,0)}^{\cK}$ for the bimodule we are denoting by ${}_{\cK} \cH_{-}^{\cK}$. Therein, ${}_{\cK} \cH_{(0,0)}^{\cK}$ denoted a larger (i.e. non-reduced) bimodule which was homotopy equivalent to ${}_{\cK} \cZ_{(0,0)}^{\cK}$.
\end{rem}

 For the proof of Theorem~\ref{thm:identity-cobordism}, it is sufficient to show that
\[
{}_{\cK} \cH_-^{\cK}\boxtimes {}_{\cK} \cH_-^{\cK}\simeq {}_{\cK} [E]^{\cK} \quad \text{and} \quad {}_{\cK} \cH_-^{\cK}\boxtimes {}_{\cK} \cH_+^{\cK}\simeq {}_{\cK} [\bI]^{\cK}. 
\]

\begin{rem}
 In \cite{ZemBordered}, we considered two modules, denoted ${}_{\cK}\cH^{\cK}$ and ${}_{\cK}\bar \cH^{\cK}$, for the negative Hopf link. This difference arose from different choices in the construction of the link surgery formula (more precisely, for different choices of \emph{arc systems}; See \cite{ZemBordered}*{Section~1.7}). These were distinguished by the presence of a $\delta_3^1$ term in the differential, and were due to the alpha/beta asymmetry in the construction of the link surgery complex. In this paper, we write $\cH_-$ for the module that was denoted $\bar \cH$ in \cite{ZemBordered}. This is the bimodule whose reduced model has no $\delta_3^1$ differential. 
\end{rem} 

\begin{rem}
\label{rem:string-orientation-reversal}
 In \cite{ZemBordered}*{Section~17}, we gave a direct computation of the module ${}_{\cK}\cH_-^{\cK}$. The same techniques also apply to the positive Hopf link, and an entirely routine modification of the aforementioned argument shows that
\begin{equation}
{}_{\cK} \cH_+^{\cK}={}_{\cK} \cH_-^{\cK}\boxtimes {}_{\cK} [E]^{\cK}.
\label{eq:string-reversal}
\end{equation}
In the subsequent sections, we will show that
\[
{}_{\cK} \cH_-^{\cK}\boxtimes {}_{\cK} \cH_-^{\cK}\simeq {}_{\cK} [E]^{\cK}.
\]
Note that combining this result with Equation~\eqref{eq:string-reversal} and the fact that ${}_{\cK} [E]^{\cK}\boxtimes {}_{\cK} [E]^{\cK}\simeq {}_{\cK} [\bI]^{\cK}$, we obtain the statement that
\[
{}_{\cK} \cH_-^{\cK}\boxtimes {}_{\cK} \cH_+^{\cK}\simeq {}_{\cK} [\bI]^{\cK}.
\]
More generally, it is straightforward to show from the description in \cite{ZemExact} that tensoring $\cX_{\Lambda}(L)^{\cK\otimes \cdots \otimes \cK}$ with ${}_{\cK} [E]^{\cK}$ has the effect of reversing the string orientation of a link component.
\end{rem}

\subsection{The Hopf link $DA$ module}

In Figure~\ref{fig:Hopf-link-diagram} we recall the type-$DA$ module of the negative Hopf link, denoted $\cH_-$. The original computation is performed in \cite{ZemBordered}*{Section~16}. Note that we are using slightly different notation. Firstly, the subscript on each generator denotes its idempotent. Secondly, in idempotent $(0,0)$, we are writing $W_{00}^i$ for what was called $W^i \ve{a}$, and we are writing $Z_{00}^i$ for what was called $Z^i \ve{d}$. Note that $W_{00}^0\neq Z_{00}^0$.  Note however that in idempotent $(0,1)$, we have $W_{01}^0=Z_{01}^0=1_{01}^0$. 
\begin{figure}[p]
\[
\begin{tikzcd}[labels=description, column sep=1.1cm, row sep=0cm]
\cdots
	\ar[r, bend left, "Z|U"]
&W^2_{01}
	\ar[r, bend left, "Z|U"]
	\ar[l, bend left, "W|1"]
& W_{01}
	\ar[r, bend left, "Z|U"]
	\ar[l, bend left, "W|1"]
&1_{01}
	\ar[r, "Z|1", bend left]
	\ar[l, "W|1", bend left]
&Z_{01}
	\ar[r, bend left, "Z|1"]
	\ar[l, bend left,"W|U"]
&Z^2_{01}
	\ar[r, bend left, "Z|1"]
	\ar[l, bend left,"W|U"]
& \cdots
	\ar[l, bend left,"W|U"]
\\[2cm]
\cdots
	\ar[r, bend left, "Z|U"]
&W^2_{00}
	\ar[r, bend left, "Z|U"]
	\ar[l, bend left, "W|1"]
	\ar[d, "\substack{\sigma|U^2Z \\ \tau|1}"]
	\ar[u, gray,"T \sigma"]
	\ar[ul, gray, "T\tau"]
& W_{00}
	\ar[r, bend left, "Z|U"]
	\ar[l, bend left, "W|1"]
	\ar[d, "\substack{\sigma|UZ \\ \tau|1}"]
	\ar[u, gray,"T \sigma"]
	\ar[ul, gray, "T\tau"]
&W_{00}^0
	\ar[r, "Z|Z", bend left]
	\ar[l, "W|1", bend left]
	\ar[d, "\substack{\sigma|Z\\ \tau|1}"]
	\ar[u, gray,"T \sigma"]
	\ar[ul, gray, "T\tau"]
&Z_{00}^0
	\ar[r, bend left, "Z|1"]
	\ar[l, bend left, "W|W"]
	\ar[d, "\substack{\sigma|1 \\ \tau|W}"]
	\ar[u,gray,"\sigma"]
	\ar[ul, gray, "\tau"]
&Z_{00}
	\ar[r, bend left, "Z|1"]
	\ar[l, bend left,"W|U"]
	\ar[d, "\substack{\sigma|1 \\ \tau|UW}"]
	\ar[u,gray,"\sigma"]
	\ar[ul, gray, "\tau"]
&\cdots
	\ar[l, bend left,"W|U"]
	\ar[d,"\substack{\sigma|1 \\ \tau|U^2W}"]
	\ar[u,gray,"\sigma"]
	\ar[ul, gray, "\tau"]
\\[2cm]
\cdots
	\ar[r, bend left, "T|1"]
&T^{-2}_{10}
	\ar[r, bend left, "T|1"]
	\ar[l, bend left, "T^{-1}|1"]
	\ar[loop below,looseness=20, "U|U"]
& T^{-1}_{10}
	\ar[r, bend left, "T|1"]
	\ar[l, bend left, "T^{-1}|1"]
	\ar[loop below,looseness=20, "U|U"]
&T^0_{10}
	\ar[r, "T|1", bend left]
	\ar[l, "T^{-1}|1", bend left]
	\ar[loop below,looseness=20, "U|U"]
&T_{10}^1
	\ar[r, bend left, "T|1"]
	\ar[l, bend left,"T^{-1}|1"]
	\ar[loop below,looseness=20, "U|U"]
&T_{10}^2
	\ar[r, bend left, "T|1"]
	\ar[l, bend left,"T^{-1}|1"]
	\ar[loop below,looseness=20, "U|U"]
& \cdots
	\ar[l, bend left,"T^{-1}|1"]
\end{tikzcd}
\]
\vspace{1cm}
\[
\begin{tikzcd}[labels=description, column sep=1.1cm, row sep=0cm]
\cdots
	\ar[r, bend left, "Z|U"]
&W^2_{01}
	\ar[r, bend left, "Z|U"]
	\ar[l, bend left, "W|1"]
	\ar[d, "\substack{\sigma|U^2\\ \tau|T^{-1} }"]
&W_{01}
	\ar[r, bend left, "Z|U"]
	\ar[l, bend left, "W|1"]
	\ar[d, "\substack{\sigma|U\\ \tau|T^{-1} }"]
&1_{01}
	\ar[r, "Z|1", bend left]
	\ar[l, "W|1", bend left]
	\ar[d, "\substack{\sigma|1\\ \tau|T^{-1} }"]
&Z_{01}
	\ar[r, bend left, "Z|1"]
	\ar[l, bend left,"W|U"]
	\ar[d, "\substack{\sigma|1\\ \tau|UT^{-1} }"]
&Z^{2}_{01}
	\ar[r, bend left, "Z|1"]
	\ar[l, bend left,"W|U"]
	\ar[d, "\substack{\sigma|1\\ \tau|U^2T^{-1} }"]
& \cdots
	\ar[l, bend left,"W|U"]
\\[2cm]
\cdots
	\ar[r, bend left, "T|1"]
&T^{-2}_{11}
	\ar[r, bend left, "T|1"]
	\ar[l, bend left, "T^{-1}|1"]
	\ar[loop below,looseness=20, "U|U"]
& T^{-1}_{11}
	\ar[r, bend left, "T|1"]
	\ar[l, bend left, "T^{-1}|1"]
	\ar[loop below,looseness=20, "U|U"]
&T^0_{11}
	\ar[r, "T|1", bend left]
	\ar[l, "T^{-1}|1", bend left]
	\ar[loop below,looseness=20, "U|U"]
&T_{11}^1
	\ar[r, bend left, "T|1"]
	\ar[l, bend left,"T^{-1}|1"]
	\ar[loop below,looseness=20, "U|U"]
&T^{2}_{11}
	\ar[r, bend left, "T|1"]
	\ar[l, bend left,"T^{-1}|1"]
	\ar[loop below,looseness=20, "U|U"]
& \cdots \ar[l, bend left,"T^{-1}|1"]
\\[2cm]
 \cdots
	\ar[r, bend left, "T|1"]
&T^{-2}_{10}
	\ar[r, bend left, "T|1"]
	\ar[l, bend left, "T^{-1}|1"]
	\ar[loop below,looseness=20, "U|U"]
	\ar[u,gray, bend right=35, "\sigma"]
	\ar[ul, "\tau",gray]
& T^{-1}_{10}
	\ar[r, bend left, "T|1"]
	\ar[l, bend left, "T^{-1}|1"]
	\ar[loop below,looseness=20, "U|U"]
	\ar[u,gray, bend right=35, "\sigma"]
	\ar[ul, "\tau",gray]
&T^{0}_{10}
	\ar[r, "T|1", bend left]
	\ar[l, "T^{-1}|1", bend left]
	\ar[loop below,looseness=20, "U|U"]
	\ar[u,gray, bend right=35, "\sigma"]
	\ar[ul, "\tau",gray]
&T^{1}_{10}
	\ar[r, bend left, "T|1"]
	\ar[l, bend left,"T^{-1}|1"]
	\ar[loop below,looseness=20, "U|U"]
	\ar[u,gray, bend right=35, "\sigma"]
	\ar[ul, "\tau",gray]
&T^{2}_{10}
	\ar[r, bend left, "T|1"]
	\ar[l, bend left,"T^{-1}|1"]
	\ar[loop below,looseness=20, "U|U"]
	\ar[u,gray, bend right=35, "\sigma"]
	\ar[ul, "\tau",gray]
& \cdots
	\ar[l, bend left,"T^{-1}|1"]
	\ar[ul, "\tau",gray]
\end{tikzcd}
\]
\caption{The $DA$-bimodule of the negative Hopf link ${}_{\cK} \cH_-^{\cK}$. The gray arrows represent structure maps $\delta^1_1: {}_{\cK} \cH_-^{\cK} \to {}_{\cK} \cH_-^{\cK} \otimes \cK$. The remaining arrows are structure maps $\delta^1_2: \cK \otimes {}_{\cK} \cH_-^{\cK} \to {}_{\cK} \cH_-^{\cK}\otimes \cK$. Subscripts denote idempotents.} 
\label{fig:Hopf-link-diagram}
\end{figure}

\subsection{Simplification of $(\cH_-\boxtimes \cH_-)^\cK$}
\label{sec:simplifying hopf link tensor hopf link}
In this section, we simplify the type-$D$ module underlying ${}_{\cK}\cH_-^{\cK}\boxtimes {}_{\cK}\cH_-^{\cK}$. We describe the computation in both the $U$-adic and chiral topologies, however we are most interested in the chiral topology.

We will write $(\cH_-\boxtimes \cH_-)^\cK$ for the underlying type-$D$ module of the tensor product. In this section, we will construct a strong deformation retraction  (see Definition~\ref{def:strong-deformation-retraction}) for the chirally completed modules
\[
\begin{tikzcd}
\ar[loop left, "H^1"](\cH_-\boxtimes \cH_-)^{\cK}\ar[r,shift left, "\Pi^1"] &\ve{I}^{\cK}, \ar[l, shift left, "I^1"]
\end{tikzcd}
\]
where $\ve{I}^{\cK}$ denotes the idempotent ring, viewed as a type-$D$ module over $\cK$ with vanishing $\delta^1$ map. 

\begin{rem} If we use the $U$-adic topology, the type-$D$ module $(\cH_-\boxtimes \cH_-)^\cK$ does not deformation retract onto $\ve{I}^{\cK}$, but instead deformation retracts onto a larger subspace.
\end{rem}

Note that by using the homological perturbation lemma for hypercubes, Lemma \ref{lem:homotopy perturbation of hypercube}, it is sufficient to construct strong deformation retractions 
\[
\ve{I}_0 \cdot (\cH_-\boxtimes \cH_-)^\cK \cdot \ve{I}_0\simeq \ve{I}_0^{\cK},\quad \ve{I}_1\cdot (\cH_-\boxtimes \cH_-)^\cK \cdot \ve{I}_0\simeq 0,
\]
\[
\ve{I}_0 \cdot (\cH_-\boxtimes \cH_-)^\cK \cdot \ve{I}_1\simeq 0,\quad \text{and} \quad \ve{I}_1\cdot (\cH_-\boxtimes \cH_-)^\cK \cdot \ve{I}_1\simeq \ve{I}_1^{\cK}.
\]
We perform the above computation in the subsequent sections. For $\veps\in \bE_2$, we will write $(I_{\veps}^1, \Pi_{\veps}^1, H_{\veps}^1)$ for the strong deformation retractions that we will construct on $\ve{I}_{\veps_1} \cdot (\cH_1\boxtimes \cH_1)^{\cK}\cdot \ve{I}_{\veps_2}$.

\subsubsection{Idempotent $(0,0)$}

We now study the $(0,0)$-idempotent of $\cH_-\boxtimes \cH_-$. In this section, we construct a strong deformation retraction
\[
\begin{tikzcd}
\ar[loop left, "H^1_{00}"]\ve{I}_0\cdot(\cH_-\boxtimes \cH_-)^{\cK}\cdot \ve{I}_0 \ar[r,shift left, "\Pi^1_{00}"] &\ve{I}_0^{\cK}. \ar[l, shift left, "I_{00}^1"]
\end{tikzcd}
\]

The differential on $(\cH_-\boxtimes \cH_-)^{\cK}$ in this idempotent is contributed from the $\delta_1^1$ map of the left $\cH_-$ being input into the $\delta_2^1$ of the right $\cH_-$, as in the following diagram: 
\[
\begin{tikzcd}[row sep=.2cm] 
\cH_-
	\ar[d] & 
\cH_-
\ar[dd]\\
\delta_1^1\ar[dr]\ar[dd]&\\
\, & \delta_2^1\ar[d]\\
\,&\,
\end{tikzcd}.
\]

Computing directly from Figure~\ref{fig:Hopf-link-diagram}, we have the following differentials of  $ \ve{I}_0\cdot \cH_-\boxtimes \cH_-\cdot \ve{I}_0$:
\[
\begin{split}
\delta_1^1[W^i_{00}|W^j_{00}] &= [W^i_{01}|T^{-j+1}_{11}]\otimes U^j Z+[W^{i+1}_{01}| T^{-j+1}_{11}]\otimes 1,
\\
\delta_1^1[W^i_{00} |Z^j_{00}]&= [W^i_{01} |T^{j+2}_{11}]\otimes 1+[W^{i+1}_{01}|T^{j+2}_{11}]\otimes U^j W
\\
\delta_1^1 [Z^i_{00}|W^j_{00}]&= [Z^i_{01}| T^{-j}_{11}]\otimes 1+[Z^{i+1}_{01}| T^{-j}_{11}]\otimes U^jZ
\\
\delta_1^1 [Z^i_{00}|Z^j_{00}]&= [Z^i_{01}|T^{j+1}_{11}]\otimes U^jW+[Z^{i+1}_{01}| T^{j+1}_{11}]\otimes 1.
\end{split}
\]
We observe that the idempotent $(0,0)$-subspace of our complex decomposes as a direct sum of staircases.
For $j\ge 0$, we write $X_j$ for the staircase complex which contains $[W^0_{00}|W^{j}_{00}]$. Also for $j\ge 0$, we write $Y_j$ for the staircase complex which contains the generator $[W^0_{00}|Z^j_{00}]$.

If $j>0$, the complex $X_{j}$ takes the following form:
\[
\begin{tikzcd}[labels=description]
\cdots
&{[W^1_{00}| W^j_{00}]}
	\ar[d, "U^j Z"]
	\ar[dl]
&{[W^0_{00}| W^j_{00}]}
	\ar[d, "U^j Z"]
	\ar[dl, "1"]
&{[Z^0_{00}| W^{j-1}_{00}]}
	\ar[dl, "1"]
	\ar[d, "U^{j-1} Z"]
&{[Z^1_{00}|W^{j-1}_{00}]}
	\ar[dl,"1"]
	\ar[d, "U^{j-1} Z"]
&\cdots
	\ar[dl]
\\
\cdots&
{[W^1_{01}| T^{-j+1}_{10}]}
&{[W^0_{01}| T^{-j+1}_{10}]}
& {[Z^1_{01}| T^{-j+1}_{10}]}
& {[Z^2_{01}|T^{-j+1}_{10}]} &\cdots
\end{tikzcd}
\]

In the following, we say that a type-$D$ module $X^\cK$ is \emph{contractible} if $\bI_X=\d(H^1)$, for some type-$D$ morphism $H^1\colon X^\cK\to X^{\cK}$.

\begin{lem}\label{lem:simplify-X_j-complexes}
For any $j>1$, the complex $X_j$ is contractible in both the chiral and $U$-adic topologies. For $j=1$, the complex $X_j$ is contractible in the chiral topology.
\end{lem}
\begin{proof} A null-homotopy may be written down using the homological perturbation lemma for chain complexes, Lemma~\ref{lem:HPL-chain-complexes}.
In terms of the notation of the present lemma, we write $X_j=(C,d^1+\a^1)$, where $d^1$ consists of the terms of the differential weighted by $1$, and $\a^1$ consists of the remaining differentials. The map $h^1$ maps backwards along the arrows weighted by $1$. The null-homotopy $H^1$ of $X_j$ is then given by
\begin{equation}
H^1=h^1\circ (1+ \a^1 \circ h^1)^{-1}=\sum_{n=0}^\infty h^1\circ (\a^1 \circ h^1)^n.
\label{eq:def-Homotopy-HPL-X_j}
\end{equation}
To apply the homological perturbation lemma, it suffices to show that $(1+\a^1\circ h^1)^{-1}=\sum_{n=0}^\infty(\a^1 \circ h^1)^n$ defines a convergent sum in the chiral topology. We write $\g^1$ for this map.

When $j>1$, the map $\g^1$ is well-defined in both the chiral and $U$-adic topologies because, given an $N$, all but finitely many terms in Equation~\eqref{eq:def-Homotopy-HPL-X_j} are in the ideal $(U^N)$.

We now consider the complex $X_1$ in the chiral topology and claim that $\g^1$ defines a well-defined and continuous map $\g^1\colon X_1\to X_1\vecotimes \bF[W,Z]$. Here $\bF[W,Z]$ is given the $I$-adic topology, where $I$ is the ideal $I=(W,Z)$, and $X_1$ is given the product topology $X_1\iso \prod_{\Z\times \{0,1\}} \bF$. (Here $\Z\times \{0\}$ enumerates the generators of the bottom level of the staircase, while $\Z\times \{1\}$ enumerates generators in the top level of the staircase). We claim that if $E\subset X_1\vecotimes \bF[W,Z]$ is open, then $\g^{-1}(E)$ is also open. Note that $E$ is open if and only if there is an open subspace $U\subset X_1$ so that $U\otimes \bF[W,Z]\subset E$ and for all $\xs\in X_1$, there is an open subspace $V_{\xs}\subset \bF[W,Z]$ so that $x\otimes V_{\xs}\subset E$. See Section \ref{sec:linear topological spaces}.

Recall that the basis of open subspaces in $X_1$ consist of the spans of cofinite collections of basis elements in $\Z\times \{0,1\}$, i.e. the subspaces $\prod_{x\in (\Z\times \{0,1\})\setminus S} \bF$ where $S$ is a finite subset of $(\Z\times \{0,1\})$. It is helpful to write $X_{\co(S)}$ for $\prod_{x\in (\Z\times \{0,1\})\setminus S} \bF$. 

Therefore, we need to show that if $S=\{\xs_1,\dots, \xs_n\}$ is a finite set of basis elements in $X_1$, and $N_{\xs}$ is a collection of positive integers for each basis element in $X_1$, then we can find a finite set of basis elements $T\subset X_1$ such that
\[
\g(X_{\co(T)})\subset X_{\co(S)}\vecotimes \bF[W,Z]+ \sum_{\xs\in X_1} \xs\otimes (W^{N_{\xs}}, Z^{N_{\xs}})
\]
where the sum $\xs\in X_1$ is taken over basis elements. Of course, this is equivalent to showing
\begin{equation}
\g(X_{\co(T)})\subset X_{\co(S)}\vecotimes \bF[W,Z]+ \sum_{\xs\in S} \xs\otimes (W^{N_{\xs}}, Z^{N_{\xs}}).
\label{eq:equivalent-condition-gamma-continuity}
\end{equation}
We let $N$ be the maximum of $N_{\xs}$. We then observe in $(1+\a \circ h)^{-1}=\sum_{n=0}^\infty (\a \circ h)^n$, only finitely many of the components pointing to a given $\xs$ will have algebra weight outside of the ideal $(W^N,Z^N)$. Therefore, we pick $T$ to contain all basis elements $\ys$ so that there is a component of $\sum_{n=0}^\infty (\a \circ h)^n$ to $\xs\otimes r$, for some $\xs\in S$ and $r\not\in (W^N,Z^N)$. There are only finitely many such $\ys$, so we know a finite $T$ can be selected so that Equation~\eqref{eq:equivalent-condition-gamma-continuity} holds. The proof is complete.
\end{proof}

Next, we study the complex $X_0$, which contains the generators $[W^0_{00}|W^0_{00}]$ and  $[Z^0_{00}|Z^0_{00}]$. This complex takes the following form:
\begin{equation*}
\begin{tikzcd}[labels=description]
\cdots
&{[W^1_{00}|W^0_{00}]}
	\ar[d, "Z"]
	\ar[dl, "1"]
&{[W^0_{00}|W^0_{00}]}
	\ar[d, "Z"]
	\ar[dl, "1"]
&{[Z^0_{00}|Z^0_{00}]}
	\ar[dl, "W"]
	\ar[d, "1"]
&{[Z^1_{00}|Z^0_{00}]}
	\ar[dl,"W"]
	\ar[d, "1"]
&\cdots
	\ar[dl, "W"]
\\
\cdots&
{[W_{01}|T^1_{10}]}
& {[1_{01}|T^1_{10}]}
& {[Z_{01}|T^1_{10}]}
& {[Z^2_{01}|T^1_{10}]} &\cdots
\end{tikzcd}
\label{eq:non-vanishing-one}
\end{equation*}

\begin{lem}
\label{lem:homotopy-equivalence-idempotent-00}
The complex $X_0$ is homotopy equivalent to the rank 1 subcomplex spanned by $[1_{01}|T_{10}^1]$ with vanishing $\delta_1^1$. This holds in both the chiral and $U$-adic topologies. 
\end{lem}

\begin{proof} The proof follows similarly to Lemma~\ref{lem:simplify-X_j-complexes}. A strong deformation retraction can be constructed by the homological perturbation lemma, canceling all of the arrows labeled with a $1$. Continuity and well-definedness of the map in the chiral topology follows as before. In the $U$-adic topology, we note that the maps, $\Pi^1$, $H^1$ and $I^1$ are well-defined before taking completions (as only finitely many homological perturbation zig-zags start at a given generator) and therefore are well-defined after completing at $U$.
\end{proof}

Next, we consider the complexes $Y_j$, for $j\ge 0$. Here, $Y_j$ is the staircase containing the generators $[W^0_{00}|Z^j_{00}]$. This complex takes the following form:
\[
\begin{tikzcd}[labels=description]
\cdots
&{[W^1_{00}|Z^j_{00}]}
	\ar[d, "1"]
	\ar[dl, "U^jW"]
&{[W^0_{00}|Z^j_{00}]}
	\ar[d, "1"]
	\ar[dl, "U^jW"]
&{[Z^0_{00}|Z^{j+1}_{00}]}
	\ar[dl, "U^{j+1}W"]
	\ar[d, "1"]
&{[Z^1_{00}|Z^{j+1}_{00}]}
	\ar[dl,"U^{j+1}W"]
	\ar[d, "1"]
&\cdots
	\ar[dl, ]
\\
\cdots&
{[W^1_{01}|T^{j+2}_{10}]}
&{[W^0_{01}|T^{j+2}_{10}]}
&{[Z^1_{01}|T^{j+2}_{10}]}
& {[Z^2_{01}|T^{j+2}_{10}]}
&\cdots
\end{tikzcd}
\]
Parallel to Lemma~\ref{lem:simplify-X_j-complexes}, we have the following:

\begin{lem} The complex $Y_j$ is contractible in both the chiral and $U$-adic topologies if $j\ge 1$. If $j=0$, it is contractible in the chiral topology. 
\end{lem}

\subsubsection{Idempotent $(0,1)$}

This idempotent is also a direct sum of staircases, similar to idempotent $(0,0)$. The top rows of the staircase are of the form $[W^i_{00}| W^j_{01}]$, $[W^i_{00}| Z^j_{01}]$, $[Z^i_{00}| W^j_{01}]$ or $[Z^i_{00}|Z^j_{01}]$. The bottom row is of the form $[W^i_{01}| T^j_{11}]$ or $[Z^i_{01}| T^j_{11}]$. 

We now describe the differentials:
\[
\begin{split}
\delta^1_1([W^i_{00}|W^j_{01}])=&[W^i_{01}|T_{11}^{-j+1}]\otimes U^j+[W^{i+1}_{01}|T_{11}^{-j+1}]\otimes T^{-1}\\
\delta_1^1([W_{00}^i|Z_{01}^j])=& [W^i_{01}|T^{j+1}_{11}]\otimes 1+[W^{i+1}_{01}|T^{j+1}_{11}]\otimes U^j T^{-1}\\
\delta_1^1([Z_{00}^i|W_{01}^j])=&[Z^{i+1}_{01}|T^{-j}_{11}]\otimes U^j+[Z_{01}^i|T^{-j}_{11}]\otimes T^{-1}\\
\delta_1^1([Z_{00}^i|Z_{01}^j])=&[Z^{i+1}_{01}|T^j_{11}]\otimes 1+[Z_{01}^i|T^j_{11}]\otimes U^j T^{-1}. 
\end{split}
\]
The above complexes split as a sum of $X_j$, where the complex has a generator $[W_{00}^0|W_{01}^{-j}]$ if $j<0 $ and $[W_{00}^0|Z_{01}^j]$ if $j\ge 0$. Note that both $Z^0_{01}$ and $W^0_{01}$ represent the same generator $1_{01}$ in the $(0,1)$ idempotent of $\cH_-$.

For $j\ge 0$ the complex $X_{j}$ takes the form
\[
\begin{tikzcd}[labels=description]
\cdots
&{[W^1_{00}|Z_{01}^j]}
	\ar[d, "1"]
	\ar[dl, "U^j T^{-1}"]
&{[W^0_{00}|Z^j_{01}]}
	\ar[d, "1"]
	\ar[dl, "U^jT^{-1}"]
&{[Z^0_{00}|Z^{j+1}_{01}]}
	\ar[dl, "U^{j+1}T^{-1}"]
	\ar[d, "1"]
&{[Z^1_{00}|Z^{j+1}_{01}]}
	\ar[dl,"U^{j+1}T^{-1}"]
	\ar[d, "1"]
&\cdots
	\ar[dl, ]
\\
\cdots&
{[W^1_{01}|T^{j+1}_{11}]}
&{[1_{01}|T^{j+1}_{11}]}
&{[Z^1_{01}|T^{j+1}_{11}]}
& {[Z^2_{01}|T^{j+1}_{11}]}
&\cdots
\end{tikzcd}
\]
Next, we describe $X_{-j}$ for $j>0$. This takes the form:
\[
\begin{tikzcd}[labels=description]
\cdots
&{[W^1_{00}|W_{01}^j]}
	\ar[d, "U^j"]
	\ar[dl, " T^{-1}"]
&{[W^0_{00}|W^j_{01}]}
	\ar[d, "U^j"]
	\ar[dl, "T^{-1}"]
&{[Z^0_{00}|W^{j-1}_{01}]}
	\ar[dl, "T^{-1}"]
	\ar[d, "U^{j-1}"]
&{[Z^1_{00}|W^{j-1}_{01}]}
	\ar[dl,"T^{-1}"]
	\ar[d, "U^{j-1}"]
&\cdots
	\ar[dl, ]
\\
\cdots&
{[W^1_{01}|T^{-j+1}_{11}]}
&{[1_{01}|T^{-j+1}_{11}]}
&{[Z^1_{01}|T^{-j+1}_{11}]}
& {[Z^2_{01}|T^{-j+1}_{11}]}
&\cdots
\end{tikzcd}
\]

\begin{lem} The complexes $X_j$ are contractible for all $j$ in the chiral topology. The complex $X_j$ is contractible in the $U$-adic topology for all $j\neq 0,-1$. 
\end{lem}

\begin{proof} The null-homotopy $H^1$ giving $\d(H^1)=\id_{X_j}$ is given via the homological perturbation lemma by reversing the arrows labeled by $1$ or $T^{-1}$, using our standard technique. It is straightforward to see that this gives well-defined and convergent maps in both both topologies. (Note some extra care is needed for $X_0$ and $X_{-1}$). 
\end{proof}

\subsubsection{Idempotent $(1,0)$}

The differential is given by
\[
\begin{split}
\delta_1^1([T_{10}^i|W^j_{00}])&=[T^i_{11}|T_{10}^{-j}]\otimes Z U^j+[T_{11}^{i-1}|T_{10}^{-j}]\otimes 1\\
\delta_1^1([T_{10}^i|Z^j_{00}])&=[T_{11}^i|T_{10}^{j+1}]\otimes 1+[T_{11}^{i-1}|T_{10}^{j+1}]\otimes W U^j. 
\end{split}
\]
The underlying type-$D$ structure in idempotent $(1,0)$ splits as a sum of $\Z$-many staircase complexes which are of one of two forms:
\[
\begin{tikzcd}[labels=description]
\cdots
&{[T^{-1}_{10}|W_{00}^j]}
	\ar[d, "ZU^j"]
	\ar[dl, "1"]
&{[T^{0}_{10}|W^j_{00}]}
	\ar[d, "ZU^j"]
	\ar[dl, "1"]
&{[T^{1}_{10}|W^{j}_{00}]}
	\ar[dl, "1"]
	\ar[d, "ZU^j"]
&{[T^{2}_{10}|W^{j}_{00}]}
	\ar[dl,"1"]
	\ar[d, "ZU^j"]
&\cdots
	\ar[dl, ]
\\
\cdots&
{[T^{-1}_{11}|T^{-j}_{10}]}
&{[T^{0}_{11}|T^{-j}_{10}]}
&{[T^{1}_{11}|T^{-j}_{10}]}
& {[T^{2}_{11}|T^{-j}_{10}]}
&\cdots
\end{tikzcd}
\]
\[
\begin{tikzcd}[labels=description]
\cdots
&{[T^{-1}_{10}|Z_{00}^j]}
	\ar[d, "1"]
	\ar[dl, "W U^j"]
&{[T^{0}_{10}|Z^j_{00}]}
	\ar[d, "1"]
	\ar[dl, "W U^j"]
&{[T^{1}_{10}|Z^{j}_{00}]}
	\ar[dl, "W U^j"]
	\ar[d, "1"]
&{[T^{2}_{10}|Z^{j}_{00}]}
	\ar[dl,"W U^j"]
	\ar[d, "1"]
&\cdots
	\ar[dl, ]
\\
\cdots&
{[T^{-1}_{11}|T^{j+1}_{10}]}
&{[T^{0}_{11}|T^{j+1}_{10}]}
&{[T^{1}_{11}|T^{j+1}_{10}]}
& {[T^{2}_{11}|T^{j+1}_{10}]}
&\cdots
\end{tikzcd}
\]

\begin{lem} In the chiral topology, all of the above complexes are contractible. In the $U$-adic topology, all of the complexes are contractible except for the ones with weights $W$ and $Z$.
\end{lem}

The proof follows from the same technique as in Lemma~\ref{lem:homotopy-equivalence-idempotent-00}, so we leave the details to the reader.

\subsubsection{Idempotent $(1,1)$}

In this idempotent, the underlying type-$D$ modules takes the following form:
\[
\begin{split}
\delta_1^1([T_{10}^i|W^j_{01}])&=[T^i_{11}|T^{-j}_{11}]\otimes U^j+[T^{i-1}_{11}|T^{-j}_{11}]\otimes T^{-1}\\
\delta_1^1([T_{10}^i|Z^j_{01}])&=[T^i_{11}|T^{j}_{11}]\otimes 1+[T^{i-1}_{11}|T^{j}_{11}]\otimes U^j T^{-1}. 
\end{split}
\]
This decomposes as a direct sum of staircases $S_j$, $j\in \Z$, where $S_j$ contains the generator $[T_{10}^0|W_{01}^{-j}]$ if $j\le0$ and $[T_{10}^-|Z_{01}^j]$ if $j\ge 0$.

We write out $S_0$:
\[
\begin{tikzcd}[labels=description]
\cdots
&{[T^{-1}_{10}|1_{01}]}
	\ar[d, "1"]
	\ar[dl, "T^{-1}"]
&{[T^{0}_{10}|1_{01}]}
	\ar[d, "1"]
	\ar[dl, "T^{-1}"]
&{[T^{1}_{10}|1_{01}]}
	\ar[dl, "1"]
	\ar[d, "T^{-1}"]
&{[T^{2}_{10}|1_{01}]}
	\ar[dl,"1"]
	\ar[d, "T^{-1}"]
&\cdots
	\ar[dl, ]
\\
\cdots&
{[T^{-1}_{11}|T^{0}_{11}]}
&{[T^{0}_{11}|T^{0}_{11}]}
&{[T^{1}_{11}|T^{0}_{11}]}
& {[T^{2}_{11}|T^{0}_{11}]}
&\cdots
\end{tikzcd}
\]

Again both $W^0_{01}$ and $Z^0_{01}$ denote the generator $1_{01}$ in the $(0,1)$-idempotent.
 
\begin{lem}
The staircases $S_j$ are contractible if $j\neq 0$ in both the chiral and $U$-adic topologies. In the $U$-adic topology, $S_0$ admits a strong deformation retraction to the rank 1 subcomplex spanned by $[T^0_{11}|T^0_{11}]$ (with vanishing $\delta^1$). In the chiral topology, the staircase $S_0$ admits a strong deformation retraction to the rank 1 subcomplex spanned by the element $\sum_{i=-\infty}^{0} [T^i_{10}|1_{01}]\otimes T^i  +\sum_{i=1}^{\infty} [T^i_{10}|1_{01}] \otimes T^{-i+1}. $
\end{lem}

In summary, using the chiral topology, by using the homological perturbation lemma, we may simplify all of the staircase complexes appearing as summands. The only staircases which are not contractible are the staircase $X_0$ in idempotent $(0,0)$ and the staircase $S_0$ in idempotent $(1,1)$. By using the homological perturbation lemma, we see that the sum of these two staircases is homotopy equivalent to the rank $2$ vector space $\ve{I}$, which we view as a type-$D$ module $\ve{I}^{\cK}$ with vanishing differential.

\subsection{The module structure}
\label{sec:module structure of elliptic involution}
The results of the previous section construct a strong deformation retraction 
\[
\begin{tikzcd}
\ar[loop left, "H^1"](\cH_-\boxtimes \cH_-)^{\cK}\ar[r,shift left, "\Pi^1"] &\ve{I}^{\cK} \ar[l, shift left, "I^1"]
\end{tikzcd}
\]
when we equip the module and algebra with the chiral topologies. By homological perturbation of $DA$-bimodules, there is an induced $DA$-bimodule structure on $\ve{I}^{\cK}$, for which we write ${}_{\cK} [\phi_{-,-}]^{\cK}$ and a strong deformation retraction
\[
{}_{\cK}\cH_-^{\cK}\boxtimes {}_{\cK}\cH_-^{\cK}\simeq {}_{\cK} [\phi_{-,-}]^{\cK}.
\]
There is of course another module ${}_{\cK}[\phi_{-,+}]^{\cK}$ which is a retraction of $\cH_{-}\boxtimes \cH_{+}$.

\begin{lem}
\label{lem:identity-d21-computation} In the chiral topology, we have 
\[
{}_{\cK} [\phi_{-,-}]^{\cK}={}_{\cK} [E]^{\cK}\quad \text{and} \quad {}_{\cK} [\phi_{-,+}]^{\cK}={}_{\cK} [\bI]^{\cK}. 
\]
\end{lem}

Note that Theorem~\ref{thm:identity-cobordism} is an immediate consequence of Lemma~\ref{lem:identity-d21-computation}.

\begin{proof}[Proof of Lemma~\ref{lem:identity-d21-computation}]
Note firstly that by Remark~\ref{rem:string-orientation-reversal} and the fact that ${}_{\cK}[E]^{\cK}\boxtimes {}_{\cK}[E]^{\cK}={}_{\cK}[\bI]^{\cK}$, it is sufficient to show the claim about $\phi_{-,-}$. 

Observe firstly that $\cH_- \boxtimes \cH_-$ admits an internal grading supported in $\{0,1\}$, which we call the \emph{staircase grading}. Here, we say a generator $\xs$ is in staircase grading 0 if it is in the bottom row of the staircase. Otherwise, we say it is in staircase grading 1. Recall that we write $I_{\veps}^1$, $H_{\veps}^1$ and $\Pi_{\veps}^1$ for the strong deformation retractions constructed above on the type-$D$ module $\ve{I}_{\veps_1}\cdot (\cH_-\boxtimes \cH_-)^{\cK}\cdot \ve{I}_{\veps_2}$, for $\veps=(\veps_1,\veps_2)\in \bE_2$. Note that the maps $I_{\veps}^1$, $H_{\veps}^1$ and $\Pi_{\veps}^1$ are homogeneously graded with respect to the staircase grading. The map $H_{\veps}^1$ increases the staircase grading by 1. Here $\epsilon \in \mathbb{E}_2=\left\{0,1\right\}\times \left\{0,1\right\}$.  Additionally, the maps $\delta_1^1$ and $\delta_2^1$ of ${}_{\cK} \cH_-\boxtimes \cH_{-}^{\cK}$ all preserve the staircase grading. Therefore, using the description of the actions from Lemma~\ref{lem:HPL-modules}, we observe that the structure maps $\delta_j^1$ on ${}_{\cK}[\phi_{-,-}]^{\cK}$ vanish if $j>2$.  It is easy also to see that $\delta_1^1=0$.

Therefore it is sufficient to show that
\[
\delta_2^1(W,i_0)=i_0\otimes Z,\quad \delta_2^1(Z,i_0)=i_0\otimes W,\quad \delta_2^1(\sigma,i_0)=i_1\otimes \tau\quad \delta_2^1(\tau,i_0)=i_1\otimes \sigma
\]
and
\[
\delta_2^1(U,i_1)=i_1\otimes U\quad \delta_2^1(T,i_1)=i_1\otimes T^{-1}. 
\]

We first verify that $\delta_2^1(W,i_0)=i_0\otimes Z$. We observe that $I_{00}^1(i_0)=[1_{01}|T_{10}^1]$. Applying $\delta_2^1(W,-)$ on $\cH_-\boxtimes \cH_-$ yields $[W_{01}|T_{10}^1]$. Finally, applying $\Pi^1_{00}$ yields $i_0\otimes Z$. (For the application of $\Pi_{00}^1$, the recipe is to travel backwards along the 1 arrow from $[W_{01}|T_{10}^1]$ to $[W_{00}^0|W_{00}^0]$, and then forward along the arrow labeled $Z$ to $[1_{01}|T_{10}^1]$. See the complex $X_0$ in Lemma \ref{lem:homotopy-equivalence-idempotent-00}.)
The computation that $\delta_2^1(Z,i_0)=i_0\otimes W$ is essentially the same.

Next we move on to the computation of $\delta_2^1(\sigma,i_0)$. There are now two paths of maps which contribute via homological perturbation:
\begin{enumerate}
\item Apply $I_{00}^1$, then apply $\delta_2^1(\sigma,-)$, then apply $H_{10}^1$, then apply a term of $\delta_1^1$ which increases cube grading in $\mathbb{E}_2$, and finally apply $\Pi_{11}^1$.
\item Apply $I_{00}^1$, then apply $\delta^1_1$ (increasing cube grading in $\mathbb{E}_2$), apply $H_{01}^1$, then apply $\delta_2^1(\sigma,-)$, and finally apply $\Pi_{11}^1$. 
\end{enumerate}
We now consider the first path. We have $I_{00}^1(i_0)=[1_{01}|T_{10}^1]$. We observe
\[
\delta_2^1(\sigma, [1_{01}|T_{10}^1])=[T_{11}^0|T^1_{10}]\otimes 1.
\]
Applying $H_{10}^1$ yields
\[
\sum_{k\ge 0} [T_{10}^{-k}|Z_{00}^0]\otimes W^{k-1}.
\]
Applying $\delta_1^1$ yields
\[
\sum_{k\ge 0} [T_{10}^{-k}|1_{01}]\otimes \tau W^{k-1}+[T_{10}^{-k}|Z_{01}^1]\otimes  \sigma W^{k-1}.
\]
Note that the projection map $\Pi_{11}^1$ is can be chosen to be projection onto any generator in the top row. For the sake of simplicity, we declare this map to be projection onto $[T_{01}^{-1}|1_{01}]$. (Other choices change the final module only by homotopy equivalence). Composing the above with $\Pi_{11}^1$, we obtain $0$. 

We now compute the second path of arrows. As before, $I_{00}^1$ sends $i_0$ to $[1_{01}|T_{10}^1]$, which is sent to $[1_{01}|T_{11}^0]\otimes \tau+[1_{01}|T_{11}^1]\otimes \sigma$ by $\delta_1^1$. We now apply $H_{01}^1$, which sends the above to
\[
\sum_{k\ge 0} [Z_{00}^k|1_{01}]\otimes T^{k+1}\tau+\sum_{k\ge 0} [W_{00}^k| 1_{01}]\otimes T^{-k} \sigma.  
\]
We now apply $\delta_2^1(\sigma,-)$ and obtain
\[
\sum_{k \ge 0} [T_{10}^{-k+1}|1_{01}^1]\otimes T^{k+1} \tau+ \sum_{k \ge0} [T_{10}^0|Z_{01}^1]\otimes T^{-k} \sigma. 
\]
With our choice of projection $\Pi_{11}^1$, the above is sent to $i_1\otimes \tau$.

An entirely similar argument shows that $\delta_2^1(\tau,i_0)=i_1\otimes \sigma$.

The computations of $\delta_2^1(U,i_1)$ and $\delta_2^1(T,i_1)$ are similar. We focus on $\delta_2^1(T,i_1)$ since it is more interesting. Our inclusion map sends $i_1$ to the element
\[
\sum_{k\in \Z} [T_{10}^k|1_{01}]\otimes T^{k+1}.
\]
Then the internal map $\delta_2^1(T,-)$ sends the above to $\sum_{k\in \Z} [T_{10}^{k+1}|1_{01}]\otimes T^{k+1}$, which the map $\Pi_{11}^1$ sends to $i_1\otimes T^{-1}$. The argument for $\delta_2^1(U,i)$ is similar.
\end{proof}

 \subsection{Quasi-invertibility}

Theorem~\ref{thm:identity-cobordism} implies the following corollary (stated as Corollary~\ref{cor:quasi-inverse} in the introduction):

\begin{cor}\label{cor:quasi-inverse-text} The bimodule ${}_{\cK|\cK} \bI^{\Supset}$ is quasi-invertible, with quasi-inverse $\cX(\Cyl_\bI)^{\cK\otimes \cK}$.
\end{cor}
\begin{proof} By definition:
\[
{}_{\cK} \cX(\Cyl_{\bI})^{\cK}:=\cX(\Cyl_{\bI})^{\cK\otimes \cK} \boxtimes {}_{\cK|\cK} \bI^{\Supset}.
\]
Theorem~\ref{thm:identity-cobordism} implies the above bimodule is homotopy equivalent to ${}_{\cK} [\bI]^{\cK}$, completing the proof.
\end{proof}

 \section{The bimodule for $(q,qn +1)$-cables}
 \label{sec:cables}

In this section, we describe our satellite formula for $(q,qn+ 1)$-cables. This could be viewed as applying the satellite operation with pattern $T_{q,1}$ torus knot to the $n$-framed companion knot. Our approach is similar to the approach in Section~\ref{sec:identity-cobordism}, where we compute the bimodules for the identity cobordism and elliptic involution. In the case of cabling, we are mostly interested in the effect of cabling on the knot Floer complex.

We define our cabling module as the tensor product
\[
{}_{\cK}\scC_{q,1}^{\bF[W,Z]}:={}_{\cK}\cH_-^{\cK}\boxtimes {}_{\cK}\cX(T_{2,2q})^{\bF[W,Z]},
\]
where ${}_{\cK}\cX(T_{2,2q})^{\bF[W,Z]}$ is described in Section~\ref{sec:T(2,2q)}. We note that $T(2,2q) $ is the link $L_{P}$ for the $(q,1)$-cabling pattern $P$, i.e. $T_{2,2q}$ is obtained by taking the $(q,1)$-cable of one component of the positive Hopf link. The analysis in this section is a prototype for the analysis carried out for general L-space patterns in Section~\ref{sec:GeneralSatellite}.

In this section, we work exclusively with the chiral topologies on the algebra and modules.

\subsection{Simplifying the underlying type-$D$ module}

In this section, we simplify the underlying chain complex of ${}_{\cK}\scC_{q,1}^{\bF[W,Z]}$. The underlying chain complex naturally splits as a direct sum of infinite staircases, all but finitely many are homotopically trivial. We first establish some notations for the complex ${}_{\cK}\cX(T_{2,2q})^{\bF[W,Z]}$. We will give a slightly asymmetrical enumeration of the generators (which will make some of the later formulas easier to write down). We will enumerate generators in idempotent 0 as $\ve{x}_{t}$, where $t\in \Z$, and we will enumerate the generators in idempotent 1 as $T^t$. Below is the $DA$-bimodule ${}_{\cK}\cX(T_{2,2q})^{\bF[W,Z]}$ with $q=3$, the same as in Figure \ref{fig:T226-diagram-H}, with index shifted by $\frac{3}{2}$.

\[
\begin{tikzcd}[labels=description, column sep=1.2cm, row sep=0cm]
	\cdots
	\ar[r, bend left, "Z|U"]
	&\xs_{-1}
	\ar[r, bend left, "Z|U"]
	\ar[l, bend left, "W|1"]
	\ar[d, "\substack{\sigma|W^3U \\ \tau|1}"]
	& \xs_{0}
	\ar[r, bend left, "Z|W"]
	\ar[l, bend left, "W|1"]
	\ar[d, "\substack{\sigma|W^3 \\ \tau|1}"]
	&\xs_{1}
	\ar[r, "Z|W", bend left]
	\ar[l, "W|Z", bend left]
	\ar[d, "\substack{\sigma|W^2\\ \tau|Z}"]
	&\xs_{2}
	\ar[r, bend left, "Z|W"]
	\ar[l, bend left, "W|Z"]
	\ar[d, "\substack{\sigma|W \\ \tau|Z^2}"]
	&\xs_{3}
	\ar[r, bend left, "Z|1"]
	\ar[l, bend left,"W|Z"]
	\ar[d, "\substack{\sigma|1\\ \tau|Z^3}"]
	&\xs_{4}
	\ar[r, bend left, "Z|1"]
	\ar[l, bend left,"W|U"]
	\ar[d,"\substack{\sigma|1\\ \tau|UZ^3}"]
	& \cdots
	\ar[l, bend left,"W|U"]
	\\[2cm]
	\cdots
	\ar[r, bend left, "T|1"]
	&T^{-1}
	\ar[r, bend left, "T|1"]
	\ar[l, bend left, "T^{-1}|1"]
	\ar[loop below,looseness=20, "U|U"]
	& T^{0}
	\ar[r, bend left, "T|1"]
	\ar[l, bend left, "T^{-1}|1"]
	\ar[loop below,looseness=20, "U|U"]
	& T^{1}
	\ar[r, "T|1", bend left]
	\ar[l, "T^{-1}|1", bend left]
	\ar[loop below,looseness=20, "U|U"]
	&T^{2}
	\ar[r, bend left, "T|1"]
	\ar[l, bend left,"T^{-1}|1"]
	\ar[loop below,looseness=20, "U|U"]
	&T^3
	\ar[r, bend left, "T|1"]
	\ar[l, bend left,"T^{-1}|1"]
	\ar[loop below,looseness=20, "U|U"]
	& T^4
	\ar[l, bend left,"T^{-1}|1"]
	\ar[loop below,looseness=20, "U|U"]
	\ar[r, "T|1", bend left]
	&\cdots 
	\ar[l, "T^{-1}|1", bend left]
\end{tikzcd}
\]

 \subsubsection{Idempotent $0$}
 
 In this idempotent, the tensor product $\cH_- \boxtimes \cX(T_{2,2q})$.  We now record the differentials of the complex:

\begin{enumerate}
\item $\delta_1^1[W_{00}^i|\ve{x}_t]=\begin{cases}[W^i_{01}|T^{t+1}]\otimes W^qU^{-t}+[W^{i+1}_{01}|T^{t+1}]\otimes 1 & \text{ if } t\le0\\
[W^i_{01}|T^{t+1}]\otimes W^{q-t} +[W^{i+1}_{01}|T^{t+1}]\otimes Z^t& \text{ if } 0\le t\le q\\
[W^i_{01}|T^{t+1}]\otimes 1+[W^{i+1}_{01}|T^{t+1}]\otimes Z^q U^{t-q}& \text{ if } q\le t.
\end{cases}$
\item $\delta_1^1 [Z^i_{00}|\ve{x}_t]=\begin{cases} [Z^{i+1}_{01}|T^t]\otimes W^q U^{-t} +[Z_{01}^i|T^t]\otimes 1& \text{ if } t \le 0\\
[Z^{i+1}_{01}|T^t]\otimes W^{q-t} +[Z_{01}^i|T^t]\otimes Z^t & \text{ if } 0\le t\le q\\
[Z_{01}^{i+1}|T^t]\otimes 1 +[Z_{01}^i|T^t]\otimes Z^q U^{q-t}& \text{ if } q \le t
\end{cases}.$
\end{enumerate}

As a chain complex, the resulting complex splits as into a direct sum of staircase complexes.  We adopt similar conventions as for the elliptic bimodule and will write $Y_t$ for the subcomplex containing the generator $[1_{01}|T^t]$. Many of these staircases are easily seen to be contractible.

For $1\le t\le q$, the complex $Y_t$ takes the following form:
\[
 \begin{tikzcd}[labels=description, column sep={2.5cm,between origins}]
 \cdots
\ar[dr]
 &{[W_{00}^1|\xs_{t-1}]}
 	\ar[d, "Z^{t-1}"]
 	\ar[dr, "W^{q-t+1}"]
 &{[W_{00}^0|\xs_{t-1}]}
 	\ar[d, "Z^{t-1}"]
 	\ar[dr, "W^{q-t+1}"]
 &{[Z^0_{00}|\xs_t]}
 	\ar[d, "Z^{t}"]
 	\ar[dr, "W^{q-t}"]
 &{[Z^1_{00}|\xs_t]}
 	\ar[d, "Z^{t}"]
 	\ar[dr, "W^{q-t}"]
 &\cdots
 \\
 \cdots&
 {[W^2_{01}|T^t]}
 &{[W_{01}|T^t]}
 & {[1_{01}|T^t]}
 & {[Z_{01}|T^t]}
 &\cdots
 \end{tikzcd}
 \]

The complexes $Y_1,\dots, Y_q$ are not contractible, however $Y_t$ is contractible if $t\not\in \{1,\dots, q\}$.

 \begin{rem} It is possible to truncate the complexes $Y_1$ and $Y_q$ to half-infinite complexes. The resulting $DA$-bimodule structure will have a non-trivial $\delta_3^1$, however.
 \end{rem}

\subsubsection{Idempotent 1}

We now investigate idempotent 1. The differential takes the following form:
\[
\delta_1^1 [T_{10}^i|\ve{x}_t]=\begin{cases}
	[T_{11}^i|T^t]\otimes W^qU^{-t} + [T_{11}^{i-1}|T^{t}]\otimes  1& \text{ if }t\le 0.\\
 [T_{11}^i|T^t]\otimes W^{q-t} + [T_{11}^{i-1}|T^{t}]\otimes Z^t & \text{ if } 0\le t\le q\\
[T_{11}^i|T^t]\otimes 1 + [T_{11}^{i-1}|T^{t}]\otimes Z^qU^{t-q} & \text{ if } q\le t
\end{cases}
\]

In idempotent 1, the module splits as a sum of staircases $S_t$, $t\in \Z$, where $S_t$ contains generators of the form $[T_{11}^i|T^t]$. It is straightforward to see that $S_t$ is contractible in the chiral topology unless $t\in \{1,\dots, q-1\}$.

\subsection{The module structure}

We now investigate the module structure. It is not hard to see that $\delta_j^1=0$ if $j>2$, so we focus on computing $\delta_2^1(a,-)$ where $a\in \{W,Z,T,U,\sigma,\tau\}$. See Section \ref{sec:module structure of elliptic involution} for a similar argument. The remaining $\delta_2^1$ actions can be worked out as in Section~\ref{sec:U-equivariant}.

Before describing the module structure, it is easier for us to relabel the generators of ${}_{\cK}\scC_{q,1}^{\bF[W,Z]}$.
In idempotent 0, we write 
\[
W^i \ve{a}_t, \quad Z^i \ve{d}_t,\quad  W^{n}Z^m \theta_t
\]
for the generators $[W^i_{00}| \ve{x}_{t-1}]$, $[Z^i_{00}|\ve{x}_t]$ and $[W_{01}^n Z_{01}^m| T^t]$, respectively.  In the above, $i\ge 0$, $t\in \{1,\dots, q\}$ and $\min(m,n)=0$.
Similarly, in idempotent 1, we write
\[
T^i \psi_t\qquad \text{and} \qquad T^i \phi_t
\]
for the generators $ [T^i_{10}|\xs_{t}]$ and $[T^i_{11}|T^t] $ respectively, where $i\in \Z$ and $t\in \{1,\dots, q-1\}$. 

We now translate the differential $\delta_1^1$ into the above notation. We have
\[
\begin{split}
\delta_1^1(W^i \ve{a}_t)&=W^{i+1} \theta_t\otimes Z^{t-1}+W^i \theta_t\otimes W^{q-t+1}
\\
\delta_1^1(Z^i \ve{d}_t)&=Z^i\theta_t\otimes Z^{t}+Z^{i+1}\theta_t\otimes W^{q-t}.
\end{split}
\]
Similarly we have
\[
\delta_1^1(T^i \psi_t)=T^i\phi_t\otimes W^{q-t}+T^{i-1}\phi_t\otimes Z^t.
\]

We now describe the module structure. In idempotent 0, we have the following module actions:
\[
\delta_2^1(W, W^i \ve{a}_t)=W^{i+1}\ve{a}_t\otimes 1\quad \text{and} \quad \delta_2^1(Z, W^i \ve{a}_t)=\begin{cases} W^{i-1} \ve{a}_t\otimes U& \text{ if } i>0\\
\ve{d}_t\otimes W& \text{ if } i=0.
\end{cases}
\]
\[
\delta_2^1(W, Z^i\ve{d}_t)=\begin{cases} Z^{i-1}\ve{d}_t\otimes U & \text{ if }i>0\\
\ve{a}_t\otimes Z & \text{ if } i=0 \end{cases}\quad \text{and} \quad \delta_2^1(Z,Z^i \ve{d}_t)=Z^{i+1} \ve{d}_t\otimes 1. 
\]
In idempotent 1, we have the following actions:
\[
\delta_2^1(T^{\pm 1}, T^i\psi_t)=T^{i\pm 1} \psi_t\otimes 1,\quad \delta_2^1(T^{\pm 1}, T^i \phi_t)=T^{i\pm 1} \phi_t\otimes 1,
\]
\[
\delta_2^1(U, T^i\psi_t)=T^i\psi_t\otimes U\quad \text{and} \quad  \delta_2^1(U, T^i \phi_t)=T^i \phi_t\otimes U.
\]
The actions of $\sigma$ and $\tau$ are given as follows:
\[
\delta_2^1(\sigma, W^i \ve{a}_t)=T^{-i} \psi_{t}\otimes U^i W,\quad \delta_2^1(\tau, W^i \ve{a}_t)=T^{-i} \psi_{t-1}\otimes 1
\]
\[
\delta_2^1(\sigma, Z^i \ve{d}_t)=T^{i+1} \psi_{t}\otimes 1,\quad  \delta_2^1(\tau, Z^i \ve{d}_t)=T^{i+1} \psi_{t-1}\otimes U^i Z.
\]
(In the above, we interpret $\psi_{t}$ as being zero unless $t\in \{1,\dots, q-1\}$).
\[
\delta_2^1(\sigma, W^i \theta_t)=T^{-i} \phi_t\otimes U^i, \quad \delta_2^1(\sigma, Z^i \theta_t)=T^i \phi_t\otimes 1,
\]
\[
\delta_2^1(\tau,W^i \theta_t)=T^{-i}\phi_{t-1}\otimes 1,\qquad \delta_2^1(\tau, Z^i \theta_t)=T^i \phi_{t-1}\otimes U^i. 
\]

The module for $q=2$ is shown in Figure~\ref{fig:C2-module}.

\begin{figure}[h]
\[
\begin{tikzcd}[ column sep={1.8cm,between origins}]
\cdots
&
W \ve{a}_1
	\ar[r, "W^{2}"]
	\ar[l, "1",swap]
	\ar[d,  "\sigma|U W"]
&
W \theta_1
	\ar[d,  "\sigma|U"]
&
\ve{a}_1
	\ar[r, "W^2"]
	\ar[l, "1",swap]
	\ar[d,  "\sigma|W"]
&
\theta_1
	\ar[d,  "\sigma|1"]
&
\ve{d}_1
	\ar[r, "W"]
	\ar[l, "Z",swap]
	\ar[d,  "\sigma|1"]
&
Z \theta_1
	\ar[d,  "\sigma|1"]
&
Z \ve{d}_1
	\ar[r, "W"]
	\ar[l, "Z",swap]
	\ar[d,  "\sigma|1"]
&\cdots
\\
\cdots
&
T^{-1}\psi_1
	\ar[r, "W"]
	\ar[l, "Z",swap] 
& 
T^{-1}\phi_1
& 
\psi_1 
	\ar[r, "W"]
	\ar[l, "Z",swap] 
& 
\phi_1 
&
T\psi_1 
	\ar[r, "W"]
	\ar[l, "Z",swap] 
& 
T\phi_1
& 
T^2\psi_1
	\ar[r, "W"]
	\ar[l, "Z",swap] 
&\cdots
\\
\cdots
&
W \ve{a}_2
	\ar[r, "W"]
	\ar[l, "Z",swap]
	\ar[u,  "\tau|1"]
& 
W \theta_2
	\ar[u,  "\tau|1"]
& 
\ve{a}_2 
	\ar[r, "W"]
	\ar[l, "Z",swap] 
	\ar[u,  "\tau|1"]
& 
\theta_2 
	\ar[u,  "\tau|1"]
&
\ve{d}_2 
	\ar[r, "1"]
	\ar[l, "Z^2",swap] 
	\ar[u,  "\tau|Z"]
& 
Z \theta_2
	\ar[u,  "\tau|U"]
& 
Z \ve{d}_2
	\ar[r, "1"]
	\ar[l, "Z^2",swap] 
	\ar[u, "\tau|UZ"]
&\cdots
\end{tikzcd}
\]
\caption{The module ${}_{\cK} \scC_{2,1}^{\bF[W,Z]}$. Only $\delta_1^1$ and $\delta_2^1(\sigma,-)$ and $\delta_2^1(\tau,-)$ are shown. Not shown are the maps $\delta_2^1(W,-)$, $\delta_2^1(Z,-)$ and $\delta_2^1(T^{\pm 1},-)$. There is no $\delta_3^1$.}
\label{fig:C2-module}
\end{figure}

  \section{General L-space satellite operators}
  \label{sec:GeneralSatellite}

 If $P$ is a satellite operator, and $L_P$ the corresponding 2-component link, then the pairing theorem
 of Theorem~\ref{thm:connected-sum}, combined with naturality of the modules under Dehn surgery  (Theorem~\ref{thm:invariance}) and the idempotent reduction formula of Equation~\eqref{eq:sublink-surgery-formula}, implies that there is a homotopy equivalence
  \begin{equation}
 \cCFK(P(K,n))^{\bF[W,Z]}\simeq \cX_n(Y,K)^{\cK}\boxtimes {}_{\cK} \cH_-^{\cK} \boxtimes {}_{\cK} \cX(L_P)^{\bF[W,Z]}. \label{eq:satellite=tensorproduct}
  \end{equation}  
  Here, ${}_{\cK}\cH_-^{\cK}$ denotes the $(0,0)$-framed bimodule for the negative Hopf link. We will write $\bX(P,K,n)^{\bF[W,Z]}$ for the right-hand side of Equation~\eqref{eq:satellite=tensorproduct}.
 
 In this section, we will study the model $\bX(P,K,n)^{\bF[W,Z]}$ of $\cCFK(P(K,n))$ when $L_P$ is an L-space link, describing it first as an infinitely generated type-$D$ module. We give explicit descriptions of the generators and differentials appearing in this complex.  In the subsequent Section~\ref{sec:truncation}, we describe how to systematically truncate and obtain finitely generated type-$D$ modules. In Section \ref{sec:Examples of satellite}, we perform a large number of example computations.

Throughout, we focus on the case that $Y$ is an integer homology 3-sphere. When $Y$ is not an integer homology sphere, one can study the tensor product in Equation~\eqref{eq:satellite=tensorproduct} using similar techniques, but the bookkeeping is slightly more complicated.

\subsection{The complex $\bX(P,K,n)^{\bF[W,Z]}$}

In this section, we describe the structure of the tensor product $\bX(P,K,n)^{\bF[W,Z]}= \cX_n(K)^{\cK}\boxtimes {}_{\cK} \cH_-^{\cK} \boxtimes {}_{\cK} \cX(L_P)^{\bF[W,Z]}$. 

Before we begin, we will write
\[
\bH(P)= \left(\Z+\frac{1}{2}\right)\times \left(\Z+\frac{\lk(\mu,P)-1}{2}\right).
\]
Note that this consists of the first two components of the lattice $\bH(K\# H\# L_P)$.

We will write the complex $\bX(P,K,n)^{\bF[W,Z]}$ as a 2-dimensional hypercube:
\[
\bX(P,K,n)^{\bF[W,Z]}=\begin{tikzcd}[column sep=1.5cm, row sep=1.5cm]
 \bE \ar[r, "\Phi^{\mu}+\Phi^{-\mu}"] \ar[d, "\Phi^{K}+\Phi^{-K}",swap] \ar[dr,gray, "\Phi^{\pm K,\pm \mu}", labels=description] & \bF\ar[d, "\Phi^{K}+\Phi^{-K}"]\\
\bJ \ar[r, swap,"\Phi^{\mu}+\Phi^{-\mu}"]& \bM
\end{tikzcd}
\]
We will have decompositions
\[
\bE=\bigoplus_{(s,t)\in \bH(P)} E_{s,t},\qquad \bF=\bigoplus_{(s,t)\in \bH(P)} F_{s}\times \{t\}, 
\]
\[
\bJ=\bigoplus_{(s,t)\in \bH(P)}  J_t\times \{s\}, \qquad
\bM=\bigoplus_{(s,t)\in \bH(P)} M\times \{s,t\}.
\]
We often write $J_{s,t}$ for the corresponding copy $J_t\times \left\{s\right\}$ of $J_t$, and we make a similar definition for $F_{s,t}$ and $M_{s,t}$. Also, each of $E_{s,t}$, $F_s$, $J_t$ and $M$ is a finitely generated type-$D$ module over $\bF[W,Z]$.

\begin{rem} If we view the full surgery formula $\cC_{\Lambda}(K\# H\# L_P)$ as being a 3-dimensional hypercube, with axis directions corresponding to the three link components (with $K$ being the first component, $H$ being the first and second, and $L_P$ the second and third), then we have the following correspondence of the complexes $\bE$, $\bF$, $\bJ$ and $\bM$ with cube points:
\begin{enumerate}
\item $\bE$ corresponds to cube point $(0,0,0)$.
\item $\bF$ corresponds to cube point $(0,1,0)$.
\item $\bJ$ corresponds to cube point $(1,0,0)$.
\item $\bM$ corresponds to cube point $(1,1,0)$. 
\end{enumerate}
\label{rem:hypercubecorner}
\end{rem}

We have
\[
\Phi^K(E_{s,t})\subset J_{s,t}\quad \Phi^{-K}(E_{s,t})\subset J_{s+n,t-1}\quad \Phi^{K}(F_{s,t})\subset M_{s,t}\quad \Phi^{-K}(F_{s,t})\subset F_{s+n, t-1},
\]
\[
\Phi^{\mu}(E_{s,t})\subset F_{s,t} \quad \Phi^{-\mu}(E_{s,t})\subset F_{s-1,t}, \quad \Phi^{\mu}(J_{s,t})\subset M_{s,t}\quad \Phi^{-\mu}(J_{s,t})\subset M_{s-1,t}.
\]
We sometimes refer to $\Phi^{\pm \mu}$ and $\Phi^{\pm K}$ as the \emph{length 1 maps} of surgery complex $\bX(P,K,n)$. 

We refer to $\Phi^{\pm K,\pm \mu}$ as the \emph{length 2} maps. The map $\Phi^{K,\mu}$ shifts the $\bH(P)$-grading the same as $\Phi^K\circ \Phi^\mu$, and similarly for $\Phi^{-K,\mu}$, $\Phi^{K,-\mu}$ and $\Phi^{-K,-\mu}$.

These maps and complexes fit together to form a chain complex of the shape shown in Figure~\ref{fig:large-model-pattern}.
\begin{figure}[h]
\[
\begin{tikzcd}[column sep={2.4cm,between origins}, row sep=1.5cm, labels=description]
\,&\vdots &\vdots &\vdots &\vdots &\vdots& \,
\\
 \cdots
& E_{s-1,t} 
	\ar[l,"\Phi^{-\mu}"]
	\ar[r, "\Phi^{\mu}"]
	\ar[d, "\Phi^{K}"]
	\ar[u, "\Phi^{-K}"]
	\ar[dr,gray, "\Phi^{K,\mu}"]
	\ar[ur,gray, "\Phi^{-K,\mu}"]
	\ar[dl,gray, "\Phi^{K,-\mu}"]
	\ar[ul,gray, "\Phi^{-K,-\mu}"]
&
F_{s-1,t}
	\ar[d, "\Phi^{K}"]
	\ar[u, "\Phi^{-K}"]
&
E_{s,t}
	\ar[l,"\Phi^{-\mu}"]
	\ar[r, "\Phi^{\mu}"]
	\ar[d, "\Phi^{K}"]
	\ar[u, "\Phi^{-K}"]
	\ar[dl,gray, "\Phi^{K,-\mu}"]
	\ar[ul,gray, "\Phi^{-K,-\mu}"]
	\ar[ur,gray, "\Phi^{-K,\mu}"]
	\ar[dr,gray, "\Phi^{K,\mu}"]
&
F_{s,t}
	\ar[d, "\Phi^{K}"]
	\ar[u, "\Phi^{-K}"]
& 
E_{s+1,t}
	\ar[l,"\Phi^{-\mu}"]
	\ar[r, "\Phi^{\mu}"]
	\ar[d, "\Phi^{K}"]
	\ar[u, "\Phi^{-K}"]
	\ar[dl,gray, "\Phi^{K,-\mu}"]
	\ar[ul,gray, "\Phi^{-K,-\mu}"]
	\ar[ur,gray, "\Phi^{-K,\mu}"]
	\ar[dr,gray, "\Phi^{K,\mu}"]
&\cdots 
\\
\cdots
&
J_{s-1,t}
	\ar[l,"\Phi^{-\mu}"]
	\ar[r, "\Phi^{\mu}"]
&
M_{s-1,t}
&
J_{s,t}
	\ar[l,"\Phi^{-\mu}"]
	\ar[r, "\Phi^{\mu}"]
&
M_{s,t}
&
J_{s+1,t}
	\ar[l,"\Phi^{-\mu}"]
	\ar[r, "\Phi^{\mu}"]
&\cdots
\\
\cdots
&
E_{s-n-1,t+1}
	\ar[l,"\Phi^{-\mu}",labels=above]
	\ar[r, "\Phi^{\mu}",labels=above]
	\ar[u, "\Phi^{-K}"]
	\ar[d, "\Phi^K"]
	\ar[ur,gray, "\Phi^{-K,\mu}"]
	\ar[dr,gray, "\Phi^{K,\mu}"]
	\ar[dl,gray, "\Phi^{K,-\mu}"]
	\ar[ul,gray, "\Phi^{-K,-\mu}"]
&
F_{s-n-1,t+1}
	\ar[u, "\Phi^{-K}"]
	\ar[d, "\Phi^K"]
&
E_{s-n,t+1}
	\ar[l,"\Phi^{-\mu}",labels=above]
	\ar[r, "\Phi^{\mu}",labels=above]
	\ar[u, "\Phi^{-K}"]
	\ar[d, "\Phi^K"]
	\ar[dl,gray, "\Phi^{K,-\mu}"]
	\ar[ul,gray, "\Phi^{-K,-\mu}"]
	\ar[ur,gray, "\Phi^{-K,\mu}"]
	\ar[dr,gray, "\Phi^{K,\mu}"]
&
F_{s-n,t+1}
	\ar[u, "\Phi^{-K}"]
	\ar[d, "\Phi^K"]
&
E_{s-n+1,t+1}
	\ar[l,"\Phi^{-\mu}",labels=above]
	\ar[r, "\Phi^{\mu}",labels=above]
	\ar[u, "\Phi^{-K}"]
	\ar[d, "\Phi^K"]
	\ar[dl,gray, "\Phi^{K,-\mu}"]
	\ar[ul,gray, "\Phi^{-K,-\mu}"]
	\ar[ur,gray, "\Phi^{-K,\mu}"]
	\ar[dr,gray, "\Phi^{K,\mu}"]
&
\cdots\\
\phantom{\vdots}&\vdots&\vdots&\vdots&\vdots&\vdots&\phantom{\vdots}
\end{tikzcd}
\]
\caption{A model of the complex $\cCFK(P(K,n))$ given by the box tensor product $\bX(P,K,n)^{\bF[W,Z]}$.}
\label{fig:large-model-pattern}
\end{figure}

\subsection{The complexes $E_{s,t}$, $F_{s}$, $J_t$ and $M$.}
\label{sec:definition of E,F,J,M}

We now describe the complexes appearing $E_{s,t}$, $F_s$, $J_t$ and $M$ appearing in $\bX(K,P,n)$. Note that we continue our assumption that $Y$ is an integer homology 3-sphere.

Our description of these complexes will be in terms of several auxiliary $DA$-bimodules. These in turn will be defined in terms of the surgery $DA$-bimodule 
  \[
  {}_{\cK}\cX_{(0,0)}(L_P)^{\bF[W,Z]}.
\]
 We will schematically describe this complex as follows:
\begin{equation}
\begin{tikzcd}[labels=description, column sep=1.1cm, row sep=0cm]
\cdots
&[-1.2cm]\cC_{t_0-2}
	\ar[r, bend left, "Z|L_Z"]
	\ar[d,"\sigma|L_\sigma", bend left=15, pos=.4]
	\ar[d, "\tau|L_\tau", bend right=15, pos=.6]
	\ar[loop above, looseness=20, "\substack{(W,Z)|h_{W,Z}\\(Z,W)|h_{Z,W}}"]
& \cC_{t_0-1}
	\ar[r, bend left, "Z|L_Z"]
	\ar[l, bend left, "W|L_W"]
	\ar[d,"\sigma|L_\sigma", bend left=15, pos=.4]
	\ar[d, "\tau|L_\tau", bend right=15, pos=.6]
	\ar[loop above, looseness=20, "\substack{(W,Z)|h_{W,Z}\\(Z,W)|h_{Z,W}}"]
&\cC_{t_0}
	\ar[r, bend left, "Z|L_Z"]
	\ar[l, bend left, "W|L_W"]
	\ar[d,"\sigma|L_\sigma", bend left=15, pos=.4]
	\ar[d, "\tau|L_\tau", bend right=15, pos=.6]
	\ar[loop above, looseness=20, "\substack{(W,Z)|h_{W,Z}\\(Z,W)|h_{Z,W}}"]
& \cC_{t_0+1}
	\ar[r, bend left, "Z|L_Z"]
	\ar[l, bend left, "W|L_W"]
	\ar[d,"\sigma|L_\sigma", bend left=15, pos=.4]
	\ar[d, "\tau|L_\tau", bend right=15, pos=.6]
	\ar[loop above, looseness=20, "\substack{(W,Z)|h_{W,Z}\\(Z,W)|h_{Z,W}}"]
& \cC_{t_0+2}
	\ar[l, bend left, "W|L_W"]
	\ar[d,"\sigma|L_\sigma", bend left=15, pos=.4]
	\ar[d, "\tau|L_\tau", bend right=15, pos=.6]
	\ar[loop above, looseness=20, "\substack{(W,Z)|h_{W,Z}\\(Z,W)|h_{Z,W}}"]
&[-1.2cm] \cdots
\\[2cm]
\cdots
& T^{t_0-2}\cS
	\ar[r, bend left, "T|1"]
	\ar[loop below,looseness=20, "U|U"]
& T^{t_0-1}\cS
	\ar[r, bend left, "T|1"]
	\ar[l, bend left, "T^{-1}|1"]
	\ar[loop below,looseness=20, "U|U"]
&T^{t_0}\cS
	\ar[r, bend left, "T|1"]
	\ar[l, bend left, "T^{-1}|1"]
	\ar[loop below,looseness=20, "U|U"]
&T^{t_0+1}\cS
	\ar[r, bend left, "T|1"]
	\ar[l, bend left, "T^{-1}|1"]
	\ar[loop below,looseness=20, "U|U"]
&T^{t_0+2}\cS
	\ar[l, bend left, "T^{-1}|1"]
	\ar[loop below,looseness=20, "U|U"]
&\cdots 
\end{tikzcd}
\label{eq:X(L_P)-complex}
\end{equation}
Here $t_0$ denotes an element of $\Z+\lk(\mu,P)/2$.
 Each $T^{t} \cS$ denotes a copy of $\cS$. Additionally, there are maps $(\sigma,W)|h_{\sigma,W}$, $(\sigma, Z)|h_{\sigma,Z}$, $(\tau,W)| h_{\tau,W}$ and $(\tau,Z)|h_{\tau,Z}$ which we are not displaying in the above diagram.

We will now build the $E_{s,t}$, $F_s$, $J_t$ and $M$ complexes out of the complexes and maps in Equation~\eqref{eq:X(L_P)-complex} and the complex $\cX_n(Y,K)^{\cK}$. We will also describe their internal $(\gr_{\ws},\gr_{\zs})$-gradings, as well as the $\sigma$-normalized Alexander grading. (See Section~\ref{sec:Alexander-gradings} for these conventions). We will write $(s,t,r)\in \bH(K\# H\# L_P)$ for the components of this Alexander grading.

 We now describe the complexes $E_{s,t}$. It is helpful to first build a $DA$-bimodule ${}_{\bF[W,Z]}\scE_{*,t}^{\bF[W,Z]}$.  This bimodule will be constructed using the same algebraic procedure as for L-space links in Section~\ref{sec:2-component-L-space}. It is specified by a diagram of the following form:
 \[
 \scE_{*,t}=\begin{tikzcd}[labels=description, column sep=1cm, row sep=0cm]
 \cdots
 &[-1cm]
 \scE_{-\frac{5}{2},t}
 	\ar[r, bend left,out=45,in=135, "Z|U"]
 & \scE_{-\frac{3}{2},t}
 	\ar[r, bend left,out=45,in=135, "Z|U"]
 	\ar[l, bend left,out=45,in=135, "W|1"]
 &\scE_{-\frac{1}{2},t}
 	\ar[r, bend left,out=45,in=135, "Z|L_Z"]
 	\ar[l, bend left,out=45,in=135, "W|1"]
 	\ar[loop above, looseness=15, "{(W,Z)|h_{W,Z}}"]
 &[.6cm]\scE_{\frac{1}{2},t}
 	\ar[r, bend left,out=45,in=135, "Z|1"]
 	\ar[l, bend left,out=45,in=135, "W|L_W"]
 	\ar[loop above, looseness=15, "{(Z,W)|h_{Z,W}}"]
 & \scE_{\frac{3}{2},t}
 	\ar[r, bend left,out=45,in=135, "Z|1"]
 	\ar[l, bend left,out=45,in=135, "W|U"]
 & \scE_{\frac{5}{2},t}
 	\ar[l, bend left,out=45,in=135, "W|U"] 
 &[-1cm] \cdots
 \end{tikzcd}
 \]
 This diagram specifies a type-$DA$ bimodule via the construction from Section~\ref{sec:U-equivariant}.  Here, $\scE_{s,t}$ denotes the type-$D$ module over $R$
 \begin{equation}
 \scE_{s,t}=\begin{cases} \cC_{t-\frac{1}{2}}[2s+1,0] & \text{ if } s<0\\
 \cC_{t+\frac{1}{2}}[0,-2s+1] & \text{ if } s>0
 \end{cases}
 \label{eq:cEst-defs}
 \end{equation}
 
 We view $\scE_{s,t}$ as having Alexander grading in $(s,t,*)$. The brackets $[x,y]$ denote an (upward) shift in the $(\gr_{\ws},\gr_{\zs})$-bigrading. We define $A_P^\sigma$ (the $r$-component of the Alexander grading  of $\scE_{s,t}$) to be the same as the corresponding component of the Alexander grading on $\cC_{t\pm \frac{1}{2}}$.

For fixed $t\in \Z+(1+\lk(\mu,P))/2$, we decompose 
\[
\cCFK(K)^{\bF[W,Z]}\boxtimes {}_{\bF[W,Z]} \scE_{*,t}^{\bF[W,Z]}=\bigoplus_{s\in \Z+\frac{1}{2}} E_{s,t}.
\]
where $E_{s,t}$ denotes the subspace where the first coordinate of the Alexander grading (corresponding to $K$) takes value $s$.

Inside of ${}_{\cK} \cH_{-}^{\cK}\boxtimes {}_{\cK} \cX_{(0,0)}(L_P)^{\bF[W,Z]}$, the bimodule $\scE_{*,t}$ appears as the following subspace:
\[
\begin{tikzcd}[labels=description, column sep=.5cm, row sep=0cm]
\cdots
&[-.7cm]W_{00}^2|\cC_{t-\frac{1}{2}}
	\ar[r,bend left,out=45,in=135, "Z|U"]
& W_{00}|\cC_{t-\frac{1}{2}}
	\ar[r, bend left,out=45,in=135, "Z|U"]
	\ar[l, bend left,out=45,in=135, "W|1"]
&W_{00}^0|\cC_{t-\frac{1}{2}}
	\ar[r, bend left,out=45,in=135, "Z|L_Z"]
	\ar[l, bend left,out=45,in=135, "W|1"]
	\ar[loop above, looseness=15, "{(W,Z)|h_{W,Z}}"]
&[0.5cm]Z_{00}^0|\cC_{t+\frac{1}{2}}
	\ar[r, bend left,out=45,in=135, "Z|1"]
	\ar[l, bend left,out=45,in=135, "W|L_W"]
	\ar[loop above, looseness=15, "{(Z,W)|h_{Z,W}}"]
& Z_{00}|\cC_{t+\frac{1}{2}}
	\ar[r, bend left,out=45,in=135, "Z|1"]
	\ar[l, bend left,out=45,in=135, "W|U"]
& Z^2_{00}|\cC_{t+\frac{1}{2}}
	\ar[l, bend left,out=45,in=135, "W|U"]
&[-.7cm] \cdots
\end{tikzcd}
\]

The grading shifts described on $\scE_{s,t}$ described in Equation~\eqref{eq:cEst-defs} are the tensorial gradings on link Floer homology, where we give $W^i_{00}$ Alexander grading $(-\tfrac{1}{2}-i,\tfrac{1}{2},0)$, and we give $Z^i_{00}$ Alexander grading $(\tfrac{1}{2}+i,-\tfrac{1}{2},0)$. We also view each $\cC_{t\pm \frac{1}{2}}$ as having Alexander gradings $(0,t\pm \tfrac{1}{2},*)$. Furthermore, the third component is the same as the corresponding component (the pattern $P$) of the Alexander grading on $\cC_{t\pm\frac{1}{2}}$.

We now define the modules $F_{s,t}$ for  $(s,t)\in \bH(P)$. Recall that $F_{s,t}$ does not depend on $t$, and we sometimes write $F_{s}$ for this complex.

 To define $F_{s,t}$, we recall the $\bF[U]$-subcomplex $A_{s+\frac{1}{2}}(K)\subset \cCFK(K)$ consisting of generators in Alexander grading $s+\tfrac{1}{2}$. In our present setting, we view $A_{s+\frac{1}{2}}(K)$ as a type-$D$ module over $\bF[U]$. We may then define $F_{s,t}$ as the tensor product:
\[
F_{s,t}:= A_{s+\frac{1}{2}}(K)^{\bF[U]}\boxtimes {}_{\bF[U]} \cS^{\bF[W,Z]}.
\]
Here, ${}_{\bF[U]}\cS^{\bF[W,Z]}$ is the $DA$-bimodule which coincides with $\cS^{\bF[W,Z]}$ as a type-$D$ module, but has the additional action $\delta_2^1(U^i,\xs)=\xs\otimes U^i$ for all  $i\ge 0$.

It is helpful to view $F_{s,t}$ as being a homogeneously $(s,t,*)$ graded summand (in the $\sigma$-normalized Alexander grading) of the tensor product $\cCFK(K)^{\bF[W,Z]}\boxtimes {}_{\bF[W,Z]}\scF_{*,t}^{\bF[W,Z]}$, where ${}_{\bF[W,Z]}\scF_{*,t}^{\bF[W,Z]}$ is the $DA$ bimodule induced by the following diagram:
\[
\scF_{*,t}=\begin{tikzcd}[labels=description, column sep=.8cm, row sep=0cm]
\cdots
&[-.8cm] \scF_{-\frac{5}{2}, t}
	\ar[r,bend left,out=45,in=135, "Z|U"]
	\ar[loop below,looseness=20, "U|U"]
& \scF_{-\frac{3}{2},t}
	\ar[r, bend left,out=45,in=135, "Z|U"]
	\ar[l, bend left,out=45,in=135, "W|1"]
	\ar[loop below,looseness=20, "U|U"]
&\scF_{-\frac{1}{2},t}
	\ar[r, bend left,out=45,in=135, "Z|1"]
	\ar[l, bend left,out=45,in=135, "W|1"]
	\ar[loop below,looseness=20, "U|U"]
&\scF_{\frac{1}{2},t}
	\ar[r, bend left,out=45,in=135, "Z|1"]
	\ar[l, bend left,out=45,in=135, "W|U"]
	\ar[loop below,looseness=20, "U|U"]
&\scF_{\frac{3}{2},t}
	\ar[l, bend left,out=45,in=135, "W|U"]
	\ar[loop below,looseness=20, "U|U"]
&[-.8cm]\cdots 
\end{tikzcd}
\]

In the above, we set
\[
\scF_{s,t}=\cS[\min(2s+1,0),\min(0,-2s-1)]
\]
for all $s,t$. Here, we are viewing $\cS$ as having $(\gr_{\ws},\gr_{\zs})$-bigrading induced by its identification with $\cCFK(P)$, and $[x,y]$ denotes a shift of the $(\gr_{\ws},\gr_{\zs})$-bigrading, as above. We put $\scF_{s,t}$ in Alexander grading $(s,t,*)$.
 The module $\scF_{*,t}$ corresponds to the following subspace of $\cH_-\boxtimes \cX(L_P)$:
\[
\begin{tikzcd}[labels=description, column sep=.6cm, row sep=0cm]
\cdots
&[-.8cm] W_{01}^2|T^{t+\frac{1}{2}} \cS
	\ar[r,bend left,out=45,in=135, "Z|U"]
	\ar[loop below,looseness=20, "U|U"]
& W_{01}^1|T^{t+\frac{1}{2}}\cS
	\ar[r, bend left,out=45,in=135, "Z|U"]
	\ar[l, bend left,out=45,in=135, "W|1"]
	\ar[loop below,looseness=20, "U|U"]
&1_{01}|T^{t+\frac{1}{2}}\cS
	\ar[r, bend left,out=45,in=135, "Z|1"]
	\ar[l, bend left,out=45,in=135, "W|1"]
	\ar[loop below,looseness=20, "U|U"]
&Z_{01}^1|T^{t+\frac{1}{2}}\cS
	\ar[r, bend left,out=45,in=135, "Z|1"]
	\ar[l, bend left,out=45,in=135, "W|U"]
	\ar[loop below,looseness=20, "U|U"]
&Z_{01}^2|T^{t+\frac{1}{2}}\cS
	\ar[l, bend left,out=45,in=135, "W|U"]
	\ar[loop below,looseness=20, "U|U"]
&[-.8cm]\cdots 
\end{tikzcd}
\]

Writing $A_P^\sigma$ for the $r$-component of the $\sigma$-normalized Alexander grading on $\scF_{s,t}$, we have
\[
A_P^\sigma=A_{\cS}+\frac{\lk(P,\mu)}{2},
\]
where $A_{\cS}$ is the ordinary Alexander grading on the staircase $\cS$, viewed as $\cCFK(P)$.

We now describe $J_{s,t}$ and $M_{s,t}$. Recall that $J_{s,t}$ is independent of $s$, and $M_{s,t}$ is independent of $s$ and $t$. For simplicity, we will assume that $\cX_n(Y,K)^{\cK}$ is described so that idempotent 1 coincides with $\CF^-(Y)\otimes \bF[T,T^{-1}]$ (in particular, we assume that the differential in idempotent 1 has no $T$-powers). With this in mind, we will write
\[
B(K)=\cX_n(Y,K)\cdot \ve{I}_1,
\]
viewed as a type-$D$ module over $\bF[U]$, concentrated in Alexander grading 0.

Finally, we define
\[
J_{s,t}:=B(K)^{\bF[U]}\boxtimes{}_{\bF[U]}\cC_{t+\frac{1}{2}}\quad \text{and}\quad M_{s,t}=B(K)^{\bF[U]}\boxtimes {}_{\bF[U]}\cS.
\]
We again can view $J_{s,t}$ as the Alexander grading $(s,t,*)$ subcomplex of the tensor product of $\cX_{n}(Y,K)\cdot \ve{I}_1$ with the complex ${}_{\bF[U,T,T^{-1}]}\scJ_{*,t}^{\bF[W,Z]}$ given by the diagram
\[
\scJ_{*,t}=\begin{tikzcd}[labels=description, column sep=1cm, row sep=0cm]
\cdots
&[-.7cm]\scJ_{-\frac{5}{2},t}
	\ar[r, bend left,out=45,in=135, "T|1"]
	\ar[loop below,looseness=20, "U|U"]
& \scJ_{-\frac{3}{2},t}
	\ar[r, bend left,out=45,in=135, "T|1"]
	\ar[l, bend left,out=45,in=135, "T^{-1}|1"]
	\ar[loop below,looseness=20, "U|U"]
&\scJ_{-\frac{1}{2},t}
	\ar[r, "T|1", bend left,out=45,in=135]
	\ar[l, "T^{-1}|1", bend left,out=45,in=135]
	\ar[loop below,looseness=20, "U|U"]
&\scJ_{\frac{1}{2},t}
	\ar[r, bend left,out=45,in=135, "T|1"]
	\ar[l, bend left,out=45,in=135,"T^{-1}|1"]
	\ar[loop below,looseness=20, "U|U"]
&\scJ_{\frac{3}{2},t}
	\ar[l, bend left,out=45,in=135,"T^{-1}|1"]
	\ar[loop below,looseness=20, "U|U"]
&[-.7cm] \cdots
\end{tikzcd}
\]
where
\[
\scJ_{s,t}=\cC_{t+\frac{1}{2}}.
\]
This appears in $\cH_-\boxtimes \cX(L_P)$ as the subcomplex
\[
\begin{tikzcd}[labels=description, column sep=.5cm, row sep=0cm]
\cdots
&[-.5cm]T^{-2}_{10}|\cC_{t+\frac{1}{2}}
	\ar[r, bend left,out=45,in=135, "T|1"]
	\ar[loop below,looseness=20, "U|U"]
& T^{-1}_{10}|\cC_{t+\frac{1}{2}}
	\ar[r, bend left,out=45,in=135, "T|1"]
	\ar[l, bend left,out=45,in=135, "T^{-1}|1"]
	\ar[loop below,looseness=20, "U|U"]
&T^{0}_{10}|\cC_{t+\frac{1}{2}}
	\ar[r, "T|1", bend left,out=45,in=135]
	\ar[l, "T^{-1}|1", bend left,out=45,in=135]
	\ar[loop below,looseness=20, "U|U"]
&T^{1}_{10}|\cC_{t+\frac{1}{2}}
	\ar[r, bend left,out=45,in=135, "T|1"]
	\ar[l, bend left,out=45,in=135,"T^{-1}|1"]
	\ar[loop below,looseness=20, "U|U"]
&T^{2}_{10}|\cC_{t+\frac{1}{2}}
	\ar[l, bend left,out=45,in=135,"T^{-1}|1"]
	\ar[loop below,looseness=20, "U|U"]
&[-.5cm] \cdots
\end{tikzcd}
\]
We give $\scJ_{s,t}$ the $A_P^\sigma$ Alexander grading (the $r$-component) equal to the last component of the natural Alexander grading on $\cC_{t+\frac{1}{2}},$ corresponding to the component $P$.

Similarly, $M_{s,t}$ is the Alexander grading $(s,t,*)$ subspace of the tensor product of $\cX_n(Y,K)\cdot \ve{I}_1$ with the $DA$-bimodule shown below ${}_{\bF[U,T,T^{-1}]} \scM_{*,t}^{\bF[W,Z]}$:
\[
\scM_{*,t}=\begin{tikzcd}[labels=description, column sep=1cm, row sep=0cm]
\cdots
&[-1cm]\scM_{-\frac{5}{2},t}
	\ar[r, bend left,out=45,in=135, "T|1"]
	\ar[loop below,looseness=20, "U|U"]
&\scM_{-\frac{3}{2},t}
	\ar[r, bend left,out=45,in=135, "T|1"]
	\ar[l, bend left,out=45,in=135, "T^{-1}|1"]
	\ar[loop below,looseness=20, "U|U"]
&\scM_{-\frac{1}{2},t}
	\ar[r, "T|1", bend left,out=45,in=135]
	\ar[l, "T^{-1}|1", bend left,out=45,in=135]
	\ar[loop below,looseness=20, "U|U"]
&\scM_{-\frac{1}{2},t}
	\ar[r, bend left,out=45,in=135, "T|1"]
	\ar[l, bend left,out=45,in=135,"T^{-1}|1"]
	\ar[loop below,looseness=20, "U|U"]
&\scM_{\frac{3}{2},t}
	\ar[l, bend left,out=45,in=135,"T^{-1}|1"]
	\ar[loop below,looseness=20, "U|U"]
&[-1cm] \cdots
\end{tikzcd}
\]
where 
\[
\scM_{s,t}=\cS
\]
for all $s,t$. The complex $\scM_{*,t}$ appears in the tensor product via the subcomplex
\[
\begin{tikzcd}[labels=description, column sep=.5cm, row sep=0cm]
\cdots
&[-.5cm]T_{11}^{-2}|T^{t+\frac{1}{2}}\cS
	\ar[r, bend left,out=45,in=135, "T|1"]
	\ar[loop below,looseness=20, "U|U"]
& T^{-1}_{11}|T^{t+\frac{1}{2}}\cS
	\ar[r, bend left,out=45,in=135, "T|1"]
	\ar[l, bend left,out=45,in=135, "T^{-1}|1"]
	\ar[loop below,looseness=20, "U|U"]
&T^{0}_{11}|T^{t+\frac{1}{2}}\cS
	\ar[r, "T|1", bend left,out=45,in=135]
	\ar[l, "T^{-1}|1", bend left,out=45,in=135]
	\ar[loop below,looseness=20, "U|U"]
&T^{1}_{11}|T^{t+\frac{1}{2}}\cS
	\ar[r, bend left,out=45,in=135, "T|1"]
	\ar[l, bend left,out=45,in=135,"T^{-1}|1"]
	\ar[loop below,looseness=20, "U|U"]
&T^{2}_{11}|T^{t+\frac{1}{2}}\cS
	\ar[l, bend left,out=45,in=135,"T^{-1}|1"]
	\ar[loop below,looseness=20, "U|U"]
&[-.5cm] \cdots
\end{tikzcd}
\]
We give $M_{s,t}$ the $(\gr_{\ws},\gr_{\zs})$ grading from $\cS$, and we define $A_P^\sigma$ on $M_{s,t}$ to be
\[
A_P^\sigma=A_{\cS}+\frac{\lk(\mu,P)}{2}.
\]

\subsection{Length 1 maps}
\label{sec:length-one-maps-general}

We now describe the maps $\Phi^{\pm K}$ and $\Phi^{\pm \mu}$.

 The map $\Phi^{K}\colon E_{s,t}\to J_{s,t}$ is constructed by tensoring $\cX_n(K)^{\cK}$ with the bimodule morphism $f^K$, determined by the following diagram:
\[
\begin{tikzcd}[labels=description,row sep=2.1cm] \scE_{*,t}\ar[d, "f^K"]
\\  \scJ_{*,t}
\end{tikzcd}
\hspace{-.2cm}
=
\hspace{-.3cm}
\begin{tikzcd}[labels=description, {column sep=2.3cm,between origins}, row sep=.4cm]
\cdots
&[-1.4cm]\cC_{t-\frac{1}{2}}^{(-\frac{5}{2}, t)}
	\ar[r,bend left,out=45,in=135, "Z|U"]
	\ar[d,"\sigma|U^2 L_Z"]
& \cC_{t-\frac{1}{2}}^{(-\frac{3}{2}, t)}
	\ar[r,bend left,out=45,in=135,  "Z|U"]
	\ar[l,bend left,out=45,in=135, "W|1"]
	\ar[d,"\sigma| U L_Z"]
&\cC_{t-\frac{1}{2}}^{(-\frac{1}{2}, t)}
	\ar[r,bend left,out=45,in=135,  "Z|L_Z"]
	\ar[l,bend left,out=45,in=135,  "W|1"]
	\ar[d,"\sigma|L_Z"]
	\ar[loop above,looseness=15, "{(W,Z)|h_{W,Z}}"]
&[.4cm] \cC_{t+\frac{1}{2}}^{(\frac{1}{2},t)}
	\ar[r,bend left,out=45,in=135,  "Z|1"]
	\ar[l,bend left,out=45,in=135, "W|L_W"]
	\ar[d,"\sigma|1"]
	\ar[loop above, looseness=15, "{(Z,W)|h_{Z,W}}"]
	\ar[dl,  "{(\sigma,W)|h_{Z,W}}", end anchor ={[xshift=-1.5ex]}, start anchor ={[xshift=1.5ex]}]
& \cC_{t+\frac{1}{2}}^{(\frac{3}{2},t)}
	\ar[r,bend left,out=45,in=135,  "Z|1"]
	\ar[l,bend left,out=45,in=135,  "W|U"]
	\ar[d,"\sigma|1"]
& \cC_{t+\frac{1}{2}}^{(\frac{5}{2},t)}
	\ar[l,bend left,out=45,in=135,  "W|U"]
	\ar[d,"\sigma|1"]
&[-1.5cm] \cdots
\\[1.7cm]
\cdots
& \cC_{t+\frac{1}{2}}^{(-\frac{5}{2}, t)}
	\ar[r,bend left,out=45,in=135, "T|1"]
	\ar[loop below,looseness=20, "U|U"]
& \cC_{t+\frac{1}{2}}^{(-\frac{3}{2}, t)}
	\ar[r,bend left,out=45,in=135,  "T|1"]
	\ar[l,bend left,out=45,in=135, "T^{-1}|1"]
	\ar[loop below,looseness=20, "U|U"]
&\cC_{t+\frac{1}{2}}^{(-\frac{1}{2}, t)}
	\ar[r,bend left,out=45,in=135, "T|1"]
	\ar[l,bend left,out=45,in=135,  "T^{-1}|1"]
	\ar[loop below,looseness=20, "U|U"]
&\cC_{t+\frac{1}{2}}^{(\frac{1}{2}, t)}
	\ar[r,bend left,out=45,in=135, "T|1"]
	\ar[l,bend left,out=45,in=135,  "T^{-1}|1"]
	\ar[loop below,looseness=20, "U|U"]
&\cC_{t+\frac{1}{2}}^{(\frac{3}{2}, t)}
	\ar[r,bend left,out=45,in=135,  "T|1"]
	\ar[l,bend left,out=45,in=135,  "T^{-1}|1"]
	\ar[loop below,looseness=20, "U|U"]
&\cC_{t+\frac{1}{2}}^{(\frac{5}{2}, t)}
	\ar[l,bend left,out=45,in=135,  "T^{-1}|1"]
	\ar[loop below,looseness=20, "U|U"]	
&\cdots 
\end{tikzcd}.
\]

In the above, the arrow labeled $\scE_{*,t}\to \scJ_{*,t}$ denotes the type-$DA$ morphism determined by the vertical arrows in the diagram above (in the same way that the actions involving $\sigma$ on the $DA$-modules of L-space links are determined in Section~\ref{sec:U-equivariant}).  Additionally, the bimodules $\cC_{t\pm \frac{1}{2}}^{s,t}$ in the top row denotes the subspace $\scE_{s,t}\subset \scE_{*,t}$, and similarly for the bottom row.

The map $\Phi^{-K}\colon E_{s,t}\to J_{s+n,t-1}$ is obtained modifying the above diagram using the natural symmetry which replaces $\sigma$ with $\tau$ and $L_Z$ with $L_W$, as in the following diagram: 

\[
\begin{tikzcd}[labels=description, row sep=2.1cm] \scE_{*,t} 
\ar[d, "f^{-K}"]
\\ 
\scJ_{*,t-1}
\end{tikzcd}
\hspace{-.2cm}
 =
\hspace{-.3cm}
\begin{tikzcd}[labels=description, {column sep=2.3cm,between origins}, row sep=.4cm]
	\cdots
	&[-1.3cm]\cC_{t-\frac{1}{2}}^{(-\frac{5}{2}, t)}
	\ar[r,bend left,out=45,in=135, "Z|U"]
	\ar[d,"\tau|1"]
	& \cC_{t-\frac{1}{2}}^{(-\frac{3}{2}, t)}
	\ar[r,bend left,out=45,in=135,  "Z|U"]
	\ar[l,bend left,out=45,in=135, "W|1"]
	\ar[d,"\tau|1"]
	&\cC_{t-\frac{1}{2}}^{(-\frac{1}{2}, t)}
	\ar[r,bend left,out=45,in=135,  "Z|L_Z"]
	\ar[l,bend left,out=45,in=135,  "W|1"]
	\ar[d,"\tau|1"]
	\ar[loop above,looseness=15, "{(W,Z)|h_{W,Z}}"]
	\ar[dr, "{(\tau,Z)|h_{W,Z}}" ,crossing over,end anchor ={[xshift=2ex]}, start anchor ={[xshift=-2ex]},shorten =1mm]
	&[.4cm] \cC_{t+\frac{1}{2}}^{(\frac{1}{2},t)}
	\ar[r,bend left,out=45,in=135,  "Z|1"]
	\ar[l,bend left,out=45,in=135, "W|L_W"]
	\ar[d,"\tau|L_W"]
	\ar[loop above, looseness=15, "{(Z,W)|h_{Z,W}}"]
	& \cC_{t+\frac{1}{2}}^{(\frac{3}{2},t)}
	\ar[r,bend left,out=45,in=135,  "Z|1"]
	\ar[l,bend left,out=45,in=135,  "W|U"]
	\ar[d,"\tau|UL_W"]
	& \cC_{t+\frac{1}{2}}^{(\frac{5}{2},t)}
	\ar[l,bend left,out=45,in=135,  "W|U"]
	\ar[d,"\tau|U^2L_W"]
	&[-1.7cm] \cdots
	\\[1.4cm]
	\cdots
	& \cC_{t-\frac{1}{2}}^{(-\frac{5}{2}, t-1)}
	\ar[r,bend left,out=45,in=135, "T|1"]
	\ar[loop below,looseness=20, "U|U"]
	& \cC_{t-\frac{1}{2}}^{(-\frac{3}{2}, t-1)}
	\ar[r,bend left,out=45,in=135,  "T|1"]
	\ar[l,bend left,out=45,in=135, "T^{-1}|1"]
	\ar[loop below,looseness=20, "U|U"]
	&\cC_{t-\frac{1}{2}}^{(-\frac{1}{2}, t-1)}
	\ar[r,bend left,out=45,in=135, "T|1"]
	\ar[l,bend left,out=45,in=135,  "T^{-1}|1"]
	\ar[loop below,looseness=20, "U|U"]
	&\cC_{t-\frac{1}{2}}^{(\frac{1}{2}, t-1)}
	\ar[r,bend left,out=45,in=135, "T|1"]
	\ar[l,bend left,out=45,in=135,  "T^{-1}|1"]
	\ar[loop below,looseness=20, "U|U"]
	&\cC_{t-\frac{1}{2}}^{(\frac{3}{2}, t-1)}
	\ar[r,bend left,out=45,in=135,  "T|1"]
	\ar[l,bend left,out=45,in=135,  "T^{-1}|1"]
	\ar[loop below,looseness=20, "U|U"]
	&\cC_{t-\frac{1}{2}}^{(\frac{5}{2}, t-1)}
	\ar[l,bend left,out=45,in=135,  "T^{-1}|1"]
	\ar[loop below,looseness=20, "U|U"]	
	&\cdots 
\end{tikzcd}.
\]

The maps 
\[
\Phi^K\colon F_{s,t}\to M_{s,t}\quad \text{and} \quad \Phi^{-K}:F_{s,t}\to M_{s+n,t-1}
\] are computed by tensoring $\cX_n(Y,K)^{\cK}$ with bimodule morphisms $f^K$ and $f^{-K}$ (respectively), shown in the following diagram:
\[
\begin{tikzcd}[labels=description, row sep=2.1cm] \scF_{*,t} \ar[d, "f^{K}"]
\\ 
\scM_{*,t}
\end{tikzcd}
\hspace{-.2cm}
 =
\hspace{-.3cm}
\begin{tikzcd}[labels=description, column sep=.8cm, row sep=0cm]
\cdots
&[-.8cm] \cS^{( -\frac{5}{2},t)}
	\ar[r,bend left,out=45,in=135, "Z|U"]
	\ar[d, "\sigma|U^2"]
& \cS^{(-\frac{3}{2},t)}
	\ar[r,out=45,in=135, "Z|U"]
	\ar[l, bend left,out=45,in=135, "W|1"]
	\ar[d, "\sigma|U"]
&\cS^{(-\frac{1}{2},t)}
	\ar[r, bend left,out=45,in=135, "Z|1"]
	\ar[l, bend left,out=45,in=135, "W|1"]
	\ar[d, "\sigma|1"]
&\cS^{(\frac{1}{2},t)}
	\ar[r, bend left,out=45,in=135, "Z|1"]
	\ar[l, bend left,out=45,in=135, "W|U"]
	\ar[d, "\sigma|1"]
&\cS^{(\frac{3}{2},t)}
	\ar[l, bend left,out=45,in=135, "W|U"]
	\ar[d, "\sigma|1"]
&[-.8cm]\cdots
\\[1.5cm]
 \cdots&[-1.3cm] \cS^{(-\frac{5}{2},t)}
	\ar[r, bend left,out=45,in=135, "T|1"]
	\ar[loop below,looseness=15, "U|U"]
& \cS^{(-\frac{3}{2},t)}
	\ar[r, bend left,out=45,in=135, "T|1"]
	\ar[l, bend left,out=45,in=135,  "T^{-1}|1"]
	\ar[loop below,looseness=15, "U|U"]
&\cS^{(-\frac{1}{2},t)}
	\ar[r, bend left,out=45,in=135, "T|1"]
	\ar[l, bend left,out=45,in=135, "T^{-1}|1"]
	\ar[loop below,looseness=15, "U|U"]
&\cS^{(\frac{1}{2},t)}
	\ar[r, bend left,out=45,in=135, "T|1"]
	\ar[l, bend left,out=45,in=135, "T^{-1}|1"]
	\ar[loop below,looseness=15, "U|U"]
&\cS^{(\frac{3}{2},t)}
	\ar[l, bend left,out=45,in=135, "T^{-1}|1"]
	\ar[loop below,looseness=15, "U|U"]
&[-1.3cm]\cdots 
\end{tikzcd}
\]
\[
\begin{tikzcd}[labels=description, row sep=2.1cm] \scF_{*,t} 
\ar[d, "f^{-K}"]
\\ 
\scM_{*,t-1}
\end{tikzcd}
\hspace{-.2cm}
 =
\hspace{-.3cm}
\begin{tikzcd}[labels=description, column sep=.8cm, row sep=0cm]
\cdots
&[-1.2cm] \cS^{(-\frac{5}{2},t)}
	\ar[r,bend left,out=45,in=135, "Z|U"]
	\ar[loop above,looseness=10, "U|U"]
	\ar[d, "\tau|1"]
& \cS^{(-\frac{3}{2},t)}
	\ar[r,out=45,in=135, "Z|U"]
	\ar[l, bend left,out=45,in=135, "W|1"]
	\ar[loop above,looseness=10, "U|U"]
	\ar[d, "\tau|1"]
&\cS^{(-\frac{1}{2},t)}
	\ar[r, bend left,out=45,in=135, "Z|1"]
	\ar[l, bend left,out=45,in=135, "W|1"]
	\ar[loop above,looseness=10, "U|U"]
	\ar[d, "\tau|1"]
&\cS^{(\frac{1}{2},t)}
	\ar[r, bend left,out=45,in=135, "Z|1"]
	\ar[l, bend left,out=45,in=135, "W|U"]
	\ar[loop above,looseness=10, "U|U"]
	\ar[d, "\tau|U"]
&\cS^{(\frac{3}{2},t)}
	\ar[l, bend left,out=45,in=135, "W|U"]
	\ar[loop above,looseness=10, "U|U"]
	\ar[d, "\tau|U^2"]
&[-1cm]\cdots
\\[1.5cm]
 \cdots&[-1.3cm] \cS^{(-\frac{5}{2},t-1)}
	\ar[r, bend left,out=45,in=135, "T|1"]
	\ar[loop below,looseness=15, "U|U"]
& \cS^{(-\frac{3}{2},t-1)}
	\ar[r, bend left,out=45,in=135, "T|1"]
	\ar[l, bend left,out=45,in=135,  "T^{-1}|1"]
	\ar[loop below,looseness=15, "U|U"]
&\cS^{(-\frac{1}{2},t-1)}
	\ar[r, bend left,out=45,in=135, "T|1"]
	\ar[l, bend left,out=45,in=135, "T^{-1}|1"]
	\ar[loop below,looseness=15, "U|U"]
&\cS^{(\frac{1}{2},t-1)}
	\ar[r, bend left,out=45,in=135, "T|1"]
	\ar[l, bend left,out=45,in=135, "T^{-1}|1"]
	\ar[loop below,looseness=15, "U|U"]
&\cS^{(\frac{3}{2},t-1)}
	\ar[l, bend left,out=45,in=135, "T^{-1}|1"]
	\ar[loop below,looseness=15, "U|U"]
&[-1.3cm]\cdots 
\end{tikzcd}
\]

Next, we consider the maps $\Phi^{\pm \mu}$. The maps
\[
 \Phi^\mu\colon  E_{s,t}\to F_{s,t}\quad \text{and} \quad \Phi^{-\mu} \colon E_{s,t}\to F_{s-1,t}
\]
 are determined by tensoring $\bI\colon \cX_n(K)^{\cK}\to \cX_n(K)^{\cK}$ with the following mapping cone complexes:
\[
\begin{tikzcd}[labels=description, row sep=2.8cm] \scE_{*,t} 
\ar[d, "f^{\mu}"]
\\ 
\scF_{*,t}
\end{tikzcd}
\hspace{-.3cm}
=
\hspace{-.3cm}\begin{tikzcd}[labels=description, column sep=.8cm, row sep=0cm]
\cdots
&[-1cm]\cC_{t-\frac{1}{2}}^{(-\frac{5}{2},t)}
	\ar[r, bend left,out=45,in=135, "Z|U"]
	\ar[d,"L_\sigma",  pos=.4]
& \cC_{t-\frac{1}{2}}^{(-\frac{3}{2},t)}
	\ar[r, bend left,out=45,in=135, "Z|U"]
	\ar[l, bend left,out=45,in=135, "W|1"]
	\ar[d,"L_\sigma",  pos=.4]
&\cC_{t-\frac{1}{2}}^{(-\frac{1}{2},t)}
	\ar[r, bend left,out=35,in=145, "Z|L_Z"]
	\ar[l, bend left,out=45,in=135, "W|1"]
	\ar[d,"L_\sigma",  pos=.4]
	\ar[loop above,looseness=15, "{(W,Z)|h_{W,Z}}"]
	\ar[dr,pos=.68, "Z|h_{\sigma,Z}", end anchor={[xshift=1ex]}]
&[.8cm] \cC_{t+\frac{1}{2}}^{(\frac{1}{2},t)}
	\ar[r, bend left,out=45,in=135, "Z|1"]
	\ar[l, bend left=15,out=35,in=145, "W|L_W"]
	\ar[d,"L_\sigma", pos=.4]
	\ar[loop above, looseness=15, "{(Z,W)|h_{Z,W}}"]
	\ar[dl, pos=.27, "W|h_{\sigma,W}",crossing over, end anchor={[xshift=-1ex]}]
& \cC_{t+\frac{1}{2}}^{(\frac{3}{2},t)}
	\ar[r, bend left,out=45,in=135, "Z|1"]
	\ar[l, bend left,out=45,in=135, "W|U"]
	\ar[d,"L_\sigma",  pos=.4]
& \cC_{t+\frac{1}{2}}^{(\frac{5}{2},t)}
	\ar[l, bend left,out=45,in=135, "W|U"]
	\ar[d,"L_\sigma", pos=.4]
&[-1cm] \cdots
\\[2.7cm]
\cdots
&[.4cm]
 \cS^{(-\frac{5}{2},t)}
	\ar[r, bend left,out=45,in=135, "Z|U"]
& \cS^{(-\frac{3}{2},t)}
	\ar[r, bend left,out=45,in=135, "Z|U"]
	\ar[l, bend left,out=45,in=135, "W|1"]
&\cS^{(-\frac{1}{2},t)}
	\ar[r, bend left,out=45,in=135, "Z|1"]
	\ar[l, bend left,out=45,in=135, "W|1"]
&\cS^{(\frac{1}{2},t)}
	\ar[r, bend left,out=45,in=135, "Z|1"]
	\ar[l, bend left,out=45,in=135, "W|U"]
&\cS^{(\frac{3}{2},t)}
	\ar[r, bend left,out=45,in=135, "Z|1"]
	\ar[l, bend left,out=45,in=135, "W|U"]
&\cS^{(\frac{5}{2},t)}
	\ar[l, bend left,out=45,in=135, "W|U"]
&\cdots 
 \end{tikzcd}
 \]
  
 \[\begin{tikzcd}[labels=description, row sep=2.1cm] 
\scE_{*,t} 
\ar[d, "f^{-\mu}"]
\\ 
\scF_{*,t}
\end{tikzcd}
\hspace{-.2cm}
=
\hspace{-.3cm}\begin{tikzcd}[labels=description, column sep=.8cm, row sep=.3cm]
 \cdots
 &[-1cm]\cC_{t-\frac{1}{2}}^{(-\frac{5}{2},t)}
 	\ar[r, bend left,out=45,in=135, "Z|U"]
 & \cC_{t-\frac{1}{2}}^{(-\frac{3}{2},t)}
 	\ar[r, bend left,out=45,in=135, "Z|U"]
 	\ar[l, bend left,out=45,in=135, "W|1"]
 	\ar[dl, "L_\tau", pos=.6]
 &\cC_{t-\frac{1}{2}}^{(-\frac{1}{2},t)}
 	\ar[r, bend left,out=45,in=135, "Z|L_Z"]
 	\ar[l, bend left,out=45,in=135, "W|1"]
 	\ar[dl, "L_\tau",  pos=.6]
 	\ar[loop above,looseness=15, "{(W,Z)|h_{W,Z}}"]
 &[.5cm] \cC_{t+\frac{1}{2}}^{(\frac{1}{2},t)}
 	\ar[r, bend left,out=45,in=135, "Z|1"]
 	\ar[l, bend left,out=45,in=135, "W|L_W"]
 	\ar[dl, "L_\tau", pos=.6]
 	\ar[loop above, looseness=15, "{(Z,W)|h_{Z,W}}"]
 	\ar[dll, pos=.37, "\substack{ W|h_{\tau,W}}", shorten =2mm,end anchor={[xshift=-2ex]}, start anchor={[xshift=2ex]} ]
 & \cC_{t+\frac{1}{2}}^{(\frac{3}{2},t)}
 	\ar[r, bend left,out=45,in=135, "Z|1"]
 	\ar[l, bend left,out=45,in=135, "W|U"]
 	\ar[dl, "L_\tau",  pos=.6]
 & \cC_{t+\frac{1}{2}}^{(\frac{5}{2},t)}
 	\ar[l, bend left,out=45,in=135, "W|U"]
 	\ar[dl, "L_\tau",  pos=.6]
 &[-1cm] \cdots
 \\[2.6cm]
 \cdots
 &[.4cm] \cS^{( -\frac{5}{2},t)}
 	\ar[r, bend left,out=45,in=135, "Z|U"]
 & \cS^{( -\frac{3}{2},t)}
 	\ar[r, bend left,out=45,in=135, "Z|U"]
 	\ar[l, bend left,out=45,in=135, "W|1"]
 &\cS^{(-\frac{1}{2},t)}
 	\ar[r, bend left,out=45,in=135, "Z|1"]
 	\ar[l, bend left,out=45,in=135, "W|1"]
  	\ar[from=u,pos=.21, "\substack{ Z|h_{\tau,Z}}",crossing over]
 &\cS^{(\frac{1}{2},t)}
 	\ar[r, bend left,out=45,in=135, "Z|1"]
 	\ar[l, bend left,out=45,in=135, "W|U"]
 &\cS^{(\frac{3}{2},t)}
 	\ar[r, bend left,out=45,in=135, "Z|1"]
 	\ar[l, bend left,out=45,in=135, "W|U"]
 &\cS^{(\frac{5}{2},t)}
 	\ar[l, bend left,out=45,in=135, "W|U"]
 &\cdots 
  \end{tikzcd}
\]

The maps
\[
\Phi^\mu\colon  J_{s,t}\to M_{s,t}\quad \text{and} \quad \Phi^{-\mu} \colon J_{s,t}\to M_{s-1,t}
\]
are determined by taking the tensor product of $\bI\colon \cX_n(K)^{\cK}\to \cX_n(K)^{\cK}$  with the diagrams
\[
\begin{tikzcd}[labels=description, row sep=2.1cm] \scJ_{*,t} 
	\ar[d, "f^{\mu}"]
	\\ 
	\scM_{*,t}
\end{tikzcd}
\hspace{-.2cm}
=
\hspace{-.3cm}
\begin{tikzcd}[labels=description, column sep=1cm, row sep=0cm]
	\cdots
	&[-1cm] \cC_{t+\frac{1}{2}}^{( -\frac{5}{2},t)}
	\ar[r,bend left,out=45,in=135, "T|1"]
	\ar[loop above,looseness=10, "U|U"]
	\ar[d, "L_\sigma"]
	& \cC_{t+\frac{1}{2}}^{(-\frac{3}{2},t)}
	\ar[r,out=45,in=135, "T|1"]
	\ar[l, bend left,out=45,in=135, "T^{-1}|1"]
	\ar[loop above,looseness=10, "U|U"]
	\ar[d, "L_\sigma"]
	&\cC_{t+\frac{1}{2}}^{(-\frac{1}{2},t)}
	\ar[r, bend left,out=45,in=135, "T|1"]
	\ar[l, bend left,out=45,in=135, "T^{-1}|1"]
	\ar[loop above,looseness=10, "U|U"]
	\ar[d, "L_\sigma"]
	&\cC_{t+\frac{1}{2}}^{(\frac{1}{2},t)}
	\ar[r, bend left,out=45,in=135, "T|1"]
	\ar[l, bend left,out=45,in=135, "T^{-1}|1"]
	\ar[loop above,looseness=10, "U|U"]
	\ar[d, "L_\sigma"]
	&\cC_{t+\frac{1}{2}}^{(\frac{3}{2},t)}
	\ar[l, bend left,out=45,in=135, "T^{-1}|1"]
	\ar[loop above,looseness=10, "U|U"]
	\ar[d, "L_\sigma"]
	&[-1cm]\cdots
	\\[1.5cm]
	\cdots&[-1.3cm] \cS^{(-\frac{5}{2},t)}
	\ar[r, bend left,out=45,in=135, "T|1"]
	\ar[loop below,looseness=10, "U|U"]
	& \cS^{(-\frac{3}{2},t)}
	\ar[r, bend left,out=45,in=135, "T|1"]
	\ar[l, bend left,out=45,in=135,  "T^{-1}|1"]
	\ar[loop below,looseness=15, "U|U"]
	&\cS^{(-\frac{1}{2},t)}
	\ar[r, bend left,out=45,in=135, "T|1"]
	\ar[l, bend left,out=45,in=135, "T^{-1}|1"]
	\ar[loop below,looseness=15, "U|U"]
	&\cS^{(\frac{1}{2},t)}
	\ar[r, bend left,out=45,in=135, "T|1"]
	\ar[l, bend left,out=45,in=135, "T^{-1}|1"]
	\ar[loop below,looseness=15, "U|U"]
	&\cS^{(\frac{3}{2},t)}
	\ar[l, bend left,out=45,in=135, "T^{-1}|1"]
	\ar[loop below,looseness=15, "U|U"]
	&[-1.3cm]\cdots 
\end{tikzcd}
\]
\[
\begin{tikzcd}[labels=description, row sep=2.1cm] \scJ_{*,t} 
	\ar[d, "f^{-\mu}"]
	\\ 
	\scM_{*,t}
\end{tikzcd}
\hspace{-.2cm}
=
\hspace{-.3cm}
\begin{tikzcd}[labels=description, column sep=1cm, row sep=0cm]
	\cdots
	&[-1cm] \cC_{t+\frac{1}{2}}^{( -\frac{5}{2},t)}
	\ar[r,bend left,out=45,in=135, "T|1"]
	\ar[loop above,looseness=10, "U|U"]
	\ar[d, "L_\tau"]
	& \cC_{t+\frac{1}{2}}^{(-\frac{3}{2},t)}
	\ar[r,out=45,in=135, "T|1"]
	\ar[l, bend left,out=45,in=135, "T^{-1}|1"]
	\ar[loop above,looseness=10, "U|U"]
	\ar[d, "L_\tau"]
	&\cC_{t+\frac{1}{2}}^{(-\frac{1}{2},t)}
	\ar[r, bend left,out=45,in=135, "T|1"]
	\ar[l, bend left,out=45,in=135, "T^{-1}|1"]
	\ar[loop above,looseness=10, "U|U"]
	\ar[d, "L_\tau"]
	&\cC_{t+\frac{1}{2}}^{(\frac{1}{2},t)}
	\ar[r, bend left,out=45,in=135, "T|1"]
	\ar[l, bend left,out=45,in=135, "T^{-1}|1"]
	\ar[loop above,looseness=10, "U|U"]
	\ar[d, "L_\tau"]
	&\cC_{t+\frac{1}{2}}^{(\frac{3}{2},t)}
	\ar[l, bend left,out=45,in=135, "T^{-1}|1"]
	\ar[loop above,looseness=10, "U|U"]
	\ar[d, "L_\tau"]
	&[-1cm]\cdots
	\\[1.5cm]
	\cdots&[-1.3cm] \cS^{(-\frac{7}{2},t)}
	\ar[r, bend left,out=45,in=135, "T|1"]
	\ar[loop below,looseness=10, "U|U"]
	& \cS^{(-\frac{5}{2},t)}
	\ar[r, bend left,out=45,in=135, "T|1"]
	\ar[l, bend left,out=45,in=135,  "T^{-1}|1"]
	\ar[loop below,looseness=15, "U|U"]
	&\cS^{(-\frac{3}{2},t)}
	\ar[r, bend left,out=45,in=135, "T|1"]
	\ar[l, bend left,out=45,in=135, "T^{-1}|1"]
	\ar[loop below,looseness=15, "U|U"]
	&\cS^{(-\frac{1}{2},t)}
	\ar[r, bend left,out=45,in=135, "T|1"]
	\ar[l, bend left,out=45,in=135, "T^{-1}|1"]
	\ar[loop below,looseness=15, "U|U"]
	&\cS^{(\frac{1}{2},t)}
	\ar[l, bend left,out=45,in=135, "T^{-1}|1"]
	\ar[loop below,looseness=15, "U|U"]
	&[-1.3cm]\cdots 
\end{tikzcd}
\]

\begin{rem} The $DA$-bimodule morphisms $f^\mu$ and $f^{-\mu}$ have terms with both no algebra inputs (i.e. $(f^{\pm \mu})_1^1$) and one algebra inputs (i.e. $(f^{\pm \mu})_2^1$). There are no terms of the form $(f^{\pm \mu})_3^1$. The components of $f^{\pm\mu}$ are encoded in the above diagrams as follows. (Compare Section~\ref{sec:U-equivariant}). Terms $(f^{\pm \mu})_1^1$ are encoded by the vertical arrows labeled $L_\sigma$ and $L_\tau$. The terms $(f^{\pm\mu})_2^1$ are encoded as follows. We view the total diagram as a type-$DA$ bimodule over $(R,R)$. We let $\delta_2^1(Z,-)$ and $\delta_2^1(W,-)$ be the arrows labeled with $W|-$ and $Z|-$, respectively. We then set $\delta_2^1(Z^n,-)$ to be the composition of $n$ copies of $\delta_2^1(Z,-)$ and similarly for $\delta_2^1(W^n,-)$. We extend these maps $U$-equivariantly. Then $(f^{\pm \mu})_2^1(a,-)$ is the vertical component of the map $\delta_2^1(a,-)$ on the total diagram.

The maps $\Phi^K: E_{s,t} \to J_{s,t}$ is described by the following schematic drawing:
\[\Phi^{K}=
\begin{tikzcd}[row sep=.3cm]
	\cX_n(K)\ar[d]
		&[-1cm]\boxtimes &[-1cm]\scE_{*,-\frac{1}{2}} \ar[dd] &[-1.2cm]\\
	\delta_1^1\ar[drr, "\sigma",labels=description] \ar[dd]
	\\
	& & (f^K)_2^1\ar[d]\ar[dr] &\\
	\cX_n(K)& [-1cm]\boxtimes & [-1cm]\scJ_{*,-\frac{1}{2}} & \otimes\,\,\,\bF[W,Z]
\end{tikzcd}
+
\begin{tikzcd}[row sep=.3cm, column sep=1.1cm]
	\cX_n(K)
	\ar[d]
		&[-1cm]
	\boxtimes 
&[-1cm]
	\scE_{*,-\frac{1}{2}}
	\ar[dd] 
	&[-1.2cm]
	\\
	\delta_1^1\ar[drr, pos=.3, "U^i W^j"]
		\ar[d]
	\\
	\delta_1^1 
		\ar[rr, "\sigma",labels=description]
		\ar[d]
		& &
		(f^K)_3^1
		\ar[d]\ar[dr] &\\
	\cX_n(K)& [-1cm]\boxtimes & [-1cm]\scJ_{*,-\frac{1}{2}} & \otimes\,\,\,\bF[W,Z]
\end{tikzcd}\]
The $(f^K)_3^1$ arrow is corresponds to the $DA$-structure maps involving $(\sigma,W)|h_{Z,W} $ obtained by applying the recipe described in Section~\ref{sec:U-equivariant} to the mapping cone of $f^K$. In the diagram above, the $\sigma$ labeled arrow denotes a differential weighted by a multiple of $\sigma$. 

 The map $\Phi^{-K}$ has a similar description, except that $\tau$ weighted differentials of $\cX_n(K)^{\cK}$ contribute instead of $\sigma$ weighted arrows, and the $U^i W^j$ labeled arrow is replaced by a $U^i Z^j$ labeled arrow. The map $\Phi^K:F_{s,t} \to M_{s,t}$ has a similar description, except there is no $(\phi^K)_3^1$ term.

Schematically, the map $\Phi^{\mu}:E_{s,t} \to F_{s,t}$ is given by the below.
	\[\Phi^{\mu}=
	\begin{tikzcd}[row sep=.3cm]
		\cX_n(K)\ar[d]& [-1cm]\boxtimes &[-1cm]\scE_{*,-\frac{1}{2}} \ar[d] &[-1.2cm]\\[4mm]
		\bI\ar[d] & & L_\sigma \ar[d] \arrow[dr]&\\[4mm]
		\cX_n(K)& [-1cm]\boxtimes & [-1cm]\scF_{*,-\frac{1}{2}} & \otimes\,\,\,\bF[W,Z]
	\end{tikzcd}+
		\begin{tikzcd}[row sep=.3cm]
			\cX_n(K)
				\ar[d]
				&[-1cm]
				\boxtimes
				 &[-1cm]\scE_{*,-\frac{1}{2}}
				\ar[dd] 
				&[-1.2cm]
				\\
				\delta_1^1\ar[drr]
				\ar[dd]
				\\
				&&
				(f^{\mu})_2^1 
					\arrow[dr]
					\ar[d]
				&
				\\[4mm]
			\cX_n(K)& [-1cm]\boxtimes & [-1cm]\scF_{*,-\frac{1}{2}} & \otimes\,\,\,\bF[W,Z]
		\end{tikzcd}
	\]
In the above diagram, the algebra inputs into $(f^\mu)_2^1$ are in idempotent $(0,0)$ (i.e. elements of $\bF[W,Z]$) and are contributed by terms involving the maps $W|h_{\sigma,W}$ and $Z|h_{\sigma,Z} $ in our description of the map $f^\mu$ above.

 The maps $\Phi^{-\mu}:E_{s,t}\to F_{s-1,t}$, $\Phi^{\mu}:J_{s,t}\to M_{s,t}$ and $\Phi^{-\mu}:J_{s,t}\to M_{s-1,t} $ have similar descriptions.
\end{rem}

\subsection{Length 2 maps}

There are four length 2 maps: $\Phi^{K,\mu}$, $\Phi^{-K,\mu}$, $\Phi^{K,-\mu}$ and $\Phi^{-K,-\mu}$. Algebraically, the term $\Phi^{K,\mu}$ arises in the tensor product 
\[
\cX_n(K)^{\cK}\boxtimes {}_{\cK} \cH_-^{\cK}\boxtimes {}_{\cK} \cX(L_P)
\]
via a diagram of configurations of differentials of the following form:
\[
\Phi^{K,\mu}=
\begin{tikzcd}[row sep=.3cm]
\cX_n(K)\ar[d]& \cH \ar[dd] & \cX(L_P) \ar[dddd] \\
\delta_1^1\ar[dr, "\sigma",labels=description]\ar[dddd]
 \\
& \delta_2^1\ar[d]\ar[ddr]&\\
&\delta_1^1\ar[dd] \ar[dr, "\sigma",labels=description]&&\\
& & \delta_3^1\ar[d] \ar[dr]\\
\cX_n(K)\,&\cH\,&\cX(L_P)\,&\bF[W,Z]\,
\end{tikzcd}
\]
In the above, an arrow labeled $\sigma$, means the terms of the differential that are multiples of $\sigma$. In more detail the length two maps are obtained by boxing $\cX_n(Y,K)^{\cK}$ with one of four $DA$-bimodule morphisms $f^{\pm K,\pm \mu}$. We show $f^{K,\mu}$ and $f^{K,-\mu}$ below: 
\[
\begin{tikzcd}[labels=description,row sep=1.7cm] 
\scE_{*,t}
	\ar[d, "\substack{f^{K,\mu}\\+f^{K,-\mu}}"]
\\  
\scM_{*,t}
\end{tikzcd}
=
\hspace{-.3cm}
\begin{tikzcd}[labels=description, {column sep=2.3cm,between origins}, row sep=1cm]
\cdots
&[-1.4cm]\cC_{t-\frac{1}{2}}^{(-\frac{5}{2}, t)}
	\ar[r,bend left,out=45,in=135, "Z|U"]
	\ar[d,"\sigma|U^2 h_{\sigma,Z}", pos=.3]
& \cC_{t-\frac{1}{2}}^{(-\frac{3}{2}, t)}
	\ar[r,bend left,out=45,in=135,  "Z|U"]
	\ar[l,bend left,out=45,in=135, "W|1"]
	\ar[d,"\sigma| U h_{\sigma,Z}", pos=.3]
	\ar[dl,"\sigma| U h_{\tau,Z}"]
&\cC_{t-\frac{1}{2}}^{(-\frac{1}{2}, t)}
	\ar[r,bend left,out=45,in=135,  "Z|L_Z"]
	\ar[l,bend left,out=45,in=135,  "W|1"]
	\ar[d,"\sigma|h_{\sigma,Z}",pos=.3]
	\ar[loop above,looseness=10, "{(W,Z)|h_{W,Z}}"]
	\ar[dl,"\sigma|h_{\tau,Z}"]
&[.4cm] \cC_{t+\frac{1}{2}}^{(\frac{1}{2},t)}
	\ar[r,bend left,out=45,in=135,  "Z|1"]
	\ar[l,bend left,out=45,in=135, "W|L_W"]
	\ar[loop above, looseness=10, "{(Z,W)|h_{Z,W}}"]
& \cC_{t+\frac{1}{2}}^{(\frac{3}{2},t)}
	\ar[l,bend left,out=45,in=135,  "W|U"]
&[-1.5cm] \cdots
\\[1.7cm]
\cdots
& \cS^{(-\frac{5}{2}, t)}
	\ar[r,bend left,out=45,in=135, "T|1"]
	\ar[loop below,looseness=10, "U|U"]
& \cS^{(-\frac{3}{2}, t)}
	\ar[r,bend left,out=45,in=135,  "T|1"]
	\ar[l,bend left,out=45,in=135, "T^{-1}|1"]
	\ar[loop below,looseness=10, "U|U"]
&\cS^{(-\frac{1}{2}, t)}
	\ar[r,bend left,out=45,in=135, "T|1"]
	\ar[l,bend left,out=45,in=135,  "T^{-1}|1"]
	\ar[loop below,looseness=10, "U|U"]
&\cS^{(\frac{1}{2}, t)}
	\ar[r,bend left,out=45,in=135, "T|1"]
	\ar[l,bend left,out=45,in=135,  "T^{-1}|1"]
	\ar[loop below,looseness=10, "U|U"]
&\cS^{(\frac{3}{2}, t)}
	\ar[l,bend left,out=45,in=135,  "T^{-1}|1"]
	\ar[loop below,looseness=10, "U|U"]
&\cdots 
\end{tikzcd}
\]

The map $f^{K,\mu}$ consists of the arrows which are multiples of $\sigma|h_{\sigma,Z}$ while $f^{K,-\mu}$ consists of the arrows which are multiples of $\sigma|h_{\tau,Z}$. We display the maps $f^{-K,\mu}$ and $f^{-K,-\mu}$ below:
\[
\begin{tikzcd}[labels=description,row sep=1.7cm] 
\scE_{*,t}
	\ar[d, "\substack{f^{-K,\mu}\\+f^{-K,-\mu}}"]
\\  
\scM_{*,t}
\end{tikzcd}
=
\hspace{-.3cm}
\begin{tikzcd}[labels=description, {column sep=2.3cm,between origins}, row sep=1cm]
\cdots
&[-1.4cm] \cC_{t-\frac{1}{2}}^{(-\frac{3}{2}, t)}
	\ar[r,bend left,out=45,in=135,  "Z|U"]
&\cC_{t-\frac{1}{2}}^{(-\frac{1}{2}, t)}
	\ar[r,bend left,out=45,in=135,  "Z|L_Z"]
	\ar[l,bend left,out=45,in=135,  "W|1"]
	\ar[loop above,looseness=10, "{(W,Z)|h_{W,Z}}"]
&[.4cm] \cC_{t+\frac{1}{2}}^{(\frac{1}{2},t)}
	\ar[r,bend left,out=45,in=135,  "Z|1"]
	\ar[l,bend left,out=45,in=135, "W|L_W"]
	\ar[d,"\tau|h_{\sigma,W}",pos=.3,]
	\ar[dl, "\tau|h_{\tau, W}"]
	\ar[loop above, looseness=10, "{(Z,W)|h_{Z,W}}"]
& \cC_{t+\frac{1}{2}}^{(\frac{3}{2},t)}
	\ar[r,bend left,out=45,in=135,  "Z|1"]
	\ar[l,bend left,out=45,in=135,  "W|U"]
	\ar[dl, "U\tau|h_{\tau, W}"]
	\ar[d,"U\tau|h_{\sigma,W}",pos=.3,]
& \cC_{t+\frac{1}{2}}^{(\frac{5}{2},t)}
	\ar[l,bend left,out=45,in=135,  "W|U"]
	\ar[pos=.3,d,"U^2\tau|h_{\sigma,W}"]
	\ar[dl, "U^2\tau|h_{\tau, W}"]
&[-1.5cm] \cdots
\\[1.7cm]
\cdots
& \cS^{(-\frac{3}{2}, t-1)}
	\ar[r,bend left,out=45,in=135,  "T|1"]
	\ar[loop below,looseness=10, "U|U"]
&\cS^{(-\frac{1}{2}, t-1)}
	\ar[r,bend left,out=45,in=135, "T|1"]
	\ar[l,bend left,out=45,in=135,  "T^{-1}|1"]
	\ar[loop below,looseness=10, "U|U"]
&\cS^{(\frac{1}{2}, t-1)}
	\ar[r,bend left,out=45,in=135, "T|1"]
	\ar[l,bend left,out=45,in=135,  "T^{-1}|1"]
	\ar[loop below,looseness=10, "U|U"]
&\cS^{(\frac{3}{2}, t-1)}
	\ar[r,bend left,out=45,in=135,  "T|1"]
	\ar[l,bend left,out=45,in=135,  "T^{-1}|1"]
	\ar[loop below,looseness=10, "U|U"]
&\cS^{(\frac{5}{2}, t-1)}
	\ar[l,bend left,out=45,in=135,  "T^{-1}|1"]
	\ar[loop below,looseness=10, "U|U"]	
&\cdots 
\end{tikzcd}
\]
The map $f^{-K,\mu}$ consists of the arrows labeled by multiples of $\tau|h_{\sigma,W}$, whereas $f^{-K,-\mu}$ consists of the maps labeled labeled by multiples of $\tau|h_{\tau,W}$.

Note that all of the maps $f^{\pm K,\pm \mu}$ have no term with no algebra inputs, and have a term with one algebra input,  $(f^{\pm K,\pm \mu})_2^1$, but no term with two algebra inputs. The maps $(f^{\pm K,\pm \mu})_2^1$ are only non-trivial when the algebra input is a multiple of $\sigma$ or $\tau$, and these are given by the above diagrams.

Summarizing the analysis described above, we have the following:

\begin{prop} The bimodule ${}_{\cK} \cH_{-}^{\cK}\boxtimes {}_{\cK} \cX(L_P)^R$ consists of the following hypercube of bimodules:
\[
\begin{tikzcd}[column sep=2cm, row sep=2cm, labels=description] \scE_{*,*}\ar[r, "f^K+f^{-K}"]\ar[d,"f^\mu+f^{-\mu}"] \ar[dr,dashed] & \scJ_{*,*}\ar[d,"f^\mu+f^{-\mu}"]\\
\scF_{*,*}\ar[r, "f^K+f^{-K}"]& \scM_{*,*}
\end{tikzcd}
\]
where $\scE_{*,*}$ denotes the sum of all of the bimodules $\scE_{*,t}$ described above, and similarly for $\scF_{*,*}$, $\scJ_{*,*}$ and $\scM_{*,*}$. The diagonal map (dashed) is the sum of the four maps $f^{K,\mu}$, $f^{-K,\mu}$, $f^{K,\mu}$ and $f^{-K,-\mu}$. 
\end{prop}

\subsection{Gradings}
\label{sec:grading}
We now describe the absolute Maslov and Alexander gradings on the complex  $\bX(P,K,n)^{\bF[W,Z]}$ introduced in the previous section. Our main tool is the general result \cite{ZemExact}*{Theorem~10.8} about grading shifts on the link surgery formula.

The topological setting is as follows. We consider the 3-component link $K\# H\# L_P$. Denote the 3-components by $K$, $\mu$ and $P$. Following the notation of \cite{ZemExact}*{Theorem~10.8}, we write $J$ for $K\cup \mu$, and $\Lambda=(n,0)$ for the framing on $J$. Note that $S^3_{\Lambda}(J)=S^3$.

We recall some notation from \cite{ZemExact}*{Section~10}. Let $W_{\Lambda}(J)$ be the trace of surgery along $J$, oriented as a cobordism from $S^3$ to $S^3_{\Lambda}(J)$. Then $H_2(W_{\Lambda}(J)) \cong \mathbb{Z}^2$, which is generated $\Sigma_K$ and $\Sigma_\mu$, the core of the $2$-handles attached along $K$ and $\mu$ respectively. We readily compute:
\[\langle \Sigma_K,\Sigma_K\rangle = n,\quad  \langle\Sigma_K,\Sigma_\mu  \rangle = -1, \quad \langle\Sigma_\mu,\Sigma_\mu \rangle = 0. \]
It follows that for any $n$, we have
\[
\chi(W_{\Lambda}(J)) = 2, \quad\text{and}\quad \sigma(W_{\Lambda}(J)) =0.
\]

For $\ve{s} = (s,t,r)\in\bH(K\# H\# L_P)$, let $\frz_{\ve{s}}^J\in \Spin^c(W_\Lambda(J))$ be the $\Spin^c$ structure such that 
\[
\frac{\langle c_1(\frz_{\ve{s}}^J), \Sigma_K\rangle -\Sigma\cdot \Sigma_K}{2}=-s\quad \text{and} \quad \frac{\langle c_1(\frz_{\ve{s}}^J), \Sigma_\mu\rangle -\Sigma\cdot \Sigma_\mu}{2}=-t. 
\]
(Note that $\frz_{\ve{s}}^J$ does not depend on $r$). Here $\Sigma=\Sigma_K+\Sigma_\mu +\Sigma_{P}$, and $\Sigma_{P} = -l \Sigma_K -nl \Sigma_{\mu}$ in $H_2(W_{\Lambda}(J))$, where $l = \lk(\mu,P)$.

We readily compute that
\[
\begin{split} 
	c_1(\frz_{\ve{s}}^J) &= (1+2t-l)\PD[\Sigma_K]+ (1+2s+(2t-l)n)\PD[\Sigma_\mu],\quad \text{and}\\
	c_1(\frz_{\ve{s}}^J)^2 &=
	(1-(2t-l)^2)n-2(1+2t-l)(1+2s).
\end{split}
\]
Using \cite{ZemExact}*{Theorem~10.8} and Remark~\ref{rem:hypercubecorner},
we compute that the absolute Maslov $\gr_{\ws}$-grading on $\cCFK(P(K,n))$  is given by the formula
\begin{equation}
\begin{split}
\gr_{\ws}^{\mathrm{abs}} =&\gr_{\ws}+\frac{c_1^2(\frz_{\ve{s}}^J)-2\chi(W_\Lambda(J))-3 \sigma(W_{\Lambda}(J))}{4}+|J|-\veps \\
=& \gr_{\ws} + \frac{(1-(2t-l)^2)n-2(1+2t-l)(1+2s)+4}{4}-\veps,
\end{split}
\label{eq:gr-w-grading-shift}
\end{equation}
where $\veps =0$ on $\bE$, $\veps =1$ on $\bF$, $\bJ$, and $\veps=2$ on $\bM$. Here $\gr_{\ws}$ denotes the internal Maslov grading on the surgery complex obtained by the complex at each cube point as a link Floer complex, tensored with one or two copies of $\bF[T,T^{-1}]$. For example, the $\gr_{\ws}$ grading on $\bE$ is obtained by identifying $\bE$ with a suitable completion of $\cCFL(K\# H\# L_P)$. On $\bJ$, the grading $\gr_{\ws}$ is obtained by identifying $\bJ$ with a suitable completion of $\bF[T,T^{-1}]\otimes \cCFL(L_P)$.

Using \cite{ZemExact}*{Theorem~10.8}, we see that the absolute Alexander grading on $\cCFK(P(K,n))$ is given by 
\begin{equation}
\begin{split}
A^{\mathrm{abs}} =& r+\frac{\langle c_1(\frz_{\ve{s}}^J),\Sigma_P\rangle-\Sigma\cdot \Sigma_P}{2}\\
=& r + (s+tn)l.
\end{split}
\label{eq:Ar-grading-shift}
\end{equation}
Note that here $r$ denotes the third component in the $\sigma$-normalized Alexander grading on $\bX(P,K,n)$. 

For our purposes, it is most convenient to compute the $\gr_{\zs}$-grading shift as well. We do this slightly indirectly. We recall from Equation~\eqref{eq:Relation-Maslov-Alexander} that
\[
\gr_{\zs}=\gr_{\ws}-2\sum_{K_i\subset L_\veps} A_i^{\veps}.
\]
On $\cCFK(P(K,n))$, we also have that $\gr_{\zs}^{\mathrm{abs}}=\gr_{\ws}^{\mathrm{abs}}-2 A^{\mathrm{abs}}$. Furthermore, we recall from Equation~\eqref{eq:A-veps-A-sigma-relation} that $A_i^{\veps}=A_i^\sigma-\lk(K_i,L\setminus L_\veps)/2$. 
It is straightforward to combine these observations with the grading shifts computed in Equations~\eqref{eq:gr-w-grading-shift} and ~\eqref{eq:Ar-grading-shift}, to obtain the $\gr_{\zs}$-grading shift for $\cCFK(P(K,n))$.

We now summarize the $(\gr_{\ws},\gr_{\zs})$-shifts in $\bX(P,K,n)^{\bF[W,Z]}$. Write
\[
d=\frac{(1-(2t-l)^2)n-2(1+2t-l)(1+2s)+4}{4},
\]
then the following shifted complexes compute the absolute  $(\gr_{\ws},\gr_{\zs})$ gradings on $\bX(P,K,n)^{\bF[W,Z]}$:
\[
\begin{split}
	&E_{s,t}[d, d+2s+2t-2(s+tn)l]\\
	& F_{s,t}[d-1,d+2s-2(s+tn)l-l]\\
	& J_{s,t}[d-1,d+2t-2(s+tn)l]\\
	& M_{s,t}[d-2,d-2-2(s+tn)l-l].
\end{split}
\]

\section{Truncations}
\label{sec:truncation}

In this section, we describe how to truncate the complex 
\[
\bX(P,K,n)^R=\cX_n(Y,K)^{\cK}\boxtimes {}_{\cK} \cH_-^{\cK} \boxtimes {}_{\cK} \cX(L_P)^R
\]
 described in the previous section to obtain a finite dimensional model for $\cCFK(Y,P(K,n))^R$.

\begin{rem} It is possible to apply Manolescu and Ozsv\'{a}th's description \cite{MOIntegerSurgery}*{Section~10} of a truncation of the link surgery formula. However, this general technique turns out to typically give models which are larger than the ones we present for the case under consideration.
\end{rem}

\begin{prop}
	\label{prop:truncation}
	Let $K$ be a knot in an integer homology 3-sphere. Let $g$ denote the Seifert genus of $K$, and define $N\in \Z +\lk(\mu,P)/2$ to be
	\[
	N = \min\left\{ Q \in \Z+\frac{\lk(\mu,P)}{2}\middle|\begin{array}{l} L_{\sigma}\colon \cC_{q}\to\cS, \text{ and} \\ 
	L_{\tau}\colon\cC_{-q}\to \cS
	\\ \text{are homotopy equivalences} \\
	\text{for all }q\ge Q
	\end{array}
	\right\}.
	\]  
	\begin{enumerate}
		\item When $n\ge0$, let $h = \max\left\{g,-g+n+1\right\}.$ Then, $\bX(P,K,n)^R$ is homotopy equivalent to the type-$D$ module consisting of 
		\[\left(\bigoplus_{\substack{(s,t)\in \bH(P)\\ -g < s<h,\\ -N < t<N }}E_{s,t}\right) 
		\oplus
		\left(\bigoplus_{\substack{(s,t)\in \bH(P)\\ -g< s< h-1,\\ -N < t<N }}F_{s,t}\right) 
		\oplus \left(\bigoplus_{\substack{(s,t)\in \bH(P)\\ -g+n < s < h,\\ -N-1 < t< N }}J_{s,t}\right) 
		\oplus \left(\bigoplus_{\substack{(s,t)\in \bH(P)\\ -g+n < s <h-1,\\ -N-1 < t < N }}M_{s,t}\right), \]
		with all the maps $\Phi^{\pm K}$, $\Phi^{\pm\mu}$ and $\Phi^{\pm K,\pm \mu}$ between them as in $\bX(P,K,n)^R$.
		\item When $n<0$, $\bX(P,K,n)^R$ is homotopy equivalent to the type-$D$ module consisting of 
			\[\left(\bigoplus_{\substack{(s,t)\in \bH(P)\\ -g < s<g,\\ -N < t<N }}E_{s,t}\right) 
		\oplus
		\left(\bigoplus_{\substack{(s,t)\in \bH(P)\\ -g-1 < s< g,\\ -N < t<N }}F_{s,t}\right) 
		\oplus \left(\bigoplus_{\substack{(s,t)\in \bH(P)\\ -g+n < s < g,\\ -N < t< N-1 }}J_{s,t}\right) 
		\oplus \left(\bigoplus_{\substack{(s,t)\in \bH(P)\\ -g+n-1 < s < g,\\ -N < t < N-1 }}M_{s,t}\right), \]
        with all the maps $\Phi^{\pm K}$, $\Phi^{\pm\mu}$ and $\Phi^{\pm K,\pm \mu}$ between them as in $\bX(P,K,n)^R$.
	\end{enumerate}
	In particular, this gives a description of a finitely generated model of $\cCFK(P(K,n))$ as a type-$D$ module over $R=\bF\left[W,Z\right]$.
\end{prop}

Examples of the truncations from Proposition~\ref{prop:truncation} are shown in Figures~\ref{fig:truncation for n=-1} and ~\ref{fig:truncation when n=1}.

	\begin{figure}[H]
		\adjustbox{scale=0.9}{
				\tikz[ 
			overlay]{
				\draw[thick,blue]
				(13.4,2.4)--(6.1,2.4)--(6.1,1.4)--(2.95,1.4)--(2.95,0.8)--(6.1,-0.47)--(6.1,-1.05)--(2.95,-1.05)--(2.95,-1.76)--(6.1,-2.98)--(6.1,-3.4)--(13.4,-3.4)--(13.4,2.4);
			}
		\begin{tikzcd}[column sep = 7mm]
			 &[-7mm]\phantom{\cdots} &[-7mm]-\frac{5}{2} &  -\frac{5}{2} & -\frac{3}{2} & -\frac{3}{2} &  -\frac{1}{2} &  -\frac{1}{2} &   \frac{1}{2} & \frac{1}{2} &  \frac{3}{2} & \frac{3}{2} & [-7mm]\phantom{\cdots}\\[-5mm]
			  -1 & \cdots &{}_{E} \arrow[r] \arrow[d]  & {}_{F} \arrow[d] &{}_{E}  \arrow[l,red,squiggly] \arrow[r]   \arrow[d] & {}_{F} \arrow[d] & {}_{E} \arrow[l] \arrow[r] \arrow[d]  &  {}_{F} \arrow[d] & {}_{E} \arrow[l] \arrow[d] \arrow[r]   & {}_{F} \arrow[d,red,squiggly] & {}_{E} \arrow[l] \arrow[r]  \arrow[d,red,squiggly]  & {}_{F} \arrow[d,red,squiggly] &\cdots	\\
			  -1 & \cdots &  {}_{J} \arrow[r] & {}_{M} &  {}_{J}\arrow[r] \arrow[l] &  {}_{M} &  {}_{J}\arrow[r] \arrow[l] &  {}_{M} &  {}_{J}\arrow[r] \arrow[l] & {}_{M} & {}_{J}\arrow[r] \arrow[l] &  {}_{M} &\cdots \\	
			  0 & \cdots& {}_{E} \arrow[r] \arrow[d]   &  {}_{F} \arrow[d] \arrow[ulll,red,squiggly]  & {}_{E} \arrow[l] \arrow[r]  \arrow[d]  \arrow[ull, red, squiggly]  &  {}_{F} \arrow[d] \arrow[ull,red, squiggly] &  {}_{E} \arrow[l] \arrow[r] \arrow[d] \arrow[ull]  &  {}_{F} \arrow[d] \arrow[ull]& {}_{E} \arrow[l] \arrow[d] \arrow[r] \arrow[ull]  & {}_{F} \arrow[d,red,squiggly] \arrow[ull]& {}_{E} \arrow[l] \arrow[r]  \arrow[d,red,squiggly] \arrow[ull] 
			  & {}_{F} \arrow[d,red,squiggly] \arrow[ull] &\cdots  	\arrow[ull,start anchor={[xshift=4ex]}] \\
			  0 & \cdots &  {}_{J}\arrow[r] &  {}_{M} & {}_{J}\arrow[r] \arrow[l] &  {}_{M} &  {}_{J}\arrow[r] \arrow[l] &  {}_{M}&  {}_{J}\arrow[r] \arrow[l] &  {}_{M}  &  {}_{J}\arrow[r] \arrow[l] &  {}_{M} & \cdots \\	
			  1 & \cdots &  {}_{E}\arrow[r] & {}_{F} \arrow[ulll,red,squiggly]& {}_{E}  \arrow[l] \arrow[r] \arrow[ull,red,squiggly]  & {}_{F} \arrow[ull,red,squiggly] &{}_{E} \arrow[l] \arrow[r] \arrow[ull] & {}_{F} \arrow[ull] &{}_{E}\arrow[l] \arrow[r ] \arrow[ull] & {}_{F} \arrow[ull]&{}_{E} \arrow[l] \arrow[r,red,squiggly] \arrow[ull]  &{}_{F} \arrow[ull]&\cdots  	\arrow[ull,start anchor={[xshift=4ex]}] 
		\end{tikzcd}
	}
\caption{The truncation from Proposition~\ref{prop:truncation} when $g=1,$ $ N=\frac{3}{2}$ and $n=-1$. The truncation consists of the complexes in the region bounded by the blue line. Length 2 arrows are not shown. The squiggly red arrows are quasi-isomorphisms by Lemmas~\ref{lem:truncation-he-cases} and~\ref{lem:truncation-edge-case-1}. }
\label{fig:truncation for n=-1}

	\adjustbox{scale=0.9}{
		
			\tikz[ 
		overlay]{
			\draw[thick,blue]
			(12,3.5)--(10.85,3.5)--(7.7,2.1)--(7.7,1.45)--(10.8,1.45)--(10.8,0.83)--(7.7,-0.48)--(7.7,-1.1)--(10.8,-1.1)--(10.8,-1.75)--(7.7,-3.04)--(7.7,-3.7)--(10.8,-3.7)--(10.8,-4.6)--(12,-4.6)--(12,3.5);
		}
		\begin{tikzcd}[column sep = 7mm]
			&[-7mm]\phantom{\cdots} &[-7mm]-\frac{5}{2} &  -\frac{5}{2} & -\frac{3}{2} & -\frac{3}{2} &  -\frac{1}{2} &  -\frac{1}{2} &   \frac{1}{2} & \frac{1}{2} &  \frac{3}{2} & \frac{3}{2} & [-7mm]\phantom{\cdots}\\[-5mm]
			-2 & \cdots & {}_{J} \arrow[r] & {}_{M} & {}_{J}\arrow[r] \arrow[l,red,squiggly] &  {}_{M} & {}_{J}\arrow[r] \arrow[l,red,squiggly] & {}_{M}& {}_{J}\arrow[r] \arrow[l,red,squiggly] & {}_{M} & {}_{J}\arrow[r] \arrow[l,red,squiggly] &  {}_{M} &\cdots \\	
			-1 \arrow[urrr,red,squiggly]& \cdots & {}_{E} \arrow[r] \arrow[d] \arrow[urr,red,squiggly] &  {}_{F} \arrow[d] \arrow[urr,red,squiggly] & {}_{E} \arrow[l,red,squiggly] \arrow[r]  \arrow[d] \arrow[urr,red,squiggly]  & {}_{F} \arrow[d] \arrow[urr,red,squiggly]& {}_{E} \arrow[l] \arrow[r] \arrow[d]   \arrow[urr] &  {}_{F} \arrow[urr] \arrow[d] & {}_{E}\arrow[l] \arrow[d] \arrow[r]   \arrow[urr]  & {}_{F} \arrow[urr] \arrow[d,red,squiggly] & {}_{E} \arrow[l] \arrow[r]  \arrow[d,red,squiggly]  \arrow[urr,end anchor={[xshift=4ex]}] & {}_{F} \arrow[d,red,squiggly] &\cdots	\\
			-1 & \cdots &  {}_{J} \arrow[r] & {}_{M}&  {}_{J}\arrow[r] \arrow[l] &  {}_{M} &  {}_{J}\arrow[r] \arrow[l] &  {}_{M} &  {}_{J}\arrow[r] \arrow[l] & {}_{M}  &  {}_{J}\arrow[r] \arrow[l] & {}_{M} &\cdots \\	
			0 \arrow[urrr,red,squiggly]& \cdots & {}_{E} \arrow[r] \arrow[d]   \arrow[urr,red,squiggly] &  {}_{F} \arrow[d] \arrow[urr,red,squiggly] & {}_{E}  \arrow[l] \arrow[r]   \arrow[d] \arrow[urr,red,squiggly]  & {}_{F} \arrow[d] \arrow[urr,red,squiggly]& {}_{E} \arrow[l] \arrow[r] \arrow[d]  \arrow[urr] &  {}_{F}\arrow[urr] \arrow[d] &{}_{E}\arrow[l] \arrow[d] \arrow[r]  \arrow[urr] &{}_{F} \arrow[urr] \arrow[d,red,squiggly] &{}_{E} \arrow[l] \arrow[r]  \arrow[d,red,squiggly]  \arrow[urr,end anchor={[xshift=4ex]}]  & {}_{F} \arrow[d,red,squiggly] &\cdots	\\	
		   0 & \cdots & {}_{J} \arrow[r] & {}_{M} &  {}_{J}\arrow[r] \arrow[l] &  {}_{M}&  {}_{J}\arrow[r] \arrow[l] &  {}_{M}& {}_{J}\arrow[r] \arrow[l] &  {}_{M}  & {}_{J}\arrow[r] \arrow[l] & {}_{M} &\cdots \\	
			1 \arrow[urrr,red,squiggly]& \cdots & {}_{E} \arrow[r] \arrow[d]   \arrow[urr,red,squiggly] &  {}_{F} \arrow[d] \arrow[urr,red,squiggly] & {}_{E}  \arrow[l] \arrow[r]  \arrow[d] \arrow[urr,red,squiggly]  & {}_{F} \arrow[d] \arrow[urr,red,squiggly]& {}_{E} \arrow[l] \arrow[r] \arrow[d]  \arrow[urr] & {}_{F} \arrow[urr] \arrow[d] &{}_{E}\arrow[l] \arrow[d] \arrow[r]   \arrow[urr]  &{}_{F} \arrow[urr] \arrow[d,red,squiggly] &{}_{E} \arrow[l] \arrow[r,red,squiggly]  \arrow[d,red,squiggly]  \arrow[urr,end anchor={[xshift=4ex]}] &{}_{F} \arrow[d,red,squiggly] &\cdots	\\
			1 & \cdots & {}_{J} \arrow[r,red,squiggly] & {}_{M} & {}_{J}\arrow[r,red,squiggly] \arrow[l] & {}_{M} & {}_{J}\arrow[r,red,squiggly] \arrow[l] & {}_{M}& {}_{J}\arrow[r,red,squiggly] \arrow[l] & {}_{M}  & {}_{J}\arrow[r,red,squiggly] \arrow[l] &{}_{M} &\cdots 
		\end{tikzcd}
	}
\caption{The truncation from Proposition~\ref{prop:truncation} when $g=1,$ $N=\frac{3}{2}$ and $n=1$. The final result of the truncation process is the region bounded by the blue line.}
\label{fig:truncation when n=1}
\end{figure}

\begin{rem}
When $L_P$ is an L-space link, $N$ is the highest power of $t_1$ in the multivariable Alexander polynomial of $\Delta_{L_P}(t_1,t_2)$, normalized as in Section \ref{sec:Examples of bimodules}. This follows from Proposition~\ref{prop:GNH-function}. 
 \end{rem}

 Before proving Proposition~\ref{prop:truncation}, we prove several lemmas. In the following lemmas, we remind the reader that $(s,t)$ takes value in \[
\bH(P)= \left(\Z+\frac{1}{2}\right)\times \left(\Z+\frac{\lk(\mu,P)-1}{2}\right),
\] 
so 
\[
s-g,t-N\in \frac{1}{2}+\Z,
\]
where $g$ and $N$ are as in Proposition~\ref{prop:truncation}.

\begin{lem}
\label{lem:truncation-he-cases}
Let $N$ be as in Proposition~\ref{prop:truncation}. If $s>g$ , then the following maps are quasi-isomorphisms: 
\begin{enumerate}
\item $\Phi^K\colon E_{s,t}\to J_t$; 
\item $\Phi^{-K}\colon E_{-s,t}\to J_{t-1}$;
\item $\Phi^K \colon F_{s-1}\to M$; 
\item $\Phi^{-K}\colon F_{-s}\to M$.
\end{enumerate}
Similarly if $t>N$, then the following maps are quasi-isomorphisms:
\begin{enumerate}
\item $\Phi^\mu\colon E_{s,t}\to F_s$;
\item $\Phi^{-\mu}\colon E_{s,-t}\to F_{s-1}$;
\item $\Phi^{\mu}\colon J_{t-1}\to M$; 
\item $\Phi^{-\mu}\colon J_{-t}\to M$.
\end{enumerate}
\label{lem:quasi-iso}
\end{lem}

\begin{proof}
	 This follows quickly from our description in Section~\ref{sec:GeneralSatellite} of the complexes $E_{s,t}$, $F_{s}$, $J_{t}$, $M$ and the length $1$ maps between them.
	 
	 For example, consider the map $\Phi^K:E_{s,t}\to J_t$. When $s>g$, $E_{s,t}$ is obtained by replacing each generator of $\cCFK(K)$ by a copy of $\cC_{t+ \frac{1}{2}}$. An arrow from the differential on $\cCFK(K)$ which is weighted by $Z^iW^j$ is replaced by an arrow which is weighted by $U^j$. We can therefore identify
	 \[
	 E_{s,t}=A_{s-\frac{1}{2}}(K)^{\bF[U]}\boxtimes {}_{\bF[U]} \cC_{t+\frac{1}{2}}^{\bF[W,Z]}
	 \]
	 where the left $\bF[U]$ action on $\cC_{t+\frac{1}{2}}$ is given by $\delta_2^1(U,\xs)=\xs\otimes U$ for all $\xs$. Note that $J_s$ also has the description
	 \[
	 J_s=B(K)^{\bF[U]}\boxtimes {}_{\bF[U]}\cC_{t+\frac{1}{2}}^{\bF[W,Z]}.
	 \]
	 From the description in Section~\ref{sec:length-one-maps-general}, we see that with respect to these identifications, we have
	 \[
	 \Phi^K=v\otimes \id_{\cC_{t+\frac{1}{2}}}.
	 \]
	 Since 
	 \[
	 v\colon A_{s-\frac{1}{2}}(K)\to B(K)
	 \]
	  is a homotopy equivalence for $s>g$, we conclude that $\Phi^K$ is as well in this range. 	 The  remaining statements about $\Phi^K$ and $\Phi^{-K}$ are proven in a similar fashion.

	 The statements about $\Phi^{\pm\mu}$ are proved in the same manner, using that\[L_\sigma \colon \cC_{t'}\to \cS\qquad L_\tau \colon \cC_{-t'}\to \cS \]  are homotopy equivalence for $t' \ge N$. Note that $J_{t}$ contains only copies of $ \cC_{t+ \frac{1}{2}}$, while $E_{s,t}$ may contain both $\cC_{t-\frac{1}{2}}$ and $\cC_{t+ \frac{1}{2}}$ (depending on $s$). Therefore we can only show that $\Phi^{\mu}:E_{s,t} \to F_s$ is a homotopy equivalence for $t \ge N+\frac{1}{2}$, while $\Phi^{\mu}:J_{t} \to M$ is a homotopy equivalence for $t\ge N-\frac{1}{2}$. 
 \end{proof}
 
 There are additionally some helpful edge cases to point out:

 \begin{lem} \label{lem:truncation-edge-case-1}The maps
 \[
 \Phi^\mu\colon E_{s,t}\to F_s\quad \text{and} \quad \Phi^{-\mu} \colon E_{-s,-t}\to F_{-s}
 \] 
 are homotopy equivalences if $s>g$ and $t>N-1$. Additionally, the maps
\[
\Phi^K\colon E_{s,t}\to J_t\quad \text{and} \quad \Phi^{-K} \colon E_{-s,-t}\to J_{-t-1}
\]
are homotopy equivalences whenever $s>g-1$ and $t>N$.

 \end{lem}
 
 In particular, the maps 
 \[
 \Phi^{\mu}\colon E_{s,N-\frac{1}{2}}\to F_{s}\quad \text{and} \quad \Phi^{-\mu} \colon E_{-s,-N+\frac{1}{2}}\to F_{-s}
\]
are homotopy equivalences whenever $s>g$.

The proof of Lemmas~\ref{lem:truncation-edge-case-1} is similar to the proof of Lemma~\ref{lem:truncation-he-cases}. We leave the verification to the reader.

To show that the homotopy equivalences needed in the truncation process converge, we require the following lemmas concerning the coefficients of the length $1$ maps $\Phi^{\pm K}, \Phi^{\pm\mu}$, which again follow directly from the descriptions of the length 1 maps in the previous section.

\begin{lem}
\label{lem:U-weight-truncations}
Suppose that $t>N+1$.
\begin{enumerate}
\item The map $\Phi^{-\mu}\colon E_{s,t}\to F_{s-1}$ has image in $F_{s-1}\otimes (U)$.
\item The map $\Phi^{-\mu}\colon J_{t-1}\to M$ has image in $M\otimes (U)$.
\item The map $\Phi^{\mu}\colon E_{s,-t}\to F_{s-1}$ has image in $F_{s-1}\otimes (U)$.
\item The map $\Phi^{\mu} \colon J_{-t}\to M$ has image in $M\otimes (U)$.
\end{enumerate}
Suppose that $s>g+1$.
\begin{enumerate}
\item The map $\Phi^K\colon E_{-s,t}\to J_{t-1}$ has image in $J_{t-1}\otimes (U)$.
\item The map $\Phi^{-K}\colon E_{s,t}\to J_{t-1}$ has image in $J_{t-1}\otimes (U)$.
\item The map $\Phi^K\colon F_{-s}\to M$ has image in $M\otimes (U)$.
\item The map $\Phi^{K}\colon F_{s-1}\to M$ has image in $M\otimes (U)$.
\end{enumerate}
\end{lem}
\begin{proof} All of the proofs are rather similar and follow quickly from our description of the maps $\Phi^{\pm K}$ and $\Phi^{\pm \mu}$ in the last section, so we consider just the claim about $\Phi^{-\mu}\colon E_{s,t}\to F_{s-1}$. When $t>N+1$, we note that $t-\tfrac{1}{2}\ge N+1$ (since $t-N\in \tfrac{1}{2} +\Z$). The complexes $\cC_{t\pm \frac{1}{2}}$ are used to build $E_{s,t}$. Furthermore, $\Phi^{-\mu}$ is gotten by applying $L_\tau$ (for $t$ in this range, there will be no $h_{\tau,Z}$ or $h_{\tau,W}$ term which contribute). Furthermore then $L_\tau\colon \cC_{r}\to \cS$ has image in $\cS\otimes (U)$ whenever $r\ge N+1$. This is because $L_{\tau Z}=L_{\tau}\circ L_Z=L_{UT} \circ L_\tau$ and $L_Z\colon \cC_r\to \cC_{r+1}$ is an isomorphism for $r\ge N$. This proves the claim in this case, and the rest are left to the reader.
\end{proof}

\begin{proof}[Proof of Proposition~\ref{prop:truncation}]
To prove the statement, it is easiest to consider separately the cases that $n>0$, $n=0$, and $n<0$. In our proof, we focus on the case that $n>0$, though the other cases are similar. Also, we focus on the $U$-adic topology for concreteness, though essentially the same argument works for the the chiral topology.

 We describe the truncation in several steps. Firstly, we describe complexes $\frL$, $\frC$ and $\frR$ whose underlying spaces are as follows:
 \[
 \begin{split}
 \frL&=\left(\bigoplus_{ s<-g}E_{s,t}\right) \oplus \left( \bigoplus_{s<-g} F_{s,t}\right) 
 		\oplus \left(\bigoplus_{s<-g+n}J_{s,t}\right)\oplus \left( \bigoplus_{s<-g+n}\oplus M_{s,t}\right) 
\\
 \frC&=\left(\bigoplus_{-g<s<h} E_{s,t}\right) \oplus \left(\bigoplus_{-g<s<h-1} F_{s,t}\right)\oplus \left(\bigoplus_{-g+n<s<h} J_{s,t}\right) \oplus \left(\bigoplus_{-g+n<s<h-1} M_{s,t}\right)
 \\
 \frR&=\left(\bigoplus_{h<s}E_{s,t}\right) \oplus \left(\bigoplus_{\substack{ h-1<s}}F_{s,t} \right) 
  		\oplus \left( \bigoplus_{h<s} J_{s,t}\right) \oplus \left(\bigoplus_{h-1<s} M_{s,t}\right) 
 \end{split}
 \]
 Here, the direct sums are over all $s$ in the given range, and all $t\in \frac{\lk(K,\mu)-1}{2}+\Z$.   We equip $\frL$, $\frC$ and $\frR$ with the differentials induced by the natural components of the differential on $\bX(P,K,n)$.

 We claim that $\bX(P,K,n)\simeq \frC$. We observe that the complex $\bX(P,K,n)$ can be described as
 \[
 \begin{tikzcd}
 \frL & \frC \ar[r,"f",swap]\ar[l, "g"] & \frR
 \end{tikzcd}
 \]
 for some chain maps $f$ and $g$.  To show that $\bX(P,K,n)\simeq \frC$, it is therefore sufficient to show that $\frL$ and $\frR$ are contractible.  We begin by showing that $\frR$ is contractible. To see this, we observe that each the maps 
 \[
 \Phi^{K}\colon E_{s,t}\to J_{s,t}\quad \text{and} \quad \Phi^K \colon F_{s,t}\to M_{s,t}
 \]
 are homotopy equivalences in the range which appears in $\frR$ by Lemma~\ref{lem:truncation-he-cases}.

 We now use the homological perturbation lemma for chain complexes. We write the differential on $\frR$ as $d^1+\a^1$, where $d^1$ is the sum of $\Phi^K$ and the internal differentials of $E_{s,t}$, $F_{s,t}$, $J_{s,t}$ and $M_{s,t}$, and $\a^1$ is all of the remaining differentials. Note that $(\frR, d^1)$ is contractible. Write $h^1$ for the null-homotopy of the identity map on $(\frR,d^1)$. We now apply the homological perturbation lemma for chain complexes as in Lemma~\ref{lem:HPL-chain-complexes}, with the morphisms $h^1\colon (\frR,d^1)\to (\frR,d^1)$, $\pi^1\colon (\frR,d^1)\to 0$ and $i^1\colon 0\to (\frR, d^1)$, where $i^1$ and $\pi^1$ are zero. Note that $\pi^1\circ h^1=0$ and $h^1\circ i^1=0$, though we may have $h^1\circ h^1\neq 0$. Nonetheless we may still apply the homological perturbation lemma by Remark~\ref{rem:relax-HPL}. Note that the map $h^1$ will have components which map $J_{s,t}$ to $E_{s,t}$, and also map $M_{s,t}$ to $F_{s,t}$. Additionally, there may be components of $h^1$ which preserve the $E_{s,t}$, $F_{s,t}$, $J_{s,t}$ and $M_{s,t}$.
  
  We define an increasing filtration $\cF$ on $\frR$ which takes values in $\{0,1\}\times \Q$. We view $\{0,1\}\times \Q$ has having the lexicographic ordering, where we order $\{0,1\}$ and $\Q$ using their standard ordering. We put $E_{s,t}$ and $J_{s,t}$ into filtration level $(0,s)$ and we put $F_{s,t}$ and $M_{s,t}$ into filtration level $(1,s)$.
  
  We consider the expression $\sum_{j=0}^\infty (h^1 \circ \a^1)^j$. 
  We claim that if $m>0$, then there is an $n_0$ so that if $j>n_0$, the map $(h^1\circ \a^1)^j$ has image in the ideal $(U^m)$.    We observe that $h^1$ preserves the $\cF$-filtration level, while $\a^1$ strictly increases the filtration level.  We observe that the only term of $\a^1$ which does not increase the $\{0,1\}$ component of the filtration is $\Phi^{-K}$. We can describe $(h^1\circ \a^1)^j$ as a sum of sequences of ``$h$-arrows'' and ``$\a$-arrows''. By the filtration argument above, in any contributing sequence, at most one of the $\a$-arrows corresponds to a map other than $\Phi^{-K}$. The only arrows which decrease the $s$ coordinate are $\Phi^{-\mu}$ and $\Phi^{K,-\mu}$, and these drop the $s$-coordinate by $1$. (Note that $\Phi^{-K}$ increases the $s$ coordinate by $n$, which we assume is at least $1$). Therefore, the map $(h^1\circ \a^1)^{j}$ increases the $s$ coordinate of the filtration by at least $j-1$. Since $\frR$ has $s$-coordinates bounded from below, it follows that there is some
    $n_1$ so that the map $(h^1\circ \a^1)^{n_1}$  has image in only filtration levels $(i,s)$ with $s > g+1$.

   Therefore, we set $n_0=n_1+m+1$ and observe that if $j>n_0$, the map $(h^1\circ \a^1)^{j}$ must have image in $(U^m)$ because each contributing arrow sequence for $(h^1\circ \a^1)^j$ must have at least $m$ arrows corresponding to $\Phi^{-K}$ with domain $E_{s,t}$ or $F_{s,t}$ with $s>g+1$, which each have image in the ideal $(U)$ by Lemma~\ref{lem:U-weight-truncations}.
  
  The argument that $\frL$ is contractible is essentially the same, and we leave the details to the reader.

Next, we truncate  $\frC$ in the direction of the $t$-coordinate. We define
\[
\begin{split}
 \frC^-&=\left(\bigoplus_{\substack{-g<s<h\\ t<-N}} E_{s,t}\right) \oplus \left(\bigoplus_{\substack{-g<s<h-1\\ t<-N}} F_{s,t}\right)\oplus \left(\bigoplus_{\substack{-g+n<s<h\\ t<-N-1}} J_{s,t}\right) \oplus \left(\bigoplus_{\substack{-g+n<s<h-1 \\ t<-N-1}} M_{s,t}\right)
\\
 \frC_0&=\left(\bigoplus_{\substack{-g<s<h\\-N<t<N }} E_{s,t}\right) \oplus \left(\bigoplus_{\substack{-g<s<h-1\\ -N<t<N}} F_{s,t}\right)\oplus \left(\bigoplus_{\substack{-g+n<s<h \\ -N-1<t<N}} J_{s,t}\right) \oplus \left(\bigoplus_{\substack{-g+n<s<h-1 \\-N-1<t<N}} M_{s,t}\right)
\\
 \frC^+&=\left(\bigoplus_{\substack{-g<s<h\\ N<t}} E_{s,t}\right) \oplus \left(\bigoplus_{\substack{-g<s<h-1\\N<t}} F_{s,t}\right)\oplus \left(\bigoplus_{\substack{-g+n<s<h \\ N<t}} J_{s,t}\right) \oplus \left(\bigoplus_{\substack{-g+n<s<h-1\\ N<t}} M_{s,t}\right).
\end{split}
\]
 We equip each of these complexes with the differential induced from $\bX(P,K,n)^R$. Note that $\frC_0$ is truncation appearing in the main statement. We observe that $\frC$ decomposes as
\[
\frC=
\begin{tikzcd}
\frC^+\ar[r, "\xi^1"] & \frC_0 & \frC^- \ar[l,swap, "\zeta^1"]
\end{tikzcd}
\]
for some type-$D$ morphisms $\xi^1$ and $\zeta^1$. 

 To show that $\bX(P,K,n)^R\simeq \frC_0$, it therefore suffices to show that $\frC^+$ and $\frC^-$ are contractible. We focus on $\frC^+$, as the argument for $\frC^-$ is essentially the same. We again use the homological perturbation lemma for chain complexes. It is helpful to do this in two steps. We write the differential on $\frC^+$ as $\delta^1=d^1+\a^1$, where $d^1$ consists of the internal differentials of $E_{s,t}$, $F_{s,t}$, $J_{s,t}$ and $M_{s,t}$, as well as the maps $\Phi^\mu$. We set $\a^1$ to be the remaining terms of the hypercube differential.

There is a strong deformation retraction from $(\frC^+, d^1)$ to the complex 
\begin{equation}
\Cone\left(\Phi^K\colon \bigoplus_{N<t} E_{h-\frac{1}{2}, t}\to \bigoplus_{N<t} J_{h-\frac{1}{2}, t}\right).
\label{eq:perturbed-complex-cone}
\end{equation}
This strong deformation retraction is obtained by using the fact that $\Phi^{\mu}$ is a homotopy equivalence when $t>N$ by Lemma~\ref{lem:truncation-he-cases}. 

Let  $h^1$ denote the homotopy. The map $h^1$ acts by the inverse of $\Phi^\mu$ in this range. We claim that the series $\sum_{j=0}^\infty (\a^1\circ h^1)^j$ converges. In fact, we claim that $(\a^1\circ h^1)^j=0$ for large $j$. We can see this by the following filtration argument. We endow the space of the complex with a filtration over $\{0,1\}\times \Q$, equipped with the lexicographic ordering. We put $E_{s,t}$ and $F_{s,t}$ in filtration level $(1,s)$, and we put $J_{s,t}$ and $M_{s,t}$ in filtration level $(0,s)$. The map $h^1$ will have some components which move backwards along some $\Phi^\mu$ (via homotopy inverses), and also some components which preserve the $E_{s,t}$, $F_{s,t}$, $J_{s,t}$ and $M_{s,t}$. On the other hand $\a^1$ strictly decreases the filtration level.  We consider the components of a summand $(\a^1\circ h^1)^j$. We can think of these are contributed by a sequence of $\a$-arrows and $h$-arrows (where an $\a$-arrow means a component of the map $\a^1$, while an $h$-arrow means a component of the map $h^1$). Note that in such a sequence, at most one $\a$ arrow can lower the $\{0,1\}$ filtration level. On the other hand, $\frC^+$ is supported in only in the filtration levels $\{0,1\}\times \{-g+\tfrac{1}{2},\dots, h-\tfrac{1}{2}\}$. Therefore $(\a^1 \circ h^1)^j=0$ if $j>2(h+g-1)$.

Applying the homological perturbation lemma for chain complexes, Lemma~\ref{lem:HPL-chain-complexes}, we obtain an induced perturbation $\b^1$ on the complex in Equation~\eqref{eq:perturbed-complex-cone} so that $\frC^+$ is homotopy equivalent to the complex in Equation~\eqref{eq:perturbed-complex-cone}, perturbed by this $\beta^1$.
 The perturbation $\b^1$ consists only of some terms which send $E_{h-\frac{1}{2}, t}$ to $J_{h-\frac{1}{2}, t+1}$.
In fact, this perturbation is straightforward to compute, it is given by
\[
\b^1= \Phi^K \circ ((\Phi^{\mu})^{-1}\circ \Phi^{-\mu})^{n-1} \circ \Phi^{-\mu}.
\]
 Therefore the complex $\frC^+$ is homotopy equivalent to the complex
\[
\begin{tikzcd}
E_{h-\frac{1}{2}, N-\frac{1}{2}} \ar[r, "\Phi^K"] &J_{h-\frac{1}{2}, N-\frac{1}{2}} &E_{h-\frac{1}{2}, N+\frac{1}{2}} \ar[r, "\Phi^K"] \ar[l, "\b^1",swap] & J_{h-\frac{1}{2}, N+\frac{1}{2}}& \ar[l] \cdots
\end{tikzcd}.
\]
This complex is contractible since $\Phi^K$ is a homotopy equivalence for $s>g-1$ and $N<t$ by Lemma~\ref{lem:truncation-edge-case-1}, and each $\b^1$ has image in the ideal $(U)$ by Lemma~\ref{lem:U-weight-truncations}. We conclude that $\frC^+$ is contractible.

A symmetric argument shows that $\frC^-$ is contractible. We conclude, therefore, that $\bX(P,K,n)^R\simeq \frC_0$, as was claimed.

The arguments for the cases that $n<0$ and $n=0$ follow from very similar lines of reasoning. We leave the details to the reader.
\end{proof}

\begin{rem} When $n\ge \max(2g-1,0)$, we can also give a more symmetric description of the truncation, which depends on the parity of $n$. If $n$ is odd, then $\bX(P,K,n)^R$ is homotopy equivalent to the complex 
\[
\left(\bigoplus_{\substack{(s,t)\in \bH(P)\\ -\frac{n}{2} \le s\le \frac{n}{2}\\ -N < t<N }}E_{s,t}\right) 
\oplus
\left(\bigoplus_{\substack{(s,t)\in \bH(P)\\ -\frac{n}{2}\le  s\le \frac{n}{2}-1\\ -N < t<N }}F_{s,t}\right) 
\oplus \left(\bigoplus_{\substack{(s,t)\in \bH(P)\\ -N-1 < t< N }}J_{\frac{n}{2},t}\right). 
\]
If $n$ is even, then the following complex gives a symmetric truncation

\[\left(\bigoplus_{\substack{(s,t)\in \bH(P)\\ -\frac{n+1}{2} \le s\le\frac{n+1}{2}\\ -N < t<N }}E_{s,t}\right) 
\oplus
\left(\bigoplus_{\substack{(s,t)\in \bH(P)\\ -\frac{n+1}{2}\le s\le \frac{n-1}{2}\\ -N < t<N }}F_{s,t}\right) 
\oplus \left(\bigoplus_{\substack{(s,t)\in \bH(P)\\ s=\frac{n-1}{2}, \frac{n+1}{2}\\ -N-1 < t< N }}J_{s,t}\right) 
\oplus \left(\bigoplus_{\substack{(s,t)\in \bH(P)\\ -N-1 < t < N }}M_{\frac{n-1}{2},t}\right).\]
Note that this complex is slightly larger than the asymmetric construction described in Proposition~\ref{prop:truncation}. 
\end{rem}

\subsection{Further simplifications}

Proposition~\ref{prop:truncation} gives a uniform way to truncate the complex $\bX(P,K,n)$, but there further reductions of these truncations which can be used to give smaller models, sometimes with more complicated differentials. We now describe several modifications. We focus on the cases $n<0$, $0\le n<2g-1$ and $n \ge 2g-1$ separately. For the computations in Section~\ref{sec:Examples of satellite}, we will typically use the modified truncations described in this section.

When $n<0$, there are homotopy equivalences on the left and right boundaries of the form:
\[
 \Phi^{-K}:F_{-g-\frac{1}{2},t}\to M_{-g+n-\frac{1}{2},t-1} \quad \text{and}\quad \Phi^{K}:F_{g-\frac{1}{2},t-1}\to M_{g-\frac{1}{2},t-1}
\] respectively, for $-N+1<t<N$. See Figure~\ref{fig:truncation-example-further-1}, below, for an example when $g=1,$ $N=\frac{3}{2}$, and $n=-1$. Here, we draw all the $\Phi^{-K}$ maps as vertical arrows (as opposed to slanted arrows, as in Figure~\ref{fig:truncation for n=-1}), so there is a shift in the $s$-coordinate between rows with different $t$ values.
\begin{figure}[h]
	\begin{tikzcd}[column sep = tiny, row sep = small, font = \small, shorten=-1mm]		
		& & & & 	F_{-\frac{3}{2},-1} \arrow[d]& E_{-\frac{1}{2},-1}\arrow[l] \arrow[r]\arrow[d]\arrow[dr,dashed] \arrow[dl,dashed] & F_{-\frac{1}{2},-1} \arrow[d] & E_{\frac{1}{2},-1} \arrow[l] \arrow[d] \arrow[dl,dashed] \arrow[r] \arrow[dr,dashed] & F_{\frac{1}{2},-1}\arrow[d,red,squiggly] \\
		& & M_{-\frac{5}{2},-1} & J_{-\frac{3}{2},-1}\arrow[l]\arrow[r]& M_{-\frac{3}{2},-1} &
		J_{-\frac{1}{2},-1} \arrow[l]\arrow[r] & M_{-\frac{1}{2},-1} &
		J_{\frac{1}{2},-1} \arrow[l]\arrow[r]& M_{\frac{1}{2},-1}\\
		& & F_{-\frac{3}{2},0} \arrow[u,red,squiggly] \arrow[d]& E_{-\frac{1}{2},0}\arrow[l]\arrow[ul,dashed]\arrow[u]\arrow[ur,dashed]\arrow[r]\arrow[dr,dashed]\arrow[d]\arrow[dl,dashed]& F_{-\frac{1}{2},0} \arrow[u]\arrow[d] &
		E_{\frac{1}{2},0} \arrow[l]\arrow[ul,dashed]\arrow[u]\arrow[ur,dashed]\arrow[r]\arrow[dr,dashed]\arrow[d]\arrow[dl,dashed] & F_{\frac{1}{2},0} \arrow[u] \arrow[d, red, squiggly]& & \\
		M_{-\frac{5}{2},0} & J_{-\frac{3}{2},0}\arrow[l]\arrow[r]& M_{-\frac{3}{2},0} &
		J_{-\frac{1}{2},0} \arrow[l]\arrow[r] & M_{-\frac{1}{2},0} &
		J_{\frac{1}{2},0} \arrow[l]\arrow[r]& 
		M_{\frac{1}{2},0} & &\\		
		F_{-\frac{3}{2},1} \arrow[u,red,squiggly]& E_{-\frac{1}{2},1}\arrow[l]\arrow[ul,dashed]\arrow[u]\arrow[ur,dashed]\arrow[r]& F_{-\frac{1}{2},1} \arrow[u]&
		E_{\frac{1}{2},1} \arrow[l]\arrow[ul,dashed]\arrow[u]\arrow[ur,dashed]\arrow[r] & F_{\frac{1}{2},1} \arrow[u]& & & &
	\end{tikzcd}
	\caption{The truncation for $g=1,$ $N = \frac{3}{2}$, and $n=-1$.}
	\label{fig:truncation-example-further-1}
\end{figure}
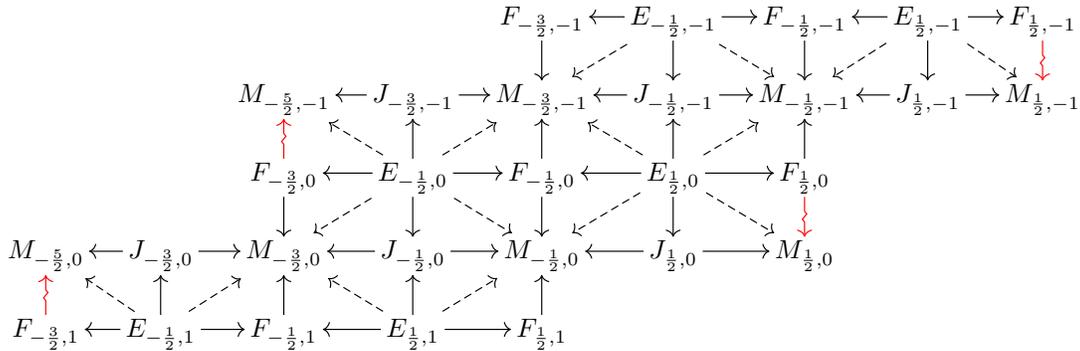

In the top right corner, 
\[\begin{tikzcd}
	F_{g-\frac{1}{2},-N+\frac{1}{2}} \arrow[r,red,squiggly, "\Phi^K"]&M_{g-\frac{1}{2},-N+\frac{1}{2}}
\end{tikzcd}
\] 
is a contractible subcomplex of the truncation, and so is 
\[
\begin{tikzcd}
F_{-g+\frac{1}{2},N-\frac{1}{2}} \arrow[r,red,squiggly, "\Phi^{-K}"]&M_{-g+n-\frac{1}{2},N-\frac{3}{2}}
\end{tikzcd} 
\] in the bottom left corner. Therefore, we can further truncate out these two subcomplexes. 

For other homotopy equivalences on the right boundary, they do not form a subcomplex or quotient complex. However, one can use the homological perturbation lemma to remove $F_{g-\frac{1}{2},t}$ and $M_{g-\frac{1}{2},t}$ for $-N+1<t<N-1$, along with all the arrows pointing to or from them, at the cost of adding an extra arrow $J_{ g-\frac{1}{2},t} \to M_{g+n-\frac{1}{2},t-1} $ for each $t$, which is the composition of the following:
	\[ 
\begin{tikzcd}[labels={description},column sep=10mm]
	J_{ g-\frac{1}{2},t} \arrow[r,"\Phi^{\mu}"] &
	M_{g-\frac{1}{2},t} \arrow[r, "\eta"] & F_{g-\frac{1}{2},t} \arrow[r,"\Phi^{-K}"] &M_{g+n-\frac{1}{2},t-1}.
\end{tikzcd}
\] 
Here, $\eta$ is a homotopy inverse of $\Phi^{K}:F_{g-\frac{1}{2},t} \to M_{g-\frac{1}{2},t}$. 

Symmetrically, one can perform similar modifications on the left boundary. See the figure below for the result of this modification, where the new arrows are drawn as blue lines.
\begin{figure}[h]
	\begin{tikzcd}[column sep = small, row sep = small, font = \small,shorten =-1mm]		
		& & & 	F_{-\frac{3}{2},-1} \arrow[d]& E_{-\frac{1}{2},-1}\arrow[l] \arrow[r]\arrow[d]\arrow[dr,dashed] \arrow[dl,dashed] & F_{-\frac{1}{2},-1} \arrow[d] & E_{\frac{1}{2},-1} \arrow[l] \arrow[d] \arrow[dl,dashed]  \\
		& & J_{-\frac{3}{2},-1}\arrow[ddl,blue ]\arrow[r]& M_{-\frac{3}{2},-1} &
		J_{-\frac{1}{2},-1} \arrow[l]\arrow[r] & M_{-\frac{1}{2},-1} &
		J_{\frac{1}{2},-1} \arrow[l]\\
		& & E_{-\frac{1}{2},0}\arrow[u]\arrow[ur,dashed]\arrow[r]\arrow[dr,dashed]\arrow[d]\arrow[dl,dashed]& F_{-\frac{1}{2},0} \arrow[u]\arrow[d] &
		E_{\frac{1}{2},0} \arrow[l]\arrow[ul,dashed]\arrow[u]\arrow[ur,dashed]\arrow[d]\arrow[dl,dashed] & & \\
		J_{-\frac{3}{2},0}\arrow[r]& M_{-\frac{3}{2},0} &
		J_{-\frac{1}{2},0} \arrow[l]\arrow[r] & M_{-\frac{1}{2},0} &
		J_{\frac{1}{2},0} \arrow[l] \arrow[uur,blue]& &\\		
		E_{-\frac{1}{2},1}\arrow[u]\arrow[ur,dashed]\arrow[r]& F_{-\frac{1}{2},1} \arrow[u]&
		E_{\frac{1}{2},1} \arrow[l]\arrow[ul,dashed]\arrow[u]\arrow[ur,dashed]\arrow[r] & F_{\frac{1}{2},1} \arrow[u]& & &
	\end{tikzcd}
	\caption{A further simplification for $g=1, N = \frac{3}{2}$, and $n=-1$. }
	\label{fig:simplified truncation when n=-1}
\end{figure}
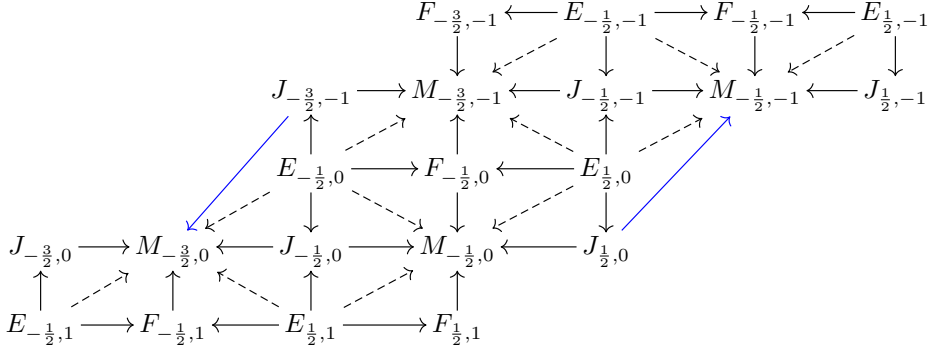

When $0\le n<2g-1$, the parameter $ h = \max\left\{g,-g+n+1\right\}$ equals $g$. See Figure \ref{fig:truncation when n=0} for an example when $g=1,N=\frac{3}{2}$ and $n=0$.
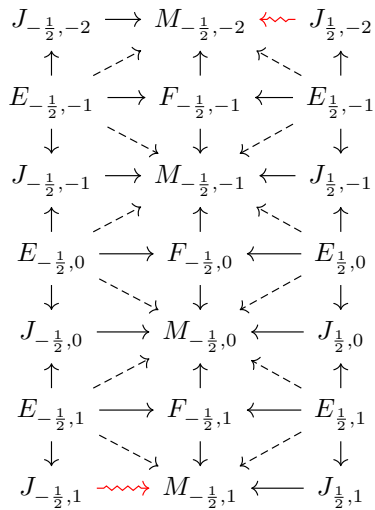
\begin{figure}[h]
	\begin{tikzcd}[column sep = small, row sep = small, font = \small]
		J_{-\frac{1}{2},-2} \arrow[r] &M_{-\frac{1}{2},-2}& J_{\frac{1}{2},-2} \arrow[l,red,squiggly] \\
		E_{-\frac{1}{2},-1}  \arrow[r]\arrow[d]\arrow[dr,dashed]\arrow[u]\arrow[ur,dashed] & F_{-\frac{1}{2},-1} \arrow[d] \arrow[u]& E_{\frac{1}{2},-1} \arrow[l] \arrow[d] \arrow[dl,dashed]\arrow[u] \arrow[ul,dashed] \\
		J_{-\frac{1}{2},-1} \arrow[r] & M_{-\frac{1}{2},-1} & J_{\frac{1}{2},-1}\arrow[l] \\
		E_{-\frac{1}{2},0} \arrow[u] \arrow[ur,dashed] \arrow[r]\arrow[d]\arrow[dr,dashed] & F_{-\frac{1}{2},0} \arrow[u] \arrow[d]& E_{\frac{1}{2},0} \arrow[u] \arrow[ul,dashed] \arrow[l] \arrow[dl,dashed] \arrow[d] \\
		J_{-\frac{1}{2},0} \arrow[r] &M_{-\frac{1}{2},0}& J_{\frac{1}{2},0} \arrow[l] \\
		E_{-\frac{1}{2},1} \arrow[u] \arrow[ur,dashed] \arrow[r] \arrow[d] \arrow[dr,dashed]&F_{-\frac{1}{2},1} \arrow[u] \arrow[d]& E_{\frac{1}{2},1} \arrow[l] \arrow[ul,dashed]\arrow[dl,dashed]\arrow[u] \arrow[d]\\
		J_{-\frac{1}{2},1} \arrow[r,red,squiggly] &M_{-\frac{1}{2},1}& J_{\frac{1}{2},1} \arrow[l]
	\end{tikzcd}
	\caption{The truncation for $g=1, N=\frac{3}{2},$ and $n=0$}
	\label{fig:truncation when n=0}
\end{figure}

 On the top boundary, there is a contractible subcomplex of the form  \[\begin{tikzcd}
	M_{-g+n+\frac{1}{2},-N-\frac{1}{2}} & J_{-g+n+\frac{3}{2},-N-\frac{1}{2}} \arrow[l,red,squiggly] \arrow[r] &\cdots \arrow[r]&M_{g-\frac{3}{2},-N-\frac{1}{2}}& J_{g-\frac{1}{2},-N-\frac{1}{2}} \arrow[l,red,squiggly].
\end{tikzcd}\]
Symmetrically, on the bottom boundary, there is a contractible subcomplex
 \[\begin{tikzcd}
 J_{-g+n+\frac{1}{2},N-\frac{1}{2}}  \arrow[r,red,squiggly] &M_{-g+n+\frac{1}{2},N-\frac{1}{2}} &\arrow[l]\cdots & J_{g-\frac{3}{2},N-\frac{1}{2}} \arrow[l] \arrow[r,red,squiggly]&M_{g-\frac{3}{2},N-\frac{1}{2}}.
 \end{tikzcd}\]

After further truncating these two subcomplexes, there is only one copy $J_{-g+n+\frac{1}{2},-N-\frac{1}{2}}$ left on the top row, and only one copy $J_{g-\frac{1}{2},N-\frac{1}{2}}$ left on the bottom row. Let $s=-g+\frac{1}{2} $ and $t=-N+\frac{1}{2} $, note that we have the following homotopy equivalence:
\[
\begin{tikzcd}[column sep=10mm]
	 J_{s+n,t-1} & [-5mm]\phantom{\cdots} &[-11mm]\phantom{\simeq} &[-11mm]M_{s+n-1,t-1} & J_{s+n,t-1} \arrow[l,red,squiggly]&[-5mm]\phantom{\cdots} &[-11mm]\phantom{\simeq} & [-11mm]\phantom{F} &\phantom{E}&[-3mm]\phantom{\cdots}\\
	 E_{s,t}\arrow[u]\arrow[r]\arrow[d]\arrow[dr,dashed]&\cdots &\simeq &F_{s-1,t} \arrow[u,red,squiggly]&E_{s,t}\arrow[l]\arrow[ul,dashed]\arrow[u]\arrow[r]\arrow[d]\arrow[dr,dashed]&\cdots &\simeq &F_{s-1,t} &E_{s,t}\arrow[l]\arrow[d]\arrow[dr,dashed] \arrow[r]&\cdots\\
	 \cdots & \cdots & & & \cdots &\cdots & & &\cdots &\cdots
\end{tikzcd},
\] 
so we can replace $ J_{-g+n+\frac{1}{2},-N-\frac{1}{2}}$ with $F_{-g-\frac{1}{2},-N+\frac{1}{2}}$ at the top left corner. Symmetrically, at the bottom right corner, we can replace $J_{g-\frac{1}{2},N-\frac{1}{2}}$ with $F_{g-\frac{1}{2},N-\frac{1}{2}}$. Hence, we can modify the truncation as in Figure~\ref{fig:truncation-swap-J-F-further}.

\begin{figure}[h]
	\begin{tikzcd}[column sep = small, row sep = small, font = \small]
		F_{-\frac{3}{2},-1} & E_{-\frac{1}{2},-1}\arrow[l] \arrow[r]\arrow[d]\arrow[dr] & F_{-\frac{1}{2},-1} \arrow[d] & E_{\frac{1}{2},-1} \arrow[l] \arrow[d] \arrow[dl]& \\
		& J_{-\frac{1}{2},-1} \arrow[r] & M_{-\frac{1}{2},-1} & J_{\frac{1}{2},-1}\arrow[l]& \\
		& E_{-\frac{1}{2},0} \arrow[u] \arrow[ur] \arrow[r]\arrow[d]\arrow[dr] & F_{-\frac{1}{2},0} \arrow[u] \arrow[d]& E_{\frac{1}{2},0} \arrow[u] \arrow[ul] \arrow[l] \arrow[dl] \arrow[d] &\\
		& J_{-\frac{1}{2},0} \arrow[r] &M_{-\frac{1}{2},0}& J_{\frac{1}{2},0} \arrow[l] & \\
		& E_{-\frac{1}{2},1} \arrow[u] \arrow[ur] \arrow[r] &F_{-\frac{1}{2},1} \arrow[u]& E_{\frac{1}{2},1} \arrow[l] \arrow[ul]\arrow[u]\arrow[r]& F_{\frac{1}{2},1}
	\end{tikzcd}
	\caption{A further simplification for $g=1, N=\frac{3}{2},$ and $n=0$.}
	\label{fig:truncation-swap-J-F-further}
\end{figure}

\subsection{Large framings}

When $n\ge 2g-1$, the truncation described in Proposition~\ref{prop:truncation} has some useful properties. Note that when $n\ge 2g-1$, the ordinary mapping cone complex $\bX_n(Y,K)$, which computes $\ve{\CF}^-(Y_n(K))$, also simplifies substantially, and becomes a direct sum of certain $A_s(K)$ complexes (with no $v$ or $h_n$ differentials) and no $B_s(K)$ complexes.	When $n\ge 2g-1$, our truncation of $\bX(P,K,n)^R$ also simplifies. 

In this case, $h =\max\left\{g,-g+n+1\right\}$ takes the value $-g+n+1$. Also, there are no $M_{s,t}$ complexes appearing in the truncation, and only one complex $J_{-g+n+1/2,t}$ for each $t$ from $-N-\frac{1}{2} $ to $N-\frac{1}{2}$. Also, none of the length 2 maps $\Phi^{\pm K,\pm \mu}$ appear in the truncation. The resulting truncation is a staircase of certain $E_{s,t}$, $F_{s,t}$ and $J_{s,t}$ complexes, as illustrated in Figure~\ref{fig:truncation when n=2}.

We can do a similar modification as in the  $0\le n<2g-1$ case to replace $J_{-g+n+\frac{1}{2},-N-\frac{1}{2}}$ with $F_{-g-\frac{1}{2},-N+\frac{1}{2}}$ and $J_{-g+n+\frac{1}{2},N-\frac{1}{2}}$ with $F_{-g+n+\frac{1}{2},N-\frac{1}{2}}$. We draw the truncation for $g=1,N=\frac{3}{2}$ and $n=2$ after this change in Figure \ref{fig:truncation when n=2}. Note that in this figure we return to the convention used in the proof of the proposition, where the horizontal coordinate corresponds to the $s$-coordinate, and $\Phi^{-K}$'s are drawn as tilted arrows. 

	\begin{figure}[h]
	\adjustbox{scale=0.85}{
		\begin{tikzcd} 
			F_{-\frac{3}{2},-1} &
			E_{-\frac{1}{2},-1} \arrow[l] \arrow[r] &
			F_{-\frac{1}{2},-1} &
			E_{\frac{1}{2},-1} \arrow[l] \arrow[r] &
			F_{\frac{1}{2},-1} &
			E_{\frac{3}{2},-1} \arrow[l] \arrow[d, red, squiggly] & \\
			& & & & & J_{\frac{3}{2},-1} & \\
		      &
			E_{-\frac{1}{2},0}  \arrow[r] \arrow[urrrr]&
			F_{-\frac{1}{2},0} &
			E_{\frac{1}{2},0} \arrow[l] \arrow[r] &
			F_{\frac{1}{2},0} &
			E_{\frac{3}{2},0} \arrow[l] \arrow[d, red, squiggly] & \\
			& & & & & J_{\frac{3}{2},1} & \\
			  &
			E_{-\frac{1}{2},1}  \arrow[r] \arrow[urrrr]&
			F_{-\frac{1}{2},1} &
			E_{\frac{1}{2},1} \arrow[l] \arrow[r] &
			F_{\frac{1}{2},1} &
			E_{\frac{3}{2},1} \arrow[l] \arrow[r,red,squiggly] & F_{\frac{3}{2},1} 
		\end{tikzcd}
	}
	\caption{The truncation for $g=1, N = \frac{3}{2}$ and the framing $n=2$}
	\label{fig:truncation when n=2}
\end{figure}

In the bottom row, if $n\ge 2g$, there is a contractible quotient complex of the form \[
\begin{tikzcd}
	E_{g+\frac{1}{2},N-\frac{1}{2}} \arrow[r,red,squiggly] & F_{g+\frac{1}{2},N-\frac{1}{2}} &\cdots \arrow[l] &E_{-g+n+\frac{1}{2},N-\frac{1}{2}}\arrow[r,red,squiggly]\arrow[l]& F_{-g+n+\frac{1}{2},N-\frac{1}{2}}
\end{tikzcd}\]
that can be truncated, where the homotopy equivalence follows from Lemma \ref{lem:truncation-edge-case-1}. Therefore, the last row will always take the form
\[\begin{tikzcd}
	E_{-g+\frac{1}{2},N-\frac{1}{2}} \arrow[r] &	F_{-g+\frac{1}{2},N-\frac{1}{2}} &\cdots \arrow[l] &E_{g-\frac{1}{2},N-\frac{1}{2}} \arrow[l]\arrow[r] &F_{g-\frac{1}{2},N-\frac{1}{2}},
\end{tikzcd}\]
as long as $n\ge 2g-1$.

For the homotopy equivalence \[\Phi^{K}:E_{-g+n+\frac{1}{2},t}\to J_{-g+n+\frac{1}{2},t}\] with $-N<t<N-1$, we can again apply the homological perturbation lemma to cancel these pairs, and add an extra arrow from $E_{-g+\frac{1}{2},t-1}$ to $F_{-g+n-\frac{1}{2},t}$, which is the composition:
\[\begin{tikzcd}[labels={description}, column sep=1.5cm]
	 E_{-g+\frac{1}{2},t-1} \arrow[r,"\Phi^{-K}"] & J_{-g+n+\frac{1}{2},t} \arrow[r,"\eta"] & E_{-g+n+\frac{1}{2},t}  \arrow[r,"\Phi^{-\mu}"] & F_{-g+n-\frac{1}{2},t},
\end{tikzcd}\]
where $\eta$ is a homotopy inverse of $\Phi^{K}:E_{-g+n+\frac{1}{2},t}\to J_{-g+n+\frac{1}{2},t}.$ The final result is shown in the following figure, with the new arrows indicated by blue lines. Note that the final outcome is a zigzag pattern consisting of $E$ terms and $F$ terms, and there is a stabilization phenomenon for $\cCFK(P(K,n))$ when $n$ grows larger than $2g-1$.
	\begin{figure}[h]
	\adjustbox{scale=0.85}{
		\begin{tikzcd} 
			F_{-\frac{3}{2},-1} &
			E_{-\frac{1}{2},-1} \arrow[l] \arrow[r] &
			F_{-\frac{1}{2},-1} &
			E_{\frac{1}{2},-1} \arrow[l] \arrow[r] &
			F_{\frac{1}{2},-1}   \\
			& E_{-\frac{1}{2},0} \arrow[urrr,blue] \arrow[r] &
			F_{-\frac{1}{2},0} &
			E_{\frac{1}{2},0} \arrow[l] \arrow[r] &
			F_{\frac{1}{2},0}\\
			& E_{-\frac{1}{2},1}  \arrow[r]  \arrow[urrr,blue]&
			F_{-\frac{1}{2},1} &
			E_{\frac{1}{2},1} \arrow[l] \arrow[r] &
			F_{\frac{1}{2},1}
		\end{tikzcd}
	}
	\caption{A further simplification for $g=1, N = \frac{3}{2}$, and $n=2$}
\end{figure}

A more symmetric choice of truncation for this case is shown in the figure below, achieved by slightly adjusting the range of truncation in the proposition to be symmetric, and using the homotopy equivalence $\Phi^{-K}:E_{-\frac{3}{2},0} \to J_{\frac{1}{2},-1}$ to do the homological perturbation.
	\begin{figure}[h]
	\adjustbox{scale=0.85}{
		\begin{tikzcd} 
			F_{-\frac{3}{2},-1} &
			E_{-\frac{1}{2},-1} \arrow[l] \arrow[r] &
			F_{-\frac{1}{2},-1} &
			E_{\frac{1}{2},-1} \arrow[l] \arrow[dlll,blue]  &   \\
			F_{-\frac{3}{2},0} & E_{-\frac{1}{2},0} \arrow[l ]\arrow[r] &
			F_{-\frac{1}{2},0} &
			E_{\frac{1}{2},0} \arrow[l] \arrow[r] &
			F_{\frac{1}{2},0}\\
			& E_{-\frac{1}{2},1}  \arrow[r]  \arrow[urrr,blue]&
			F_{-\frac{1}{2},1} &
			E_{\frac{1}{2},1} \arrow[l] \arrow[r] &
			F_{\frac{1}{2},1}
		\end{tikzcd}
	}
\end{figure}

\section{Examples}
\label{sec:Examples of satellite}

In this section, we perform a number of example computations using the truncations described in Section~\ref{sec:truncation}.

\subsection{Unknot companion}

\label{exm:unknot companion}

We begin by considering the case that the companion $K$ is an unknot (so $g=0$), to illustrate the techniques from Section~\ref{sec:truncation} in a simple case. 

Note that these examples can also be computed by using the rationally framed solid tori $\cD_{-1/n}^\cK$ from \cite{ZemBordered}*{Section~18.2}. This is because topologically, we can blow down the unknot $K=U$ to see that $P(U,n)$ is the knot $P$ in the $-1/n$ surgery on the meridianal component of $L_P$. Therefore
\[
\cCFK(P(U,n))^R\simeq \cD_{-1/n}^{\cK} \boxtimes {}_{\cK} \cX(L_P)^R.
\]

Note that in this case, $F_{s,t} = M_{s,t} = \cS$, $J_{s,t} = \cC_{t+\frac{1}{2}}, $  $E_{s,t} = \cC_{t-\frac{1}{2}} $ if $s\le -\frac{1}{2}$ and $E_{s,t} = \cC_{t+\frac{1}{2}}$ if $s\ge \frac{1}{2}$. 

The maps $\Phi^{\mu}:E_{s,t} \to F_{s,t}$ and $\Phi^{\mu}:J_{s,t} \to M_{s,t}$ are given by $L_{\sigma}$. The maps $\Phi^{-\mu}:E_{s,t} \to F_{s-1,t}$ and  $\Phi^{-\mu}:J_{s,t} \to M_{-s,t}$ are given by $L_{\tau}$.

The explicit description depends on the framing $n$ as follows.
\begin{enumerate}
	\item When $n=0$, the truncated complex consists of \[\bigoplus_{\substack{t\in N+\frac{1}{2}+\mathbb{Z} \\-N<t<N} }E_{\frac{1}{2},t} \oplus \bigoplus_{\substack{t\in N+\frac{1}{2}+\mathbb{Z} \\-N-1<t<N}}J_{\frac{1}{2},t},\]
	with $\Phi^{K}: E_{\frac{1}{2},t} \to J_{\frac{1}{2},t}$ being the identity map $\bI:\cC_{t+\frac{1}{2}} \to \cC_{t+\frac{1}{2}} $, and \newline$\Phi^{-K}: E_{\frac{1}{2},t} \to J_{\frac{1}{2},t-1} $ being the map $L_{W}:\cC_{t+\frac{1}{2}} \to \cC_{t-\frac{1}{2}}.$ Therefore, the truncated complex \[\begin{tikzcd}[labels={description}]
	\hspace*{10mm}	\cC_{-N} & \cC_{-N+1}\arrow[l,"L_{W}"] \arrow[r,red,"\bI"] & \cC_{-N+1} & \cdots \arrow[l,"L_{W}"] \arrow[r,red,"\bI"] & \cC_{N-1} & \cC_{N}\arrow[l,"L_{W}"]\arrow[r,red,"\bI"]&\cC_{N} 
	\end{tikzcd}\]
	is homotopy equivalent to $\cC_{-N}$. Note that $\cC_{-N}$ is homotopy equivalent to $\cS$, as $L_{\tau}: \cC_{-N} \to \cS$ is a homotopy equivalence by the definition of $N$ in Proposition \ref{prop:truncation}, so we get the expected result 
	\[
	\cCFK(P(U,0)) = \cCFK(P) \simeq \cS. 
	\]

	\item When $n>0$, one can use a slightly different truncation than the one in Proposition \ref{prop:truncation}, which consists of the following, and maps between them: 
	\[ \bigoplus_{\substack{(s,t)\in \bH(P)\\-1<s<n\\ -N<t<N}}E_{s,t} \oplus \bigoplus_{\substack{(s,t)\in \bH(P)\\-1<s<n-1\\ -N<t<N}}F_{s,t}  \oplus \bigoplus_{\substack{t\in N+\frac{1}{2}+\mathbb{Z}\\ -N-1<t<N}}J_{n-\frac{1}{2},t},\]
	with $\Phi^{K}: E_{n-\frac{1}{2},t} \to J_{n-\frac{1}{2},t}$ being the identity
	map $\bI:\cC_{t+\frac{1}{2}} \to \cC_{t+\frac{1}{2}}$, and $\Phi^{-K}:E_{-\frac{1}{2},t} \to J_{n-\frac{1}{2},t-1}$ being the identity map $\bI:\cC_{t-\frac{1}{2}} \to \cC_{t-\frac{1}{2}}$. An example of this truncation when $n=1$ and $N=\frac{3}{2}$ is drawn in the following diagram, where red arrows represent homotopy equivalences.

\begin{figure}[h]
		\hspace*{10mm}
	\adjustbox{scale = 0.7}{
		\begin{tikzcd}[labels={description}, column sep=7mm,shorten =-1mm]
			J_{\frac{1}{2},-2} & & & & & &\\
			E_{-\frac{1}{2},-1}\arrow[u,"\Phi^{-K}",red] \arrow[r,"\Phi^{\mu}"]& F_{-\frac{1}{2},-1} & E_{\frac{1}{2},-1} \arrow[l,"\Phi^{-\mu}"] \arrow[d,"\Phi^{K}",red]& & & & \\
			& & J_{\frac{1}{2},-1}  & & & &\\
			& & E_{-\frac{1}{2},0}\arrow[u,"\Phi^{-K}",red] \arrow[r,"\Phi^{\mu}"] & F_{-\frac{1}{2},0} & E_{\frac{1}{2},0} \arrow[l,"\Phi^{-\mu}"] \arrow[d,"\Phi^{K}",red]& &\\
			& &  & & J_{\frac{1}{2},0}  & & \\
			& &  & & E_{-\frac{1}{2},1}\arrow[u,"\Phi^{-K}",red] \arrow[r,"\Phi^{\mu}"] & F_{-\frac{1}{2},1} & E_{\frac{1}{2},1}\arrow[l,"\Phi^{-\mu}"] \arrow[d,"\Phi^{K}",red]\\
			& &  & &  & &  J_{\frac{1}{2},1}\\
		\end{tikzcd}
	\hspace*{-25mm}
	=
	\hspace*{-15mm}
	\begin{tikzcd}[labels={description}, column sep=7mm,shorten =-1mm]
		\cC_{-\frac{3}{2}} & & & & & &\\
		\cC_{-\frac{3}{2}}\arrow[u,"\bI",red] \arrow[r,"L_{\sigma}"]& \cS & \cC_{-\frac{1}{2}} \arrow[l,"L_{\tau}"] \arrow[d,"\bI",red]& & & & \\
		& & \cC_{-\frac{1}{2}} & & & &\\
		& & \cC_{-\frac{1}{2}}\arrow[u,"\bI",red] \arrow[r,"L_{\sigma}"] & \cS & \cC_{\frac{1}{2}} \arrow[l,"L_{\tau}"] \arrow[d,"\bI",red]& &\\
		& &  & & \cC_{\frac{1}{2}}  & & \\
		& &  & & \cC_{\frac{1}{2}}\arrow[u,"\bI",red] \arrow[r,"L_{\sigma}"] & \cS & \cC_{\frac{3}{2}}\arrow[l,"L_{\tau}"] \arrow[d,"\bI",red]\\
		& &  & &  & &  \cC_{\frac{3}{2}}\\
	\end{tikzcd}
	}
\end{figure}

This is homotopy equivalent to 
\[\begin{tikzcd}[labels={description}]
\cS & \cC_{-\frac{1}{2}} \arrow[l,"L_{\tau}"]\arrow[r,"L_{\sigma}"] & \cS & \cC_{\frac{1}{2}} \arrow[l,"L_{\tau}"] \arrow[r,"L_{\sigma}"] &\cS.
\end{tikzcd}\]

	In general, $\cCFK(P(U,n))$ is homotopy equivalent to the following complex:
	\[\begin{tikzcd}[labels={description}]
		\cS & \mathcal{C}'_0 \arrow[l,"L_{\tau}"] \arrow[r,"L_{\sigma}"] &\cS & \mathcal{C}'_1 \arrow[l,"L_{\tau}"] \arrow[r,"L_{\sigma}"]&\cdots & \mathcal{C}'_{(2N-1)n-1} \arrow[l,"L_{\tau}"] \arrow[r,"L_{\sigma}"] &\cS,
	\end{tikzcd}\]
where  \[\mathcal{C}'_{i} = \cC_{-N+1+\lfloor \frac{i}{n}\rfloor},\]
$\lfloor \frac{i}{n}\rfloor$ is the largest integer smaller or equal to $\frac{i}{n}$. In words, the sequence $\left\{\mathcal{C}'_i\right\}_{i=0}^{(2N-1)n-1}$ consists of $n$ copies of $ \cC_{-N+1}$, followed by $n$ copies of $\cC_{-N+2}$, and so on, ending with $n$ copies of $\cC_{N-1}$.
	\item When $n<0$, the truncated region is given by
	\[\bigoplus_{\substack{t\in N+\frac{1}{2}+\mathbb{Z} \\-N<t<N }}F_{-\frac{1}{2},t} \oplus \bigoplus_{\substack{(s,t)\in \bH(P)\\n<s<0 \\ -N<t<N-1}} J_{s,t} \oplus \bigoplus_{\substack{(s,t)\in \bH(P) \\ n-1<s<0 \\ -N<t<N-1}} M_{s,t},\] such that the maps $\Phi^{K}:F_{-\frac{1}{2},t} \to M_{-\frac{1}{2},t}$ and $\Phi^{-K}:F_{-\frac{1}{2},t} \to M_{-\frac{1}{2}+n,t-1}$ are both the identity map $\bI:\cS \to \cS$. An example of the truncation when $n=-1$ and $N = \frac{3}{2}$ is drawn in the following diagram, where red arrows are homotopy equivalences.
\begin{figure}[h]
		\hspace*{10mm}
	\adjustbox{scale = 0.8}{
		\begin{tikzcd}[labels={description}]
		& &    & & F_{-\frac{1}{2},-1}\arrow[d,red,"\Phi^{K}"] \\
		& & M_{-\frac{3}{2},-1} & J_{-\frac{1}{2},-1}\arrow[l,"\Phi^{-\mu}"]\arrow[r,"\Phi^{\mu}"]& M_{-\frac{1}{2},-1} \\
		& & F_{-\frac{1}{2},0} \arrow[u,red,"\Phi^{-K}"]\arrow[d,red,"\phi^{K}"]& & \\
		M_{-\frac{3}{2},0} & J_{-\frac{1}{2},0}\arrow[l,"\Phi^{-\mu}"] \arrow[r,"\Phi^{\mu}"]& M_{-\frac{1}{2},0}& &\\
		 F_{-\frac{1}{2},1}\arrow[u,red,"\Phi^{-K}"] & & & & 
	\end{tikzcd} 	\hspace*{-20mm} = 	\hspace*{-10mm} \begin{tikzcd}[labels={description}]
		& &    & & \cS\arrow[d,red,"\bI"] \\
		& &\cS & \cC_{-\frac{1}{2}}\arrow[l,"L_{\tau}"]\arrow[r,"L_{\sigma}"]& \cS \\
		& & \cS \arrow[u,red,"\bI"]\arrow[d,red,"\bI"]& & \\
		\cS & \cC_{\frac{1}{2}}\arrow[l,"L_{\tau}"] \arrow[r,"L_{\sigma}"]& \cS & &\\
		\cS\arrow[u,red,"\bI"] & & & & 
	\end{tikzcd}
}
\end{figure}

This is homotopy equivalent to 
\[ \begin{tikzcd}[labels={description}]
	\cC_{\frac{1}{2}} \arrow[r, "L_{\sigma}"] &\cS & \cC_{-\frac{1}{2}} \arrow[l,"L_{\tau}"].
\end{tikzcd} \]
	
	 In general, $\cCFK(P(U,n))$ is homotopy equivalent to the following complex:
	\[\begin{tikzcd}[labels={description}]
 \mathcal{C}''_0  \arrow[r,"L_{\sigma}"] &\cS & \mathcal{C}''_1 \arrow[l,"L_{\tau}"] \arrow[r,"L_{\sigma}"]&\cdots & \mathcal{C}''_{-(2N-1)n-1} \arrow[l,"L_{\tau}"]  ,
	\end{tikzcd}\]
where \[\cC''_i = \cC_{N-1-\lfloor \frac{i}{-n}\rfloor}. \] In words, $\left\{\mathcal{C}''_i\right\}_{i=0}^{-(2N-1)n-1}$ consists of $-n$ copies of $ \cC_{N-1}$, followed by $-n$ copies of $\cC_{N-2}$, and so on, ending with $-n$ copies of $\cC_{-N+1}$.
\end{enumerate}

\subsection{The $(2,2n+1)$-cable of the right hand trefoil}

In this example, we perform the computation of $\cCFK$ for the $(2,2n+1)$-cables of the right hand trefoil $T_{2,3}$ in detail to illustrate how the algorithm works.
	
Recall that the $(q,qn+1)$-cable of a knot $K$ is obtained when the link $L_P$ is the torus link $T(2,2q)$ and the input knot has framing $n$. The $DA$-bimodule structure of $T(2,2q)$ is described in Section \ref{sec:T(2,2q)}. Here, we redraw the special case when $q=2$.
	
	\[
	\begin{tikzcd}[labels=description, column sep=1.2cm, row sep=0cm]
		\cdots
		\ar[r, bend left, "Z|U"]
		&\xs_{-2}
		\ar[r, bend left, "Z|U"]
		\ar[l, bend left, "W|1"]
		\ar[d, "\substack{\sigma|UW^2 \\ \tau|1}"]
		& \xs_{-1}
		\ar[r, bend left, "Z|W"]
		\ar[l, bend left, "W|1"]
		\ar[d, "\substack{\sigma|W^2 \\ \tau|1}"]
		&\xs_0
		\ar[r, "Z|W", bend left]
		\ar[l, "W|Z", bend left]
		\ar[d, "\substack{\sigma|W  \\ \tau|Z}"]
		&\xs_1
		\ar[r, bend left, "Z|1"]
		\ar[l, bend left, "W|Z"]
		\ar[d, "\substack{\sigma|1 \\ \tau|Z^2}"]
		&\xs_2
		\ar[r, bend left, "Z|1"]
		\ar[l, bend left,"W|U"]
		\ar[d, "\substack{\sigma|1\\ \tau|UZ^2}"]
		& \cdots
		\ar[l, bend left,"W|U"]
		\\[2cm]
		\cdots
		\ar[r, bend left, "T|1"]
		&T^{-2}\cS
		\ar[r, bend left, "T|1"]
		\ar[l, bend left, "T^{-1}|1"]
		\ar[loop below,looseness=20, "U|U"]
		& T^{-1}\cS
		\ar[r, bend left, "T|1"]
		\ar[l, bend left, "T^{-1}|1"]
		\ar[loop below,looseness=20, "U|U"]
		& T^0\cS
		\ar[r, "T|1", bend left]
		\ar[l, "T^{-1}|1", bend left]
		\ar[loop below,looseness=20, "U|U"]
		&T^1\cS
		\ar[r, bend left, "T|1"]
		\ar[l, bend left,"T^{-1}|1"]
		\ar[loop below,looseness=20, "U|U"]
		&T^2\cS
		\ar[r, bend left, "T|1"]
		\ar[l, bend left,"T^{-1}|1"]
		\ar[loop below,looseness=20, "U|U"]
		&\cdots 
		\ar[l, "T^{-1}|1", bend left]
	\end{tikzcd}
	\]

	In the above diagram, each $\cC_{t_0} = \xs_{t_0}$, for $t_0 \in \Z+\lk(\mu,P)/2 = \mathbb{Z}$, and $\cS$ is a staircase with one generator and trivial differential.

	The maps \[L_{\sigma}: \cC_{t_0} \to \cS \text{   and   }L_{\tau}:\cC_{-t_0} \to \cS \] are homotopy equivalences for $t_0 \ge 1$, so $N=1$.
		
	The surgery complex $\cX_n(S^3,K)^{\cK}$ for the $n$-framed right-hand trefoil takes the following form, and the Seifert genus of $K$ is $g=1$.
	\begin{figure}[h]
		\begin{tikzcd}[column sep=large]
			& \ve{a} \arrow[dl,"Z"description]\arrow[dr,"W"description] & \\
			\ve{b} \arrow[dr," UT^{-2}\sigma + T^{n-2}\tau"description]& & \ve{c} \arrow[dl, "\sigma+UT^n\tau"description]\\
			& \ve{y} &
		\end{tikzcd}
	\end{figure}
	
	By the truncation procedure described in Section \ref{sec:truncation}, the non-trivial rows of interest are $t=-\frac{1}{2}$ and $t= \frac{1}{2}$. 
	
	For $t= -\frac{1}{2}$, $E_{s,t}$ is obtained by taking the first Alexander grading $s$ part of the box tensor product of $\cCFK(K)$ with  $\scE_{*,-\frac{1}{2}}$, which is 
	\[
	\begin{tikzcd}[labels=description, column sep=.5cm, row sep=0cm]
		\scE_{*,-\frac{1}{2}}:=\cdots
		&[-.7cm]W^2_{00}|\cC_{-1}
		\ar[r,bend left,out=45,in=135, "Z|U"]
		& W_{00}|\cC_{-1}
		\ar[r, bend left,out=45,in=135, "Z|U"]
		\ar[l, bend left,out=45,in=135, "W|1"]
		&[0.3cm]W^0_{00}|\cC_{-1}
		\ar[r, bend left,out=45,in=135, "Z|W"]
		\ar[l, bend left,out=45,in=135, "W|1"]
		&[0.5cm]Z^0_{00}|\cC_{0}
		\ar[r, bend left,out=45,in=135, "Z|1"]
		\ar[l, bend left,out=45,in=135, "W|Z"]
		&[0.5cm]Z_{00}|\cC_{0}
		\ar[r, bend left,out=45,in=135, "Z|1"]
		\ar[l, bend left,out=45,in=135, "W|U"]
		& [0.2cm]Z^2_{00}|\cC_{0}
		\ar[l, bend left,out=45,in=135, "W|U"]
		&[-.7cm] \cdots
	\end{tikzcd}.
	\]
Recall that $W^i_{00}|\cC_{-1}$ has the first Alexander grading $-1/2-i$, and $Z^i_{00}|\cC_{0}$ has the first Alexander grading $1/2+i$, and the Alexander grading on $ \cCFK(K)$ is such that \[A(\ve{a})=0,\,\,\,\,A(\ve{b}) = -1, \,\,\,\, A(\ve{c}) =1.\] Thus, $E_{s,-\frac{1}{2}}$ takes the form as in Figure \ref{fig:E for (2,2n+1)-cabling of RHT)}.

\begin{figure}[h]
	\adjustbox{scale = 0.8}{
\begin{tikzcd}[column sep=small]
	 & \bullet \arrow[dl,"U"description]\arrow[dr,"1"description] & & &\bullet \arrow[dl,"W"description]\arrow[dr,"1"description] & & & \bullet \arrow[dl,"1"description]\arrow[dr,"Z"description] & & & \bullet \arrow[dl,"1"description]\arrow[dr,"U"description]\\
	 \bullet & & \bullet & \bullet & & \bullet & \bullet & & \bullet & \bullet & & \bullet\\[-0.7cm]
      & [-0.7cm]s\leq -\frac{3}{2} &  &  & [-0.7cm]s=-\frac{1}{2} &  &   & [-0.7cm]s=\frac{1}{2}  &  &   &  [-0.7cm]s\ge \frac{3}{2} & \\
\end{tikzcd}
}
\caption{$E_{s,-\frac{1}{2}}$ for different values of $s$. Each $\bullet$ is a single generator.}
\label{fig:E for (2,2n+1)-cabling of RHT)}
\end{figure}
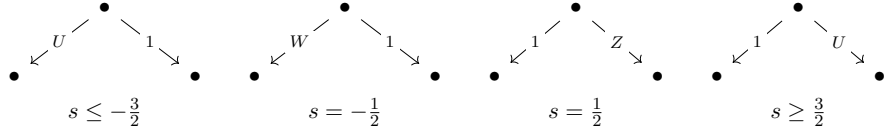

Similarly, $F_{s,-\frac{1}{2}}$ is obtained by taking the first Alexander grading $s$ part of the box tensor product of $\cCFK(K)$ with $\scF_{*,-\frac{1}{2}}$, which is  
\[
\begin{tikzcd}[labels=description, column sep=.6cm, row sep=0cm]
		\scF_{*,-\frac{1}{2}}:=\cdots
	&[-.8cm] W_{01}^2|T^{0} \cS
	\ar[r,bend left,out=45,in=135, "W|U"]
	& W_{01}^1|T^{0}\cS
	\ar[r, bend left,out=45,in=135, "Z|U"]
	\ar[l, bend left,out=45,in=135, "W|1"]
	&1_{01}|T^{0}\cS
	\ar[r, bend left,out=45,in=135, "Z|1"]
	\ar[l, bend left,out=45,in=135, "W|1"]
	&Z_{01}^1|T^{0}\cS
	\ar[r, bend left,out=45,in=135, "Z|1"]
	\ar[l, bend left,out=45,in=135, "W|U"]
	&Z_{01}^2|T^{0}\cS
	\ar[l, bend left,out=45,in=135, "W|U"]
	&[-.8cm]\cdots ,
\end{tikzcd}
\]

The first component of the Alexander grading of $W_{01}^i|T^{0}\cS$ is $-\frac{1}{2}-i$, while the first component of the Alexander grading of $Z_{01}^i|T^{0}\cS$ is $-\frac{1}{2}+i$, for $i\ge 0$. Therefore, $F_{s,-\frac{1}{2}} = F_{s}$ takes the form as in Figure \ref{fig:F for (2,2n+1)-cabling of RHT)}. (By the definition, $F_{s,t} = F_{s}$, which is independent of the $t$-coordinate.)
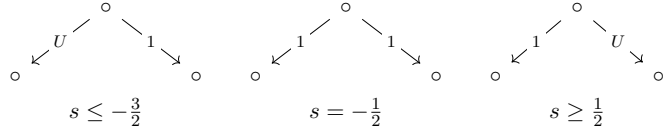
\begin{figure}[h]
	\adjustbox{scale = 0.8}{
		\begin{tikzcd}[column sep=small]
			& \circ \arrow[dl,"U"description]\arrow[dr,"1"description] & & &\circ \arrow[dl,"1"description]\arrow[dr,"1"description] & & & \circ \arrow[dl,"1"description]\arrow[dr,"U"description]&\\
		\circ & & \circ & \circ & & \circ & \circ & & \circ\\[-0.7cm]
			& [-0.7cm]s\leq -\frac{3}{2} &  &  & [-0.7cm]s=-\frac{1}{2} &  &   & [-0.7cm]s\ge\frac{1}{2}  & \\
		\end{tikzcd}
	}
	\caption{$F_{s}$ for different values of $s$. Each $\circ$ is a single generator.}
	\label{fig:F for (2,2n+1)-cabling of RHT)}
\end{figure}

The rightward horizontal map $\Phi^{\mu}: E_{s,-\frac{1}{2}}\to F_{s,-\frac{1}{2}}$ is given by tensoring $\cX_n(S^3,K)^{\cK}$ with the mapping cone of $f^\mu$. We note that in this case $f^\mu$ acts by multiplication by $W^2$ from $\cC_{-1}$ to $\cS$, and multiplication by $W$ from $\cC_{0}$ to $\cS$.  Note also that $h_{Z,W}, h_{W,Z}, h_{\sigma,Z}$ and $h_{\sigma,W}$ all vanish in this case. The map $f^\mu$ is shown below
\[
 \begin{tikzcd}[labels=description, column sep=0.9cm, row sep=1.7cm]
	\scE_{*,-\frac{1}{2}} \ar[d, "f^{\mu}="]:\,\,\cdots
	&[-1cm]\cC_{-1}
	\ar[r, bend left,out=45,in=135, "Z|U"]
	\ar[d,"W^2",  pos=.4]
	& \cC_{-1}
	\ar[r, bend left,out=45,in=135, "Z|U"]
	\ar[l, bend left,out=45,in=135, "W|1"]
	\ar[d,"W^2",  pos=.4]
	&\cC_{-1}
	\ar[r, bend left,out=45,in=135, "Z|W"]
	\ar[l, bend left,out=45,in=135, "W|1"]
	\ar[d,"W^2",  pos=.4]
	& \cC_{0}
	\ar[r, bend left,out=45,in=135, "Z|1"]
	\ar[l, bend left=15,out=45,in=135, "W|Z"]
	\ar[d,"W", pos=.4]
	& \cC_{0}
	\ar[r, bend left,out=45,in=135, "Z|1"]
	\ar[l, bend left,out=45,in=135, "W|U"]
	\ar[d,"W",  pos=.4]
	& \cC_{0}
	\ar[l, bend left,out=45,in=135, "W|U"]
	\ar[d,"W", pos=.4]
	&[-.8cm] \cdots
	\\
	\scF_{*,-\frac{1}{2}}:\,\,\cdots
	&[-1cm] \cS
	\ar[r, bend left,out=45,in=135, "Z|U"]
	& \cS
	\ar[r, bend left,out=45,in=135, "Z|U"]
	\ar[l, bend left,out=45,in=135, "W|1"]
	&\cS
	\ar[r, bend left,out=45,in=135, "Z|1"]
	\ar[l, bend left,out=45,in=135, "W|1"]
	&\cS
	\ar[r, bend left,out=45,in=135, "Z|1"]
	\ar[l, bend left,out=45,in=135, "W|U"]
	&\cS
	\ar[r, bend left,out=45,in=135, "Z|1"]
	\ar[l, bend left,out=45,in=135, "W|U"]
	&\cS
	\ar[l, bend left,out=45,in=135, "W|U"]
	&\cdots 
\end{tikzcd}
\]

Similarly, the leftward horizontal map $\Phi^{-\mu}: E_{s,-\frac{1}{2}}\to F_{s-1,-\frac{1}{2}}$ is giving by tensoring $\cX_n(S^3,K)^{\cK}$ with the mapping cone of $L_\tau$, which is multiplication by $1$ from $\cC_{-1}$ to $\cS$, and multiplication by $Z$ from $\cC_{0}$ to $\cS$. 

From the truncation procedure, the region that will appear in the truncated complex is 
\[\begin{tikzcd}[column sep=1.3cm, labels=description]
	F_{-\frac{3}{2},-\frac{1}{2}} & E_{-\frac{1}{2},-\frac{1}{2}} \arrow[l,"\Phi^{-\mu}"] \arrow[r,"\Phi^{\mu}"] & F_{-\frac{1}{2},-\frac{1}{2}} & E_{\frac{1}{2},-\frac{1}{2}} \arrow[l,"\Phi^{-\mu}"]
\end{tikzcd},
\]
which is \[
\adjustbox{scale = 0.9}{
\begin{tikzcd}[column sep =small, labels=description]
	 & \circ\arrow[dl,"U"] \arrow[dr, "1"]& & & \bullet \arrow[dl,"W"]\arrow[dr,"1"] \arrow[lll, bend right,"1"]\arrow[rrr,bend left, "W^2"]& & & \circ \arrow[dl,"1"]\arrow[dr,"1"]& & & \bullet \arrow[dl,"1"]\arrow[dr,"Z"]\arrow[lll, bend right,"Z"]& \\
	 \circ & &\circ & \bullet \arrow[lll,bend left=60,pos=0.55,"Z"] \arrow[rrr,bend right=60,pos = 0.55, "W"] & &\bullet \arrow[lll,bend left=65, pos=0.55, "1"] \arrow[rrr,bend right=60,pos =0.4,"W^2"]& \circ & &\circ & \bullet \arrow[lll,bend left=65, pos=0.45,"Z"]& &\bullet \arrow[lll,bend left=60,pos=0.45,"1"]\\
\end{tikzcd}.
}
\]

Similarly, when $t=\frac{1}{2}$, the region that will appear in the truncated complex is
\[\begin{tikzcd}[column sep=1.3cm, labels=description]
 E_{-\frac{1}{2},\frac{1}{2}}  \arrow[r,"\Phi^{\mu}"] & F_{-\frac{1}{2},\frac{1}{2}} & E_{\frac{1}{2},\frac{1}{2}} \arrow[l,"\Phi^{-\mu}"]
 \arrow[r,"\Phi^{\mu}"] &
 F_{\frac{1}{2},\frac{1}{2}} 
\end{tikzcd}.
\] 
Using $\cC_{0}$ and $\cC_{1}$ instead of $\cC_{-1}$ and $\cC_{0}$, one can compute that it equals to the following:

\[
\adjustbox{scale = 0.9}{
	\begin{tikzcd}[column sep =small, labels=description]
        \bullet \arrow[rrr,bend left=60,pos = 0.4, "1"] & &\bullet \arrow[rrr,bend left=60,pos =0.4,"W"]& \circ & &\circ & \bullet \arrow[lll,bend right=60, pos=0.4,"Z^2"] \arrow[rrr, bend left =60, pos = 0.6, "1"]& &\bullet \arrow[lll,bend right=60,pos=0.6,"Z"] \arrow[rrr,bend left=60,pos=0.6, "W"] &\circ & & \circ\\
        & \bullet \arrow[ul,"W"]\arrow[ur,"1"] \arrow[rrr,bend right, "W"]& & & \circ \arrow[ul,"1"]\arrow[ur,"1"]& & & \bullet \arrow[ul,"1"]\arrow[ur,"Z"]\arrow[lll, bend left,"Z^2"]\arrow[rrr, bend right, "1"] & & & \circ \arrow[ul,"1"]\arrow[ur,"U"]& \\
	\end{tikzcd}.
}
\]

The complexes $J_{s,t}$ and $M_{s,t}$ are simpler. $J_{s,t} = J_{t}$ is obtained by the tensor $B(K) \boxtimes \cC_{t+\frac{1}{2}}$. For knots $K$ in $S^3$, $B(K)$ has a single generator. In this example, $\cC_{t+\frac{1}{2}}$ also has a single generator for all $t$, so $J_{s,t}$ has a single generator and we will denote it by $\star$. $M_{s,t} = M$ is obtained by the tensor product $B(K) \boxtimes \cS$, which again has a single generator, and we will denote it by $\ast$. From the truncation, we only need $J_{s,t}$ and $M_{s,t}$ for $t= -\frac{1}{2}$.

The map $\Phi^{\mu}: J_{s,-\frac{1}{2}} \to M_{s,-\frac{1}{2}}$ is given by $\bI \boxtimes L_{\sigma}$, where $L_{\sigma}: \cC_{0} \to\cS$ is multiplication by $W$. Symmetrically, the map $\Phi^{-\mu}:J_{s,-\frac{1}{2}}\to M_{s-1,-\frac{1}{2}}$ is given by $\bI \boxtimes L_{\tau}$, where $L_{\tau}: \cC_{0} \to\cS$ is multiplication by $Z$.

The map $\Phi^{K}: E_{s,-\frac{1}{2}} \to J_{s,-\frac{1}{2}}$ is schematically described by the diagram
		\[\Phi^{K}=
		\begin{tikzcd}[row sep=.3cm]
			\cX_n(K)\cdot \ve{I}_0\ar[d]& [-1cm]\boxtimes &[-1cm]\scE_{*,-\frac{1}{2}} \ar[dd] &[-1.2cm]\\
			\delta_1^1\ar[drr, "\sigma",labels=description] \ar[dd]
			\\
			& & (f^K)_2^1\ar[d]\ar[dr] &\\
			\cX_n(K)\cdot \ve{I}_1& [-1cm]\boxtimes & [-1cm]\scJ_{*,-\frac{1}{2}} & \otimes\,\,\,\bF[W,Z]
		\end{tikzcd},\]
 where $\cX_n(K)\cdot \ve{I}_0$ is $\cCFK(K)$, $\cX_n(K)\cdot \ve{I}_1$ has a single generator, and the action $(f^K)^1_2(\sigma,-)$ on $\scE_{*,-\frac{1}{2}}$ is given by 
\[
\begin{tikzcd}[labels=description, {column sep=2cm,between origins}, row sep=.4cm]
	\scE_{*,-\frac{1}{2}} \ar[d, "f^K"]:\,\,\cdots
	&[-0.7cm]\cC_{-1}
	\ar[r,bend left,out=45,in=135, "Z|U"]
	\ar[d,"\sigma|U^2W"]
	& \cC_{-1}
	\ar[r,bend left,out=45,in=135,  "Z|U"]
	\ar[l,bend left,out=45,in=135, "W|1"]
	\ar[d,"\sigma| UW"]
	&\cC_{-1}
	\ar[r,bend left,out=45,in=135,  "Z|W"]
	\ar[l,bend left,out=45,in=135,  "W|1"]
	\ar[d,"\sigma|W"]
	& \cC_{0}
	\ar[r,bend left,out=45,in=135,  "Z|1"]
	\ar[l,bend left,out=45,in=135, "W|Z"]
	\ar[d,"\sigma|1"]
	& \cC_{0}
	\ar[r,bend left,out=45,in=135,  "Z|1"]
	\ar[l,bend left,out=45,in=135,  "W|U"]
	\ar[d,"\sigma|1"]
	& \cC_{0}
	\ar[l,bend left,out=45,in=135,  "W|U"]
	\ar[d,"\sigma|1"]
	&[-1.2cm] \cdots
	\\[1cm]
	\scJ_{*,-\frac{1}{2}}:\,\,\cdots
	& \cC_{0}
	\ar[r,bend left,out=45,in=135, "T|1"]
	\ar[loop below,looseness=20, "U|U"]
	& \cC_{0}
	\ar[r,bend left,out=45,in=135,  "T|1"]
	\ar[l,bend left,out=45,in=135, "T^{-1}|1"]
	\ar[loop below,looseness=20, "U|U"]
	&\cC_{0}
	\ar[r,bend left,out=45,in=135, "T|1"]
	\ar[l,bend left,out=45,in=135,  "T^{-1}|1"]
	\ar[loop below,looseness=20, "U|U"]
	&\cC_{0}
	\ar[r,bend left,out=45,in=135, "T|1"]
	\ar[l,bend left,out=45,in=135,  "T^{-1}|1"]
	\ar[loop below,looseness=20, "U|U"]
	&\cC_{0}
	\ar[r,bend left,out=45,in=135,  "T|1"]
	\ar[l,bend left,out=45,in=135,  "T^{-1}|1"]
	\ar[loop below,looseness=20, "U|U"]
	&\cC_{0}
	\ar[l,bend left,out=45,in=135,  "T^{-1}|1"]
	\ar[loop below,looseness=20, "U|U"]	
	&\cdots 
\end{tikzcd}.
\]
Note that $h_{Z,W}$ is trivial in this example. Therefore, $\Phi^K:E_{s,-\frac{1}{2}}\to J_{s,-\frac{1}{2}}$ is the following:

\begin{figure}[h]
	\adjustbox{scale = 0.8}{
		\begin{tikzcd}[labels=description,column sep=small]
			& \bullet \arrow[dl,"U"]\arrow[dr,"1"] & & &\bullet \arrow[dl,"W"]\arrow[dr,"1"] & & & \bullet \arrow[dl,"1"]\arrow[dr,"Z"] & & & \bullet \arrow[dl,"1"]\arrow[dr,"U"]\\
			\bullet \arrow[dr,"U^{-s-1/2}W"]& & \bullet \arrow[dl,"U^{-s+1/2}W"]& \bullet \arrow[dr,"U"]& & \bullet \arrow[dl,"UW"]& \bullet \arrow[dr,"U"]& & \bullet \arrow[dl,"W"] & \bullet \arrow[dr,"U"]& & \bullet\arrow[dl,"1"]\\[0.5cm]
			& \star & & & \star & & & \star & & & \star &\\[-0.7cm]
			& [-0.7cm]s\leq -\frac{3}{2} &  &  & [-0.7cm]s=-\frac{1}{2} &  &   & [-0.7cm]s=\frac{1}{2}  &  &   &  [-0.7cm]s\ge \frac{3}{2} & \\
		\end{tikzcd}
	}.
	\caption{$\Phi^K:E_{s,-\frac{1}{2}}\to J_{s,-\frac{1}{2}}$ for different values of $s$. }
\end{figure}
The map $\Phi^K:F_{s,-\frac{1}{2}} \to M_{s,-\frac{1}{2}}$ is obtained similarly by replacing $\delta^1_2(\sigma,-):\scE_{*,-\frac{1}{2}}\to \scJ_{*,-\frac{1}{2}}$ with $\delta^1_2(\sigma,-):\scF_{*,-\frac{1}{2}}\to \scM_{*,-\frac{1}{2}}$, which is
\[
\begin{tikzcd}[labels=description, column sep=1.2cm, row sep=0cm]
	\scF_{*,-\frac{1}{2}}:\,\,\cdots
	&[-1.2cm] \cS
	\ar[r,bend left,out=45,in=135, "W|U"]
	\ar[d, "\sigma|U^2"]
	& \cS
	\ar[r,out=45,in=135, "Z|U"]
	\ar[l, bend left,out=45,in=135, "W|1"]
	\ar[d, "\sigma|U"]
	&\cS
	\ar[r, bend left,out=45,in=135, "Z|1"]
	\ar[l, bend left,out=45,in=135, "W|1"]
	\ar[d, "\sigma|1"]
	&\cS
	\ar[r, bend left,out=45,in=135, "Z|1"]
	\ar[l, bend left,out=45,in=135, "W|U"]
	\ar[d, "\sigma|1"]
	&\cS
	\ar[l, bend left,out=45,in=135, "W|U"]
	\ar[d, "\sigma|1"]
	&[-1.2cm]\cdots
	\\[1.5cm]
	\scM_{*,-\frac{1}{2}}:\,\,\cdots&[-1.3cm] \cS
	\ar[r, bend left,out=45,in=135, "T|1"]
	\ar[loop below,looseness=20, "U|U"]
	& \cS
	\ar[r, bend left,out=45,in=135, "T|1"]
	\ar[l, bend left,out=45,in=135,  "T^{-1}|1"]
	\ar[loop below,looseness=20, "U|U"]
	&\cS
	\ar[r, bend left,out=45,in=135, "T|1"]
	\ar[l, bend left,out=45,in=135, "T^{-1}|1"]
	\ar[loop below,looseness=20, "U|U"]
	&\cS
	\ar[r, bend left,out=45,in=135, "T|1"]
	\ar[l, bend left,out=45,in=135, "T^{-1}|1"]
	\ar[loop below,looseness=20, "U|U"]
	&\cS
	\ar[l, bend left,out=45,in=135, "T^{-1}|1"]
	\ar[loop below,looseness=20, "U|U"]
	&\cdots 
\end{tikzcd}.
\]
So $\Phi^{K}:F_{s,-\frac{1}{2}}\to M_{s,-\frac{1}{2}}$ is the following:

\begin{figure}[h]
	\adjustbox{scale = 0.8}{
		\begin{tikzcd}[labels = description,column sep=small]
			& \circ \arrow[dl,"U"]\arrow[dr,"1"] & & &\circ \arrow[dl,"1"]\arrow[dr,"1"] & & & \circ \arrow[dl,"1"]\arrow[dr,"U"]&\\
			\circ \arrow[dr,"U^{-s-1/2}"]& & \circ \arrow[dl,"U^{-s+1/2}"]& \circ \arrow[dr,"U"]& & \circ \arrow[dl,"U"]& \circ \arrow[dr,"U"]& & \circ\arrow[dl,"1"]\\[0.5cm]
			& \ast & & & \ast & & & \ast\\[-0.7cm]
			& [-0.7cm]s\leq -\frac{3}{2} &  &  & [-0.7cm]s=-\frac{1}{2} &  &   & [-0.7cm]s\ge\frac{1}{2}  & \\
		\end{tikzcd}
	}.
	\caption{$\Phi^{-K}: F_{s,\frac{1}{2}}\to M_{s+n,-\frac{1}{2}}$ for different values of $s$.}
\end{figure}

The map $\Phi^{-K}: E_{s,\frac{1}{2}} \to J_{s+n,-\frac{1}{2}}$ is obtained by tensoring the identity map on $\cX_n(K)$ with the morphism $f^{-K}$. In our present case, this is encoded schematically by the following diagram:
\[
\Phi^{-K}=
\begin{tikzcd}[row sep=.3cm]
	\cX_n(K)\cdot \ve{I}_0\ar[d]& [-1cm]\boxtimes &[-1cm]\scE_{\frac{1}{2}} \ar[dd] &[-1.2cm]\\
	\delta_1^1\ar[drr, "\tau",labels=description] \ar[dd]
	\\
	& & (f^{-K})_2^1\ar[d]\ar[dr] &\\
	\cX_n(K)\cdot \ve{I}_1& [-1cm]\boxtimes & [-1cm]\scJ_{-\frac{1}{2}} & \otimes\,\,\,\bF[W,Z]
\end{tikzcd}
\]
We recall the map $f^{-K}$ is given by the diagram 
\[\begin{tikzcd}[labels=description, {column sep=2cm,between origins}, row sep=.4cm]
	\scE_{*,\frac{1}{2}} \ar[d, "f^{-K}"]:\,\,\cdots
	&[-0.7cm]\cC_{0}
	\ar[r,bend left,out=45,in=135, "Z|U"]
	\ar[d,"\tau|1"]
	& \cC_{0}
	\ar[r,bend left,out=45,in=135,  "Z|U"]
	\ar[l,bend left,out=45,in=135, "W|1"]
	\ar[d,"\tau| 1"]
	&\cC_{0}
	\ar[r,bend left,out=45,in=135,  "Z|W"]
	\ar[l,bend left,out=45,in=135,  "W|1"]
	\ar[d,"\tau|1"]
	& \cC_{1}
	\ar[r,bend left,out=45,in=135,  "Z|1"]
	\ar[l,bend left,out=45,in=135, "W|Z"]
	\ar[d,"\tau|Z"]
	& \cC_{1}
	\ar[r,bend left,out=45,in=135,  "Z|1"]
	\ar[l,bend left,out=45,in=135,  "W|U"]
	\ar[d,"\tau|UZ"]
	& \cC_{1}
	\ar[l,bend left,out=45,in=135,  "W|U"]
	\ar[d,"\tau|U^2Z"]
	&[-1.2cm] \cdots
	\\[1cm]
	\scJ_{*,-\frac{1}{2}}:\,\,\cdots
	& \cC_{0}
	\ar[r,bend left,out=45,in=135, "T|1"]
	\ar[loop below,looseness=20, "U|U"]
	& \cC_{0}
	\ar[r,bend left,out=45,in=135,  "T|1"]
	\ar[l,bend left,out=45,in=135, "T^{-1}|1"]
	\ar[loop below,looseness=20, "U|U"]
	&\cC_{0}
	\ar[r,bend left,out=45,in=135, "T|1"]
	\ar[l,bend left,out=45,in=135,  "T^{-1}|1"]
	\ar[loop below,looseness=20, "U|U"]
	&\cC_{0}
	\ar[r,bend left,out=45,in=135, "T|1"]
	\ar[l,bend left,out=45,in=135,  "T^{-1}|1"]
	\ar[loop below,looseness=20, "U|U"]
	&\cC_{0}
	\ar[r,bend left,out=45,in=135,  "T|1"]
	\ar[l,bend left,out=45,in=135,  "T^{-1}|1"]
	\ar[loop below,looseness=20, "U|U"]
	&\cC_{0}
	\ar[l,bend left,out=45,in=135,  "T^{-1}|1"]
	\ar[loop below,looseness=20, "U|U"]	
	&\cdots 
\end{tikzcd}.
\]

So $\Phi^{-K}: E_{s,\frac{1}{2}}\to J_{s+n,-\frac{1}{2}}$ is the following:
\begin{figure}[h]
	\adjustbox{scale = 0.8}{
		\begin{tikzcd}[labels=description,column sep=small]
			& \star & & & \star & & & \star & & & \star &\\[0.6cm]			
			\bullet \arrow[ur,"1"]& & \bullet \arrow[ul,"U"]& \bullet \arrow[ur,"Z"]& & \bullet \arrow[ul,"U"]& \bullet \arrow[ur,"UZ"]& & \bullet \arrow[ul,"U"] & \bullet \arrow[ur,pos=0.4,"U^{s+1/2}Z"]& & \bullet\arrow[ul,pos=0.4,"U^{s-1/2}Z"]\\
			& \bullet \arrow[ul,"U"]\arrow[ur,"1"] & & &\bullet \arrow[ul,"W"]\arrow[ur,"1"] & & & \bullet \arrow[ul,"1"]\arrow[ur,"Z"] & & & \bullet \arrow[ul,"1"]\arrow[ur,"U"]\\[-0.7cm]
			& [-0.7cm]s\leq -\frac{3}{2} &  &  & [-0.7cm]s=-\frac{1}{2} &  &   & [-0.7cm]s=\frac{1}{2}  &  &   &  [-0.7cm]s\ge \frac{3}{2} & \\
		\end{tikzcd}
	}.
	\caption{$\Phi^{-K}:E_{s,\frac{1}{2}}\to J_{s+n,-\frac{1}{2}}$ for different values of $s$. }
\end{figure}

The maps $\Phi^{-K}:F_{s,\frac{1}{2}} \to M_{s+n,-\frac{1}{2}}$ is obtained by using $\delta_2^1(\tau,-):\scF_{*,\frac{1}{2}}\to \scM_{*,-\frac{1}{2}}$, which is 
\[
\begin{tikzcd}[labels=description, column sep=1cm, row sep=0cm]
	\scF_{\frac{1}{2}}: \cdots
	&[-.8cm] \cS
	\ar[r,bend left,out=45,in=135, "W|U"]
	\ar[d, "\tau|1"]
	& \cS
	\ar[r,out=45,in=135, "Z|U"]
	\ar[l, bend left,out=45,in=135, "W|1"]
	\ar[d, "\tau|1"]
	&\cS
	\ar[r, bend left,out=45,in=135, "Z|1"]
	\ar[l, bend left,out=45,in=135, "W|1"]
	\ar[d, "\tau|1"]
	&\cS
	\ar[r, bend left,out=45,in=135, "Z|1"]
	\ar[l, bend left,out=45,in=135, "W|U"]
	\ar[d, "\tau|U"]
	&\cS
	\ar[l, bend left,out=45,in=135, "W|U"]
	\ar[d, "\tau|U^2"]
	&[-.8cm]\cdots
	\\[1.5cm]
	\scM_{-\frac{1}{2}}: \cdots&[-1.3cm] \cS
	\ar[r, bend left,out=45,in=135, "T|1"]
	\ar[loop below,looseness=20, "U|U"]
	& \cS
	\ar[r, bend left,out=45,in=135, "T|1"]
	\ar[l, bend left,out=45,in=135,  "T^{-1}|1"]
	\ar[loop below,looseness=20, "U|U"]
	&\cS
	\ar[r, bend left,out=45,in=135, "T|1"]
	\ar[l, bend left,out=45,in=135, "T^{-1}|1"]
	\ar[loop below,looseness=20, "U|U"]
	&\cS
	\ar[r, bend left,out=45,in=135, "T|1"]
	\ar[l, bend left,out=45,in=135, "T^{-1}|1"]
	\ar[loop below,looseness=20, "U|U"]
	&\cS
	\ar[l, bend left,out=45,in=135, "T^{-1}|1"]
	\ar[loop below,looseness=20, "U|U"]
	&[-1.3cm]\cdots 
\end{tikzcd}
\]
So $\Phi^{-K}:F_{s,\frac{1}{2}}\to M_{s+n,-\frac{1}{2}}$ is the following:
\begin{figure}[h]
	\adjustbox{scale = 0.8}{
		\begin{tikzcd}[labels = description,column sep=small]
			& \ast & & & \ast & & & \ast\\[0.5cm]
			\circ \arrow[ur,"1"]& & \circ \arrow[ul,"U"]& \circ \arrow[ur,"U"]& & \circ \arrow[ul,"U"]& \circ \arrow[ur,pos=0.4,"U^{s+3/2}"]& & \circ\arrow[ul,pos=0.4,"U^{s+1/2}"]\\
			& \circ \arrow[ul,"U"]\arrow[ur,"1"] & & &\circ \arrow[ul,"1"]\arrow[ur,"1"] & & & \circ \arrow[ul,"1"]\arrow[ur,"U"]&\\[-0.7cm]
			& [-0.7cm]s\leq -\frac{3}{2} &  &  & [-0.7cm]s=-\frac{1}{2} &  &   & [-0.7cm]s\ge\frac{1}{2}  & \\
		\end{tikzcd}
	}.
	\caption{$\Phi^K: F_{s,-\frac{1}{2}}\to M_{s,-\frac{1}{2}}$ for different values of $s$.}
\end{figure}

There is no length $2$ map in this example, since the actions of $\sigma$ and $\tau$ commute with the actions of $Z$ and $W$.

Putting everything together, we obtain the truncated complex for the $(2,2n+1)$-cable of the right hand trefoil as follows.

\begin{enumerate}
	\item When $n=-1$, the truncated complex for the (2,-1)-cable of the right hand trefoil takes the following form:
	\[
	\begin{tikzcd}[labels=description]
	 & 	F_{-\frac{3}{2},-\frac{1}{2}} \arrow[d,"\Phi^{K}"] &  E_{-\frac{1}{2},-\frac{1}{2}} \arrow[l,"\Phi^{-\mu}"] \arrow[r,"\Phi^{\mu}"] \arrow[d,"\Phi^K"]& F_{-\frac{1}{2},-\frac{1}{2}} \arrow[d,"\Phi^K"]& E_{\frac{1}{2},-\frac{1}{2}} \arrow[l,"\Phi^{-\mu}"] \arrow[d,"\Phi^K"] \\
		J_{-\frac{3}{2},-\frac{1}{2}} \arrow[r,"\Phi^{\mu}"]& M_{-\frac{3}{2},-\frac{1}{2}}& J_{-\frac{1}{2},-\frac{1}{2}} \arrow[r,"\Phi^{\mu}"] \arrow[l,"\Phi^{-\mu}"] & M_{-\frac{1}{2},-\frac{1}{2}} & J_{\frac{1}{2},-\frac{1}{2}}\arrow[l,"\Phi^{\mu}"]& \\
       E_{-\frac{1}{2},\frac{1}{2}} \arrow[r,"\Phi^{\mu}"] \arrow[u,"\Phi^{-K}"] & F_{-\frac{1}{2},\frac{1}{2}} \arrow[u,"\Phi^{-K}"] & E_{\frac{1}{2},\frac{1}{2}} \arrow[r,"\Phi^{\mu}"] \arrow[u,"\Phi^{-K}"]\arrow[l,"\Phi^{-\mu}"] & F_{\frac{1}{2},\frac{1}{2}} \arrow[u,"\Phi^{-K}"] &
	\end{tikzcd},\]
where we quotient out the contractible subcomplex $F_{-\frac{3}{2},\frac{1}{2}}\xrightarrow{\Phi^{-K}}M_{-\frac{5}{2},-\frac{1}{2}}$ and $F_{\frac{1}{2},-\frac{1}{2}}\xrightarrow{\Phi^{K}}M_{\frac{1}{2},-\frac{1}{2}}$ from the result in Proposition \ref{prop:truncation}; see Figure \ref{fig:simplified truncation when n=-1}.

The full complex is:
\[
\adjustbox{scale=0.9}{
\begin{tikzcd}[column sep =small, labels=description]
	& & & & \circ\arrow[dl,"U"] \arrow[dr, "1"]& & & \bullet \arrow[dl,"W"]\arrow[dr,"1"] \arrow[lll, bend right,"1"] \arrow[rrr,bend left, "W^2"]& & & \circ \arrow[dl,"1"]\arrow[dr,"1"]& & & \bullet \arrow[dl,"1"]\arrow[dr,"Z"]\arrow[lll, bend right,"Z"]& \\
	& & & \circ \arrow[dr,bend right=30,pos =0.6, "U"]& &\circ \arrow[dl,bend left=30,pos =0.6, "U^2"]& \bullet \arrow[lll,bend left=60,pos=0.6,"Z"] \arrow[dr,bend right =30, pos=0.6, "U"] \arrow[rrr,bend right=60,pos = 0.55, "W"] & &\bullet \arrow[lll,bend left=60, pos=0.5, "1"] \arrow[dl,bend left =30, pos=0.6, "UW"]\arrow[rrr,bend right=60,pos =0.45,"W^2"]& \circ \arrow[dr,bend right =30, pos=0.6,"U"]& &\circ \arrow[dl,bend left =30, pos=0.6,"U"]& \bullet \arrow[lll,bend left=60, pos=0.45,"Z"]\arrow[dr, bend right =30, pos=0.6, "U"]& &\bullet \arrow[lll,bend left=60,pos=0.4,"1"] \arrow[dl,bend left =30, pos=0.6,"W"]\\[1.5cm]
	& \star \arrow[rrr,"W"] & & & \ast & & & \star \arrow[lll,"Z"] \arrow[rrr,"W"] & & & \ast & & & \star \arrow[lll,"Z"] &\\[1.5cm]
     \bullet \arrow[rrr,bend left=60,pos = 0.4, "1"] \arrow[ur,bend left=30, pos=0.6,"Z"]& &\bullet \arrow[rrr,bend left=60,pos =0.45,"W"] \arrow[ul,bend right=30,pos=0.6, "U"]& \circ \arrow[ur,bend left=30,pos =0.6, "U"] & &\circ \arrow[ul,bend right=30,pos=0.6,"U"] & \bullet \arrow[lll,bend right=60, pos=0.44,"Z^2"] \arrow[rrr, bend left =60, pos = 0.55, "1"] \arrow[ur,bend left=30, pos =0.6, "UZ"]& &\bullet \arrow[lll,bend right=60,pos=0.53,"Z"] \arrow[rrr,bend left=60,pos=0.6, "W"] \arrow[ul,bend right=30,pos =0.6, "U"] &\circ \arrow[ur,bend left =30,pos =0.6, "U^2"]& & \circ \arrow[ul,bend right =30,pos =0.6, "U"]& & &\\
     & \bullet \arrow[ul,"W"]\arrow[ur,"1"] \arrow[rrr,bend right, "W"]& & & \circ \arrow[ul,"1"]\arrow[ur,"1"]& & & \bullet \arrow[ul,"1"]\arrow[ur,"Z"]\arrow[lll, bend left,"Z^2"]\arrow[rrr, bend right, "1"] & & & \circ \arrow[ul,"1"]\arrow[ur,"U"]& & & &\\
\end{tikzcd}.
}
\]

Reducing arrows labeled by $1$, it simplifies to 
\[
\begin{tikzcd}[row sep = 1cm, column sep =1cm,labels=description]
	& \circ \arrow[d,"U"]& \bullet \arrow[d,"U"]\arrow[l,"Z"]\arrow[drr,bend left, "W^2"] & &\\
	\star \arrow[r,"W"] & \ast & \star \arrow[l,"Z"]\arrow[r,"W"] & \ast & \star \arrow[l,"Z"]\\
	& & \bullet \arrow[ull,bend left, "Z^2"] \arrow[u,"U"]\arrow[r,"W"]& \circ \arrow[u,"U"]& 
\end{tikzcd} \hspace*{2mm} = \hspace*{2mm}\begin{tikzcd}
[row sep =  0.1mm, column sep =0.1mm, labels={description}]
\star \arrow[dd,shorten=-1mm,"Z"]&[1.2cm]\phantom{\circ} &[1.2cm]\bullet \arrow[ll,shorten=-1mm, "W^2"] \arrow[ddl,shorten=-1mm,pos =0.6,"U"] \arrow[dd,shorten=-1mm,"Z"] &[-2mm]\phantom{\bullet}\\[-2mm]
& \circ \arrow[dl,shorten=-1mm,"U"]& &\bullet\arrow[ll,shorten=-1mm,"W"] \arrow[dll,shorten=-1mm,pos = 0.52,"U"] \arrow[dd,shorten=-1mm,"Z^2"]\\[1.2cm]
\ast & \star \arrow[l,shorten=-1mm,"W"] \arrow[d,shorten=-1mm,"Z"]& \circ \arrow[dl,shorten=-1mm,"U",pos = 0.3]& \\[1.2cm]
& \ast & & \star\arrow[ll,shorten=-1mm,"W"]
\end{tikzcd}\hspace*{2mm}.\]
	\item When $n=0$,
	the truncated complex for the $(2,1)$-cable of the right hand trefoil takes the following form:\[
	\begin{tikzcd}[labels=description]
		F_{-\frac{3}{2},-\frac{1}{2}} &  E_{-\frac{1}{2},-\frac{1}{2}} \arrow[l,"\Phi^{-\mu}"] \arrow[r,"\Phi^{\mu}"] \arrow[d,"\Phi^K"]& F_{-\frac{1}{2},-\frac{1}{2}} \arrow[d,"\Phi^K"]& E_{\frac{1}{2},-\frac{1}{2}} \arrow[l,"\Phi^{-\mu}"] \arrow[d,"\Phi^K"]& \\
		& J_{-\frac{1}{2},-\frac{1}{2}} \arrow[r,"\Phi^{\mu}"] & M_{-\frac{1}{2},-\frac{1}{2}} & J_{\frac{1}{2},-\frac{1}{2}}\arrow[l,"\Phi^{\mu}"]& \\
    	&	E_{-\frac{1}{2},\frac{1}{2}} \arrow[r,"\Phi^{\mu}"] \arrow[u,"\Phi^{-K}"] & F_{-\frac{1}{2},\frac{1}{2}} \arrow[u,"\Phi^{-K}"] & E_{\frac{1}{2},\frac{1}{2}} \arrow[r,"\Phi^{\mu}"] \arrow[u,"\Phi^{-K}"]\arrow[l,"\Phi^{-\mu}"] & F_{\frac{1}{2},\frac{1}{2}}
	\end{tikzcd},\]
which is 
\[
\adjustbox{scale = 0.9}{
	\begin{tikzcd}[column sep =small, labels=description]
		& \circ\arrow[dl,"U"] \arrow[dr, "1"]& & & \bullet \arrow[dl,"W"]\arrow[dr,"1"] \arrow[lll, bend right,"1"] \arrow[rrr,bend left, "W^2"]& & & \circ \arrow[dl,"1"]\arrow[dr,"1"]& & & \bullet \arrow[dl,"1"]\arrow[dr,"Z"]\arrow[lll, bend right,"Z"]& & & & &\\
		\circ & &\circ & \bullet \arrow[lll,bend left=60,pos=0.5,"Z"] \arrow[dr,bend right =30, pos=0.6, "U"] \arrow[rrr,bend right=60,pos = 0.55, "W"] & &\bullet \arrow[lll,bend left=60, pos=0.5, "1"] \arrow[dl,bend left =30, pos=0.6, "UW"]\arrow[rrr,bend right=60,pos =0.45,"W^2"]& \circ \arrow[dr,bend right =30, pos=0.6,"U"]& &\circ \arrow[dl,bend left =30, pos=0.6,"U"]& \bullet \arrow[lll,bend left=60, pos=0.45,"Z"]\arrow[dr, bend right =30, pos=0.6, "U"]& &\bullet \arrow[lll,bend left=60,pos=0.4,"1"] \arrow[dl,bend left =30, pos=0.6,"W"]& & & &\\[1.5cm]
		& & & & \star \arrow[rrr,"W"]& & & \ast & & & \star\arrow[lll,"Z"] & & & &\\[1.5cm]
		 & & &\bullet \arrow[rrr,bend left=60,pos = 0.4, "1"] \arrow[ur,bend left=30, pos=0.6,"Z"]& &\bullet \arrow[rrr,bend left=60,pos =0.45,"W"] \arrow[ul,bend right=30,pos=0.6, "U"]& \circ \arrow[ur,bend left=30,pos =0.6, "U"] & &\circ \arrow[ul,bend right=30,pos=0.6,"U"] & \bullet \arrow[lll,bend right=60, pos=0.44,"Z^2"] \arrow[rrr, bend left =60, pos = 0.55, "1"] \arrow[ur,bend left=30, pos =0.6, "UZ"]& &\bullet \arrow[lll,bend right=60,pos=0.53,"Z"] \arrow[rrr,bend left=60,pos=0.5, "W"] \arrow[ul,bend right=30,pos =0.6, "U"] &\circ & & \circ\\
 	   	& & & & \bullet \arrow[ul,"W"]\arrow[ur,"1"] \arrow[rrr,bend right, "W"]& & & \circ \arrow[ul,"1"]\arrow[ur,"1"]& & & \bullet \arrow[ul,"1"]\arrow[ur,"Z"]\arrow[lll, bend left,"Z^2"]\arrow[rrr, bend right, "1"] & & & \circ \arrow[ul,"1"]\arrow[ur,"U"]& \\
	\end{tikzcd}.
}
\]A further simplification using arrows labeled by $1$ gives the following complex:

\[
\begin{tikzcd}[row sep = 1cm, column sep =1cm,labels=description]
	\circ &\bullet \arrow[l, "Z"] \arrow[d,"U"] \arrow[drr,bend left,"W^2"]& & & \\
	& \star \arrow[r,"W"]& \ast & \star\arrow[l,"Z"] & \\
	& & & \bullet \arrow[ull,bend left, "Z^2"]\arrow[u,"U"]\arrow[r, "W"]& \circ
\end{tikzcd}\hspace*{2mm} =\hspace*{6mm}
\begin{tikzcd}[arrows={shorten=-1mm},labels={description}]
	 \circ & \bullet \arrow[l,"W"] \arrow[dl,"U"] \arrow[dd,"Z^2",pos=0.3]&  \\
	 \star \arrow[d,"Z"]&                   &\bullet \arrow[ll,"W^2",pos=0.3] \arrow[d,"Z"]\arrow[dl,"U"]\\
	 \ast & \star \arrow[l,"W"] &\circ
\end{tikzcd} \hspace*{2mm}.
\]
	\item When $n=1$, the truncated complex for the $(2,3)$-cable of the right hand trefoil takes the following form: 
	
	\[
	\begin{tikzcd}[labels=description]
		F_{-\frac{3}{2},-\frac{1}{2}} &  E_{-\frac{1}{2},-\frac{1}{2}} \arrow[l,"\Phi^{-\mu}"] \arrow[r,"\Phi^{\mu}"]& F_{-\frac{1}{2},-\frac{1}{2}} & E_{\frac{1}{2},-\frac{1}{2}} \arrow[l,"\Phi^{-\mu}"] \arrow[d,"\Phi^K"]& & & \\
		&  & & J_{\frac{1}{2},-\frac{1}{2}}& & &\\
		& & &	E_{-\frac{1}{2},\frac{1}{2}} \arrow[r,"\Phi^{\mu}"] \arrow[u,"\Phi^{-K}"] & F_{-\frac{1}{2},\frac{1}{2}} & E_{\frac{1}{2},\frac{1}{2}} \arrow[r,"\Phi^{\mu}"] \arrow[l,"\Phi^{-\mu}"] & F_{\frac{1}{2},\frac{1}{2}}
	\end{tikzcd}.\]
The full complex is 
\[
\adjustbox{scale = 0.8}{
	\begin{tikzcd}[column sep =small, labels=description]
		& \circ\arrow[dl,"U"] \arrow[dr, "1"]& & & \bullet \arrow[dl,"W"]\arrow[dr,"1"] \arrow[lll, bend right,"1"] \arrow[rrr,bend left, "W^2"]& & & \circ \arrow[dl,"1"]\arrow[dr,"1"]& & & \bullet \arrow[dl,"1"]\arrow[dr,"Z"]\arrow[lll, bend right,"Z"]& & & & & & & & & & &\\
		\circ & &\circ & \bullet \arrow[lll,bend left=60,pos=0.5,"Z"] \arrow[rrr,bend right=60,pos = 0.55, "W"] & &\bullet \arrow[lll,bend left=60, pos=0.5, "1"] \arrow[rrr,bend right=60,pos =0.45,"W^2"]& \circ & &\circ & \bullet \arrow[lll,bend left=60, pos=0.45,"Z"]\arrow[dr, bend right =30, pos=0.6, "U"]& &\bullet \arrow[lll,bend left=60,pos=0.4,"1"] \arrow[dl,bend left =30, pos=0.6,"W"]& & & & & & & & & &\\[1.5cm]
		& & & & & & &  & & & \star & & & & & & & & & &\\[1.5cm]
	 & & & & & &	& & &\bullet \arrow[rrr,bend left=60,pos = 0.4, "1"] \arrow[ur,bend left=30, pos=0.6,"Z"]& &\bullet \arrow[rrr,bend left=60,pos =0.45,"W"] \arrow[ul,bend right=30,pos=0.6, "U"]& \circ  & &\circ & \bullet \arrow[lll,bend right=60, pos=0.44,"Z^2"] \arrow[rrr, bend left =60, pos = 0.55, "1"] & &\bullet \arrow[lll,bend right=60,pos=0.53,"Z"] \arrow[rrr,bend left=60,pos=0.5, "W"] &\circ & & \circ\\
	& & & & & &	& & & & \bullet \arrow[ul,"W"]\arrow[ur,"1"] \arrow[rrr,bend right, "W"]& & & \circ \arrow[ul,"1"]\arrow[ur,"1"]& & & \bullet \arrow[ul,"1"]\arrow[ur,"Z"]\arrow[lll, bend left,"Z^2"]\arrow[rrr, bend right, "1"] & & & \circ \arrow[ul,"1"]\arrow[ur,"U"]& \\
	\end{tikzcd},
} 
\]
which simplifies to \[
\begin{tikzcd}[row sep = 1cm, column sep =1cm,labels=description]
\circ & \bullet \arrow[l,"Z"] \arrow[dr,bend left,"W^2"]& & &  \\
&  & \star & & \\
& & &  \bullet \arrow[ul,bend left,"Z^2"] \arrow[r,"W"] &\circ\\
\end{tikzcd}
\hspace*{2mm} = \hspace*{4mm}
\begin{tikzcd}[row sep = 1cm, column sep =1cm,labels={description}]
	\circ & \bullet \arrow[l,"W"]\arrow[d,"Z^2"] &[5mm]\\[5mm]
	& \star & \bullet\arrow[l,"W^2"] \arrow[d,"Z"]\\
	& & \circ 
\end{tikzcd}.
\]
\end{enumerate}

With the same procedure, one can get that for $n\leq-2$, $\cCFK$ of the $(2,2n+1)$-cable of the right hand trefoil is homotopy equivalent to 
\[
\begin{tikzcd}[labels=description]
	& & & & & & & \circ \arrow[d,"U"]& \bullet \arrow[l,"Z"]\arrow[d,"U"]\arrow[drr,bend left,"W^2"]& &\\
	\star \arrow[r,"W"] & \ast & \star \arrow[l,"Z"]\arrow[r,"W"]& \ast & \star \arrow[l,"Z"]\arrow[r,"W"]& [1cm]\cdots & \star \arrow[l,"Z"]\arrow[r,"W"]& \ast& 
	\star \arrow[l,"Z"]\arrow[r,"W"]& \ast & \star \arrow[l,"Z"]\\
	& & \bullet \arrow[u,"U"]\arrow[ull,bend left,"Z^2"]\arrow[r,"W"] & \circ \arrow[u,"U"] & & & & &
\end{tikzcd},
\]
with $-n-2$ many $\star$'s between the middle two vertical arrows labeled $U$.

When $n=2$, $\cCFK$ of the $(2,5)$-cable of the right hand trefoil is homotopy equivalent to \[
\begin{tikzcd}[labels={description},column sep=1cm]
	\circ & \bullet\arrow[l,"Z"]\arrow[r,"W^3"] & \circ & \bullet \arrow[l,"Z^3"]\arrow[r,"W"] &\circ
\end{tikzcd}.\]
The form that looks more familiar with the previous case is 
\[
\begin{tikzcd}[labels={description},column sep=1cm]
	\circ & \bullet\arrow[l,"Z"]\arrow[r,"W^3"] \arrow[ddr,bend right, "W^3"]& \circ \arrow[d,"1"]& &\\ 
	& & \star & &\\
	& & \circ \arrow[u,"1"]&
	\bullet \arrow[l,"Z^3"]\arrow[r,"W"] \arrow[uul, bend right, "Z^3"] &\circ
\end{tikzcd}.\]

For $n\ge 3$, $\cCFK$ of the $(2,2n+1)$-cable of the right hand trefoil is homotopy equivalent to \[\hspace*{-2mm}\begin{tikzcd}[labels={description},column sep=1cm]
	\circ & \bullet\arrow[l,"Z"]\arrow[r,"W^3"] & \circ &\bullet \arrow[l,"Z"]\arrow[r,"W"]&\circ  &\cdots \arrow[l,"Z"]\arrow[r,"W"]&\circ&\bullet\arrow[l,"Z"]\arrow[r,"W"] &\circ& \bullet \arrow[l,"Z^3"]\arrow[r,"W"] &\circ
\end{tikzcd},\]
with $n-2$ many $\bullet$'s between the arrow labeled $W^3$ and the arrow labeled $Z^3$.  

\subsection{Whitehead doubles of the right hand trefoil }

	In this example, we compute $\cCFK$ of the Whitehead double of the right hand trefoil with framing $n$.
	
	The $DA$-bimodule of $L_P$, which is the positive Whitehead link, is shown in Figure \ref{fig:7}.	Note that $\cC_0$ is a non-trivial staircase instead of a single generator, and there are some non-trivial $\delta_3^1$-actions if one starts and ends at $\cC_{0}$. The complex $\cS$ still has only one generator, with vanishing differential, as the pattern is an unknot in $S^3$. Again, $N=1$, $g=1$, so the shape of the truncation is the same as in the previous example.

	When $t=-\frac{1}{2}$, the mapping cone $f^{\mu}: \scE_{*,-\frac{1}{2}}\to \scF_{*,-\frac{1}{2}}$ takes the following form:	
	\[\begin{tikzcd}[labels=description]
		 &[-0.7cm] \phantom{\cdots}&[-1.4cm]\phantom{\bullet} & \phantom{\bullet}& \phantom{\bullet}&[0.4cm]   \bullet\arrow[dl,bend right, "W|Z"] \arrow[r, bend left, "Z|1"] 
		  &[1.1cm] \bullet  \arrow[r, bend left, "Z|1"] \arrow[l,bend left, "W|U"] &[1.1cm] \bullet\arrow[l,bend left, "W|U"]& [-0.3cm] \phantom{\cdots}\\
	\scE_{*,-\frac{1}{2}}: \arrow[dd,shorten = 1.3mm,"f^{\mu}"]& \cdots & 
	\bullet\ar[r,bend left, "Z|U"]
	 \arrow[dd,shorten = 3mm,"U"]&\bullet \ar[r,bend left,  "Z|U"]
	 \ar[l,bend left, "W|1"]
	  \arrow[dd,shorten = 3mm,"U"]
	  &\bullet  \ar[l,bend left, "W|1"] \arrow[ur,bend right, pos = 0.4,"Z|W"]
	    \arrow[dd,shorten = 3mm,"U"]& \bullet \ar[r,bend left,  "Z|1"]\arrow[u,"W"]\arrow[d,"Z"]& \bullet  \ar[r,bend left,  "Z|1"]\arrow[u,"W"]\arrow[d,"Z"] \ar[l,bend left, "W|U"]& \bullet  \arrow[u,"W"]\arrow[d,"Z"] \ar[l,bend left, "W|U"]& \cdots \\
		 & & & & &  \bullet \arrow[ul,bend left,"W|W"] \arrow[r, bend left,"Z|1"] 
		 & \bullet \arrow[r, bend left, "Z|1"]\arrow[l,bend left, "W|U"]& \bullet \arrow[l,bend left,  "W|U"]& \\[1cm]
  	\scF_{*,-\frac{1}{2}}:& \cdots
  &[.4cm] \cS
  \ar[r, bend left, "Z|U"]
  & \cS
  \ar[r, bend left, "Z|U"]
  \ar[l, bend left, "W|1"]
  &\cS
  \ar[r, bend left, "Z|1"]
  \ar[l, bend left, "W|1"]
  &\cS  \arrow[from = uuu, bend left, crossing over, pos = 0.82, "Z"]
  \arrow[from = u, bend right, "W"]
  \ar[r, bend left, "Z|1"]
  \ar[l, bend left, "W|U"]
  &\cS
  \ar[r, bend left, "Z|1"]
  \ar[l, bend left, "W|U"]
   \arrow[from = uuu, bend left, crossing over, pos = 0.82, "Z"]
     \arrow[from = u, bend right, "W"]
  &\cS
  \ar[l, bend left, "W|U"]
   \arrow[from = uuu, bend left, crossing over, pos = 0.82, "Z"]
     \arrow[from = u, bend right, "W"]
  &\cdots 
	\end{tikzcd}.\]
All the $\delta_3^1(\sigma,Z,-)$ or $\delta_3^1(\sigma,W,-)$ in this example are again trivial. The mapping cone $f^{-\mu}:\scE_{*,-\frac{1}{2}}\to \scF_{*,-\frac{1}{2}}$ is similar,  which is the identity on the left half, and it's the same as $L_\sigma$ on the right half. The row 
	\[\begin{tikzcd}[column sep=1.3cm, labels=description]
		F_{-\frac{3}{2},-\frac{1}{2}} & E_{-\frac{1}{2},-\frac{1}{2}} \arrow[l,"\Phi^{-\mu}"] \arrow[r,"\Phi^{\mu}"] & F_{-\frac{1}{2},-\frac{1}{2}} & E_{\frac{1}{2},-\frac{1}{2}} \arrow[l,"\Phi^{-\mu}"]
	\end{tikzcd}
	\]
is given by
\[
\adjustbox{scale=0.9}{
\begin{tikzcd}[every label/.append style = {font = \tiny},labels=description,column sep =small]
	 \phantom{\circ}&\phantom{\circ}  & \phantom{\circ} &\phantom{\space} & [-0.3cm]\phantom{\bullet} &[-0.3cm]\phantom{\bullet} & [-0.9cm] \phantom{\bullet}& [0.5cm]\phantom{\bullet}& &  &  &  & & & &  &[-0.9cm]\bullet \arrow[dll,"Z"]\arrow[dr,"W"] \arrow[ddll,bend right=10,pos = 0.8,"1"]& &  \\[-0.1cm]
		  & \circ \arrow[ddl,"U"]\arrow[ddr,"1"]&  & &  & & \bullet \arrow[ddl,bend  right=10, "W"] \arrow[lllll,bend right, "1"]\arrow[rrrr,bend left, "U"]\arrow[ddr, bend left=30,"1"] & & &  &  \circ \arrow[ddl,"1"]\arrow[ddr,"1"]&  & & &\bullet \arrow[ddl, bend right =10,"1"]  \arrow[llll,bend right, "W"]&  & & \bullet \arrow[ddr,bend left=10, "Z"] \arrow[ddll,bend right =10, pos =0.6, "1"]  \arrow[lllllll,bend right=70, "Z"]&  \\
	      &  &      &      & \bullet \arrow[dl,"Z"] \arrow[dr,"W"]&  & &  & & & &  & &  & \bullet \arrow[dl,"Z"] \arrow[dr,"W"] & & & & \\[-0.1cm]
	\circ &      &\circ & \bullet \arrow[lll, bend left, "W"] \arrow[rrrrrr, bend right=40, pos=0.6, "W"]&  & \bullet  \arrow[lllll, bend left =40, "Z"] \arrow[rrrr,bend right,"Z"]& & \bullet \arrow[lllll,bend left=40,crossing over,pos = 0.55, "1"] & &\circ \arrow[from =rrrr,bend left = 40, pos = 0.4,"W"]\arrow[from = rrrrrr,bend left =40, "Z"] & & \circ \arrow[from = llll,bend right =40, crossing over, "U"]& & \bullet & & \bullet & & & \bullet \arrow[from = uullll,crossing over,bend left = 10, pos = 0.8,"W"]  \arrow[lllllll, bend left =40, crossing over, "1"]
\end{tikzcd}.
}
\]

At this stage, one can simplifying it by removing a contractible quotient complex, (since all the vertical arrows are leaving this row, a quotient complex of this row is also a quotient complex of the full truncated complex) which gives

\[\begin{tikzcd}[labels =description]
\phantom{\circ} & & &[-0.5cm] \bullet \arrow[dl, "Z"]  \arrow[dr,"W"]& [-0.5cm]& &  &[-0.5cm]\circ \arrow[dl,"1"]\arrow[dr,"1"]&[-0.5cm] & &\\
	\circ & & \bullet \arrow[ll,bend left=40, pos= 0.4,"W"] \arrow[rrrr, bend right=40, pos = 0.4, "W"]& & \bullet \arrow[llll, bend left=40, pos=0.4,"Z"]  \arrow[rr, bend right =40, pos = 0.4, "Z"]& & \circ & & \circ & & \bullet \arrow[ll, bend left = 40, "1"]
\end{tikzcd}.\]

By a change of basis summing up the two $\circ$'s with arrows labeled by $1$ pointing to them, and note that there is no $\Phi^{K}$ map leaving the top $\circ$, we can further simplify it to the following:

\[\begin{tikzcd}[labels =description]
	\phantom{\circ} & & &[-0.5cm] \bullet \arrow[dl, "Z"]  \arrow[dr,"W"]& [-0.5cm]& & &\\
	\circ & & \bullet \arrow[ll,bend left=40, pos= 0.4,"W"] \arrow[rrrr, bend right=40, pos = 0.4, "W"]& & \bullet \arrow[llll, bend left=40, pos=0.4,"Z"]  \arrow[rr, bend right =40, pos = 0.4, "Z"]& & \circ & & \bullet \arrow[ll, bend left = 40, "1"]
\end{tikzcd}\]

Symmetrically (rotating $180^\circ$ and switching $W$ with $Z$), the row \[\begin{tikzcd}[column sep=1.3cm, labels=description]
	E_{-\frac{1}{2},\frac{1}{2}}  \arrow[r,"\Phi^{\mu}"] & F_{-\frac{1}{2},\frac{1}{2}} & E_{\frac{1}{2},\frac{1}{2}} \arrow[l,"\Phi^{-\mu}"]
	\arrow[r,"\Phi^{\mu}"] &
	F_{\frac{1}{2},\frac{1}{2}} 
\end{tikzcd}
\]  
can be simplified to 
\[
\begin{tikzcd}[labels = description]
	\bullet \arrow[rr, bend left=40, "1"]& & \circ & & \bullet \arrow[ll,  bend right =40, pos =0.4, "W"] \arrow[rrrr, bend left =40, pos =0.4,"W"]&[-0.5cm] & [-0.5cm]\bullet\arrow[llll,bend right =40, pos =0.4, "Z"] \arrow[rr, bend left=40, pos =0.4, "Z"] & & \circ\\
		  & &  & &  & \bullet\arrow[ul,"Z"]\arrow[ur,"W"] &  & & 
\end{tikzcd}.
\]

The row 	
\[\begin{tikzcd}[column sep=1.3cm, labels=description]
	\cdots& [-1cm]M_{-\frac{3}{2},-\frac{1}{2}} & J_{-\frac{1}{2},-\frac{1}{2}} \arrow[l,"\Phi^{-\mu}"] \arrow[r,"\Phi^{\mu}"] & M_{-\frac{1}{2},-\frac{1}{2}} & J_{\frac{1}{2},-\frac{1}{2}} \arrow[l,"\Phi^{-\mu}"]&[-1cm] \cdots
\end{tikzcd}
\]
takes the form 

\[   \begin{tikzcd}[column sep=1.3cm, labels=description]
     &[-1cm]  & \star \arrow[dl,bend right, "Z"] \arrow[dr, bend left, "Z"]&   & \star \arrow[dl,bend right, "Z"] &[-1cm]\\
	\cdots & \ast & \star\arrow[u,"W"]\arrow[d,"Z"] & \ast & \star \arrow[u,"W"]\arrow[d,"Z"]& \cdots\\
     &   & \star \arrow[ul,bend left, "W"] \arrow[ur, bend right, "W"]&   & \star \arrow[ul,bend left, "W"]& 
\end{tikzcd},\]
and is independent of the $s$-coordinate, i.e., invariant under horizontal translation.

The maps $\Phi^{K} : E_{s,-\frac{1}{2}} \to J_{s,-\frac{1}{2}}$ and $\Phi^{K} : F_{s,-\frac{1}{2}} \to M_{s,-\frac{1}{2}}$ for the simplified model in the range of interest look like the following:
\begin{figure}[h]
	 \adjustbox{scale=0.8}{\begin{tikzcd}[labels =description, column sep = 1cm]
	 		\phantom{\circ} & & &[-0.5cm] \bullet \arrow[dl, "Z"]  \arrow[dr,"W"] & [-0.5cm]& & &\\
	 		\circ \arrow[dd,"U"]& & \bullet \arrow[ll,bend left=40, pos= 0.4,"W"] \arrow[rrrr, bend right=40, pos = 0.6, "W"]& & \bullet \arrow[llll, bend left=40, pos=0.6,"Z"]  \arrow[rr, bend right =40, pos = 0.4, "Z"]& & \circ \arrow[dd,"U"] & & \bullet \arrow[ll, bend left = 40, "1"]\arrow[d,"W"]\\[1cm]
	 		& & & \star \arrow[dlll,bend right, "Z"] \arrow[drrr,bend left, "Z"] \arrow[from =ur, bend left =30, crossing over, pos = 0.6, "U"]& & &  & & \star  \arrow[dll,bend right, "Z"] \\
	 		\ast & & & \star \arrow[u,"W"]\arrow[d,"Z"]\arrow[from = uuu, bend right=28, crossing over, pos =0.65,"U"]& & & \ast & & \star\arrow[u,"W"]\arrow[d,"Z"] \\
	 		& & & \star \arrow[ulll, bend left, "W"] \arrow[urrr, bend right, "W"] \arrow[from = uuul, bend right, crossing over, pos = 0.6, "U"]& & &  & & \star\arrow[ull, bend left, "W"]
	 	\end{tikzcd}.
	 
 }
\end{figure}

Symmetrically, the maps $\Phi^{-K}: E_{s,\frac{1}{2}}\to J_{s+n,-\frac{1}{2}}$ and  $\Phi^{-K}: F_{s,\frac{1}{2}}\to M_{s+n,-\frac{1}{2}}$ for the simplified model in the range of interest looks like:

\begin{figure}[h]
	\adjustbox{scale=0.8}{ 
	\begin{tikzcd}[labels = description, column sep = 1cm]
		\star \arrow[drr, bend left, "Z"]& & & & & [-0.4cm]\star \arrow[dlll,bend right, "Z"] \arrow[drrr,bend left, "Z"]& [-0.4cm] & &\\
		\star \arrow[d,"Z"]\arrow[u,"W"] &  & \ast & & & \star \arrow[d,"Z"] \arrow[u,"W"]& & &\ast\\
		\star \arrow[urr,bend right, "W"]& & & & & \star \arrow[ulll,bend left,"W"]  \arrow[urrr,bend right, "Z"]& & &\\[1cm]
		\bullet \arrow[rr, bend left=40, "1"] \arrow[u , "Z"]& & \circ \arrow[uu,"U"] & & \bullet \arrow[ll,  bend right =40, pos =0.4, "W"] \arrow[rrrr, bend left =40, pos =0.63,"W"] \arrow[ur,bend left, crossing over, pos=0.6,"U"]& & [-0.5cm]\bullet\arrow[llll,bend right =40, pos =0.6, "Z"] \arrow[rr, bend left=40, pos =0.4, "Z"] \arrow[uuul,bend right, crossing over, pos=0.6,"U"]& & \circ \arrow[uu,"U"]\\
		& &  & &  & \bullet\arrow[ul,"Z"]\arrow[ur,"W"]  \arrow[uuu,bend right=25, crossing over,pos =0.65, "U"]&  & & \\
	\end{tikzcd}.
}
\end{figure}

There would have been arrows in $\Phi^{K}:E_{\frac{1}{2},-\frac{1}{2}}\to J_{\frac{1}{2},-\frac{1}{2}}$ corresponding to the non-trivial $\delta_3^1(Z,W,-)$-action on $\cC_{0}$ in the $DA$-bimodule, coming from the term $\ve{y}\otimes \sigma \otimes W$ in $\delta^1(\delta^1(\ve{a})) $ in  $\cX_n(S^3,T_{2,3})^{\cK}$, but since we have made some simplification first, these arrows no longer exist in the simplified version. See the following diagram, where the contribution from $(\sigma,W)|h_{Z,W}$ is the blue arrow. The situation for $\Phi^{-K}: E_{-\frac{1}{2},\frac{1}{2}} \to J_{-\frac{1}{2}+n,-\frac{1}{2}} $ is symmetric.

\begin{figure}[h]
		\adjustbox{scale=0.8}{
	\begin{tikzcd}[labels={description}]
		E_{\frac{1}{2},-\frac{1}{2}} \arrow[d,"\Phi^{K}"]\\[10mm]
		J_{\frac{1}{2},-\frac{1}{2}}
	\end{tikzcd}
	\hspace*{5mm}= \hspace*{5mm}
	\begin{tikzcd}[labels={description}]
		 & \cC_{0} \arrow[dl,"\bI"]\arrow[dr, "L_W"] \arrow[dd,blue, "h_{Z,W}"]&\\
		\cC_{0} \arrow[dr,"U"]& & \cC_{-1}\arrow[dl,"L_Z"]\\[10mm]
		& \cC_{0}&
	\end{tikzcd}		
			
\hspace*{5mm}	=  \hspace*{5mm}\begin{tikzcd}[labels={description}]
	 	& & \bullet\arrow[dl,"Z"]\arrow[dr,"W"] \arrow[ddl, bend right=10,"1",pos = 0.8]& &\\[-1mm]
	 	& \bullet\arrow[ddl,bend right=10,"1",pos = 0.6] \arrow[dddr,blue,"W",pos = 0.7, bend left =15,crossing over]& & \bullet\arrow[ddl, bend right =10, "1",pos = 0.6]\arrow[ddr,bend left =10, "Z"] &\\[5mm]
	 	&\bullet\arrow[dl,"Z"]\arrow[dr,"W"] \arrow[ddr, bend right =10,"U"] & & & \\[-1mm]
	 	\bullet \arrow[ddr, bend right =10, "U"] & & \bullet \arrow[ddr, bend right =10, "U"] & & \bullet\arrow[from = uulll, crossing over, bend left =10, "W",pos = 0.7]\arrow[ddl, bend left =10, "W"]\\[10mm]
	 	& & \star \arrow[dl,"Z"]\arrow[dr,"W"]& &\\[-1mm]
	 	&\star & & \star &
	 \end{tikzcd}
\hspace*{5mm}  $\simeq$\hspace*{5mm}
 \begin{tikzcd}[labels={description}]
 	\\
 	\\
 	\\
 	\\
 	\\
 	\\[5mm]
 	\bullet\arrow[d,"W"] \\[10mm]
 	\star \\
 	\star \arrow[u,"W"]\arrow[d,"Z"]\\
 	\star \\
 \end{tikzcd}
 }
\end{figure}

The final step is stacking the above two diagrams of maps $\Phi^K$ and $\Phi^{-K}$  together along the row of $J_{*,-\frac{1}{2}}$ and $M_{*,-\frac{1}{2}}$, with a relative horizontal shift specified by the framing $n$. The shape is described by the truncation formula. Note that for $n\ge2$,  we need to extend the top row consisting of $E_{*,-\frac{1}{2}}$ and $F_{*,-\frac{1}{2}}$ to include terms $E_{s,-\frac{1}{2}}$ and $F_{s-1,\frac{1}{2}}$ for $s=\frac{3}{2},\frac{5}{2},...,n-\frac{1}{2}$ as well. They take the following form after simplification:
\[	\begin{tikzcd}[column sep=1cm, labels=description]
	E_{\frac{1}{2},-\frac{1}{2}} \arrow[r,"\Phi^{\mu}"] & F_{\frac{1}{2},-\frac{1}{2}} & E_{\frac{3}{2},-\frac{1}{2}} \arrow[l,"\Phi^{-\mu}"] \arrow[r,"\Phi^{\mu}"] & F_{\frac{3}{2},-\frac{1}{2}} \arrow[from=r, "\Phi^{-\mu}"]& E_{\frac{5}{2},-\frac{1}{2}} \arrow[r,"\Phi^{\mu}"]&\cdots& E_{n-\frac{1}{2},-\frac{1}{2}} \arrow[l,"\Phi^{-\mu}"] \arrow[d,"\Phi^{K}"] \\
	 &   &     &  &  & & J_{n-\frac{1}{2},-\frac{1}{2}}
\end{tikzcd}\]
\begin{figure}[h]
	\adjustbox{scale=0.8}{ 
	$\simeq$ 	\begin{tikzcd}[labels=description,column sep=small]
		& & & &  &[-0.2cm]\bullet\arrow[dl,"Z"]\arrow[dr,"W"] &[-0.1cm]  & & & &  &[-0.1cm]\bullet \arrow[dl,"Z"]\arrow[dr,"W"] &[-0.1cm] & \cdots &  &[-0.1cm]\bullet \arrow[dl,"Z"]\arrow[dr,"W"] \arrow[ddd, bend right = 28, "1"] &[-0.2cm] \\
	\hspace*{5mm}\bullet \arrow[rr,bend right, "U"]& & \circ & & \bullet \arrow[ll,bend left, "W"] \arrow[rrrr,bend right=50, "W"]& & \bullet \arrow[llll,bend left =50,"Z"] \arrow[rr, bend right, "Z"]& &\circ & & \bullet \arrow[ll,bend left, "W"] & & \bullet \arrow[llll,bend left =50,"Z"] & \cdots & \bullet \arrow[dddr, bend right, "1"] & &\bullet  \arrow[dl, bend left, "1"]\\[-0.3cm]
		& & & &  &  &  & & & &  & & &  &  &\star & \\
		& & & &  &  &  & & & &  & & &  &  &\star \arrow[u, "W"] \arrow[d,"Z"] & \\
		& & & &  &  &  & & & &  & & &  &  &\star &
	\end{tikzcd}.
}
\end{figure}

Below are some examples of the knot complex for $n$-framed Whitehead double of the right hand trefoil $(W^+(T_{2,3},n))$, with some further simplification.
\begin{enumerate}
	\item When $n=-1$, $\cCFK(W^+(T_{2,3},-1))$ is homotopy equivalent to the complex in Figure \ref{fig:CFK(Wh_{-1}(RHT))}.
	\begin{figure}[h]
			\adjustbox{scale=0.8}{
		\begin{tikzcd}[labels =description, column sep = 1cm]
	& &\phantom{\circ} & & &[-0.5cm] \bullet \arrow[dl, "Z"]  \arrow[dr,"W"] & [-0.5cm]& & &\\
	& &	\circ \arrow[dd,"U"]& & \bullet \arrow[ll,bend left=40, pos= 0.4,"W"] \arrow[drrrrrr, pos = 0.65, crossing over, bend right=10, "W^2"]& & \bullet \arrow[llll, bend left=40, pos=0.6,"Z"]  \arrow[drrrr, bend right=10, pos = 0.4, "U"]& &  & &  \\[1cm]
	\star \arrow[drr, bend left, "Z"]& &	& & & \star \arrow[dlll,bend right, "Z"] \arrow[drrr,bend left, "Z"] \arrow[from =ur, bend left =30, crossing over, pos = 0.6, "U"]& & &  & & \star  \arrow[dll,bend right, "Z"] \\
	\star \arrow[u,"W"]\arrow[d,"Z"]& &	\ast & & & \star \arrow[u,"W"]\arrow[d,"Z"]\arrow[from = uuu, bend right=28, crossing over, pos =0.65,"U"]& & & \ast & & \star\arrow[u,"W"]\arrow[d,"Z"] \\
	\star \arrow[urr, bend right,"W"] \arrow[from=drrrrrr, bend right =10, "Z^2"]& &	& & & \star \arrow[ulll, bend left, "W"] \arrow[urrr, bend right, "W"] \arrow[from = uuul, bend right, crossing over, pos = 0.6, "U"]& & &  & & \star\arrow[ull, bend left, "W"] \\[1cm]
	& &   & & \bullet \arrow[ullll,  bend right =10, pos =0.4, "U"] \arrow[rrrr, bend left =40, pos =0.63,"W"] \arrow[ur,bend left, crossing over, pos=0.6,"U"]& & [-0.5cm]\bullet   \arrow[rr, bend left=40, pos =0.4, "Z"] \arrow[uuul,bend right, crossing over, pos=0.6,"U"]& & \circ \arrow[uu,"U"] & &\\
	& &  & &  & \bullet\arrow[ul,"Z"]\arrow[ur,"W"]  \arrow[uuu,bend right=22, crossing over,pos =0.65, "U"]&  & & & &\\
\end{tikzcd}
}
\adjustbox{scale=0.7}{
\begin{tikzcd}
	[row sep =  -1mm, column sep =-1mm, arrows={shorten = -1mm}]
	\phantom{\bullet} & [3mm] \phantom{\bullet} &[3mm]  \phantom{\bullet} &[1.5cm] \phantom{\bullet} & \bullet \arrow[ddd,"Z"]  & \phantom{\bullet}&\phantom{\bullet}&\phantom{\bullet} & [1.5cm] \bullet\arrow[llll,"W"] \arrow[ddd,"Z"]  & \phantom{\bullet}&\phantom{\bullet}\\[1mm]
	\phantom{\bullet} & \phantom{\bullet} & \phantom{\bullet} &[1.8cm] \phantom{\bullet} & \phantom{\bullet} & \phantom{\bullet}&\bullet \arrow[dddd,"Z"]  &\phantom{\bullet} & [1.8cm] \phantom{\bullet} &\phantom{\bullet} & \bullet \arrow[dddd,"Z"] \arrow[llll,"W"]\\[1.8cm]
	\bullet \arrow[ddddd,"Z"]& \phantom{\bullet} & \phantom{\bullet} &[1.5cm] \bullet \arrow[lll,"W"] \arrow[ddddd,pos =0.6,"Z"] & \phantom{\bullet} & \phantom{\bullet}& &\phantom{\bullet} & [1.5cm]   &\phantom{\bullet} &  \\
	& \phantom{\bullet} & \phantom{\bullet} &[1.5cm]  & \bullet & \phantom{\bullet}& &\phantom{\bullet} & [1.5cm]  \bullet \arrow[llll,"W"]  &\phantom{\bullet} &  \\
	& \bullet& \phantom{\bullet} &[1.5cm]  &   & \bullet \arrow[llll,"W"] \arrow[dddd,"Z"] & &\phantom{\bullet} & [1.5cm]    &\phantom{\bullet}& \\
	&  & \phantom{\bullet} &[1.5cm]  &   &  & \bullet &\phantom{\bullet} & [1.5cm]    &\phantom{\bullet}& \bullet \arrow[llll,"W"]\\	
	&  & \bullet \arrow[ddd,pos =0.4,"Z"] &[1.5cm]  &   &  &   &\bullet \arrow[lllll,pos =0.6,"W"] \arrow[ddd,"Z"] & [1.5cm]    &\phantom{\bullet}& \\[1.8cm] 	
	\bullet&  &  &[1.5cm] \bullet \arrow[lll,pos =0.4,"W"]&   &  &   &  & [1.5cm]    &\phantom{\bullet}& \\[2mm] 
	&  &  &[1.5cm] &    &\bullet  &   &  & [1.5cm]    &\phantom{\bullet}& \\ [2mm]
	&  &\bullet  &[1.5cm] &   &  &    &\bullet \arrow[lllll,"W"] & [1.5cm]    &\phantom{\bullet}& \\
\end{tikzcd} 
 }
\caption{A model of $\cCFK(W^+(T_{2,3},-1))$, and a simplified version }
\label{fig:CFK(Wh_{-1}(RHT))}
	\end{figure}

	\item When $n=0$, $\cCFK(W^+(T_{2,3},0))$ is homotopy equivalent to the complex in Figure \ref{fig:CFK(Wh_{0}(RHT))}.
	
	\begin{figure}[h]
		\adjustbox{scale=0.8}{
\begin{tikzcd}[labels =description, column sep = 1cm]
	\phantom{\circ} & & &[-0.5cm] \bullet \arrow[dl, "Z"]  \arrow[dr,"W"] & [-0.5cm]& & & & \phantom{\bullet} & [-0.5cm]\phantom{\bullet} & [-0.5cm] & \\
	\circ  & & \bullet \arrow[ll,bend left=40, pos= 0.4,"W"] \arrow[drrrrrrr, bend right=7, pos = 0.4, "W^2"]& & \bullet \arrow[llll, bend left=40, pos=0.6,"Z"]  \arrow[drrrrr, bend right=7, pos = 0.4, "U"]& & & &  && & &\\[1cm]
	& & & \star\arrow[drrr,bend left, "Z"] \arrow[from =ur, bend left =30, crossing over, pos = 0.6, "U"]& & &  & & & \star  \arrow[dlll,bend right, "Z"] & & & \\
	\phantom{\ast }& & & \star \arrow[u,"W"]\arrow[d,"Z"]\arrow[from = uuu, bend right=28, crossing over, pos =0.65,"U"]& & & \ast & & & \star\arrow[u,"W"]\arrow[d,"Z"] & & & \\
	& & & \star\arrow[urrr, bend right, "W"] \arrow[from = uuul, bend right, crossing over, pos = 0.6, "U"] \arrow[from = drrrrrrr, bend right=7, pos = 0.4,"Z^2"]& & &  & &  &\star\arrow[ulll, bend left, "W"]
	& & & \\[1cm]
	& & &  & & &   & & \bullet \arrow[ulllll,  bend right =7, pos =0.5, "U"] \arrow[rrrr, bend left =40, pos =0.63,"W"] \arrow[ur,bend left, crossing over, pos=0.6,"U"]& & [-0.5cm]\bullet\arrow[rr, bend left=40, pos =0.4, "Z"] \arrow[uuul,bend right, crossing over, pos=0.6,"U"] & & \circ \\
	& & & & & &  & &  & \bullet \arrow[ul,"Z"]\arrow[ur,"W"]  \arrow[uuu,bend right=22, crossing over,pos =0.65, "U"]&  & & \\
\end{tikzcd}
}
\adjustbox{scale=0.7}{
\begin{tikzcd}
	[row sep =  -1mm, column sep =-1mm, arrows={shorten = -1mm}]
	\phantom{\bullet} & [3mm] \phantom{\bullet} &[3mm]  \phantom{\bullet} &[1.2cm] \phantom{\bullet} & \bullet \arrow[ddd,"Z"]  & \phantom{\bullet}&\phantom{\bullet}&\phantom{\bullet} & [1.5cm] \bullet\arrow[llll,"W"] \arrow[ddd,"Z"]  & \phantom{\bullet}&\phantom{\bullet}\\[1mm]
	\phantom{\bullet} & \phantom{\bullet} & \phantom{\bullet} &[1.8cm] \phantom{\bullet} & \phantom{\bullet} & \phantom{\bullet}&\bullet \arrow[dddd,"Z"]  &\phantom{\bullet} & [1.8cm] \phantom{\bullet} &\phantom{\bullet} & \bullet \arrow[dddd,"Z"] \arrow[llll,"W"]\\[1.8cm]
	\phantom{\bullet}& \phantom{\bullet} & \phantom{\bullet} &[1.5cm] \phantom{\bullet} & \phantom{\bullet} & \phantom{\bullet}& &\phantom{\bullet} & [1.5cm]   &\phantom{\bullet} &  \\
	& \phantom{\bullet} & \phantom{\bullet} &[1.5cm]  & \bullet & \phantom{\bullet}& &\phantom{\bullet} & [1.5cm]  \bullet \arrow[llll,"W"]  &\phantom{\bullet} &  \\
	& \bullet \arrow[dddd,"Z"]& \phantom{\bullet} &[1.5cm]  &   & \bullet \arrow[llll,swap,"W"] \arrow[dddd,"Z"] & &\phantom{\bullet} & [1.5cm]    &\phantom{\bullet}& \\
	\bullet  &  & \phantom{\bullet} &[1.5cm]  & \bullet \arrow[llll,"W"] \arrow[dddd,swap,"Z"]  &  & \bullet &\phantom{\bullet} & [1.5cm]    &\phantom{\bullet}& \bullet \arrow[llll,"W"]\\	
	&  &  &[1.5cm]  &   &  &   & & [1.5cm]    &\phantom{\bullet}& \\[1.5cm] 	
	&  &  &[1.5cm] &   &  &   &  & [1.5cm]    &\phantom{\bullet}& \\[2mm] 
	&  \bullet &  &[1.5cm] &    &\bullet \arrow[llll,"W"] &   &  & [1.5cm]    &\phantom{\bullet}& \\ [1mm]
	&  &  &[1.5cm] &\bullet    & &    & & [1.5cm]    &\phantom{\bullet}& \\
\end{tikzcd}
}
\caption{A model of $\cCFK(W^+(T_{2,3},0))$, and the simplified version }
\label{fig:CFK(Wh_{0}(RHT))}
\end{figure}

\item When $n=1$, $\cCFK(W^+(T_{2,3},1))$ is homotopy equivalent to the complex in Figure \ref{fig:CFK(Wh_{1}(RHT))}.
	
\begin{figure}[h]
	\adjustbox{scale=0.8}{
		\begin{tikzcd}[labels =description, column sep = 0.8cm]
	\phantom{\circ} & & &[-0.5cm] \bullet \arrow[dl, "Z"]  \arrow[dr,"W"] & [-0.5cm]& & & & & & & & & & &\\
	\circ & & \bullet \arrow[ll,bend left=40, pos= 0.4,"W"] \arrow[drrrrrr, bend right=40, pos = 0.3, "W^2"]& & \bullet \arrow[llll, bend left=40, pos=0.6,"Z"]  \arrow[drrrr, bend right =40, pos = 0.4, "U"]& &    & &  & & & & & & & &\\[-0.5cm]
	& & & \phantom{\star} & & &  & & \star  & & & & & & & &\\
	\phantom{\ast }& & & \phantom{\star} & & & \phantom{\ast }& & \star\arrow[u,"W"]\arrow[d,"Z"] & & & & & & & &\\
	& & & \phantom{\star}& & &  & & \star& & & & & & & & \\  [-0.5cm]
	& & & & & & & &  & &   & & \bullet \arrow[ullll,  bend right =40, pos =0.4, "U"] \arrow[rrrr, bend left =40, pos =0.63,"W"] & & [-0.5cm]\bullet\arrow[ullllll,bend right =40, pos =0.6, "Z^2"] \arrow[rr, bend left=40, pos =0.4, "Z"] & & \circ\\
	& & & & & & & & & &  & &  & \bullet\arrow[ul,"Z"]\arrow[ur,"W"] &  & & \\
\end{tikzcd}
	}

\adjustbox{scale=0.8}{
 \begin{tikzcd}
	 [row sep =  -1mm, column sep =-1mm, arrows={shorten = -1mm}]
	 \phantom{\bullet} & [3mm] \phantom{\bullet} &[3mm]  \phantom{\bullet} &[1.4cm] \phantom{\bullet} & \bullet \arrow[ddd,"Z"]  & \phantom{\bullet}&\phantom{\bullet}&\phantom{\bullet} & [1.2cm] \bullet\arrow[llll,"W"] \arrow[ddd,"Z"]  & \phantom{\bullet}&\phantom{\bullet}\\[1mm]
	 	 \phantom{\bullet} & \phantom{\bullet} & \phantom{\bullet} &[1.8cm] \phantom{\bullet} & \phantom{\bullet} & \phantom{\bullet}&\bullet \arrow[dddd,"Z"]  &\phantom{\bullet} & [1.5cm] \phantom{\bullet} &\phantom{\bullet} & \bullet \arrow[dddd,"Z"] \arrow[llll,"W"]\\[1.5cm]
	 	  \phantom{\bullet}& \phantom{\bullet} & \phantom{\bullet} &[1.2cm] \phantom{\bullet} & \phantom{\bullet} & \phantom{\bullet}& &\phantom{\bullet} & [1.5cm]   &\phantom{\bullet} &  \\
	 	  & \phantom{\bullet} & \phantom{\bullet} &[1.2cm]  & \bullet & \phantom{\bullet}& &\phantom{\bullet} & [1.2cm]  \bullet \arrow[llll,"W"]  &\phantom{\bullet} &  \\
	 	    & &\bullet  &[1.2cm]  &   &\bullet \arrow[lll,"W"] \arrow[dd,"Z"] & & & [1.2cm]    &\phantom{\bullet}& \\
	 	  &  &  &[1.2cm]  &    &  & \bullet &\phantom{\bullet} & [1.2cm]    &\phantom{\bullet}& \bullet \arrow[llll,"W"]\\[1.7cm]	&  & &[1.5cm]  &    &\bullet  &  & & [1.2cm]    && 
\end{tikzcd}
}
	\caption{A model of $\cCFK(W^+(T_{2,3},1))$, and a simplified version }
	\label{fig:CFK(Wh_{1}(RHT))}
\end{figure}
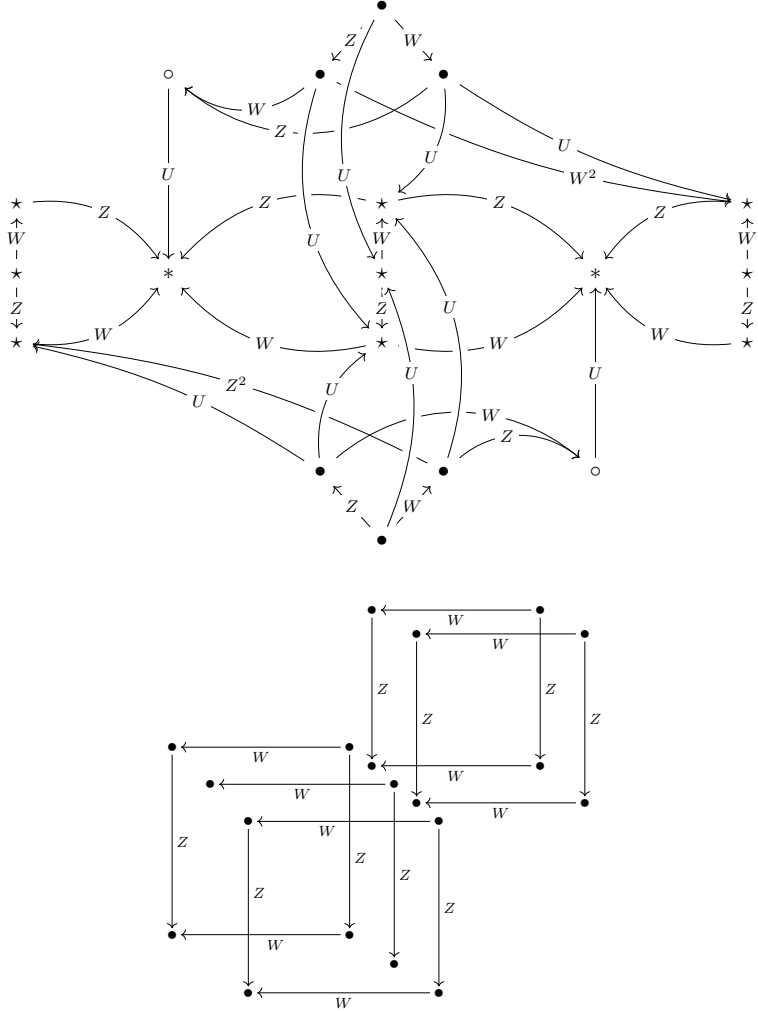

\item When $n=2$, $\cCFK(W^+(T_{2,3},2))$ is homotopy equivalent to the complex in Figure \ref{fig:CFK(Wh_{2}(RHT))}.

\begin{figure}[h]
	\adjustbox{scale=0.8}{
		\begin{tikzcd}[labels =description, column sep = 0.8cm]
	& &  & [-0.5cm]\bullet \arrow[dl, "Z"] \arrow[dr,"W"]& [-0.5cm]  & &   & &  &[-0.5cm] & [-0.5cm]   & &  \\
	\circ & & \bullet\arrow[ll, bend left=40, "W"] \arrow[rrrr, bend right=50, "UW"]& & \bullet \arrow[llll, bend left =50, "Z"] \arrow[rr, bend right =40, "UZ"]& & \circ & & \bullet\arrow[ll, bend right =40, "UW"] \arrow[rrrr, bend left =50, "W"] & & \bullet \arrow[llll, bend right =50, "UZ"]  \arrow[rr, bend left =40, "Z"]& & \circ\\
	& &  & &   & &  & &   & [-0.5cm]\bullet \arrow[ul, "Z"] \arrow[ur, "W"] &[-0.5cm]   & &  
\end{tikzcd}
	}
\adjustbox{scale=0.8}{
\begin{tikzcd}
	 [row sep =  -1mm, column sep =-1mm, arrows={shorten = -1mm}]
	 \phantom{\star} &  \phantom{\star} &[3mm]  \phantom{\star} & \phantom{\ast} & \bullet \arrow[ddd,"Z"]  & \phantom{\star}&\phantom{\bullet}&\phantom{\star} & [1.2cm] \bullet\arrow[llll,"W"] \arrow[ddd,"Z"]  & \phantom{\bullet}&\phantom{\bullet}\\[1mm]
	 	 \phantom{\star} & \phantom{\star} & \phantom{\star} &[1.5cm] \phantom{\ast} & \phantom{\bullet} & \phantom{\star}&\bullet \arrow[dddd,"Z"]  &\phantom{\star} & [1.5cm] \phantom{\bullet} &\phantom{\bullet} & \bullet \arrow[dddd,"Z"] \arrow[llll,"W"]\\[1.5cm]
	 	  \phantom{\star}& \phantom{\star} & \phantom{\star} &[1.2cm] \phantom{\star} & \phantom{\bullet} & \phantom{\star}& &\phantom{\star} & [1.5cm]   &\phantom{\bullet} &  \\
	 	  & \phantom{\star} & \phantom{\star} &[1.2cm]  & \bullet & \phantom{\star}& &\phantom{\star} & [1.5cm]  \bullet \arrow[llll,"W"]  &\phantom{\bullet} &  \\
	 	    & & \phantom{\star} &[1.5cm]  &   &  & &\phantom{\star} & [1.2cm]    &\phantom{\bullet}& \\
	 	  &  & \phantom{\star} &[1.5cm]  &\bullet    &  & \bullet &\phantom{\star} & [1.2cm]    &\phantom{\bullet}& \bullet \arrow[llll,"W"]	
\end{tikzcd}
 }
	\caption{A model of $\cCFK(W^+(T_{2,3},2))$, and a simplified version}
	\label{fig:CFK(Wh_{2}(RHT))}
\end{figure}

\item When $n\ge 3$, $\cCFK(W^+(T_{2,3},n))$ is homotopy equivalent to the complex in Figure \ref{fig:CFK(Wh_n(RHT))}, with $n-2$ many staircase of length $1$, $\bullet \xleftarrow{Z} \bullet \xrightarrow{W} \bullet$, in the middle.

\begin{figure}[h]
	\adjustbox{scale=0.7}{
		\begin{tikzcd}[labels =description, column sep = 0.8cm]
			& &  & [-0.5cm]\bullet \arrow[dl, "Z"] \arrow[dr,"W"]& [-0.5cm]  & &   &\bullet\arrow[dl, bend right=20, "W"] \arrow[dr, bend left=20, "W"]& & [-0.5cm]   &[-0.5cm] & \bullet\arrow[dl, bend right=20, "W"] \arrow[dr, bend left=20, "Z"]& & &  &[-0.5cm] & [-0.5cm]   & &  \\
			\circ & & \bullet\arrow[ll, bend left=40, "W"] \arrow[rrrr, bend right=50, "UW"]& & \bullet \arrow[llll, bend left =50, "Z"] \arrow[rr, bend right =40, "UZ"]& & \circ & \bullet\arrow[u,"Z"] \arrow[d,"W"]& \circ & \cdots& \circ & \bullet\arrow[u,"Z"] \arrow[d,"W"]& \circ & & \bullet\arrow[ll, bend right =40, "UW"] \arrow[rrrr, bend left =50, "W"] & & \bullet \arrow[llll, bend right =50, "UZ"]  \arrow[rr, bend left =40, "Z"]& & \circ\\
			& &  & &  & & &\bullet \arrow[ul, bend left=20,"Z"]\arrow[ur, bend right=20, "Z"]& &  & & \bullet \arrow[ul, bend left=20,"Z"]\arrow[ur, bend right=20, "Z"]  & & &   & [-0.5cm]\bullet \arrow[ul, "Z"] \arrow[ur, "W"] &[-0.5cm]   & &  
		\end{tikzcd}
	}
\caption{A model of $\cCFK(W^+(T_{2,3},n))$ when $n\ge 3$}
\label{fig:CFK(Wh_n(RHT))}
\end{figure}
\end{enumerate}

\subsection{Mazur patterns of the right hand trefoil}

	In this example, we compute the Mazur pattern of the right hand trefoil for some framings. A $DA$-bimodule structure of $L_P$ is drawn in Figure \ref{fig:Mazur}. The process is exactly the same as before except $N=\frac{3}{2}$ instead of $N=1$ in this case, so we are interested in the rows $t= -1,$ $0$ and $1$ respectively. 
	
	The region
	\[\begin{tikzcd}[labels = description]
		F_{-\frac{3}{2},-1} \arrow[d,"\Phi^K"]& E_{-\frac{1}{2},-1} \arrow[l,"\Phi^{-\mu}"]\arrow[r,"\Phi^{\mu}"]\arrow[d,"\Phi^K"]& 
		F_{-\frac{1}{2},-1} \arrow[d,"\Phi^K"] &
		E_{\frac{1}{2},-1} \arrow[l,"\Phi^{-\mu}"]\arrow[d,"\Phi^K"] \\
		M_{-\frac{3}{2},-1} & J_{-\frac{1}{2},-1}\arrow[l,"\Phi^{-\mu}"]\arrow[r,"\Phi^{\mu}"] & 
		M_{-\frac{1}{2},-1} &
		J_{\frac{1}{2},-1} \arrow[l,"\Phi^{-\mu}"]
	\end{tikzcd}\]
of the truncated complex, after similar simplification as before in the case of Whitehead double, takes the form shown in Figure \ref{fig: top two rows of the Mazur pattern}.

\begin{figure}[h]
	\adjustbox{scale=0.8}{
	 \begin{tikzcd}[labels =description, column sep = .8cm, row sep = 0.6cm]
	 	\phantom{\circ} & & &[-0.5cm] \bullet \arrow[dl, "Z"]  \arrow[dr,"W"] & [-0.5cm]& & &\\
	 	\circ \arrow[dd,"U"]& & \bullet \arrow[ll,bend left=40, pos= 0.4,"W"] \arrow[rrrr, bend right=40, pos = 0.6, "W^2"]& & \bullet \arrow[llll, bend left=40, pos=0.6,"Z"]  \arrow[rr, bend right =40, pos = 0.4, "U"]& & \circ \arrow[dd,"U"] & & \bullet \arrow[ll, bend left = 40, "1"]\arrow[d,"W"]\\[1cm]
	 	& & & \star \arrow[dlll,bend right, "Z"] \arrow[drrr,bend left, "U"] \arrow[from =ur, bend left =30,  pos = 0.6, "U"]& & &  & & \star  \arrow[dll,bend right, "Z"] \\
	 	\ast & & & \star \arrow[u,"W"]\arrow[d,"Z"]\arrow[from = uuu, bend right=24,  pos =0.65,"U"]& & & \ast & & \star\arrow[u,"W"]\arrow[d,"Z"] \\
	 	& & & \star \arrow[ulll, bend left, "W"] \arrow[urrr, bend right, "W^2"] \arrow[from = uuul, bend right, pos = 0.6, "U"]& & &  & & \star\arrow[ull, bend left, "W"]
	 \end{tikzcd}
  }
\caption{The region consisting of $E_{s,-1},F_{s-1,-1},J_{s,-1}, M_{s-1,-1} $ with $s= -\frac{1}{2},\frac{1}{2}$ for the Mazur pattern of the right hand trefoil }
\label{fig: top two rows of the Mazur pattern}
\end{figure}
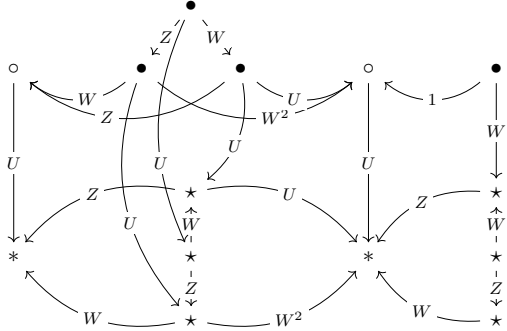

The below region, after simplification, takes the form as in Figure \ref{fig: second and third rows of the Mazur pattern}. \[\begin{tikzcd}[labels={description}]
	M_{-\frac{3}{2}+n,-1} & 
	J_{-\frac{1}{2}+n,-1} \arrow[l,"\Phi^{-\mu}"]\arrow[r,"\Phi^{\mu}"]&
	M_{-\frac{1}{2}+n,-1} & 
	J_{\frac{1}{2}+n,-1} \arrow[l,"\Phi^{-\mu}"]\arrow[r,"\Phi^{\mu}"] &
	M_{\frac{1}{2}+n,-1} \\
	F_{-\frac{3}{2},0} \arrow[u,"\Phi^{-K}"] & 
	E_{-\frac{1}{2},0} \arrow[l,"\Phi^{-\mu}"]\arrow[r,"\Phi^{\mu}"] \arrow[u,"\Phi^{-K}"]&
	F_{-\frac{1}{2},0} \arrow[u,"\Phi^{-K}"]& 
	E_{\frac{1}{2},0} \arrow[l,"\Phi^{-\mu}"]\arrow[r,"\Phi^{\mu}"] \arrow[u,"\Phi^{-K}"]&
	F_{\frac{1}{2},0}  \arrow[u,"\Phi^{-K}"]
\end{tikzcd},\]

\begin{figure}[h]
	\vspace*{10mm}
	\adjustbox{scale=0.8}{
	 \begin{tikzcd}[labels =description, column sep = 1.5cm, row sep = 0.6cm]
	 	& \star & & \star & \\
	 	\ast \arrow[from=ur, bend right=20, "Z"]
	 	\arrow[from=dr, bend left=20, "W"]& \star \arrow[u,"W"]\arrow[d,"Z"] &\ast  \arrow[from=ur, bend right=20, "Z"]
	 	\arrow[from=dr, bend left=20, "W"] 
	 	\arrow[from=ul, bend left=20, "U"]
	 	\arrow[from=dl, bend right=20, "W^2"]& \star\arrow[u,"W"]\arrow[d,"Z"] &\ast   \arrow[from=ul, bend left=20, "U"]
	 	\arrow[from=dl, bend right=20, "W^2"]\\
	 	& \star & & \star & \\[0.5cm]
	 	& \bullet \arrow[uuu, bend left, crossing over, pos = 0.7, "Z"] 
	 	\arrow[dr, bend left=20, "Z"]& & \bullet \arrow[uuu, bend left, crossing over, pos = 0.7, "U"] \arrow[dr, bend left=20, "U"]& \\
	 	\circ \arrow[from=ur, bend right=20, "Z^2"]
	 	\arrow[from=dr, bend left=20, "U"] \arrow[uuu, "1"]& \bullet \arrow[u,"W"]\arrow[d,"Z"] 
	 	\arrow[uuu, bend right, crossing over, "Z"] & \circ \arrow[from=ur, bend right=20, "Z"]
	 	\arrow[from=dr, bend left=20, "W"] 
	 	\arrow[from=dl, bend right=20, "W"] \arrow[uuu, "U"]& \bullet \arrow[u,"W"]\arrow[d,"Z"] \arrow[uuu, bend right, crossing over, "U"]& \circ   
	 	\arrow[from=dl, bend right=20, "W^2"]
	 	\arrow[uuu,"U"]\\
	 	& \bullet\arrow[uuu, bend left, crossing over, pos = 0.3,"Z"] & & \bullet \arrow[uuu, bend left, crossing over, pos = 0.3, "U"]& \\
	 \end{tikzcd}
  }
\caption{The region of interest $E_{s,0},F_{s-1,0},J_{s+n,-1}, M_{s+n-1,-1} $ for the Mazur pattern of the right hand trefoil. Note that $J$ and $M$ are independent of the $s$-coordinate.  }
\label{fig: second and third rows of the Mazur pattern}
\end{figure}
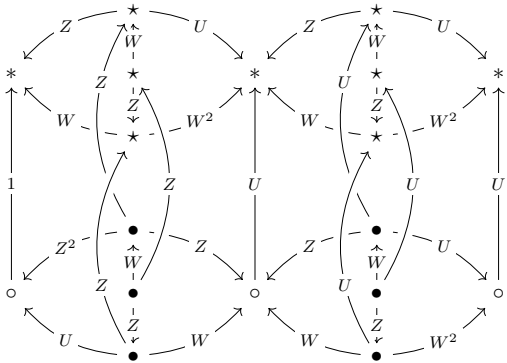

The information needed to build the lower half of the truncated complex can be obtained by the symmetry of rotation $180^{\circ}$ about the plane, and then switching $W$ with $Z$. We show some final results after some simplifications. Let $\Maz(T_{2,3},n)$ denote the result of the satellite operation of the Mazur pattern on the $n$-frame right hand trefoil.

\begin{enumerate}
	\item When $n=-1$, $\cCFK(\Maz(T_{2,3},-1))$ is homotopy equivalent to the complex in Figure \ref{fig:CFK(Maz_{-1}(RHT))}. We used the further simplification as in Figure \ref{fig:simplified truncation when n=-1}.
	\begin{figure}[htb!]
		\adjustbox{scale=0.8}{
			\begin{tikzcd}[labels =description, column sep = 1cm]
		    \phantom{\bullet} &[-0.5cm]\phantom{\bullet}
		    &[-0.5cm]\phantom{\bullet} &\phantom{\circ} &\phantom{\bullet} 
		    &[-0.5cm]\phantom{\bullet}
		    &[-0.5cm]\phantom{\bullet}
		    &\phantom{\circ} &\phantom{\bullet} 
		    &[-0.5cm]\bullet\arrow[dl, "Z"]\arrow[dr, "W"]
		    &[-0.5cm]\phantom{\bullet}
		    &\phantom{\circ} &\phantom{\bullet} 
		    &[-0.5cm]\phantom{\bullet}
		    &[-0.5cm]\phantom{\bullet}\\
		    \phantom{\bullet} &[-0.5cm]\phantom{\bullet}
		    &[-0.5cm]\phantom{\bullet} &\phantom{\circ} &\phantom{\bullet} 
		    &[-0.5cm]\phantom{\bullet}
		    &[-0.5cm]\phantom{\bullet}
		    &\circ &\bullet \arrow[l, bend left =20, "W"] 
		    &[-0.5cm]
		    &[-0.5cm]\bullet \arrow[lll, bend left =30, pos =0.7, "Z"]
		    &\phantom{\circ} &\phantom{\bullet} 
		    &[-0.5cm]\phantom{\bullet}
		    &[-0.5cm]\phantom{\bullet}\\
		    \phantom{\bullet} &[-0.5cm]\phantom{\bullet}
		    &[-0.5cm]\phantom{\bullet} &\phantom{\circ} &\phantom{\bullet} 
		    &[-0.5cm]\phantom{\bullet}
		    &[-0.5cm]\phantom{\bullet}
		    & &  
		    &[-0.5cm] \star \arrow[dl,"Z"]\arrow[dr,"W"]  \arrow[from =uu, "U"]
		    &[-0.5cm]
		    &\phantom{\circ} &\phantom{\bullet} 
		    &[-0.5cm]\star \arrow[dl, pos =0.4,"Z"]\arrow[dr,pos = 0.4,"W"]
		    &[-0.5cm]\phantom{\bullet}\\
		    \phantom{\bullet} &[-0.5cm]\phantom{\bullet}
		    &[-0.5cm]\phantom{\bullet} &\phantom{\circ} &\star \arrow[rrr, bend left =30, pos =0.7, "W^2"]
		    &[-0.5cm]\phantom{\bullet}
		    &\star \arrow[r, bend left =20, "U"]
		    &\ast \arrow[from = uu, "U"] &  \star\arrow[from =uu, "U"]\arrow[l, bend left =20, "W"]\arrow[rrr, bend right =30, "W^2" pos = 0.7,]
		    &[-0.5cm] 
		    &[-0.5cm]\star \arrow[from = uu, "U"] \arrow[lll, bend left =30, pos =0.7, "Z"] \arrow[r, bend right =20, "U"]
		    &\ast &\star \arrow[l, bend left =20, "W"]
		    &[-0.5cm]
		    &\star \arrow[lll, bend left =30, "Z"] \arrow[from = uullll, bend left =45, "UW", crossing over] \arrow[from = uullllll, bend right=3, "W^3",crossing over]\\
		    \phantom{\bullet} &[-0.5cm]\phantom{\bullet}
		    &[-0.5cm]\phantom{\bullet} &\phantom{\circ} &\phantom{\bullet} 
		    &\star \arrow[ul, "Z"]\arrow[ur, "W"]
		    &[-0.5cm]\phantom{\bullet}
		    &\phantom{\circ} &\phantom{\bullet} 
		    & 
		    &[-0.5cm]\phantom{\bullet}
		    &\phantom{\circ} &\phantom{\bullet} 
		    &[-0.5cm]\phantom{\bullet}
		    &[-0.5cm]\phantom{\bullet}\\[-0.5cm]
		    \phantom{\bullet} &[-0.5cm]\phantom{\bullet}
		    &[-0.5cm]\phantom{\bullet} &\phantom{\circ} &\phantom{\bullet} 
		    &[-0.5cm]\phantom{\bullet}
		    &[-0.5cm]\phantom{\bullet}
		    &\phantom{\circ} &\phantom{\bullet} 
		    &[-0.5cm]\bullet\arrow[dl, "Z"]\arrow[dr, "W"] \arrow[uuu,pos =0.2,"U"]
		    &[-0.5cm]\phantom{\bullet}
		    &\phantom{\circ} &\phantom{\bullet} 
		    &[-0.5cm]\phantom{\bullet}
		    &[-0.5cm]\phantom{\bullet}\\
		    \phantom{\bullet} &[-0.5cm]\phantom{\bullet}
		    &[-0.5cm]\phantom{\bullet} &\phantom{\circ} &\bullet  \arrow[rrr, bend left =30, pos=0.7, "W"] \arrow[uuu, "Z"]
		    &[-0.5cm]\phantom{\bullet}
		    &\bullet \arrow[r, bend left =20, "Z"]\arrow[uuu,"Z"]
		    &\circ \arrow[uuu, "U"]&\bullet \arrow[uuu,"U"] \arrow[l, bend left=20, "W"]
		    &[-0.5cm] 
		    &\bullet \arrow[uuu, "U"]\arrow[lll, bend left =30,pos = 0.7, "Z"]
		    &\phantom{\circ} &\phantom{\bullet} 
		    &[-0.5cm]\phantom{\bullet}
		    &[-0.5cm]\phantom{\bullet}\\
		    \phantom{\bullet} &[-0.5cm]
		    &[-0.5cm]\phantom{\bullet} &\phantom{\circ} &\phantom{\bullet} 
		    &[-0.5cm]\bullet \arrow[uuu, pos = 0.3,"Z"] \arrow[ul, "Z"] \arrow[ur, "W"]
		    &[-0.5cm]\phantom{\bullet}
		    &\phantom{\circ} &\phantom{\bullet} 
		    &[-0.5cm] 
		    &[-0.5cm]\phantom{\bullet}
		    &\phantom{\circ} &\phantom{\bullet} 
		    &[-0.5cm]\phantom{\bullet}
		    &[-0.5cm]\phantom{\bullet}\\[-0.5cm]
		    \phantom{\bullet} &[-0.5cm]\phantom{\bullet}
		    &[-0.5cm]\phantom{\bullet} &\phantom{\circ} &\phantom{\bullet} 
		    &[-0.5cm]\phantom{\bullet}
		    &[-0.5cm]\phantom{\bullet}
		    &\phantom{\circ} &\phantom{\bullet} 
		    &[-0.5cm]\star\arrow[dl, "Z"]\arrow[dr, "W"] \arrow[from = uuu, pos = 0.3, "W"]
		    &[-0.5cm]\phantom{\bullet}
		    &\phantom{\circ} &\phantom{\bullet} 
		    &[-0.5cm]\phantom{\bullet}
		    &[-0.5cm]\phantom{\bullet}\\
		    \star \arrow[rrr, bend left =30, "W"] &[-0.5cm]\phantom{\bullet}
		    &\star\arrow[r, bend left =20, "Z"] &\ast \arrow[from = uuuuuur, bend right, "UW"] \arrow[from = uuuuuurrr, "UZ", crossing over, pos = 0.8, bend left =9]&\star \arrow[l, bend right =20, "U"] \arrow[rrr, bend left =30, pos =0.7,"W"] \arrow[from = uuu, "U"]
		    &[-0.5cm]\phantom{\bullet}
		    &\star \arrow[lll, bend right =30, "Z^2",pos =0.7] \arrow[r, bend left =20, "Z"] \arrow[from = uuu, "U"]
		    &\ast  \arrow[from = uuu, "U"]&\star \arrow[l, bend left =20, "U"] \arrow[from = uuu, "W"] \arrow[uuuuuurrr, "UW", crossing over, pos = 0.8, bend left =8 ]
		    &[-0.5cm] 
		    &\star \arrow[lll, bend left =30, "Z^2"] \arrow[from = uuu, "W"] \arrow[uuuuuur, bend right, "UZ"]
		    &\phantom{\circ} &\phantom{\bullet} 
		    &[-0.5cm]\phantom{\bullet}
		    &[-0.5cm]\phantom{\bullet}\\
		    \phantom{\bullet} &[-0.5cm]\star \arrow[ul, pos = 0.4,"Z"]\arrow[ur,pos = 0.4,"W"]
		    &[-0.5cm]\phantom{\bullet} &\phantom{\circ} &\phantom{\bullet} 
		    &[-0.5cm]\star \arrow[ul, "Z"]\arrow[ur,"W"] \arrow[from =uuu,pos = 0.25,"U"] \arrow[from = dd, pos = 0.4, "U"]
		    &[-0.5cm]\phantom{\bullet}
		    &\phantom{\circ} &\phantom{\bullet} 
		    &[-0.5cm] 
		    &[-0.5cm]\phantom{\bullet}
		    &\phantom{\circ} &\phantom{\bullet} 
		    &[-0.5cm]\phantom{\bullet}
		    &[-0.5cm]\phantom{\bullet}\\
		    \phantom{\bullet} &[-0.5cm]\phantom{\bullet}
		    &[-0.5cm]\phantom{\bullet} &\phantom{\circ} &\bullet \arrow[uu, "U"] \arrow[rrr ,bend left=30, "W", pos =0.75]
		    \arrow[uullll, bend left =45, "UZ", crossing over] 
		    &[-0.5cm]\phantom{\bullet}
		    &[-0.5cm]\bullet \arrow[uu, "U"] \arrow[r, bend right =20, "Z" ] \arrow[uullllll, bend right=3, "Z^3",crossing over]
		    &\circ \arrow[uu,"U"]& 
		    &[-0.5cm] 
		    &[-0.5cm] 
		    &  &\phantom{\bullet} 
		    &[-0.5cm]\phantom{\bullet}
		    &[-0.5cm]\phantom{\bullet}\\
		    \phantom{\bullet} &[-0.5cm]\phantom{\bullet}
		    &[-0.5cm]\phantom{\bullet} &\phantom{\circ} &\phantom{\bullet} 
		    &[-0.5cm]\bullet  \arrow[ul,"Z"]\arrow[ur,"W"]
		    &[-0.5cm]\phantom{\bullet}
		    &\phantom{\circ} & 
		    &[-0.5cm] 
		    &[-0.5cm] 
		    &  &\phantom{\bullet} 
		    &[-0.5cm]\phantom{\bullet}
		    &[-0.5cm]\phantom{\bullet}\\
			\end{tikzcd}
		}
	\adjustbox{scale=0.8}{
\begin{tikzcd}
	 [row sep =  -1mm, column sep =-1mm, arrows={shorten = -1mm}]
&   &[1.5cm]   &   &\bullet \arrow[dddd]    &  &   &   &[1.2cm]\bullet \arrow[llll] \arrow[dddd]   &  &   & \\
&   &   &   &   &\bullet \arrow[dddd]    &  &   &   &\bullet \arrow[llll] \arrow[dddd]   &  & \\
\bullet &   &   &\bullet  \arrow[lll] \arrow[dddddd] &   &    &\bullet  \arrow[dddd]   &  &   &   &\bullet \arrow[llll] \arrow[dddd]   & \\
\bullet\arrow[ddd]&   & \bullet \arrow[ll] \arrow[ddd]   &   &  &    &   &\bullet  \arrow[dddd]   &  &   &   &\bullet \arrow[llll]  \arrow[dddd] \\[1.2cm]
&   &   &   &\bullet   &      &   &   &\bullet \arrow[llll]   &   &   & \\
&   &   &   &   &\bullet    &  &   &   &\bullet \arrow[llll]   &  & \\
\bullet&   &\bullet\arrow[ll]  &   &   &    &\bullet    &    &   &   &\bullet \arrow[llll]   &  \\
&   &   &   &   &    &    &\bullet   &   &   &   &\bullet \arrow[llll] \\
\bullet \arrow[dd]&   & \bullet \arrow[ll] \arrow[dd] &\bullet   &   &    &   &   &   & \bullet \arrow[llllll] \arrow[ddd] &   & \\
&\bullet \arrow[dd] &   &\bullet \arrow[ll] \arrow[dd]  &  &\bullet\arrow[dd]    &   &   & \bullet \arrow[dd] \arrow[lll] &   &   & \\[1.5 cm]
\bullet&   & \bullet \arrow[ll]  &   &   &    &   &   &   &   &   & \\
&\bullet   &  &\bullet \arrow[ll]  &  &\bullet    &   &   &\bullet \arrow[lll]  & \bullet  &   & 
\end{tikzcd}
		}
		\caption{A model of $\cCFK(\Maz (T_{2,3},-1))$.  In the second diagram, each vertical arrow is marked by $Z$ and  each horizontal arrow is marked by $W$.}
		\label{fig:CFK(Maz_{-1}(RHT))}
	\end{figure}	
	\item When $n=0$, $\cCFK(\Maz(T_{2,3},0))$ is homotopy equivalent to the complex in Figure \ref{fig:CFK(Maz_0(RHT))}.
\begin{figure}[htb!]
		\adjustbox{scale=0.8}{
			\begin{tikzcd}[labels =description, column sep = 1cm]
               \phantom{\circ} &\phantom{\bullet}&[-0.5cm]\bullet \arrow[dl,"Z"]\arrow[dr,"W"]\arrow[dd,"U"] & [-0.5cm]\phantom{\bullet} & \phantom{\circ} &\phantom{\bullet} &[-0.5cm] \phantom{\bullet} &[-0.5cm]\phantom{\bullet} & \phantom{\circ}\\
               \circ &\bullet \arrow[l,bend left=20, "W"] \arrow[dd, pos=0.7, "U"]&  & \bullet \arrow[lll, bend left =30, pos=0.7,"Z"] \arrow[dd, pos =0.7, "U"]& \phantom{\circ} &\phantom{\bullet} &[-0.5cm] \phantom{\bullet} &[-0.5cm]\phantom{\bullet} & \phantom{\circ}\\
               \phantom{\circ} &\phantom{\bullet}&[-0.5cm]\star \arrow[dl,"Z"]\arrow[dr,"W"]& [-0.5cm]\phantom{\bullet} & \phantom{\circ} &\phantom{\bullet} &[-0.5cm] \phantom{\bullet} &[-0.5cm]\phantom{\bullet} & \phantom{\circ}\\
               \phantom{\circ} &\star\arrow[rrr, bend right =40, pos=0.75,"W^2"]& & \star \arrow[r, bend right=20, "U"]&\ast &\star \arrow[l, bend right=20, "W"] &[-0.5cm] \phantom{\bullet} &\star \arrow[lll, bend right=40,pos =0.5,"Z"] \arrow[from = uullll, bend left, "UW", crossing over] \arrow[from = uullllll, bend right=10, "W^3",crossing over]& \phantom{\circ}\\
               \phantom{\circ} & & &  &  &  &\star \arrow[ul, "Z"]\arrow[ur,"W"] & & \phantom{\circ}\\[-0.5cm]
               \phantom{\circ} & & \bullet \arrow[uuu,pos =0.3,"Z"] \arrow[dl, "Z"]\arrow[dr, "W"]&  &  &  &  & & \phantom{\circ}\\
               \phantom{\circ} &\bullet \arrow[uuu, "Z"] \arrow[rrr, bend right = 40, pos = 0.75, "W"] &  & \bullet\arrow[uuu, "Z"] \arrow[r, bend right = 20, "Z"]& \circ\arrow[uuu,"U"] & \bullet \arrow[uuu, "U"] \arrow[l, bend right = 20, "W"] &  &\bullet \arrow[uuu, "U"] \arrow[lll, bend right =40, pos = 0.75, "Z"] & \phantom{\circ}\\
               \phantom{\circ} & &  &  &  &  &  \bullet \arrow[uuu, "U", pos = 0.3]\arrow[ul, "Z"]\arrow[ur, "W"]& & \phantom{\circ}\\[-0.5cm]
               \phantom{\circ} & & \star \arrow[from = uuu,"U", pos = 0.3] \arrow[dl, "Z"]\arrow[dr, "W"]&  &  &  &  & & \phantom{\circ}\\
               \phantom{\circ} & \star \arrow[from = uuu, "U"] \arrow[rrr, bend right =40, pos = 0.5, "W"] &  & \star \arrow[from = uuu, "U"] \arrow[r, bend right =20, "Z"]& \ast \arrow[from = uuu, "U"] & \star\arrow[l ,bend right =20, "U"] \arrow[from = uuu, "W"]&  & \star \arrow[lll, bend right =40, pos = 0.75, "Z^2"]\arrow[from= uuu, "W"] & \phantom{\circ}\\
               \phantom{\circ} & &  &  &  &  & \star \arrow[from = uuu, "W", pos = 0.3]  \arrow[ul, "Z"]\arrow[ur, "W"]& & \phantom{\circ}\\
               \phantom{\circ} & &  &  &  & \bullet  \arrow[rrr, bend left =30,pos = 0.7, "W"]\arrow[uu, pos =0.6, "U"] \arrow[uullll, bend left, "UZ", crossing over]&  & \bullet\arrow[r, bend left = 20, "Z"] \arrow[uu, pos=0.6, "U"]
               \arrow[uullllll, bend right=10, "Z^3",crossing over] & \circ \\
                  \phantom{\circ} & &  &  &  &  &\bullet \arrow[ul, "Z"]\arrow[ur, "W"]\arrow[uu, "U"]  &  &  \\
			\end{tikzcd}
		}
	
	\adjustbox{scale=0.8}{
      \begin{tikzcd}
	 [row sep =  -1mm, column sep =-1mm, arrows={shorten = -1mm}]
\bullet &  &[1.2cm]    &\bullet  \arrow[lll] \arrow[dddddd] &\bullet\arrow[dddd]   &    &   &  &[1.6cm]\bullet\arrow[llll] \arrow[dddd]   &   &   & \\
\bullet\arrow[ddd]&   & \bullet\arrow[ll] \arrow[ddd]   &   &  &\bullet  \arrow[dddd]    &   &   &  &\bullet \arrow[llll]  \arrow[dddd]   &   & \\[1.5cm]
&   &   &   &   &    &   &   &   &  &   & \\
&   &   &   &   &    &  &   &   &   &  & \\
\bullet&   &\bullet\arrow[ll]  &   &\bullet   &    &   &    &\bullet \arrow[llll]   &   &   &  \\
&   &   &   &   &\bullet     &   &   &   &\bullet \arrow[llll]   &   & \\
\bullet \arrow[dd]&   & \bullet \arrow[ll] \arrow[dd] &\bullet   &   &    &   &   &   & \bullet \arrow[llllll] \arrow[ddd] &   & \\
 &\bullet \arrow[dd] &   &\bullet \arrow[ll] \arrow[dd]  &  &\bullet\arrow[dd]    &   &   & \bullet\arrow[dd] \arrow[lll] &   &   & \\[1.2 cm]
\bullet&   & \bullet \arrow[ll]  &   &   &    &   &   &   &   &   & \\
&\bullet   &  &\bullet \arrow[ll]  &  &\bullet    &   &   &\bullet \arrow[lll]  & \bullet  &   & 
\end{tikzcd}
}		
\caption{A model of $\cCFK(\Maz (T_{2,3},0))$. In the second diagram, each vertical arrow is marked by $Z$ and  each horizontal arrow is marked by $W$.}
		\label{fig:CFK(Maz_0(RHT))}
	\end{figure}
	\item When $n=1$, $\cCFK(\Maz (T_{2,3},1))$ is homotopy equivalent to the complex in Figure \ref{fig:CFK(Maz_1(RHT))}.
				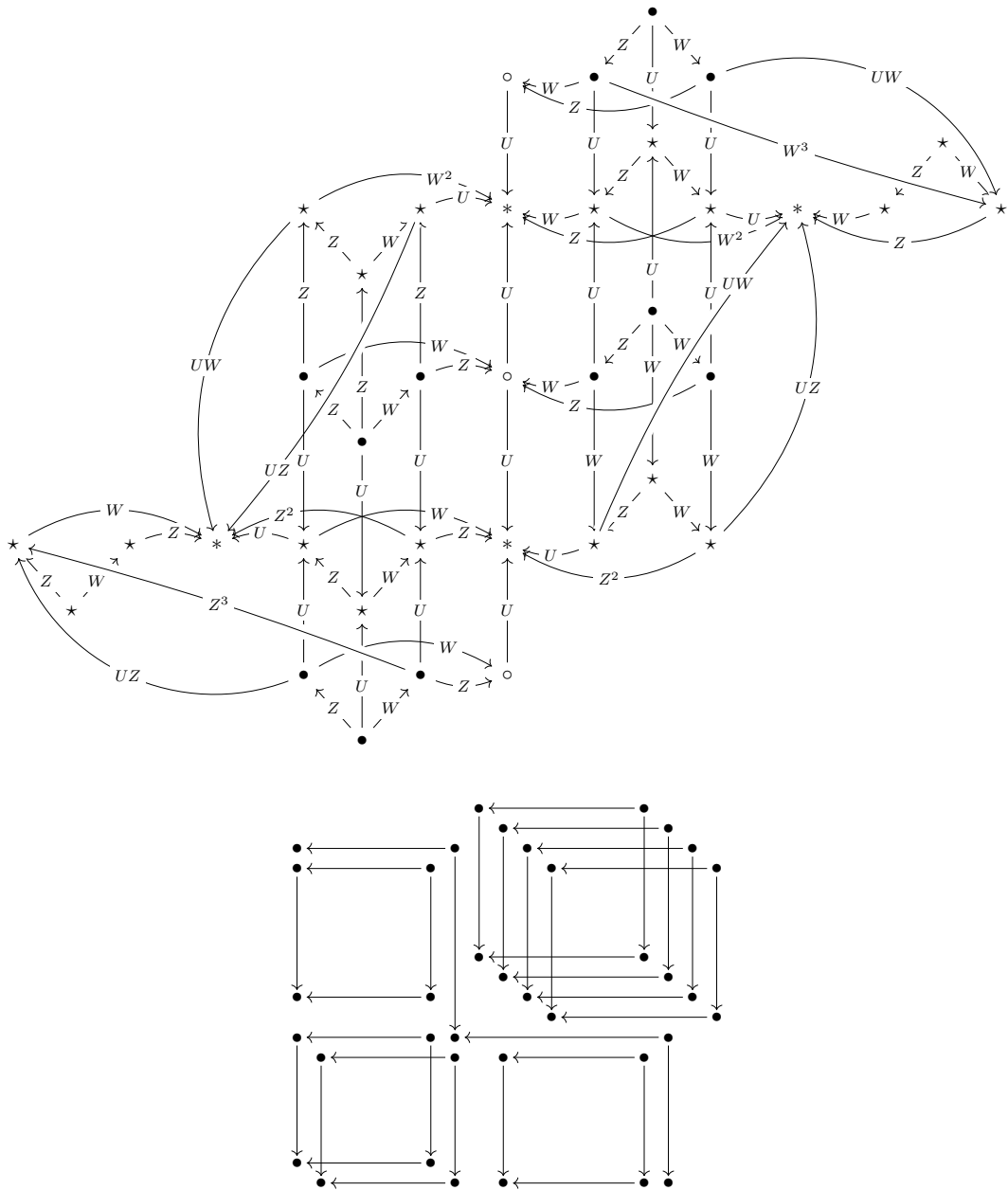
\begin{figure}[htb!]
		\adjustbox{scale=0.8}{
			\begin{tikzcd}[labels =description, column sep = 0.8cm]
               \phantom{\circ}& \phantom{\bullet} & [-0.5cm] \bullet \arrow[dl, "Z"] \arrow[dr, "W"] &[-0.5cm]\phantom{\bullet} & \phantom{\circ} &[-0.5cm]\phantom{\bullet} &[-0.5cm]\phantom{\bullet} &[-0.5cm]\phantom{\bullet} 
               & \phantom{\circ}&\phantom{\bullet} &[-0.5cm]\phantom{\bullet} &[-0.5cm]\phantom{\bullet} & \phantom{\circ}
               &\phantom{\bullet}  &[-0.5cm]\phantom{\bullet}  &[-0.5cm]\phantom{\bullet} &\phantom{\circ} \\               
              \circ& \bullet \arrow[l, bend left=20, "W"]& [-0.5cm]  &[-0.5cm]\bullet \arrow[lll,bend left =50, pos =0.25, "Z"] & \phantom{\circ} &[-0.5cm]\phantom{\bullet} &[-0.5cm]\phantom{\bullet} &\phantom{\bullet} 
               & \phantom{\circ}&[-0.5cm]\phantom{\bullet} &[-0.5cm]\phantom{\bullet} &\phantom{\bullet} & \phantom{\circ}
               &[-0.5cm]\phantom{\bullet}  &[-0.5cm]\phantom{\bullet}  &\phantom{\bullet} &\phantom{\circ} \\               
               \phantom{\circ}& \phantom{\bullet} & [-0.5cm]  &[-0.5cm]\phantom{\bullet} & \phantom{\circ} &[-0.5cm] &[-0.5cm]\star\arrow[dl,"Z"]\arrow[dr,"W"] \arrow[from = dd,pos=0.5,"Z"]&\phantom{\bullet} 
               & \phantom{\circ}&[-0.5cm]\phantom{\bullet} &[-0.5cm]\phantom{\bullet} &\phantom{\bullet} & \phantom{\circ}
               &[-0.5cm]\phantom{\bullet}  &[-0.5cm]\phantom{\bullet}  &\phantom{\bullet} &\phantom{\circ}  \\
               \phantom{\circ}& \phantom{\bullet} & [-0.5cm]  &[-0.5cm]\phantom{\bullet} & \phantom{\circ} &[-0.5cm]\star \arrow[from =dd, "Z",pos =0.4] &[-0.5cm]\phantom{\bullet} &\star  \arrow[from = uullll, bend left=50, "UW",crossing over] \arrow[from = uullllll, bend right =60, "W^3", crossing over]
               \arrow[from =dd, "Z",pos =0.4]
               & \phantom{\circ}&[-0.5cm]\phantom{\bullet} &[-0.5cm]\phantom{\bullet} &\phantom{\bullet} & \phantom{\circ}
               &[-0.5cm]\phantom{\bullet}  &[-0.5cm]\phantom{\bullet}  &\phantom{\bullet} &\phantom{\circ}  \\
               \phantom{\circ}& \phantom{\bullet} & [-0.5cm]  &[-0.5cm]\phantom{\bullet} & \phantom{\circ} &[-0.5cm] &[-0.5cm]\bullet\arrow[dl,"Z"]\arrow[dr,"W"] &\phantom{\bullet} 
               & \phantom{\circ}&[-0.5cm]\phantom{\bullet} &[-0.5cm]\phantom{\bullet} &\phantom{\bullet} & \phantom{\circ}
               &[-0.5cm]\phantom{\bullet}  &[-0.5cm]\phantom{\bullet}  &\phantom{\bullet} &\phantom{\circ}  \\
               \phantom{\circ}& \phantom{\bullet} & [-0.5cm]  &[-0.5cm]\phantom{\bullet} & \phantom{\circ} &[-0.5cm] \bullet  \arrow[rrr,bend right =50, "W"]&  &\bullet \arrow[r, bend right =20, "Z"]
               & \circ&\bullet \arrow[l, bend right=20, "W"] &[-0.5cm]\phantom{\bullet} &\bullet  \arrow[lll, bend right =50, "Z"]& \phantom{\circ}
               &[-0.5cm]\phantom{\bullet}  &[-0.5cm]\phantom{\bullet}  &\phantom{\bullet} &\phantom{\circ}  \\
                \phantom{\circ}& \phantom{\bullet} & [-0.5cm]  &[-0.5cm]\phantom{\bullet} & \phantom{\circ} &[-0.5cm]\phantom{\bullet} &[-0.5cm]\phantom{\bullet} &[-0.5cm]\phantom{\bullet} 
               & \phantom{\circ}&\phantom{\bullet} &\bullet \arrow[ul, "Z"] \arrow[ur, "W"] &[-0.5cm]\phantom{\bullet} & \phantom{\circ}
               &\phantom{\bullet}  &[-0.5cm]\phantom{\bullet}  &[-0.5cm]\phantom{\bullet} &\phantom{\circ} \\
               \phantom{\circ}& \phantom{\bullet} & [-0.5cm]  &[-0.5cm]\phantom{\bullet} & \phantom{\circ} &[-0.5cm]\phantom{\bullet} &[-0.5cm]\phantom{\bullet} &[-0.5cm]\phantom{\bullet} 
               & \phantom{\circ}&\star \arrow[from = uu,pos=0.3, "W"]&  & \star \arrow[from = uu, pos=0.3,"W"]& \phantom{\circ}
               &\phantom{\bullet}  &[-0.5cm]\phantom{\bullet}  &[-0.5cm]\phantom{\bullet} &\phantom{\circ}\\
               \phantom{\circ}& \phantom{\bullet} & [-0.5cm]  &[-0.5cm]\phantom{\bullet} & \phantom{\circ} &[-0.5cm]\phantom{\bullet} &[-0.5cm]\phantom{\bullet} &[-0.5cm]\phantom{\bullet} 
               & \phantom{\circ}& & \star \arrow[from = uu, pos =0.5, "W"] \arrow[ul, "Z"]\arrow[ur, "W"]&  & \phantom{\circ}
               &\phantom{\bullet}  &[-0.5cm]\phantom{\bullet}  &[-0.5cm]\phantom{\bullet} &\phantom{\circ}\\
               \phantom{\circ}& \phantom{\bullet} & [-0.5cm]  &[-0.5cm]\phantom{\bullet} & \phantom{\circ} &[-0.5cm]\phantom{\bullet} &[-0.5cm]\phantom{\bullet} &[-0.5cm]\phantom{\bullet} 
               & \phantom{\circ}& &  &  & \phantom{\circ}
               &\bullet \arrow[rrr, bend left =50,pos=0.25, "W"] \arrow[uullll, crossing over, bend left =50, "UZ"] &[-0.5cm]\phantom{\bullet}  &\bullet \arrow[r,bend left=20, "Z"] \arrow[uullllll, crossing over, bend right =60, "Z^3"]&\circ\\
               \phantom{\circ}& \phantom{\bullet} & [-0.5cm]  &[-0.5cm]\phantom{\bullet} & \phantom{\circ} &[-0.5cm]\phantom{\bullet} &[-0.5cm]\phantom{\bullet} &[-0.5cm]\phantom{\bullet} 
               & \phantom{\circ}& &  &  & \phantom{\circ}
               &  &[-0.5cm]\bullet \arrow[ul, "Z"]\arrow[ur, "W"]  & & \\
         \end{tikzcd}}
\vspace*{-3mm}
	\adjustbox{scale=0.8}{
\begin{tikzcd}
	 [row sep =  -1mm, column sep =-1mm, arrows={shorten = -1mm}]
\bullet\arrow[d] &[1cm] \bullet \arrow[l]\arrow[d]  &     &  &[1cm]   &       &[1cm]  \\[1cm]
\bullet&\bullet\arrow[l]    &     &\bullet\arrow[dd]  &   &\bullet \arrow[ll]\arrow[dd]  & \\
\bullet&\bullet\arrow[l] \arrow[ddd]   &\bullet \arrow[dd]     &  &\bullet \arrow[ll]\arrow[dd]  &   & \\[1cm]
&   &    &\bullet  &   &\bullet \arrow[ll]   & \\
&   &\bullet    &  &\bullet\arrow[ll]   &   & \\
& \bullet  &   &  &\bullet \arrow[lll] \arrow[d] &\bullet\arrow[d]   &\bullet \arrow[l]\arrow[d] \\[1cm]
&   &    &  &\bullet   &\bullet  &\bullet \arrow[l] 
\end{tikzcd}		}
		\caption{A model of $\cCFK(\Maz (T_{2,3},1)).$  In the second diagram, each vertical arrow is marked by $Z$ and  each horizontal arrow is marked by $W$.}
		\label{fig:CFK(Maz_1(RHT))}
	\end{figure}
	\item When $n\ge 2$, $\cCFK(\Maz(T_{2,3},n))$ is homotopy equivalent to the one in Figure \ref{fig:CFK(Maz_n(RHT))}, with $n-2$ many copies of the left square $S_1$ on the left $\cdots$ (including the drawn one), and $n-2$ many copies of the right square $S_2$ on the right $\cdots$. 
	\[ 
	\begin{tikzcd}[labels = description]
		& \bullet\arrow[dl,"Z"]
		\arrow[dr,"U"] & &[1cm] & &\bullet \arrow[dl,"Z^2"]
		\arrow[dr,"Z"]& \\
		\circ& \bullet\arrow[u,"W"]\arrow[d,"Z"]&\circ & &\circ&\bullet\arrow[u,"W"]\arrow[d,"Z"]&\circ \\
		& \bullet \arrow[ul,"W"]\arrow[ur,"W^2"]& &  & &\bullet\arrow[ul,"U"]\arrow[ur,"W"] & \\[-5mm]
		& S_1& & & &S_2&
	\end{tikzcd}
	\]
	\begin{figure}[h]
		\adjustbox{scale=0.8}{
			\begin{tikzcd}[labels =description, column sep = 1cm]
		  \phantom{\bullet}& [-0.5cm]\bullet \arrow[dl,"Z"]\arrow[dr,"W"] &[-0.5cm] \phantom{\bullet}& \phantom{\circ}& [-0.5cm]\bullet \arrow[dl,"Z"]\arrow[dr,"U"]& [-0.5cm]\phantom{\bullet}& [-1cm]\phantom{\cdots}& [-1cm]\phantom{\circ}& [-0.5cm]\bullet \arrow[dl,"Z^2"]\arrow[dr,"Z"] &[-0.5cm] \phantom{\circ}&[-0.5cm]\bullet \arrow[dl,"Z"]\arrow[dr,"U"] &[-0.5cm]\phantom{\circ}&[-1cm]\phantom{\cdots}&[-1cm]\phantom{\circ}&[-0.5cm]\bullet \arrow[dl,"Z^2"]\arrow[dr,"Z"] &[-0.5cm]\phantom{\circ}&\phantom{\bullet} &[-0.5cm]\circ\arrow[from = dl,"W"] \arrow[from = dr,"Z"] & [-0.5cm]\phantom{\bullet}\\ 
		   \bullet\arrow[dr,"W"]&    &\bullet \arrow[dl,"Z"] & \circ \arrow[from =l, bend left =30, "U^2W"] \arrow[from = lll, bend right =40,  pos = 0.7,"UW^3"] \arrow[from =dr,"W"]&  \bullet\arrow[u,"W"] \arrow[d,"Z"]& \circ \arrow[from = dl, "W^2"]& \cdots& \circ \arrow[from=dr, "U"]&  \bullet \arrow[u,"W"] \arrow[d,"Z"] &\circ \arrow[from =dl, "W"] \arrow[from =dr,"W"]& \bullet\arrow[u,"W"] \arrow[d,"Z"] &\circ \arrow[from = dl, "W^2"]&\cdots&\circ\arrow[from = dr, "U"]& \bullet \arrow[u,"W"] \arrow[d,"Z"]&\circ\arrow[from =dl,"W"]&\bullet\arrow[from = dr,"Z"] \arrow[l,bend left =30, "U^2Z"]&  & \bullet\arrow[lll, bend right =40, pos = 0.7, "UZ^3"]\arrow[from=dl,"W"]\\ 
		    \phantom{\bullet}& [-0.5cm]\circ &[-0.5cm] \phantom{\bullet}& \phantom{\circ}& [-0.5cm]\bullet& [-0.5cm]\phantom{\bullet}& \phantom{\cdots}& \phantom{\circ}& [-0.5cm]\bullet &[-0.5cm] \phantom{\circ}&[-0.5cm]\bullet &[-0.5cm]\phantom{\circ}&\phantom{\cdots}&\phantom{\circ}&[-0.5cm]\bullet &[-0.5cm]\phantom{\circ}&\phantom{\bullet} &[-0.5cm]\bullet & [-0.5cm]\phantom{\bullet}\\  
		\end{tikzcd}
		}
		\caption{A model of $\cCFK(\Maz(T_{2,3},n))$ when $n\ge 2$}
		\label{fig:CFK(Maz_n(RHT))}
	\end{figure}
\end{enumerate}

\newpage
\subsection{L-space satellite operations on the right hand trefoil}

	Summarizing what we have done in the previous three examples, here is the general pattern for the L-space satellite operations when the companion knot $K$ is the right hand trefoil. 
	
   The row \[\begin{tikzcd}
   	\cdots&[-1cm] E_{-\frac{3}{2},t} \arrow[r,"
   	\Phi^{\mu}"] & F_{-\frac{3}{2},t} & E_{-\frac{1}{2},t} \arrow[l,"\Phi^{-\mu}"]\arrow[r,"
   	\Phi^{\mu}"] & F_{-\frac{1}{2},t} & E_{\frac{1}{2},t} \arrow[l,"\Phi^{-\mu}"]\arrow[r,"
   	\Phi^{\mu}"] & F_{\frac{1}{2},t} & E_{\frac{3}{2},t} \arrow[l,"\Phi^{-\mu}"]&[-1cm] \cdots 
   \end{tikzcd}, \] which combines the tensor product of $\cCFK(T_{2,3})$ with $f^{\mu}:\scE_{*,t} \to \scF_{*,t}$ and the tensor product of $\cCFK(T_{2,3})$ with $f^{-\mu}:\scE_{*,t} \to \scF_{*,t}$ , takes the following form as in Figure \ref{fig:E_F_rows_RHT}.
   \begin{figure}[h]
   	\adjustbox{scale=0.7}{
   		\begin{tikzcd}
   			[labels=description, column sep=3mm, shorten=-1mm]
   		\cdots&[-2mm]	& [-4mm]\cC_{t-\frac{1}{2}} \arrow[dl, "U"]\arrow[dr, "1"]\arrow[rrr, bend left=40, "L_{\sigma}"]  &[-4mm] & & [-2mm]\cS\arrow[dl, "U"] \arrow[dr,"1"] &[-2mm] & &[-4mm] \cC_{t-\frac{1}{2}}\arrow[dl,"L_Z"]\arrow[dr, "1"] \arrow[rrr, bend left=40, "L_{\sigma}"]  \arrow[lll ,bend right =40, "L_{\tau}"] \arrow[dllll, crossing over, pos = 0.4,"h_{\tau, Z}"]  \arrow[drr, crossing over, "h_{\sigma,Z}"]&[-4mm]& &[-2mm] \cS \arrow[dl, "1"]\arrow[dr,"1"] &[-2mm] & &[-4mm] \cC_{t+\frac{1}{2}} \arrow[dl,"1"] \arrow[dr, "L_W"]\arrow[rrr, bend left=40, "L_{\sigma}"] \arrow[lll ,bend right =40, "L_{\tau}"] \arrow[dll, crossing over, "h_{\tau, W}"]&[-4mm]   & & [-2mm]\cS\arrow[dl, "1"]\arrow[dr,"U"] & [-2mm] & &[-4mm]\cC_{t+\frac{1}{2}}\arrow[dl, "1"] \arrow[dr,"U"] \arrow[lll ,bend right =40, "L_{\tau}"] &[-4mm] &[-2mm] \cdots\\
   			\cdots&\cC_{t-\frac{1}{2}} \arrow[rrr, bend right=40, pos = 0.4, "L_{\sigma}"]  & &\cC_{t-\frac{1}{2}} \arrow[rrr, bend right=50, pos = 0.4, "L_{\sigma}"] & \cS & & \cS & \cC_{t+\frac{1}{2}} \arrow[rrr, bend right=40, pos = 0.4, "L_{\sigma}"] \arrow[lll, bend left =50, "L_{\tau}"]& & \cC_{t-\frac{1}{2}}  \arrow[rrr, bend right=50, pos = 0.4, "L_{\sigma}"] \arrow[lll, bend left =40, "L_{\tau}"] & \cS & & \cS & \cC_{t+\frac{1}{2}} \arrow[rrr, bend right=40, pos = 0.4, "L_{\sigma}"] \arrow[lll, bend left =50, "L_{\tau}"]& & \cC_{t-\frac{1}{2}}  \arrow[rrr, bend right=50, pos = 0.4, "L_{\sigma}"] \arrow[lll, bend left =40, "L_{\tau}"]& \cS & & \cS \arrow[from =ullll, crossing over, "h_{\sigma,W}",pos= 0.4]&\cC_{t+\frac{1}{2}} \arrow[lll, bend left =50, "L_{\tau}"] & &\cC_{t+\frac{1}{2}} \arrow[lll, bend left =40, "L_{\tau}"] &\cdots
   		\end{tikzcd}
   	}
   \caption{The $E_{*,t}, F_{*,t}$ rows when the companion knot $K$ is the right hand trefoil}
   \label{fig:E_F_rows_RHT}
   \end{figure}
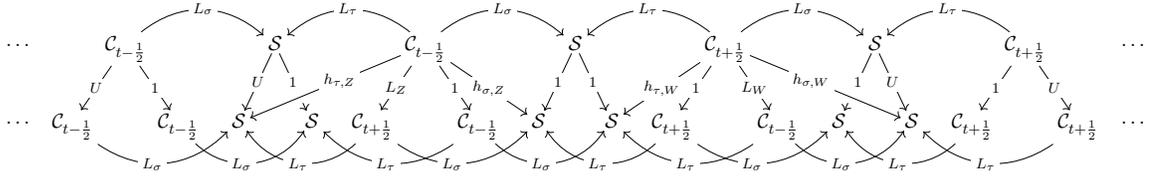

 Then, we can remove arrows $\cC_{t-\frac{1}{2}}\xrightarrow{1}\cC_{t-\frac{1}{2}}$ and $\cC_{t+\frac{1}{2}}\xrightarrow{1}\cC_{t+\frac{1}{2}}$, as they are contractible quotient complex of the row (and hence quotient complex of the whole $2-$dimensional grid as well). After that, we can remove arrows $\cS \xrightarrow{1}\cS$, as they are now null homotopic quotient complex of the new row. (For the middle $F_{\frac{1}{2},t}$, we do a change of basis by adding the two copies of $\cS$ together first.) 
 
 After the simplification, the $E_{*,t}$ and $F_{*,t}$ rows can be replaced by  
 \[\begin{tikzcd}[labels={description},column sep=8mm]
 	\cdots&[-8mm]\cC_{t-\frac{1}{2}} \arrow[r,"
 	L_{\sigma}"] & \cS & \cC_{t+\frac{1}{2}} \arrow[l,"L_{\tau}"]\arrow[r,"
 	L_{\sigma}"] & \cS & \cC_{t- \frac{1}{2}} \arrow[l,"L_{\tau}"]\arrow[r,"
 	L_{\sigma}"] & \cS & \cC_{t+\frac{1}{2}} \arrow[l,"L_{\tau}"]  & [-8mm] \cdots 
 \end{tikzcd}, \]
which continues as $\cdots\xleftarrow{L_{\tau}} \cC_{t-\frac{1}{2}} \xrightarrow{L_\sigma} \cS \cdots$ to the left end, and $\cdots\xleftarrow{L_{\tau}} \cC_{t+\frac{1}{2}} \xrightarrow{L_\sigma} \cS\cdots$ to the right end.

The $J_{*,t}$ and $M_{*,t}$ rows are by definition:
\[\begin{tikzcd}[labels={description}]
	\cdots&[-1cm] \cC_{t+\frac{1}{2}} \arrow[r,"
	L_{\sigma}"] & \cS & \cC_{t+\frac{1}{2}} \arrow[l,"L_{\tau}"]\arrow[r,"
	L_{\sigma}"] & \cS & \cC_{t+\frac{1}{2}} \arrow[l,"L_{\tau}"]\arrow[r,"
	L_{\sigma}"] & \cS & \cC_{t+\frac{1}{2}} \arrow[l,"L_{\tau}"]&[-1cm] \cdots 
\end{tikzcd}. \] 

After this simplification, the maps $\Phi^{K}$ from the $E_{*,t},F_{*,t}$ rows to the $J_{*,t},M_{*,t}$ rows are given by the vertical maps:
\[
\begin{tikzcd}[labels={description}]
	\cdots&[-1cm] \cC_{t-\frac{1}{2}} \arrow[r,"L_{\sigma}"] \arrow[d, "UL_Z"] \arrow[dr, dashed, "Uh_{\sigma, Z}"] & \cS  \arrow[d, pos = 0.6,"U"]& \cC_{t+\frac{1}{2}} \arrow[l,"L_{\tau}"]\arrow[r,"L_{\sigma}"]\arrow[d, "U"] & \cS\arrow[d, pos=0.6,"U"] & \cC_{t- \frac{1}{2}} \arrow[l,"L_{\tau}"]\arrow[r,"L_{\sigma}"] \arrow[d,"L_Z"] \arrow[dl, dashed, "h_{\tau, Z}"] \arrow[dr,dashed, "h_{\sigma,Z}"]& \cS \arrow[d,pos=0.6,"1"]& \cC_{t+\frac{1}{2}} \arrow[l,"L_{\tau}"] \arrow[d,"1"]&[-1cm] \cdots \\
	\cdots&[-1cm] \cC_{t+\frac{1}{2}} \arrow[r,"L_{\sigma}"] & \cS & \cC_{t+\frac{1}{2}} \arrow[l,"L_{\tau}"]\arrow[r,"
	L_{\sigma}"] & \cS & \cC_{t+\frac{1}{2}} \arrow[l,"L_{\tau}"]\arrow[r,"
	L_{\sigma}"] & \cS & \cC_{t+\frac{1}{2}} \arrow[l,"L_{\tau}"]&[-1cm] \cdots 
\end{tikzcd},
\]
where the leftmost vertical arrow is $\Phi^K:E_{-\frac{3}{2},t} \to J_{-\frac{3}{2},t}$, and the rightmost vertical arrow is $\Phi^K:E_{\frac{3}{2},t} \to J_{\frac{3}{2},t}$.

Symmetrically, the maps $\Phi^{-K}$ from the $E_{*,t},F_{*,t}$ rows to the $J_{*,t-1},M_{*,t-1}$ rows are the vertical maps:
\[
\begin{tikzcd}[labels={description}]
	\cdots&[-1cm] \cC_{t-\frac{1}{2}} \arrow[r,"L_{\sigma}"] & \cS   & \cC_{t-\frac{1}{2}} \arrow[l,"L_{\tau}"]\arrow[r,"L_{\sigma}"]\  & \cS & \cC_{t-\frac{1}{2}} \arrow[l,"L_{\tau}"]\arrow[r,"L_{\sigma}"]   & \cS  & \cC_{t-\frac{1}{2}} \arrow[l,"L_{\tau}"]  &[-1cm] \cdots \\
	\cdots&[-1cm] \cC_{t-\frac{1}{2}} \arrow[u,"1"] \arrow[r,"L_{\sigma}"] & \cS \arrow[u,pos = 0.4,"1"] & \cC_{t+\frac{1}{2}}\arrow[u,"L_W"] \arrow[l,"L_{\tau}"]\arrow[r,"
	L_{\sigma}"] \arrow[ul, dashed, "h_{\tau, W}"] \arrow[ur,dashed, "h_{\sigma,W}"]& \cS \arrow[u,pos = 0.4,"U"]& \cC_{t-\frac{1}{2}} \arrow[l,"L_{\tau}"]\arrow[r,"
	L_{\sigma}"]  \arrow[u,"U"]& \cS \arrow[u,pos = 0.4,"U"]& \cC_{t+\frac{1}{2}} \arrow[l,"L_{\tau}"] \arrow[u,"UL_W"] \arrow[ul,dashed, "Uh_{\tau,W}"]&[-1cm] \cdots 
\end{tikzcd},
\]
where the leftmost vertical arrow is $\Phi^K:E_{-\frac{3}{2},t} \to J_{-\frac{3}{2}+n,t-1}$, and the rightmost vertical arrow is $\Phi^K:E_{\frac{3}{2},t} \to J_{\frac{3}{2}+n,t-1}$.

The length $2$ maps in this simplification are the dashed diagonal maps in the above, assuming we choose our $DA$-bimodule structure to be $U$-equivariant, which can always be achieved via the construction in Section~\ref{sec:U-equivariant}. The absolute $(\gr_{\ws},\gr_{\zs})$-grading could be worked out using the formula in Section \ref{sec:grading}.

The final answer of the truncated complex for the L-space satellite of right hand trefoil is then determined by $N,n,\cC_{t_0}, \cS$ and maps between them. For example, when $N=\frac{3}{2}$ and $n=0$, $\cCFK(P(T_{2,3},0))$ is homotopy equivalent to the following complex in Figure \ref{fig:general formula for L-space satellite operation on RHT}.

\begin{figure}[h]
	 \begin{tikzcd}[labels={description}]
	 	\cS & \cC_{-\frac{1}{2}} \arrow[l, "L_{\tau}"] \arrow[r, "L_{\sigma}"] \arrow[d, "U"]& \cS \arrow[d, pos  =0.6,"U"] & \cC_{-\frac{3}{2}}  \arrow[l,"L_{\tau}"] \arrow[d, "L_Z"] \arrow[dl,"h_{\tau,Z}"]& \\
	 	& \cC_{-\frac{1}{2}} \arrow[r, "L_{\sigma}"]& \cS & \cC_{-\frac{1}{2}}  \arrow[l, "L_{\tau}"]& \\
	 	& \cC_{\frac{1}{2}} \arrow[u,"L_W"] \arrow[r,"L_{\sigma}"] \arrow[d,"U"] \arrow[ur, "h_{\sigma,W}"]& \cS \arrow[u,pos =0.4,"U"]\arrow[d,pos = 0.6,"U"] & \cC_{-\frac{1}{2}} \arrow[u,"U"] \arrow[l,"L_{\tau}"] \arrow[d,"L_Z"] \arrow[dl, "h_{\tau,Z}"]& \\
	 	& \cC_{\frac{1}{2}} \arrow[r,"L_{\sigma}"] & \cS & \cC_{\frac{1}{2}} \arrow[l,"L_{\tau}"]&	 \\
	 	& \cC_{\frac{3}{2}} \arrow[u,"L_{W}"] \arrow[r,"L_{\sigma}"] \arrow[ur, "h_{\sigma,W}"]& \cS \arrow[u, pos =0.4, "U"] & \cC_{\frac{1}{2}} \arrow[l,"L_{\tau}"] \arrow[u,"U"] \arrow[r,"L_{\sigma}"] & \cS.
	 \end{tikzcd}
 \caption{A formula for $\cCFK(P(T_{2,3},0))$ when $N= \frac{3}{2}$.}
 \label{fig:general formula for L-space satellite operation on RHT}
\end{figure}

We perform another example and compute the $(3,2)$-cable of the right hand trefoil. The $DA$-bimodule structure of $L_P$ is drawn in Figure \ref{fig:(3,2)_torus}. We plug  the corresponding information into the above diagram shown in Figure~\ref{fig:general formula for L-space satellite operation on RHT}. One gets that the knot Floer complex of the $(3,2)$-cable of the right hand trefoil is homotopy equivalent to the one in Figure \ref{fig:(3,2) cable of RHT}. A further simplification gives the diagram in Figure \ref{fig:simplified C_{3,2}(T_{2,3})}.
\begin{figure}[h]
	\adjustbox{scale=0.8}{
	\begin{tikzcd}[labels={description}, shorten=-1mm]
		 \phantom{\circ}& [-5mm]\circ\arrow[dl, "Z"] \arrow[dr,"W"]& [-5mm] \phantom{\circ} & [-5mm] \phantom{\bullet}&[-5mm] \phantom{\bullet}&[-5mm]\phantom{\bullet} & \phantom{\circ} & [-4mm]\circ \arrow[dl, "Z"] \arrow[dr,"W"] \arrow[dd, pos = 0.8,"U"]& [-4mm] \phantom{\circ} &\phantom{\bullet}&[-5mm] \bullet \arrow[lll, bend right =20, "1"] \arrow[dl, "Z"] \arrow[dr,"W"]&[-5mm]\phantom{\bullet} &[-5mm]\phantom{\circ}& [-5mm]\phantom{\circ}& [-5mm] \phantom{\circ} \\[-1mm]
		 \circ & \phantom{\circ}&\circ& \phantom{\bullet}&\bullet \arrow[ll, "1"]\arrow[rr,"W"] \arrow[dd, "U"]& \phantom{\bullet} & \circ \arrow[dd,"U"]&  \phantom{\circ}& \circ \arrow[dd,"U"] &\bullet \arrow[lll, bend left=20, pos = 0.55, "1"] \arrow[ddr,bend right =20,"W^2"] \arrow[dll, pos = 0.8,"W"]&  \phantom{\bullet}&\bullet  \arrow[lll, bend left=20,pos = 0.25, "1"] \arrow[ddl,bend left =20, "U"] &\phantom{\circ}&  \phantom{\circ}&   \phantom{\circ} \\	
		 \phantom{\circ} & \phantom{\circ}&\phantom{\circ}& \phantom{\bullet}&\phantom{\bullet}  & \phantom{\bullet} & \phantom{\circ} &  \ast\arrow[dl, "Z"]\arrow[dr, "W"]& \phantom{\circ} &\phantom{\bullet}&  \phantom{\bullet}&\phantom{\bullet} &\phantom{\circ}&  \phantom{\circ}&   \phantom{\circ} \\
		 \phantom{\circ} & \phantom{\circ}&\phantom{\circ}& \phantom{\bullet}&\star\arrow[rr, "W"] & \phantom{\bullet} & \ast &  \phantom{\circ}& \ast &\phantom{\bullet}&  \star \arrow[ll, "1"]&\phantom{\bullet} &\phantom{\circ}&  \phantom{\circ}&   \phantom{\circ} \\[5mm]
		 \phantom{\circ} & \phantom{\circ}&\phantom{\circ}& \phantom{\bullet}&\bullet\arrow[rr, "1"] \arrow[u,"Z"] & \phantom{\bullet} & \circ \arrow[u,"U"]&  \circ \arrow[l,"Z"] \arrow[r,"W"] \arrow[uu,"U", pos = 0.4]& \circ \arrow[u,"U"] &\phantom{\bullet}&  \bullet \arrow[ll, "1"] \arrow[u,"U"]&\phantom{\bullet} &\phantom{\circ}&  \phantom{\circ}&   \phantom{\circ} \\[5mm]
		 \phantom{\circ} & \phantom{\circ}&\phantom{\circ}& \phantom{\bullet}&\star\arrow[rr, "1"] \arrow[from = u,"U"] & \phantom{\bullet} & \ast \arrow[from = u,"U"]&  \phantom{\ast}  & \ast \arrow[from = u,"U"] &\phantom{\bullet}&  \star \arrow[ll, "Z"] \arrow[from = u,"W"]&\phantom{\bullet} &\phantom{\circ}&  \phantom{\circ}&   \phantom{\circ} \\
		 \phantom{\circ} & \phantom{\circ}&\phantom{\circ}& \phantom{\bullet}&\phantom{\star}  & \phantom{\bullet} & \phantom{\ast} &  \ast \arrow[ul, "Z"] \arrow[ur,"W"] \arrow[from = uu, "U", pos = 0.4]  & \phantom{\ast}  &\phantom{\bullet}&  \phantom{\star} &\phantom{\bullet} &\phantom{\circ}&  \phantom{\circ}&   \phantom{\circ} \\
		 \phantom{\circ} & \phantom{\circ}&\phantom{\circ}& \bullet \arrow[uur, bend left =20, "U"] \arrow[rrr, bend left=25, pos = 0.25, "1"]&\phantom{\bullet}  & \bullet \arrow[uul, bend right =20, "Z^2"]  \arrow[rrr, bend left=20, pos = 0.55, "1"] \arrow[urr, pos = 0.8,"Z"]& \circ \arrow[uu,"U"]&  \phantom{\ast}  & \circ \arrow[uu,"U"] &\phantom{\bullet}&  \bullet\arrow[uu,"U"] \arrow[rr,"1"] \arrow[ll, "Z"]&\phantom{\bullet} &\circ&  \phantom{\circ}&   \circ \\
		  \phantom{\circ} & \phantom{\circ}&\phantom{\circ}& \phantom{\bullet}&\bullet \arrow[ul,"Z"] \arrow[ur,"W"] \arrow[rrr, bend right =20, "1"]  & \phantom{\bullet} & \phantom{\ast} &  \circ\arrow[ul, "Z"] \arrow[ur,"W"] \arrow[from = uu, "U", pos = 0.4]  & \phantom{\ast}  &\phantom{\bullet}&  \phantom{\star} &\phantom{\bullet} &\phantom{\circ}&  \circ \arrow[ul, "Z"] \arrow[ur, "W"]&   \phantom{\circ} \\		 
	\end{tikzcd}
}
\caption{A model of $\cCFK(C_{3,2}(T_{2,3})).$}
\label{fig:(3,2) cable of RHT}
\end{figure}

\begin{figure}[h]
	\adjustbox{scale=0.7}{
\begin{tikzcd}[row sep =  -1mm, column sep =-1mm, arrows={shorten = -1mm},labels={description}]
	\circ &[1.5cm]\phantom{ } & \circ \arrow[ll,"W"] \arrow[ddd,"Z^3"] \arrow[dll,"UZ"] & [1.5cm] \phantom{ } & [3.5cm] \phantom{ }\\[3.5cm]
	\star \arrow[dd,"Z"]& & &\circ \arrow[lll,"W^2",pos =0.6] \arrow[dll,"U"]\arrow[ddl,"U",pos = 0.6] \arrow[ddd,"Z^2",pos = 0.6]& \\[1.5cm]
	& \ast \arrow[dd,"Z"]& & & \circ\arrow[lll,"W^3"] \arrow[ddl,"UW"] \arrow[dd,"Z"] \\
	\ast & & \ast\arrow[ll,"W"] & &\\[1.5cm]
	& \ast & & \star\arrow[ll,"W"] & \circ
\end{tikzcd}
}
\caption{A simplified model of $\cCFK(C_{3,2}(T_{2,3}))$}
\label{fig:simplified C_{3,2}(T_{2,3})}
\end{figure}

\subsection{L-space satellite operations on L-space knots}

\label{exm:L-space satellite for L-space companion}
The observation in the previous example can be applied in the more general situation when the companion knot $K$ is itself an L-space knot, i.e., when $\cCFK(K)$ is a staircase. 

Recall from Section \ref{sec:definition of E,F,J,M} that $E_{s,t}$ is obtained by taking the first Alexander grading-$s$ part of the tensor product $\cCFK(K)\boxtimes\scE_{*,t}$. In words, it means a generator $\xs \in \cCFK(K)$ is replaced by the staircase $\cC_{t-\frac{1}{2}}$ if $A(x) \ge s+\frac{1}{2}$, and replaced by the staircase $\cC_{t+\frac{1}{2}}$ if $A(x) < s+ \frac{1}{2}$, where $A(\xs)$ is the Alexander grading of $\xs$. The maps in $E_{s,t}$ are such that if there is a differential \[\ys\xrightarrow{W^iZ^j} \xs\] in $\cCFK(K)$, then it is replaced by a map between staircases by the following rule: 
\begin{enumerate}
	\item If $A(\ys), A(\xs) \ge s+\frac{1}{2}$, it is replaced by $\cC_{t- \frac{1}{2}} \xrightarrow{U^j} \cC_{t-\frac{1}{2}}$.
	\item If $A(\ys),A(\xs)< s+\frac{1}{2}$, it is replaced by $\cC_{t+\frac{1}{2}} \xrightarrow{U^i} \cC_{t+\frac{1}{2}}$.
	\item If $A(\ys)<s+\frac{1}{2} \le A(\xs)$, it is replaced by $\cC_{t+ \frac{1}{2}} \xrightarrow{U^{s-1/2-A(\ys)}L_W} \cC_{t-\frac{1}{2}}$.
	\item If $A(\xs)<s+\frac{1}{2} \le A(\ys)$, it is replaced by $\cC_{t- \frac{1}{2}} \xrightarrow{U^{A(\ys)-s-1/2 }L_Z} \cC_{t+\frac{1}{2}}$.
\end{enumerate}
Also, for a general knot $K$, the differential on $E_{s,t}$ will contain some additional terms which correspond to the $\delta_3^1$ differentials on idempotent 0 of ${}_{\cK} \cX(L_P)^R$. If $\cCFK(K)$ is a staircase, then there are no composable arrows in the differential of $\cCFK(K)$, and therefore these $\delta_3^1$ will make no contribution to the differential on $E_{s,t}$.

 We now use the notation from Section~\ref{sec:staircase-complexes} for staircase complexes, i.e., we say that $\cCFK(K)$ has an $R$-basis $\xs_0,\ys_1,\xs_2,\dots, ,\ys_{2n-1},\xs_{2n}$ with respect to which the differential takes the following form:
\[
\d(\xs_{2i})=0\quad \text{and} \quad \d(\ys_{2i+1})=\xs_{2i}\otimes W^{\a_{i}}+\xs_{2i+2}\otimes Z^{\b_{i}}
\]
for some $\a_{i},\b_{i}>0$. Then, we can replace $E_{s,t}$ by a single staircase as in the previous example. More explicitly:
\begin{enumerate}
	\item If $s<A(\xs_{2n})=-g(K)$, then all the horizontal maps in the original staircase (weighted by $W^{\a_i}$) become identity maps between $\cC_{t- \frac{1}{2}}$ in $E_{s,t}$, so $E_{s,t}$ strongly deformation retracts to the staircase $\xs_{2n}\boxtimes\cC_{t- \frac{1}{2}}$.
	\item If $s>A(\xs_{0})=g(K)$, then all the vertical maps in the original staircase (weighted by $Z^{\b_i}$) become identity maps between $\cC_{t+\frac{1}{2}}$ in $E_{s,t}$, so  $E_{s,t}$ strongly deformation retracts to the staircase $\xs_0\boxtimes\cC_{t+ \frac{1}{2}}$.
	\item If $A(\ys_{2i+1})<s<A(\xs_{2i})$ for some $i$, (note that $s \in \mathbb{Z}+\frac{1}{2}$, so it will never equal to $A(\xs_{2i})$ or $A(\ys_{2i+1})$), then all the horizontal maps $ \ys_{2j+1} \to \xs_{2j}$ become identity map between $\cC_{t-\frac{1}{2}}$ for $j<i$, and all the vertical maps $ \ys_{2j+1} \to \xs_{2j+2}$ becomes identity map between $\cC_{t+\frac{1}{2}}$ for $j\ge i$. Therefore, $E_{s,t}$ strongly deformation retracts to the staircase $\xs_{2i}\boxtimes\cC_{t- \frac{1}{2}}$.
	\item If $A(\xs_{2i+2})<s<A(\ys_{2i+1})$ for some $i$, the situation is symmetric to the previous case under flipping about the slope $ 45^{\circ}$ line passing through the generator $\ve{x}_n$, 
	and $E_{s,t}$ strongly deformation retracts to the staircase $\xs_{2i+2}\boxtimes\cC_{t+ \frac{1}{2}}$.
\end{enumerate}
In Figure \ref{fig:simplification of E and F for staircase input}, we have illustrated the situation when $\cCFK(K)$ is given by the one in Figure~\ref{fig:cable6}.

The situation for $F_{s,t}$ is similar and simpler. Each generator $\xs$ in $\cCFK(K)$ is replaced by the staircase $\cS$. Each map is replaced by a multiplication by certain multiple of $U$ from $\cS$ to itself. The specific multiplicity $U$ is determined by a formula as before. Also, there is no non-trivial $\delta_3^1$ on $\scF_{*,t}$, so no extra differentials in $F_{s,t}$ corresponding to the $\delta_3^1$-action. Again, see Figure \ref{fig:simplification of E and F for staircase input} for an illustration.
Therefore, each $F_{s,t}$ strongly deformation retracts to a copy of the staircase $\cS$, which corresponds to $\xs_{i}\boxtimes \cS$ for different $i$'s depending on the value of $s$, determined the same as in the case of $E_{s,t}$.

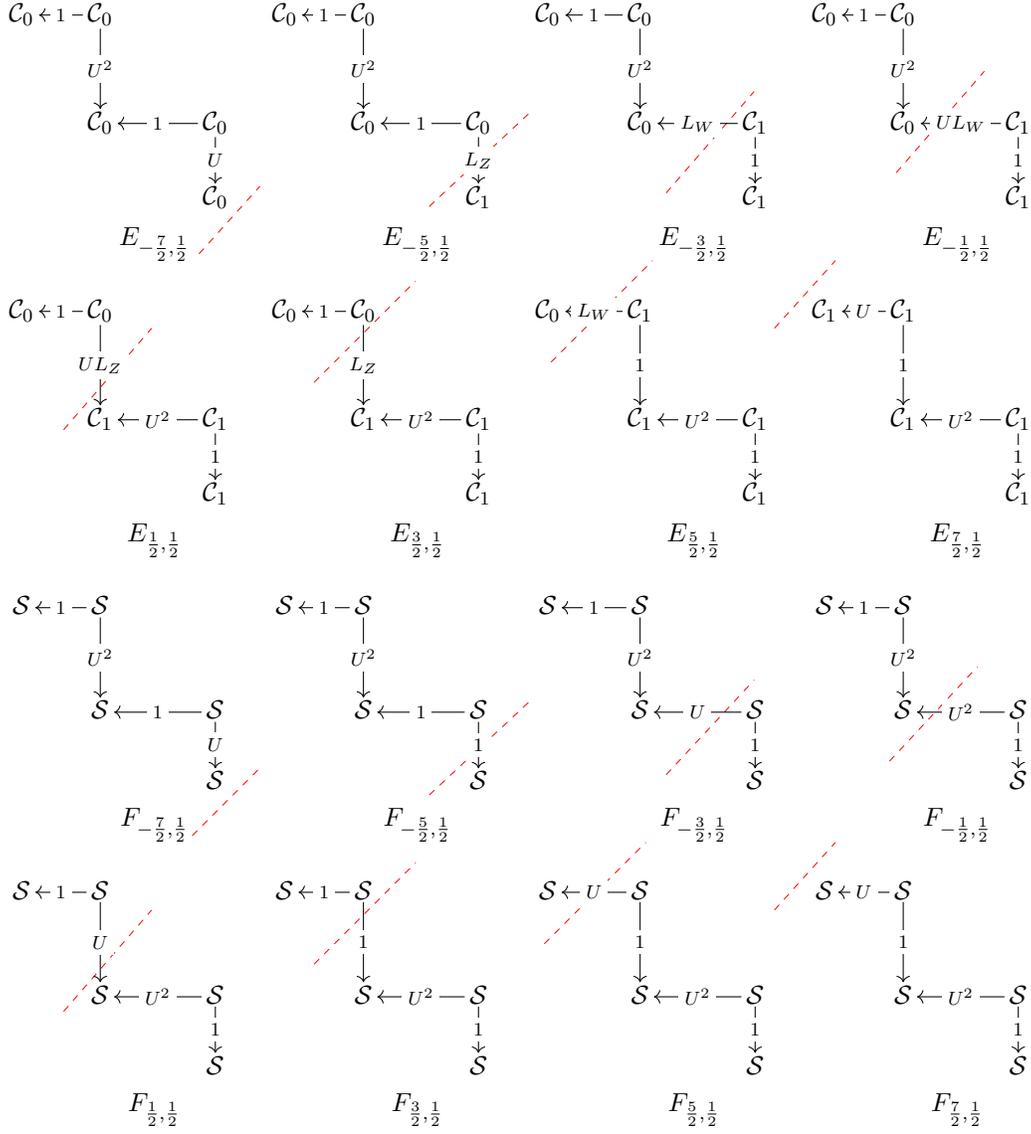
\begin{figure}[h]
		\adjustbox{scale=0.9}{
			\tikz[
			overlay]{
				\draw[dashed,red] (3.1,4.6)--(4,5.6);
				\draw[dashed, red]
				(3,-4)--(4,-3);
				\draw[dashed,red](6.5,5.3 )--(8,6.7);
				\draw[dashed,red](6.5,-3.4 )--(8,-2);
				\draw[dashed,red](10,5.5)--(11.3,7);
				\draw[dashed,red](10,-3.1)--(11.3,-1.7);
				\draw[dashed,red](13.4,5.8)--(14.7,7.3);
				\draw[dashed,red](13.3,-2.9)--(14.6,-1.5);
				\draw[dashed,red](1.1,2)--(2.4,3.5);
				\draw[dashed,red](1.1,-6.6)--(2.4,-5.1);
				\draw[dashed,red](4.8,2.7)--(6.3,4.2);
				\draw[dashed,red](4.8,-5.9)--(6.3,-4.4);
				\draw[dashed,red](8.3,3)--(9.8,4.5);
				\draw[dashed,red](8.2,-5.6)--(9.7,-4.1);
				\draw[dashed,red](11.6,3.5)--(12.5,4.5);
				\draw[dashed,red](11.6,-5.1)--(12.5,-4.1);
			}
	\begin{tikzcd}[labels={description}, row sep=5mm, column sep=small,shorten=-1mm]
   \cC_0& \cC_0 \arrow[l,"1"] \arrow[dd,"U^2"] &[-7mm] &[-6mm] \phantom{\cC}&[-3mm] &[-3mm]   \cC_0& \cC_0 \arrow[l,"1"] \arrow[dd,"U^2"]& [-7mm] &[-6mm]   \phantom{\cC}&[-3mm] &[-3mm]   \cC_0 &[2mm] \cC_0 \arrow[l,"1"] \arrow[dd,"U^2"]&[-7mm] &[-6mm] \phantom{\cC}& [-3mm] &[-3mm]\cC_0& \cC_0 \arrow[l,"1"] \arrow[dd,"U^2"]&[-7mm] &[-6mm] & \phantom{\cC}\\
   &  & & & &    &  & &  & &  &  & &   & &    & & &\\
   & \cC_0 & & \cC_0 \arrow[ll,"1"] & &    & \cC_0 & & \cC_0 \arrow[ll,"1"]  & &  &\cC_0  & & \cC_1 \arrow[ll, "L_W"]  & &    &\cC_0 &  &\cC_1 \arrow[ll,"UL_{W}"] \\
    & & & \cC_0 \arrow[from = u,"U"] & &    &   & & \cC_1\arrow[from =u,"L_Z"]  & &  &\  & & \cC_1 \arrow[from = u, "1"]  & &    &  &  &\cC_1 \arrow[from =u,"1"] \\[-5mm]
    &  &E_{-\frac{7}{2},\frac{1}{2}} & & &    &  &E_{-\frac{5}{2},\frac{1}{2}} &  & &  &  &E_{-\frac{3}{2},\frac{1}{2}} &   & &    & &E_{-\frac{1}{2},\frac{1}{2}} &\\[-2mm]
    \cC_0& \cC_0 \arrow[l,"1"] \arrow[dd,"UL_Z"] &[-7mm] &[-6mm] \phantom{\cC}&[-3mm] &[-3mm]   \cC_0& \cC_0 \arrow[l,"1"] \arrow[dd,"L_Z"]& [-7mm] &[-6mm]   \phantom{\cC}&[-3mm] &[-3mm]   \cC_0 & \cC_1 \arrow[l,"L_W"] \arrow[dd,"1"]&[-7mm] &[-6mm] \phantom{\cC}& [-3mm] &[-3mm]\cC_1& \cC_1 \arrow[l,"U"] \arrow[dd,"1"]&[-7mm] &[-6mm] & \phantom{\cC}\\
    &  & & & &    &  & &  & &  &  & &   & &    & & &\\
    & \cC_1 & & \cC_1 \arrow[ll,"U^2"] & &    & \cC_1 & & \cC_1 \arrow[ll,"U^2"]  & &  &\cC_1  & & \cC_1 \arrow[ll, "U^2"]  & &    &\cC_1 &  &\cC_1 \arrow[ll,"U^2"] \\
    & & & \cC_1 \arrow[from = u,"1"] & &    &   & & \cC_1\arrow[from =u,"1"]  & &  &\  & & \cC_1 \arrow[from = u, "1"]  & &    &  &  &\cC_1 \arrow[from =u,"1"] \\[-5mm]
    &  &E_{\frac{1}{2},\frac{1}{2}} & & &    &  &E_{\frac{3}{2},\frac{1}{2}} &  & &  &  &E_{\frac{5}{2},\frac{1}{2}} &   & &    & &E_{\frac{7}{2},\frac{1}{2}} &\\[-2mm]
    \cS& \cS \arrow[l,"1"] \arrow[dd,"U^2"] &[-7mm] &[-6mm] \phantom{\cC}&[-3mm] &[-3mm]   \cS& \cS \arrow[l,"1"] \arrow[dd,"U^2"]& [-7mm] &[-6mm]   \phantom{\cC}&[-3mm] &[-3mm]   \cS &[2mm] \cS \arrow[l,"1"] \arrow[dd,"U^2"]&[-7mm] &[-6mm] \phantom{\cC}& [-3mm] &[-3mm]\cS& \cS \arrow[l,"1"] \arrow[dd,"U^2"]&[-7mm] &[-6mm] & \phantom{\cC}\\
    &  & & & &    &  & &  & &  &  & &   & &    & & &\\
    & \cS & & \cS \arrow[ll,"1"] & &    & \cS & & \cS \arrow[ll,"1"]  & &  &\cS & & \cS \arrow[ll, "U"]  & &    &\cS &  &\cS \arrow[ll,"U^2"] \\
    & & & \cS \arrow[from = u,"U"] & &    &   & & \cS\arrow[from =u,"1"]  & &  &\  & & \cS \arrow[from = u, "1"]  & &    &  &  &\cS \arrow[from =u,"1"] \\[-5mm]
    &  &F_{-\frac{7}{2},\frac{1}{2}} & & &    &  &F_{-\frac{5}{2},\frac{1}{2}} &  & &  &  &F_{-\frac{3}{2},\frac{1}{2}} &   & &    & &F_{-\frac{1}{2},\frac{1}{2}} &\\[-2mm]
    \cS& \cS \arrow[l,"1"] \arrow[dd,"U"] &[-7mm] &[-6mm] \phantom{\cC}&[-3mm] &[-3mm]   \cS& \cS \arrow[l,"1"] \arrow[dd,"1"]& [-7mm] &[-6mm]   \phantom{\cC}&[-3mm] &[-3mm]   \cS &[2mm] \cS \arrow[l,"U"] \arrow[dd,"1"]&[-7mm] &[-6mm] \phantom{\cC}& [-3mm] &[-3mm]\cS& \cS \arrow[l,"U"] \arrow[dd,"1"]&[-7mm] &[-6mm] & \phantom{\cC}\\
    &  & & & &    &  & &  & &  &  & &   & &    & & &\\
    & \cS & & \cS \arrow[ll,"U^2"] & &    & \cS & & \cS \arrow[ll,"U^2"]  & &  &\cS & & \cS \arrow[ll, "U^2"]  & &    &\cS &  &\cS \arrow[ll,"U^2"] \\
    & & & \cS \arrow[from = u,"1"] & &    &   & & \cS\arrow[from =u,"1"]  & &  &\  & & \cS \arrow[from = u, "1"]  & &    &  &  &\cS \arrow[from =u,"1"] \\[-5mm]
    &  &F_{\frac{1}{2},\frac{1}{2}} & & &    &  &F_{\frac{3}{2},\frac{1}{2}} &  & &  &  &F_{\frac{5}{2},\frac{1}{2}} &   & &    & &F_{\frac{7}{2},\frac{1}{2}} & 
	\end{tikzcd}
}
\caption{Examples of $E_{s,t}$ and $F_{s,t}$ when the input companion knot is an L-space knot and $t= \frac{1}{2}$. The red dashed line indicating where the Alexander grading equals $s$.}
\label{fig:simplification of E and F for staircase input}
\end{figure}

Applying the homological perturbation lemma as in Lemma \ref{lem:HPL-chain-complexes}, the row consisting of $E_{*,t}$ and  $F_{*,t}$ is homotopy equivalent to one such that each $F_{s,t}$ is replaced by a copy of $\cS$, and each $E_{s,t}$ is replaced by either
\begin{enumerate}
	\item a copy of $\cC_{t-\frac{1}{2}}$, if $s<-g(K)$, or $A(\ys_{2i+1})<s<A(\xs_{2i})$ for some $i$; or
	\item a copy of $\cC_{t+ \frac{1}{2}}$, if $s>g(K)$, or $A(\xs_{2i+2})<s<A(\ys_{2i+1})$ for some $i$.
\end{enumerate}

The maps $\Phi^{\mu}$ are replaced by $L_{\sigma}$, the maps $\Phi^{-\mu}$ are replaced by $L_{\tau}$. The maps $\Phi^{K}$ and $\Phi^{-K}$ could be worked out similarly, and are either $U^j\bI, U^jL_Z$ or $U^jL_W$ for some power of $U$. Length-$2$ maps $ \Phi^{\pm K,\pm \mu}$ are similarly given by $U^jh_{\sigma/\tau,W/Z} $ for some powers of $U$, which compensate for the non-commutativity of the square.

More explicitly, the map $\Phi^{K}:E_{s,t} \to J_{s,t} $ in the simplified model has the following description. Recall that $H_K$ is the $H$-function of the companion knot $K$ as in Definition \ref{def:H function}. 
 \begin{enumerate}
 	\item  If $s<-g(K)$, or $A(\ys_{2i+1})<s<A(\xs_{2i})$ for some $i$, then $\Phi^{K}:E_{s,t} \to J_{s,t} $ is given by	\[ U^{H_K(s+\frac{1}{2})}L_{Z}:\cC_{t-\frac{1}{2}} \to \cC_{t+\frac{1}{2}}.\]
 	\item If $s>g(K)$, or $A(\xs_{2i+2})<s<A(\ys_{2i+1})$ for some $i$, then $\Phi^{K}:E_{s,t} \to J_{s,t} $ is given by\[U^{H_K(s+ \frac{1}{2})}\bI :\cC_{t+\frac{1}{2}} \to \cC_{t+\frac{1}{2}}.\]
 \end{enumerate}

The map $\Phi^K:F_{s,t} \to M_{s,t}$ in the simplified model is given by \[U^{H_K(s+\frac{1}{2})}\bI:\cS \to \cS .\]

The maps $\Phi^{-K}: E_{s,t} \to J_{s+n,t-1}$ and $\Phi^{-K}:F_{s,t} \to M_{s+n,t-1}$ have a similar description, which could be obtained using the symmetry. More explicitly, the map $\Phi^{-K}:E_{s,t}\to J_{s+n,t-1}$ in the simplified model has the description:

\begin{enumerate}
	\item If If $s<-g(K)$, or $A(\ys_{2i+1})<s<A(\xs_{2i})$ for some $i$, then $\Phi^{-K}: E_{s,t} \to J_{s+n,t-1}$ is given by	\[ U^{s-\frac{1}{2}+H_K(s-\frac{1}{2})}\bI:\cC_{t-\frac{1}{2}} \to \cC_{t-\frac{1}{2}}.\]
	\item If $s>g(K)$, or $A(\xs_{2i+2})<s<A(\ys_{2i+1})$ for some $i$, then $\Phi^{-K}: E_{s,t} \to J_{s+n,t-1}$ is given by \[ U^{s-\frac{1}{2}+H_K(s-\frac{1}{2})}L_{W} :\cC_{t+\frac{1}{2}} \to \cC_{t-\frac{1}{2}}.\]
\end{enumerate}

The map $\Phi^K:F_{s,t} \to M_{s,t}$ in the simplified model is given by \[U^{s+\frac{1}{2}+H_K(s+\frac{1}{2})}\bI:\cS \to \cS.\]

Therefore, when the input companion knot $K$ is an L-space knot, the $2$-dimensional grid 
\[ \bX(P,K,n)^{\bF[W,Z]} =\cX_n(Y,K)^{\cK}\boxtimes {}_{\cK} \cH_-^{\cK} \boxtimes {}_{\cK} \cX(L_P)^{\bF[W,Z]} \]
is homotopy equivalent to another one, which consists of a copy of $\cC_{t_0}$ or $\cS$ at each grid point, and the maps between them are given by maps between these staircases, without further tensoring. Then, one can apply the truncation process described in Proposition~\ref{prop:truncation} to get a finitely generated model of $\cCFK(P(K,n))$.

\subsection{L-space satellite operations on the figure-eight knot}

	\label{exm:L-space satellite for square complex}
	The type-$D$ module for the figure-eight knot splits as a direct sum of the type-$D$ module for an unknot, as well as a one-by-one box:
	\[
	\begin{tikzcd}[labels=description] \bullet \ar[d, "\sigma+\tau"] \\ \bullet
	\end{tikzcd}
	\begin{tikzcd}[labels=description]
	\bullet \ar[d, "Z"] &\bullet\ar[d, "Z"]  \ar[l, "W"]\\
	\bullet & \bullet \ar[l, "W"] 
	\end{tikzcd}.
	\]
	Tensoring with the unknot summand gives the same result as applying this satellite operation to the unknotted companion, which gives $\cCFK(P(U,n))$ for the $n$-twisted pattern. See Section~\ref{exm:unknot companion} for the result.

	Therefore it suffices to compute the effect of the satellite operation on the box complex. Note that since the one-by-one box is supported in idempotent 0, the $J_{s,t}$ and $M_{s,t}$ complexes will vanish. In particular, each row of $E_{*,t}$ and $F_{*,t}$ complexes forms a summand of $\bX(P,K,n)$.
	 The rows consisting of $E_{*,t}$ and $F_{*,t}$, are shown in Figure~\ref{fig:E and F rows for box input}.   Here, we illustrate the case that $t=\frac{1}{2}$, though the row for all other $t$ take a similar form.
	 \begin{figure}[h]
		\adjustbox{scale=0.9}{
		\begin{tikzcd}[labels=description,column sep=9mm,shorten =-1mm]
		 \cdots& [-9mm]\cS \arrow[d,"U" ]&\cS\arrow[l,"1"] \arrow[d,"U"]& \cC_0 \arrow[d,"U"] \arrow[ll, bend right=30, "L_{\tau}"] \arrow[rr, bend left =30,"L_{\sigma}"]& \cC_0 \arrow[l,"1"] \arrow[d,"L_Z"]\arrow[dl, "h_{W,Z}"] \arrow[ll, bend right=30, "L_{\tau}"] \arrow[rr, bend left =30,"L_{\sigma}"] \arrow[dll,"h_{\tau,Z}",pos= 0.75, bend right =15] \arrow[drr,"h_{\sigma,Z}",pos = 0.2, bend left =15]&\cS \arrow[d,"U"]& \cS\arrow[d,"1"] \arrow[l,"1"]& \cC_0 \arrow[d,"L_Z"] \arrow[ll, bend right=30, "L_{\tau}"] \arrow[rr, bend left =30,"L_{\sigma}"] \arrow[dll, pos=0.2, "h_{\tau,Z}", bend right =15] \arrow[drr, pos = 0.75, "h_{\sigma, Z}", bend left =13] &\cC_1\arrow[l,"L_W"] \arrow[dl,"h_{Z,W}"] \arrow[ll, bend right=30, "L_{\tau}"] \arrow[rr, bend left =30,"L_{\sigma}"] \arrow[d,"1"] \arrow[r, bend right =20, "h_{\sigma,W}"] \arrow[lll, bend right =40, "h_{\tau,W}"]&\cS \arrow[d,"1"]&\cS \arrow[l,"U"] \arrow[d,"1"]&[-9mm]\cdots\\
		 \cdots& \cS &\cS\arrow[l,"1"]& \cC_0 \arrow[rr, bend right =30,"L_{\sigma}"] \arrow[ll, bend left=30, "L_{\tau}"]& \cC_1 \arrow[l,"L_W"] \arrow[rr, bend right =30,"L_{\sigma}"] \arrow[ll, bend left=30, "L_{\tau}"] \arrow[r, bend left=20, "h_{\sigma,W}"] \arrow[lll, bend left =40, "h_{\tau,W}"] &\cS & \cS \arrow[l,"U"]& \cC_1 \arrow[rr, bend right =30,"L_{\sigma}"] \arrow[ll, bend left=30, "L_{\tau}"] &\cC_1\arrow[l,"U"] \arrow[rr, bend right =30,"L_{\sigma}"] \arrow[ll, bend left=30, "L_{\tau}"]&\cS&\cS \arrow[l,"U"]&[-9mm]\cdots 
	 \end{tikzcd}
		} 
	\caption{The $E_{*,\frac{1}{2}}$, $F_{*,\frac{1}{2}}$ rows when the input is a square complex.}
	\label{fig:E and F rows for box input}
	\end{figure}
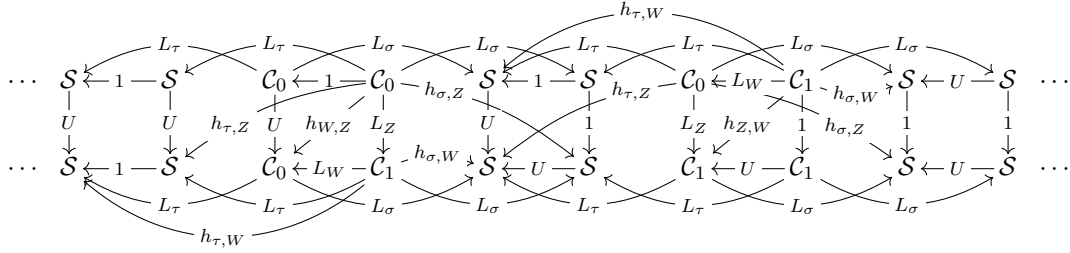

After canceling contractible sub and quotient complexes, the row consisting of $E_{*,t}$ and $F_{*,t}$ is homotopy equivalent to the complex $\frC_t$ shown in Figure~\ref{fig:simplified complex for square input}.
\begin{figure}[h]
\[
\frC_t=\begin{tikzcd}[labels={description},shorten=-1mm]
	 \cC_{t+\frac{1}{2}} \arrow[r, "L_{\sigma}"] \arrow[d, "L_{W}"] \arrow[dr, "h_{\sigma, W}"]& \cS \arrow[d,"U"]& \cC_{t-\frac{1}{2}} \arrow[l,"L_{\tau}"] \arrow[d,"L_{Z}"] \arrow[dl,"h_{\tau,Z}"]\\
	 \cC_{t-\frac{1}{2}} \arrow[r,"L_{\sigma}"] & \cS & \cC_{t+\frac{1}{2}} \arrow[l, "L_{\tau}"] 
\end{tikzcd}.
\]
\caption{The complex $\frC_t$, which is a simplification of the $E_{*,t}$, $F_{*,t}$ rows when the input is a box complex.}
\label{fig:simplified complex for square input}
\end{figure}

We observe that the complex in Figure~\ref{fig:simplified complex for square input} is contractible when  $t\ge N+\frac{1}{2}$. Indeed, when $t\ge N+\frac{1}{2}$, the maps 
\[
L_Z\to \cC_{t-\frac{1}{2}}\to \cC_{t+\frac{1}{2}} \quad \text{and} \quad L_\sigma\colon \cC_{t\pm \frac{1}{2}}\to \cS
\]
 are homotopy equivalences by definition of $N$. A symmetric argument shows that the complex in Figure~\ref{fig:simplified complex for square input} in contractible if $t\le -N-\frac{1}{2}$.

Therefore, the final result when the companion knot is the figure-eight knot $4_1$ is homotopy equivalent to
\[\cCFK(P(4_1,n))\simeq \cCFK(P(U,n)) \oplus \bigoplus_{-N<t<N}\mathfrak{C}_t,\]
where $ P(U,n)$ is the $n$-twisted pattern ($n$-framed satellite operation with pattern $P$ on the unknotted companion), and $\mathfrak{C}_t$ is the complex as in Figure \ref{fig:simplified complex for square input}.

One can also use the grading shift formula in Section in \ref{sec:grading} to compute the absolute $(\gr_{\ws},\gr_{\zs})$ grading of each $\mathfrak{C}_t$. Let \[ \xi =\frac{\left(1-(2t-l)^2\right)n}{4}+1,\] which is the expression $d$ in Section \ref{sec:grading} when setting $s = -\frac{1}{2}$, then the graded version of $\mathfrak{C}_t$ takes the following form, where $\left[a,b\right]$ denotes an upward shift of $a$ in $\gr_{\ws}$-grading and upward shift of $b$ in $\gr_{\zs}$-grading.  
\begin{figure}[h]
	\adjustbox{scale=0.8}{
	\begin{tikzcd}[labels={description},shorten=-1mm]
		\cC_{t+\frac{1}{2}} \left[\xi-1,\xi+2t+l-2tnl\right] \arrow[r, "L_{\sigma}"] \arrow[d, "L_{W}"] \arrow[dr, "h_{\sigma, W}"]& \cS\left[\xi-2,\xi-2-2tnl\right] \arrow[d,"U"]& \cC_{t-\frac{1}{2}} \left[\xi-2t+l,\xi-1-2tnl\right] \arrow[l,"L_{\tau}"] \arrow[d,"L_{Z}"] \arrow[dl,"h_{\tau,Z}"]\\
		\cC_{t-\frac{1}{2}}\left[\xi,\xi-1+2t+l-2tnl\right] \arrow[r,"L_{\sigma}"] & \cS\left[\xi-1,\xi-1-2tnl\right] & \cC_{t+\frac{1}{2}}\left[\xi-1-2t+l,\xi-2tnl\right] \arrow[l, "L_{\tau}"] 
	\end{tikzcd}
}
\end{figure}

\begin{rem}
	As the knot Floer complex of a homologically thin knot $K$ with $\tau(K)\ge 0$ is a direct sum of box complexes with a staircase, the result in Section \ref{exm:L-space satellite for L-space companion} and \ref{exm:L-space satellite for square complex} allows an efficient computation for L-space satellite for such knot. Unfortunately, when the input companion knot is the mirror of an L-space knot, we can not perform a simplification as in Section \ref{exm:L-space satellite for L-space companion} to get a much simpler expression.
\end{rem}
\bibliographystyle{custom}
\def\MR#1{}
\bibliography{biblio}

\end{document}